\documentclass[a4paper,11pt]{article}
\usepackage{subfiles}
\usepackage{template}

\usepackage{tikz}
\usetikzlibrary{arrows.meta}

\usepackage{natbib}
\bibpunct{(}{)}{;}{a}{,}{,}

\definecolor{modelcolor}{RGB}{215,236,255}
\definecolor{predcolor}{RGB}{255,238,204}
\definecolor{rigorcolor}{RGB}{225,245,225}
\definecolor{validcolor}{RGB}{245,220,255}

\begin{document}
\title{\textbf{High-Dimensional Analysis of Gradient Flow for Extensive-Width Quadratic Neural Networks}}
\date{}

\author[1,2]{Simon Martin\thanks{To whom correspondence should be sent: \texttt{smartin@di.ens.fr}.}}
\author[1]{Giulio Biroli}
\author[2]{Francis Bach}

\affil[1]{Laboratoire de Physique de l’\'Ecole Normale Sup\'erieure, ENS, Universit\'e PSL, CNRS, Sorbonne Universit\'e, Universit\'e de Paris, F-75005 Paris, France.}
\affil[2]{INRIA - \'Ecole Normale Sup\'erieure, PSL Research University, Paris, France.}

\maketitle

\begin{abstract}
We study the high-dimensional training dynamics of a shallow neural network with quadratic activation in a teacher--student setup. We focus on the extensive-width regime, where the teacher and student network widths scale proportionally with the input dimension, and the sample size grows quadratically. This scaling aims to describe overparameterized neural networks in which feature learning still plays a central role. In the high-dimensional limit, we derive a dynamical characterization of the gradient flow, in the spirit of dynamical mean-field theory (DMFT). Under $\ell_2$-regularization, we analyze these equations at long times and characterize the performance and spectral properties of the resulting estimator. This result provides a quantitative understanding of the effect of overparameterization on learning and generalization, and reveals a double descent phenomenon in the presence of label noise, where generalization improves beyond interpolation. In the small regularization limit, we obtain an exact expression for the perfect recovery threshold as a function of the network widths, providing a precise characterization of how overparameterization influences recovery.
\end{abstract}

\tableofcontents

\newpage

\section{Introduction} \label{Sec:Intro}
Deep neural networks have achieved remarkable success across many domains such as image and speech recognition \citep{krizhevsky2012imagenet, hinton2012deep}, protein structure prediction \citep{jumper2021highly}, natural language processing \citep{vaswani2017attention, brown2020language} and autonomous systems \citep{bojarski2016end}. On the theoretical side, while the expressive power of neural architectures has been clarified \citep{cybenko1989approximation, hornik1989multilayer}, many fundamental challenges related to their loss landscapes and training dynamics remain open. The main obstacles to developing a thorough and general analysis are the high-dimensionality and the non-convexity of the loss landscape.  

\paragraph{Overparameterized neural networks.} This gap has motivated a line of work focused on gradient-based training and its success, despite the non-convexity of neural networks objectives. In theory, gradient descent could get stuck in poor local minima, even for relatively simple architectures. Indeed, several works have shown that loss landscapes exhibit spurious local minima in underparameterized networks \citep{christof2023omnipresence} and shallow ReLU models \citep{safran2018spurious, yun2018small}. However, recent works suggest that in some regimes, optimization algorithms often find global minimizers of their training loss. For instance, in the overparameterized regime, \citet{chizat2018global} have shown that two-layer networks converge globally, while \citet{jacot2018neural} and \citet{chizat2019lazy} have described the lazy and neural tangent kernel (NTK) regimes, in which the network behaves nearly linearly during training. In addition, in heavily overparameterized regimes, several works have shown that deep architectures enjoy polynomial-time convergence guarantees \citep{allen2019convergence, kawaguchi2019gradient, du2019gradient}. Most of these investigations focus on the case in which the number of hidden nodes diverges but the dimension of data is kept fixed. 

\paragraph{High-dimensional proportional regime.} In contrast to overparameterized networks, the proportional regime corresponds to a fixed number of hidden units while the dimension and the number of samples diverge proportionally. In this regime, a long line of work has analyzed the high-dimensional learning abilities of simple models such as multi-index and shallow neural networks. More precisely, tools from statistical physics have been applied to predict Bayes-optimal errors, phase transitions, and information-theoretic limits \citep{loureiro2022fluctuations, aubin2019committee, maillard2020phase, maillard2020landscape}. Many of these predictions have since been confirmed through rigorous approaches such as adaptive interpolation \citep{barbier2019adaptive, barbier2019optimal, barbier2020mutual} and approximate message passing \citep{donoho2009message, gerbelot2023graph, cornacchia2023learning, gerbelot2022asymptotic}. Beyond these static analyses, dynamical mean-field methods have studied gradient-based trajectories \citep{mignacco2020dynamical, sarao2019afraid, agoritsas2018out}, and more recent rigorous works have clarified the high-dimensional behavior of stochastic gradient descent \citep{ben2022high, gerbelot2024rigorous} and Langevin dynamics \citep{fan2025dynamical}. This line of work has built a precise understanding of learning in the regime of high-dimensional data and finite number of hidden nodes. 

\paragraph{Quadratic networks and matrix sensing.} Shallow neural networks with quadratic activation functions provide a simple model to study nonconvex optimization, as they reduce learning to the estimation of a positive semidefinite matrix. In this setting, several works studied the impact of overparameterization on the loss landscape \citep{soltanolkotabi2018theoretical, du2018power,gamarnik2019stationary,venturi2019spurious} and consequences of the learning dynamics with gradient-based methods, with an emphasis on population dynamics \citep{sarao2020optimization, martin2024impact}, and stochastic gradient descent \citep{arous2025learning}. In addition, this setting is equivalent to a low-rank matrix sensing problem, with rank-one measurements: in this more general framework, convex relaxation guarantees recovery \citep{recht2010guaranteed, candes2011tight, gross2011recovering}. In the case of nonconvex factorizations, several results established the absence of spurious local minima and global convergence of gradient descent under restricted isometry properties \citep{zheng2015convergent, ge2017no, park2017non, li2018algorithmic}, with an implicit bias toward minimal nuclear norm solutions \citep{gunasekar2017implicit, arora2019implicit}. As we shall show, quadratic networks offer a complex but tractable framework to investigate the challenging regime in which the number of hidden nodes, the dimension of the data and the number of data points diverge (in fixed ratios to be precised below).  

\subsection{Related Works}

Quadratic networks have been the object of recent attention: they provide a framework allowing to study the role of overparameterization in a nonconvex setting. Several works have already clarified convergence properties of gradient-based methods in the highly overparameterized regime, where the problem enjoys a convex relaxation. In this regime, it has been shown that gradient flow converges globally and several works studied the generalization properties of the minimizers found by the dynamics (\citet{sarao2020optimization} for a rank-one teacher, \citet{gamarnik2019stationary} for a full-rank teacher and \citet{soltanolkotabi2018theoretical} when the output weights are also optimized). However, in the case where the network has fewer neurons than the dimension, existing dynamical studies only include population dynamics, i.e., the infinite sample limit, with recent works by \citet{martin2024impact} for gradient flow and \citet{arous2025learning} for stochastic gradient descent (SGD). This paper aims to go beyond these results by providing an exact description of gradient flow dynamics on the empirical loss for an arbitrary, yet extensive, number of neurons. 

The extensive-width regime, in which the number of hidden nodes, the input dimension and the number of data points diverge, has recently been studied with the goal of describing overparameterized neural networks in a setting where feature learning is still present. In this setting, several works already clarified Bayes-optimal learning and empirical risk minimization: \citet{maillard2024bayes, erba2025nuclear, defilippis2025scaling} in the case of quadratic networks, \citet{cui2023bayes, barbier2025statistical} for deep networks, \citet{erba2025bilinear} for bilinear regression, and \citet{boncoraglio2025inductive, boncoraglio2025bayes} for attention networks. Additionally, note that similar scalings were originally studied in the context of matrix denoising \citep{maillard2022perturbative, semerjian2024matrix, barbier2025phase}. However, in the dynamical setting, it is still an open question to know if gradient-based algorithms match these static predictions. In recent works, \citet{montanari2025dynamical} leverage dynamical mean-field theory (DMFT) to identify a separation of timescales between learning and overfitting in the gradient flow setting, while \citet{ren2025emergence} characterize SGD dynamics through precise scaling laws. Our work provides a complementary perspective to these studies by deriving sharp asymptotic results in the extensive-width regime. 

One goal of the present work is also to connect the dynamical perspective with the static predictions obtained in the analyses of \citet{maillard2024bayes} and \citet{erba2025nuclear}, whose settings are essentially identical to ours. More precisely, our dynamical approach is inspired by the replica calculation introduced by \citet{maillard2024bayes}, while the regularized formulation we adopt follows the framework proposed by \citet{erba2025nuclear}. Within this setting, we are able to recover several conclusions reached in these studies. Together, these observations clarify how gradient flow and Langevin dynamics connect to the underlying landscape structure and optimal learning predictions. 

\subsection{Contributions}

In this work, we analyze the learning dynamics of a shallow quadratic neural network trained with gradient flow and Langevin dynamics on a regularized empirical loss. To study the role of overparameterization on learning and generalization, we work in the \emph{extensive-width} regime, where the widths of the teacher and student networks scale proportionally with the dimension. For analytical tractability, we introduce a Gaussian surrogate model that preserves the covariance structure of the rank-one measurements associated with the quadratic network model. We have verified by numerical experiments that this simplification does not affect our predictions. Our main results are:
\begin{itemize}[label={}, leftmargin=0pt]
\item \mbox{\textbf{1. Description of the high-dimensional dynamics.}} In the high-dimensional limit, we derive a set of dynamical equations that is distributionally equivalent to gradient flow and Langevin dynamics (\cref{Subsubsec:GeneralDMFT,Subsubsec:SimplifiedDMFT}). This result provides a self-consistent description of the learning dynamics in this asymptotic regime. Our equations are formulated and derived in the spirit of DMFT equations, but still involve high-dimensional objects. We believe that a reduction to a set of low-dimensional equations is impossible while keeping track of the relevant averaged quantities of the dynamics.
\item \mbox{\textbf{2. Reduction to a denoising dynamics.}} Under $\ell_2$-regularization and an asymptotic simplification of the dynamics, we argue that the gradient flow dynamics can be reduced to a denoising gradient flow, known as the Oja flow (\cref{Subsubsec:SteadyState}). We interpret this dynamics as a noisy version of the population dynamics (the infinite-sample gradient flow), where the effective noise variance arises from finite sample complexity and label noise. We analyze the Oja flow and obtain new results that are essential for our theoretical analysis (\cref{Sec:OjaFlow}).   
\item \mbox{\textbf{3. Long-time analysis of the dynamics.}} Building on this reduction, we study the long-time behavior of the gradient flow dynamics in high dimensions. This leads to a set of scalar equations describing the performance and spectral properties of the gradient flow estimator (\cref{Subsubsec:EquationsLongTimes}). We support these results with theoretical insights (\cref{Subsubsec:Stability}) and numerical confirmations, while noting that we leave a rigorous treatment of the simplifying assumptions for further work.
\item \mbox{\textbf{4. Langevin dynamics.}} We study the Langevin dynamics on the regularized empirical loss and derive its stationary measure under similar simplifications to those used for gradient flow (\cref{Subsubsec:Langevin}). In the zero-temperature limit, the stationary measure both recovers the long-time equations of the deterministic dynamics and concentrates on global minimizers, providing evidence for convergence of gradient flow to a global minimum. 
\item \mbox{\textbf{5. Impact of overparameterization.}} The long-time equations reveal that the overparameterized regime, in which the network has more neurons than the input dimension, persists down to a smaller width that we identify (\cref{Subsubsec:Regions}). Above this threshold, gradient flow converges to the global minimizer of the regularized empirical loss over the set of positive semidefinite (PSD) matrices. Below it, the performance and behavior of the estimator depend on the effective number of neurons per dimension. These equations illustrate how overparameterization enables global convergence, whereas insufficient width traps the dynamics at higher loss solutions. 
\item \mbox{\textbf{6. Overfitting and double descent.}} We provide numerical evidence of overfitting along the training dynamics. In the presence of label noise, our long-time equations capture a double descent phenomenon in the test error as a function of the sample complexity (\cref{Subsubsec:Overfitting}). In the small regularization limit, we identify the associated interpolation threshold (\cref{Subsubsec:InterpolationThreshold}). This result constitutes a theoretical characterization of double descent based on the dynamical study of a nonlinear model.
\item \mbox{\textbf{7. Perfect recovery thresholds.}} In the small regularization limit, we derive an explicit expression of the perfect recovery threshold as a function of the teacher and student network widths, quantifying how overparameterization affects recovery (\cref{Subsubsec:PRThresholdReg}). For the unregularized case, we conjecture an expression of this threshold (\cref{Subsubsec:PRThresholdUnreg}) and provide numerical evidence supporting this conjecture. 
\item \mbox{\textbf{8. Insights into Gaussian equivalence.}} We investigate the equivalence between the Gaussian surrogate model and the rank-one formulation induced by the quadratic network model (\cref{subsec:GaussianUniversality}). We conjecture their asymptotic equivalence and provide numerical and theoretical evidence.
\end{itemize}

\begin{figure}[ht]
\centering
\begin{tikzpicture}[
    scale=0.95,
    >=Latex,
    mainnode/.style={
        draw,
        rounded corners,
        align=center,
        minimum height=1.05cm,
        text width=4.1cm,
        font=\small
    },
    sidenode/.style={
        draw,
        rounded corners,
        align=center,
        minimum height=0.95cm,
        text width=3.3cm,
        font=\small
    },
    bottomnode/.style={
        draw,
        rounded corners,
        align=center,
        minimum height=0.95cm,
        text width=4.cm,
        font=\small
    },
    mainarrow/.style={->, thick},
    sidearrow/.style={->, thick, dashed},
    edgelabel/.style={
        font=\small,
        align=center,
        inner sep=5pt
    }
]

\node[mainnode, fill=modelcolor] (qnn) at (0,0)
    {Quadratic neural networks};

\node[mainnode, fill=modelcolor] (gauss) at (0,-2.3)
    {Gaussian surrogate model};

\node[mainnode, fill=predcolor] (dmft) at (0,-5.0)
    {High-dimensional \\ dynamical equations\\ (\cref{subsec:DMFT})};

\node[mainnode, fill=predcolor] (longtime) at (0,-7.8)
    {Long-time scalar equations\\(\cref{Subsubsec:EquationsLongTimes})};

\node[sidenode, fill=predcolor] (langevin) at (-5.9,-4.2)
    {Langevin stationary \\ measure \\ (\cref{Subsubsec:Langevin})};

\node[sidenode, fill=predcolor] (aging) at (-5.9,-5.9)
    {Aging equations \\ (\cref{App:Aging})};

\node[sidenode, fill=rigorcolor] (oja) at (5.9,-7.3)
    {Oja flow results \\(\cref{Sec:OjaFlow})};

\node[bottomnode, fill=rigorcolor] (overparam) at (-5.5,-10.7)
    {Impact of \\ overparameterization \\(\cref{Subsubsec:Regions})};

\node[bottomnode, fill=validcolor] (overfit) at (0,-10.7)
    {Overfitting and double descent\\(\cref{Subsubsec:Overfitting})};

\node[bottomnode, fill=rigorcolor] (thresholds) at (5.5,-10.7)
    {Small regularization \\ limit
    and thresholds \\(\cref{subsec:SmallReg})};

\draw[mainarrow] (qnn) --
    node[right,edgelabel]
    {Gaussian equivalence\\(\cref{subsec:GaussianUniversality})}
    (gauss);

\draw[mainarrow] (gauss) --
    node[right,edgelabel]
    {DMFT formalism\\(\cref{App:DMFT})}
    (dmft);

\draw[mainarrow] (dmft) --
    node[right,edgelabel]
    {Steady-state ansatz\\(\cref{Subsubsec:SteadyState})}
    (longtime);

\draw[mainarrow] (oja.west) --
    node[midway,above,edgelabel] 
    {}
    (longtime.east);

\draw[sidearrow] (dmft.west) --
    node[pos=0.35,above,edgelabel,yshift=7pt]
    {Stationarity \\ ansatz}
    (langevin.east);

\draw[sidearrow] (dmft.west) --
    node[pos=0.3,below,edgelabel, yshift=-3pt]
    {Slow \\ relaxation}
    (aging.east);

\draw[mainarrow] (longtime) -- (overparam);
\draw[mainarrow] (longtime) -- (overfit);
\draw[mainarrow] (longtime) -- (thresholds);

\end{tikzpicture}
\caption{Logical structure of the main contributions. Solid arrows indicate the main chain of derivations and consequences, while dashed arrows represent complementary analyses. Node colors indicate the nature of the corresponding contribution: model formulation (blue), DMFT predictions and ansatz-based reductions (orange), rigorous analyses (green), and phenomena validated numerically from the theory (purple).}
\label{fig:Logicalstructure}
\end{figure}

\newpage

\section{Setting} \label{Sec:Setting}
In this section, we introduce the notation and conventions used throughout the paper and describe the model and main assumptions of our work.

\subsection{Notation and Conventions} \label{Subsub:NotationsConventions}

For a matrix $A \in \R^{d \times m}$, we denote $A^\top$ its transpose, $\tr(A)$ its trace and $\|A\|_F = \big[\tr(AA^\top) \big]^{1/2}$ its Frobenius norm. We denote $\mathcal{S}_d(\R), \mathcal{S}_d^+(\R)$ and $\mathcal{S}_d^{++}(\R)$ the sets of $d \times d$ symmetric, positive semidefinite (PSD) and positive definite matrices. For $A \in \R^{d \times d}$, we denote $\mathrm{Sym}(A) = (A + A^\top) / 2 \in \mathcal{S}_d(\R)$ its symmetrization. Given $E$ a Euclidean space and $L \colon E \to \R$ continuously differentiable, we let $\nabla L(x) \in E$ be its gradient at $x \in E$.

\subsection{Model}
 
Consider a data set composed of standard Gaussian inputs $x_1, \dots, x_n \overset{\mathrm{i.i.d.}}{\sim} \N(0, I_d)$ and labels $z_1, \dots, z_n \in \R$ generated as:
\begin{equation} \label{eq2:model}
    z_k \sim P \big( \cdot \, \big| \,  y_k^* \big), \hspace{1.5cm} y_k^* = \frac{1}{m^*} \sum_{i=1}^{m^*} \frac{(x_k^\top w_i^*)^2 - \|w_i^*\|^2}{\sqrt{d}},  
\end{equation}
where the vectors $w_1^*, \dots, w_{m^*}^* \in \R^d$ will be referred to as teacher vectors, and the distribution $P$ encodes possible nonlinearity and noise applied to the true labels $(y_k^*)_{1 \leq k \leq n}$. This structure corresponds to a two-layer neural network with a centered quadratic activation function and uniform output weights. The centering subtracts the expectation of the quadratic activation under the input distribution, since $\E \big[ (x^\top w)^2 \big] = \|w\|^2$ for $x \sim \N(0, I_d)$, therefore removing an uninformative contribution. Equivalently, the centered quadratic corresponds to the second-order Hermite polynomial, leading to an activation function with information exponent 2.  

Unlike in generalized linear models \citep{mccullagh2019generalized, barbier2019optimal}, our objective is not to recover the individual teacher vectors $w_i^*$. Indeed, the quadratic network is invariant under any orthogonal transformation of these vectors, leading to a degeneracy: many different sets of teacher vectors yield the same model output. As a result, only their second-order structure can be identified. Therefore, we instead aim to recover the \emph{teacher matrix}:
\begin{equation} \label{eq:DefTeacher}
    Z^* = \frac{1}{m^*} \sum_{i=1}^{m^*} w_i^* w_i^{*\top} \in \mathcal{S}_d^+(\R),
\end{equation}
which captures all the information about the teacher that we can recover from the observations. This matrix plays a central role in the quadratic model, as it allows one to express the true labels as:
\begin{equation} \label{eq:MatrixSensingFormulation}
    y_k^* = \tr\big( X_k Z^* \big), \hspace{1.5cm} X_k = \frac{x_kx_k^\top - I_d}{\sqrt{d}}.  
\end{equation}
This formulation reveals an underlying linear structure: each true label $y_k^*$ is a linear measurement of the teacher matrix $Z^*$ through the sensing matrices $X_k$. The problem can therefore be viewed as a generalized \emph{matrix sensing} task, where one aims to recover $Z^*$ from the observations of $z_1, \dots, z_n$ that may depend nonlinearly on the true labels. In the particular case where~$P$ is Gaussian, this reduces to the standard matrix sensing model with additive Gaussian noise on the labels. 

Under this representation, the effective parameter of the model becomes the matrix $Z^*$, that gathers all second-order information about the teacher. In the following, we will no longer refer to the individual teacher vectors, as the model's invariance makes them non-identifiable. Instead, we directly work with the matrix $Z^*$, whose structure is determined by the model: it is positive semidefinite and of rank $m^*$, with typically $m^* < d$. 

\subsection{Optimization} \label{Subsec:optimization}

Given some cost $\ell \colon \R^2 \to \R^+$, we leverage the low-rank structure of $Z^*$ and define the empirical loss function:
\begin{equation} \label{eq:Loss}
    \L(W) = \frac{1}{2n} \sum_{k=1}^n \ell \Big( \tr(X_k WW^\top), z_k \Big),
\end{equation}
for $W \in \R^{d \times m}$. We assume that $m^*$ is not known, therefore we potentially have $m \neq m^*$. While optimizing with respect to $W$ directly exploits this low-rank structure, it is often simpler to analyze optimization with respect to the matrix $Z = WW^\top$, which can lead to convex formulations. Relevant works on this problem include the results of \citet{burer2003nonlinear} and \citet{gunasekar2017implicit} on factorized approaches, and in a more general setting \citep{edelman1998geometry, journee2010low, massart2020quotient} on the geometry and optimization of functions of $WW^\top$. 

To optimize the student matrix $W$, we add a regularization $\Omega \colon \R^{d \times m} \to \R^+$ to the empirical loss and perform Langevin dynamics:
\begin{equation} \label{eq:GFdynamics}
    \d W(t) = - d \, \nabla \L \big( W(t) \big) \d t - \nabla \Omega \big( W(t) \big) \d t + \frac{1}{\sqrt{\beta d}} \d B(t), 
\end{equation}
where $B$ is a standard Brownian motion over $\R^{d \times m}$, and the inverse temperature $\beta \geq 0$ controls the intensity of the noise. When setting $\beta = \infty$, we recover the plain gradient flow dynamics, which is the setting for most of our results. The prefactor $1 / \sqrt{d}$ in front of $B$ ensures that the effect of the noise remains of order one in the high-dimensional limit, consistently with the scaling we introduce in \cref{Assumption1}. 

In addition, although our analysis is formulated in continuous time, the results presented in \cref{subsec:DMFT} remain valid for gradient descent, corresponding to the discretized version of the dynamics~\eqref{eq:GFdynamics}, with $\beta = \infty$:
\begin{equation} \label{eq:GDdynamics}
    W_{k+1} = W_k - \eta d \, \nabla \L(W_k) - \eta \nabla \Omega(W_k), 
\end{equation}
where $\eta > 0$ corresponds to the stepsize of the algorithm. 

\paragraph{Gradient descent simulations.} The gradient descent algorithm (with small stepsize) will also be used as a proxy for the gradient flow dynamics in our numerical experiments. As some of our results will be derived using non-rigorous methods, numerical simulations are essential to confirm the validity of our claims. Throughout this paper, several figures are presented in order to compare the numerical integration of equation \eqref{eq:GDdynamics} to our theoretical claims. Unless specified in the corresponding figure, all of our simulations were performed at dimension $d = 100$, with stepsize $\eta = 5 \times 10^{-3}$, and error bars indicate the standard deviation under several realizations of the random initialization, teacher matrix and sensing matrices. However our results hold for a general class of initialization and teacher distributions (see \cref{Assumption1} for more details), we have chosen the Gaussian distribution for our simulations:
\begin{equation} \label{eq:DistribInitTeacher}
    W_{ij}(t = 0) \overset{\mathrm{i.i.d.}}{\sim} \N \left( 0, \frac{1}{m} \right), \hspace{1.5cm} (w_i^*)_j \overset{\mathrm{i.i.d.}}{\sim} \N(0,1). 
\end{equation}
In this case, the teacher matrix $Z^*$, expressed as the average of the rank-one matrices $w_i^* w_i^{*\top}$ in equation \eqref{eq:DefTeacher}, is therefore distributed as a Wishart matrix (see \cref{App:Subsec:RandomMatrices}).

We give more details on our simulation setup in \cref{App:Subsubsec:SimulationsGD} and we provide a code to reproduce the numerical experiments of the paper in the freely accessible GitHub repository: \url{https://github.com/simonmartin15/QuadraticNets}.

\subsection{High-Dimensional Limit}

In this work, we analyze the gradient flow dynamics~\eqref{eq:GFdynamics} in the limit where the dimension~$d$ diverges. One key feature of our analysis is to study the \textit{extensive-width regime} where the widths $m, m^*$ of the student and teacher networks diverge proportionally to the dimension. This scaling is designed to reflect architectures with both large input dimension and hidden layers size, representing settings in which the feature space increases with the scale of the problem.

Moreover, one key feature of our analysis consists in replacing the centered rank-one matrices $(X_k)_{1 \leq k \leq n}$ in equation \eqref{eq:MatrixSensingFormulation} by symmetric Gaussian matrices with the same mean and covariance. This distribution corresponds to the Gaussian orthogonal ensemble (denoted $\mathrm{GOE}(d)$ in the following), that is, the random symmetric matrices $X \in \mathcal{S}_d(\R)$ whose coefficients are distributed as:
\begin{equation}
    \hspace{3cm} X_{ij} \overset{\mathrm{i.i.d.}}{\sim} \N \left(0, \frac{1 + \delta_{ij}}{d} \right), \hspace{2.5cm} \big(1 \leq i \leq j \leq d \big). 
\end{equation}
Replacing the centered rank-one sensing matrices by Gaussian matrices is justified by recent universality and exact-asymptotics results in static problems \citep{maillard2024bayes, erba2025nuclear, xu2025fundamental}. These works show that, for purposes of sharp risk and Bayes-optimal asymptotics, the quadratic network model can be traded for a sensing matrix problem with Gaussian observations. We claim that the same equivalence holds for the gradient flow trajectory in the high-dimensional limit, and that all our results apply to the observations introduced in equation \eqref{eq:MatrixSensingFormulation}. We refer to this property as \emph{Gaussian equivalence}, which is supported by theoretical considerations and numerical experiments, presented in \cref{subsec:GaussianUniversality}. 

\begin{assumption} \label{Assumption1}
We make the following assumptions throughout this paper:

\begin{assumpitem}[High-dimensional scaling.]\label{Assumption1:Scaling} The number of observations $n$ and the number of neurons $m,m^*$ scale with the dimension $d$ as:
\begin{equation}
    n \underset{d \to \infty}{\sim} \alpha d^2, \hspace{1.5cm} m \underset{d \to \infty}{\sim} \kappa d, \hspace{1.5cm} m^* \underset{d \to \infty}{\sim} \kappa^* d,
\end{equation}
where $\alpha, \kappa, \kappa^*$ are fixed positive constants, referred to respectively as the sample complexity and the effective width of the student and the teacher. 
\end{assumpitem}
\begin{assumpitem}[Data distribution.]\label{Assumption1:Data} The measurement matrices $X_1, \dots, X_n \in \mathcal{S}_d(\R)$ are independent and drawn from the $\mathrm{GOE}(d)$ distribution. 
\end{assumpitem}
\begin{assumpitem}[Teacher distribution.]\label{Assumption1:Teacher} The spectral distribution of the teacher matrix $Z^* \in \mathcal{S}_d^+(\R)$ almost surely converges as $d \to \infty$ to some measure $\mu^*$ with compact support. Moreover, $\|Z^*\|_\mathrm{op} \leq C$ almost surely for some $C$ independent of $d$. When necessary, we will decompose $\mu^* = (1 - \min(\kappa^*, 1)) \delta + \min(\kappa^*, 1) \nu^*$, where $\nu^*$ has a bounded support on $\R^+$ with no mass at zero, and $\delta$ is the Dirac point mass at zero. 
\end{assumpitem}
\begin{assumpitem}[Initialization.]\label{Assumption1:Initialization} The dynamics~\eqref{eq:GFdynamics} is initialized with a random matrix $W_0 \in \R^{d \times m}$ such that the empirical spectral distribution of $W_0W_0^\top \in \mathcal{S}_d^+(\R)$ almost surely converges as $d \to \infty$ to some measure $\mu_0$ with compact support. Moreover, the distribution of $W_0$ admits a density with respect to the Lebesgue measure on $\R^{d \times m}$.
\end{assumpitem}
\end{assumption}

Therefore, we let the width of the teacher and student to vary proportionally with the dimension, leading to what we refer to as the extensive-width regime. In this work, we quantify overparameterization through the width of the student network, as measured by $\kappa$, which reflects how the number of parameters grows with the dimension. Additionally, since the teacher has $m^* d \sim \kappa^* d^2$ coefficients, it is natural to require $\Theta(d^2)$ measurements in order to recover it. Together with the choice of $\kappa, \kappa^*$, this setting allows us to study the effect of overparameterization on the performance of gradient flow dynamics and signal recovery. 

From the observations~\eqref{eq2:model}, we aim to recover the teacher matrix $Z^*$. We will be interested in several scalar quantities in the high-dimensional limit, including the mean-square error (MSE) between the teacher and the student, and the empirical loss value (or training loss):
\begin{equation} \label{eq:DefMSELoss}
    \mathrm{MSE} = \frac{1}{d} \E \, \big\| Z - Z^* \big\|_F^2, \hspace{1.5cm} \mathrm{Loss}_\mathrm{train} = \frac{1}{2n} \E \sum_{k=1}^n \ell \big( \tr(X_kZ), z_k \big),
\end{equation}
with $Z = WW^\top$, and the expectation is taken with respect to the distribution of the data and the teacher. Note that, when the cost $\ell$ in equation~\eqref{eq:Loss} is quadratic, the MSE is directly proportional to the generalization error, since the matrices $X_1, \dots, X_n$ are Gaussian with isotropic covariance. When writing $\mathrm{MSE}(t)$ and $\mathrm{Loss}_\mathrm{train}(t)$ we will refer to the MSE and the loss evaluated with $W(t)$ being the solution of the gradient flow dynamics~\eqref{eq:GFdynamics}.

\section{Main Results} \label{Sec:Results}
This section is dedicated to our main results. We briefly outline its organization:
\begin{itemize}
\item In \cref{subsec:DMFT}, we derive a set of effective dynamical equations governing gradient flow and Langevin dynamics in the high-dimensional limit. These equations provide a self-consistent description of the learning dynamics and serve as the basis for the analyses that follow. 
\item In \cref{Subsec:LongTimes}, under a long-time simplification, we reduce the high-dimensional dynamics to an effective denoising gradient flow and analyze its asymptotic behavior. We characterize the spectral properties and performance of the estimator, and discuss the impact of overparameterization through the network width, as well as the emergence of a double descent phenomenon. 
\item In \cref{subsec:SmallReg}, we investigate the behavior of the long-time equations as the regularization vanishes. This leads to an exact characterization of the perfect recovery and interpolation thresholds. 
\item In \cref{subsec:Unregularized}, we study the unregularized setting. While a complete analytical characterization remains challenging, we formulate a conjecture for the perfect recovery threshold and support it with numerical evidence.
\item In \cref{subsec:GaussianUniversality}, we study the relationship between the Gaussian surrogate model and the original quadratic neural network formulation. We provide theoretical and numerical evidence supporting their asymptotic equivalence.
\end{itemize}

\subsection{High-Dimensional Dynamics} \label{subsec:DMFT}

In this section, we study the high-dimensional limit associated with the gradient flow dynamics~\eqref{eq:GFdynamics}. More precisely, we present a system of stochastic differential equations that is equivalent in distribution to the gradient flow dynamics. 

\subsubsection{General High-Dimensional Equations} \label{Subsubsec:GeneralDMFT}

We now present a high-dimensional characterization of the gradient flow dynamics. In the high-dimensional limit, the training dynamics admits an equivalent description in terms of a self-consistent stochastic system. In this description, the dynamics is driven by Gaussian processes and deterministic scalar functions, both determined self-consistently from expectations along the training trajectory.

The resulting system involves two coupled random processes. The first is a matrix-valued process $\overline W(t) \in \mathbb{R}^{d \times m}$, which can be interpreted as the student matrix and remains high-dimensional. The second is a scalar process $\overline y(t) \in \mathbb{R}$ corresponding to a typical label, defined as one of the training labels sampled uniformly at random. Together, these processes describe the evolution of the network weights and the associated labels.

For conciseness, we do not report the full system of equations here and refer to \cref{App:Subsec:SummaryDMFT} for a complete statement.

\begin{claim} \label{Result0}
Let $W(t) \in \R^{d \times m}$ be the solution of the dynamics~\eqref{eq:GFdynamics}, with observed labels $z_1, \dots, z_n$ and initial condition $W_0 \in \R^{d \times m}$. Consider an index $K$ uniformly drawn in $\{1, \dots, n\}$, and the typical label defined as:
\begin{equation}
    y(t) = \tr \big( X_K W(t)W(t)^\top \big).
\end{equation}
Then, in the $d \to \infty$ limit, under \cref{Assumption1}, $\big( W, y \big) \overset{\mathrm{distrib}}{=} \big( \overline W, \overline y \big)$, solution of the stochastic equations:
\begin{align}
    \d \overline W(t) &= \left( \mathcal{H}(t) + r(t) Z^* - \int_0^t \Gamma(t,t') \overline Z(t') \d t' \right) \overline W(t) \d t - \nabla \Omega \big( \overline W(t) \big) \d t + \frac{1}{\sqrt{\beta d}} \d B(t), \label{eq:DMFT0W} \\
    0 &= \int_0^t R(t,t') \overline y(t') \d t' + \eta(t) - m(t) y^* + \frac{2}{\alpha} \ell' \big( \overline y(t), z \big), \label{eq:DMFT0Y}
\end{align}
where $\overline Z(t) = \overline W(t) \overline W(t)^\top$, $z \sim P( \, \cdot \, | \, y^*)$, and $\overline W(t = 0) = W_0$. The functions $\mathcal{H}$ and $\eta$ are independent Gaussian processes respectively belonging to $\mathcal{S}_d(\R)$ and $\R$, with covariances:
\begin{equation}
\begin{aligned}
    \E \, \mathcal{H}_{ij}(t) \mathcal{H}_{i'j'}(t') &= \frac{1}{2d}\big( \delta_{ii'} \delta_{jj'} + \delta_{ij'} \delta_{i'j} \big) \, \mathcal{K}_Z(t,t'), \hspace{1.5cm} \E \, \eta(t) \eta(t') = \mathcal{K}_y(t,t'),
\end{aligned}
\end{equation}
and $r, \Gamma, R, m, \mathcal{K}_Z, \mathcal{K}_y$ are deterministic scalar functions expressed as expectations over the stochastic processes $\overline W, \overline y$.
\end{claim}

In these equations, the random variables $y^*$ and $z$ respectively correspond to the true label and its noisy version, associated with the index $K$ picked at random in \cref{Result0}. In the high-dimensional limit, as a consequence of the central-limit theorem, $y^*$ is a Gaussian variable with zero mean and variance:
\begin{equation}
    \E \, y^{*2} = \lim_{d \to \infty} \frac{2}{d} \E \, \tr\big(Z^{*2} \big),
\end{equation}
and $z \sim P\big( \, \. \big| \, y^* \big)$. 

This result suggests that, in the high-dimensional limit, the dynamics of $\big(W(t),y(t)\big)$ can be accurately described by $\big(\overline W(t),\overline y(t)\big)$. More precisely, for any fixed time horizon, averaged quantities of the original system can be computed from the previous equations with asymptotically vanishing error. To illustrate this fact, consider for instance the MSE and the empirical loss defined in~\eqref{eq:DefMSELoss}. Then, for any fixed $T > 0$, both the differences:
\begin{equation}
    \sup_{t \in [0,T]} \left| \mathrm{MSE}(t) - \frac{1}{d} \E \, \big\| \overline Z(t) - Z^* \big\|_F^2 \right|, \hspace{1cm} \sup_{t \in [0,T]} \left| \mathrm{Loss}_\mathrm{train}(t) - \frac{1}{2} \E \, \ell \big( \overline y(t),z \big) \right|,
\end{equation}
converge to zero as $d \to \infty$. 

In this set of equations, the dynamics for the student matrix $\overline W$ and the typical label $\overline y$ are driven by scalar functions. As mentioned, these functions can be themselves computed from the law of the random processes $\overline W, \overline y$, leading to a highly nonlinear and self-consistent dynamics, in the spirit of McKean--Vlasov processes \citep{chaintron2022propagation}.

We emphasize the generality of this result: the derivation of \cref{Result0} only makes use of the high-dimensional scaling chosen for the number of samples $n$ and the Gaussian distribution of the data. It remains valid for any choice of regularization $\Omega$, cost function $\ell$, and noise distribution $P$. Furthermore, since the equations we obtain still explicitly depend on the teacher matrix $Z^*$ and the initialization of the dynamics $W_0$, our result remains independent of the choice of the distributions of these matrices: the requirements made in \cref{Assumption1} are only necessary to guarantee that the dynamical equations in \cref{Result0} are well-posed.  

\paragraph{Derivation of the equations.} We derive the equations of \cref{Result0} in \cref{App:Subsec:DMFTDerivation}. Our calculation is based on the rewriting of the dynamical partition function associated with the dynamics~\eqref{eq:GFdynamics}. One key point of our analysis is the possibility to average the partition function with respect to the sensing matrices, thanks to the Gaussian distribution hypothesis in \cref{Assumption1}. Finally, a saddle-point calculation in our high-dimensional setting allows to obtain the self-consistent set of equations in \cref{Result0}. Therefore, as $d \to \infty$, we obtain the equality between the dynamical partition function of the gradient flow dynamics and that of the stochastic process $\big( \overline W, \overline y \big)$ of \cref{Result0}. This result turns out to be equivalent to equality in distribution in the high-dimensional limit. More details can be found in \cref{App:Subsec:DMFTBackground}. 

We insist on the fact that several steps in our calculation are non-rigorous, but build on objects and methods that have been applied in similar problems, see for instance \citet{agoritsas2018out} for the perceptron model, \citet{mignacco2020dynamical} for Gaussian mixture classification and \citet{bordelon2022self} for wide neural networks. More generally, our calculation falls inside the class of dynamical mean field theory (DMFT), which gathers several methods originally used to derive dynamical equations for spin glass models \citep{sompolinsky1982relaxational, cugliandolo1993analytical}. Since then, several works have succeeded in showing rigorously that these asymptotic equations are exact \citep{arous1995large, ben2006cugliandolo, celentano2021high, gerbelot2024rigorous}.

\paragraph{Dimensionality reduction.} In similar studies of the finite-rank case $m = O(1)$, a common procedure is to write a low-dimensional set of equations on the correlations and the overlaps between the neurons, leading to a finite number of summary statistics \citep[see for instance][]{celentano2021high, gerbelot2024rigorous, montanari2025dynamical}. In our case, where $m$ grows with the dimension, we believe that our system of equations cannot be reduced to a finite-dimensional one while still capturing the relevant dynamical quantities such as the MSE.

\subsubsection{Learning the Second Layer} \label{Subsubsec:SecondLayer}

The method used to derive the high-dimensional dynamics in \cref{Result0} extends to the more general setting where a second layer of weights is also learned. To be more precise, consider the following predictor:
\begin{equation} \label{eq:PredictorSecondLayer}
    \tr \big( X_k W D_a W^\top \big) = \frac{1}{m} \sum_{i=1}^m a_i \frac{(x_k^\top w_i)^2 - \|w_i\|^2}{\sqrt{d}},  \hspace{1.2cm} W = \frac{1}{\sqrt{m}} \big( w_1, \dots, w_m \big) \in \R^{d \times m},
\end{equation}
where $a \in \R^m$ and $D_a \in \R^{m \times m}$ is the diagonal matrix with the same diagonal coefficients as $a$. As it is common in machine learning settings where all layers of the network possess trainable parameters, this vector can also be optimized. We then consider the joint dynamics:
\begin{align}
    \d a(t) &= - \vartheta d \, \nabla_a \L \big( a(t), W(t) \big) \d t  - \nabla_a \Omega \big( a(t), W(t) \big) \d t + \frac{1}{\sqrt{\beta d}} \d B_a(t), \label{eq:Dynamics2LayersA} \\
    \d W(t) &= - d \, \nabla_W \L \big( a(t), W(t) \big) \d t - \nabla_W \Omega \big( a(t), W(t) \big) \d t + \frac{1}{\sqrt{\beta d}} \d B_W(t), \label{eq:Dynamics2LayersW}
\end{align}
where the loss function $\L$ from equation \eqref{eq:Loss} is considered with the predictor given in equation~\eqref{eq:PredictorSecondLayer}. The regularization $\Omega$ is now a function of $a, W$, the constant $\vartheta > 0$ controls the learning rate of the dynamics of $a(t)$ with respect to $W(t)$, and $B_a, B_W$ are independent standard Brownian motions over $\R^m$ and $\R^{d \times m}$, respectively.

Using the same method as in \cref{Result0}, we derive an equivalent description of the dynamics in the high-dimensional limit, under \cref{Assumption1}. We recover similar self-consistent equations (given in \cref{App:Subsec:SummaryDMFT}), but the effective dynamics in equation \eqref{eq:DMFT0W} is replaced by a joint stochastic evolution for two processes $\overline a, \overline W$:
\begin{align}
\begin{aligned}
    \d \overline a(t) &= \frac{\vartheta}{2} \mathrm{diag} \left( \overline W(t)^\top \left[ \H(t) + r(t) Z^* - \int_0^t \Gamma(t,t') \overline Z(t') \d t' \right] \overline W(t) \right) \d t \\
    &\hspace{5cm} - \nabla_a \Omega \big( \overline a(t), \overline W(t) \big) \d t + \frac{1}{\sqrt{\beta d}} \d B_a(t)
\end{aligned} \label{eq:DMFT2LayersA} \\
\begin{aligned}
    \d \overline W(t) &= \left( \H(t) + r(t) Z^* - \int_0^t \Gamma(t,t') \overline Z(t') \d t' \right) \overline W(t) D_{\overline a(t)} \d t  \\
    &\hspace{5cm} - \nabla_W \Omega \big( \overline a(t), \overline W(t) \big) \d t + \frac{1}{\sqrt{\beta d}} \d B_W(t),
\end{aligned} \label{eq:DMFT2LayersW}
\end{align}
where $\overline Z(t) = \overline W(t) D_{\overline a(t)} \overline W(t)^\top$, and for $A \in \R^{m \times m}$, $\mathrm{diag}(A) \in \R^m$ denotes the vector composed of the diagonal elements of $A$. We refer to \cref{App:Subsec:SecondLayer} for a derivation of these equations. In the following, we do not analyze this particular setting and leave this study for further work. 

Another option is to train the output layer only, while keeping the inner weights fixed. This corresponds to the random feature model, originally introduced by \citet{rahimi2007random}. This setting has been extensively analyzed in high dimensions \citep{gerace2020generalisation, mei2022generalization, hastie2022surprises}. In this regime, the optimization is convex and therefore simpler than when training the inner weights. 

\subsubsection{Simplified Setting: Quadratic Cost and Gaussian Label Noise} \label{Subsubsec:SimplifiedDMFT}

We now give an illustration of the system of equations of \cref{Result0} in a less general setting where we specify the cost function and the noisy channel that generates the labels. 
In this case, the full set of equations and order parameters is less cumbersome and will be presented. This setup will be studied in more detail when analyzing the gradient flow dynamics at long times, in \cref{Subsec:LongTimes}. 
\begin{assumption} \label{Assumption2}

\begin{assumpitem}[Cost function.]\label{Assumption2:Cost} Optimization is performed using the quadratic cost: $\ell(y, z) = \dfrac{1}{2} (y-z)^2$.
\end{assumpitem} 
\begin{assumpitem}[Noisy channel.]\label{Assumption2:Channel} The labels are generated using an additive Gaussian channel with variance $\Delta$:
\begin{equation}
    P \big( z \, \big| \, y \big) = \frac{1}{\sqrt{2 \pi \Delta}} \exp \left( - \frac{1}{2\Delta} (y-z)^2 \right). 
\end{equation}
\end{assumpitem}
\end{assumption}

Under this assumption, \cref{Result0} can be simplified and leads to the following set of self-consistent equations:
\begin{claim} \label{Result1} 
Under \cref{Assumption1,Assumption2}, the system of equations~\eqref{eq:DMFT0W}, \eqref{eq:DMFT0Y} can be written in the form: 
\begin{align}
    \d \overline W(t) &= 2 \left( \int_0^t R(t,t') \Big(\G(t') + Z^* - \overline Z(t') \Big) \d t' \right) \overline W(t) \d t  - \nabla \Omega \big( \overline W(t) \big) \d t + \frac{1}{\sqrt{\beta d}} \d B(t),  \label{eq:DMFT1W} \\
    \overline y(t) &= y^* + \sqrt{\Delta} \zeta + \int_0^t \d t' \, R(t,t') \left( \xi(t') - \sqrt{\Delta} \zeta + \left(\frac{m_Z(t')}{Q_*} - 1 \right) y^* \right),  \label{eq:DMFT1Y}
\end{align}
where $\overline Z(t) = \overline W(t) \overline W(t)^\top$, and $y^* \sim \N(0, 2Q_*), \zeta \sim \N(0,1)$, $\G(t) \in \mathcal{S}_d(\R)$ and $\xi(t) \in \R$ are independent centered Gaussian variables and processes with covariances:
\begin{align}
    \E \, \G_{ij}(t) \G_{i'j'}(t') &= \frac{1}{2 \alpha d} \big( \delta_{ii'} \delta_{jj'} + \delta_{ij'} \delta_{i'j} \big) \left(C_Z(t,t') - m_Z(t) - m_Z(t') + Q_* + \frac{\Delta}{2} \right), \label{eq:Covariance1G}\\
    \E \, \xi(t) \xi(t') &= 2C_Z(t,t') - \frac{2}{Q_*} m_Z(t) m_Z(t'). \label{eq:Covariance1xi} 
\end{align}
Finally, the following relationships close the system of equations:
\begin{equation} \label{eq:AveragedQuantities}
\begin{aligned}
    Q_* &= \frac{1}{d} \E \,  \tr(Z^{*2}), &\hspace{0.5cm} &C_Z(t,t') = \frac{1}{d} \E \, \tr \big( \overline Z(t) \overline Z(t') \big), \\
    m_Z(t) &= \frac{1}{d} \E \, \tr \big( \overline Z(t) Z^* \big), &\hspace{1.5cm} &R(t,t') = \delta(t-t') - \frac{1}{\alpha d^2} \tr \left( \left. \frac{\partial \, \E \, \overline Z(t)}{\partial H(t')} \right|_{H = 0} \right). 
\end{aligned}
\end{equation}
\end{claim}

This last set of equations defines averaged quantities with respect to $\overline W(t)$: covariances, overlaps and responses. $\delta$ is the Dirac delta distribution supported at 0 and the equation defining $R$ is to be interpreted in the sense of distributions (see \cref{App:Subsec:DMFTBackground} for more details). The response $R(t,t')$ quantifies the average change of $\overline Z(t)$ in response to an infinitesimal perturbation $H(t') \in \mathcal{S}_d(\R)$ added to the Gaussian noise $\G(t') \to \G(t') + H(t')$ into equation~\eqref{eq:DMFT1W}. More precisely, the partial derivative in equation \eqref{eq:AveragedQuantities} can be seen as the differential of the function mapping the perturbation $H(t')\in \mathcal{S}_d(\R)$ to the corresponding perturbed solution $\overline Z(t)$. Then $R$ is computed by taking the trace over linear maps acting on $\mathcal{S}_d(\R)$.

\begin{figure}[ht]
    \centering
    \includegraphics[width=\linewidth]{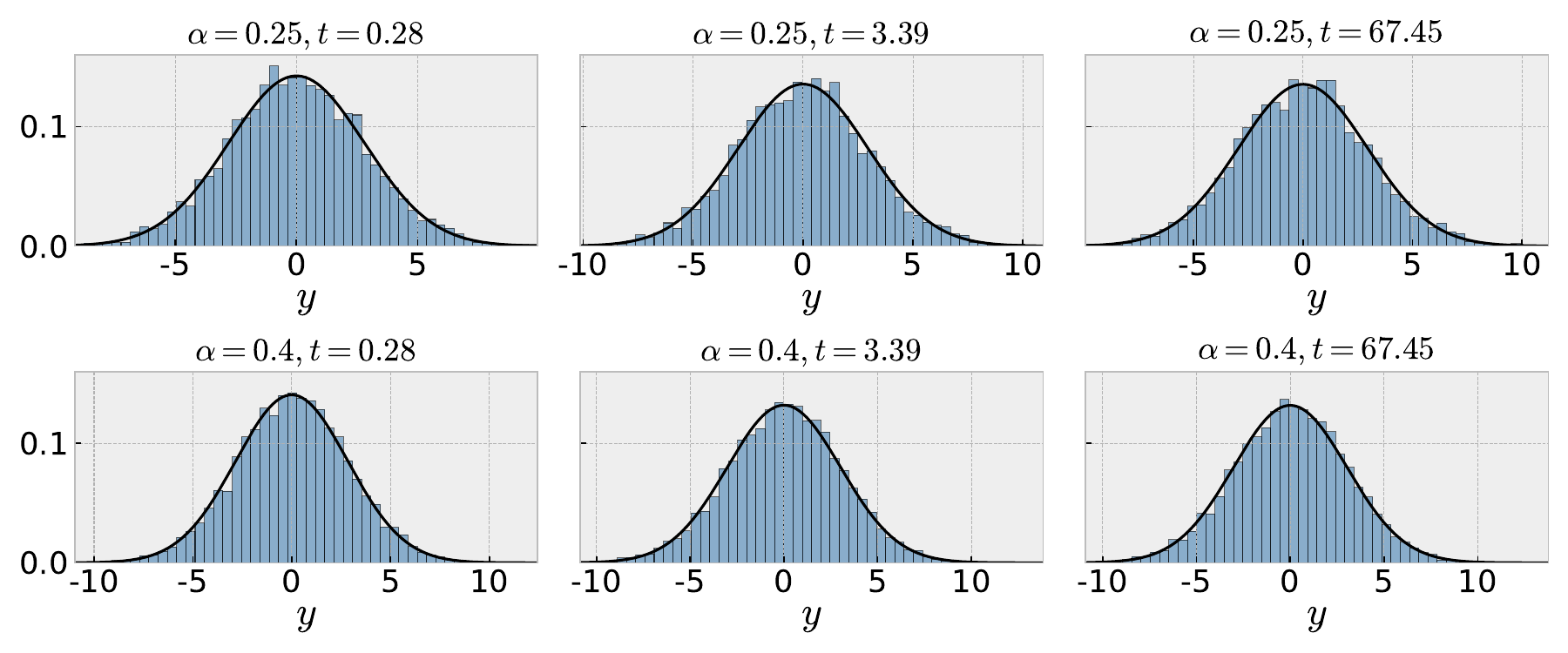}
    \vspace*{-0.7cm}
    \caption{Empirical distribution of the student labels $y_k(t) = \tr \big( W(t)W(t)^\top X_k \big)$ during optimization with gradient descent, defined in equation \eqref{eq:GDdynamics}, with parameters $\kappa = 0.4, \kappa^* = 0.3$, $d = 150$, quadratic cost and no regularization, for two values of $\alpha$ and three values of time $t$. The black curve corresponds to the Gaussian density with zero mean and variance equal to the empirical variance of the labels.}
    \label{fig:Labels}
\end{figure}

Finally, we remark that in equation \eqref{eq:DMFT1Y} the label $\overline y(t)$ is explicitly expressed as a linear combination of the independent centered Gaussian variables $y^*, \zeta, \xi(t)$. Therefore, $(\overline y(t))_{t\geq 0}$ is itself a Gaussian process with zero mean. This shows that in the high-dimensional limit, a label drawn at random remains Gaussian along the gradient flow dynamics. \cref{fig:Labels} presents a numerical confirmation of this result. 

\subsubsection{Recovering Population Dynamics} \label{Subsubsec:Population}

A natural setting to simplify the equations corresponds to the population case, where the student has access to an infinite number of observations. In the expression of the loss in equation~\eqref{eq:Loss}, this corresponds to replacing the average over the $n$ data points with the expectation over the data distribution. Regarding our high-dimensional dynamics, these equations can be recovered by taking the $\alpha \to \infty$ limit in those of \cref{Result1}. In this case we obtain that the Gaussian noise $\G(t)$ vanishes and the response $R$ simplifies to $R(t,t') = \delta(t-t')$, i.e., the memory effect disappears and equations~\eqref{eq:DMFT1W}, \eqref{eq:DMFT1Y} simply write:
\begin{align}
    \d \overline W(t) &= 2 \Big( Z^* - \overline W(t) \overline W(t)^\top \Big) \overline W(t) \d t - \nabla \Omega \big( \overline W(t) \big) \d t + \frac{1}{\sqrt{\beta d}} \d B(t), \label{eq:PopulationW} \\
    \overline y(t) &= \frac{m_Z(t)}{Q_*} y^* + \xi(t). \label{eq:Populationy}
\end{align}
This first equation on $\overline W$ is precisely the one we would obtain by writing the Langevin dynamics~\eqref{eq:GFdynamics} on the population loss. This evolution is very similar to those found in the population limit for shallow quadratic neural networks \citep{gamarnik2019stationary, sarao2020optimization, martin2024impact}. In addition, it can be shown that the equation on the typical label $\overline y(t)$ corresponds to the evolution of the student matrix $\overline W(t) \overline W(t)^\top$ projected on a random direction: indeed, when it has access to an infinite number of observations, the student does not correlate with any particular example. 

In the gradient flow setting ($\beta = \infty$), and with the choice of $\ell_2$-regularization, the dynamics~\eqref{eq:PopulationW} on $W(t)$ can be interpreted as an Oja flow (see \cref{Sec:OjaFlow} for a definition and properties). In the following section, we study the equations given in \cref{Result1} in the long-time limit. One special case of our analysis will cover this population limit $\alpha \to \infty$. Further analysis can be found in \cref{App:Subsec:Population}. 

\subsection{Long-Time Analysis of the Regularized Dynamics} \label{Subsec:LongTimes}

This section is dedicated to the study of the long-time asymptotics of the set of equations from \cref{Result1}, with the choice of $\ell_2$-regularization:
\begin{equation}
    \Omega(W) = \lambda \tr(WW^\top). 
\end{equation}
The following results focus on the gradient flow setting ($\beta = \infty$), and we refer to \cref{Subsubsec:Langevin} for an analysis of the Langevin dynamics at finite temperature.

The goal is to understand the $t \to \infty$ limit of the gradient flow dynamics~\eqref{eq:GFdynamics}, and to describe it in the high-dimensional limit. To do so, we start from the dynamical equations of \cref{Result1}. Such equations, often referred to as DMFT equations, are known to be difficult to analyze and typically require a guess on the structure of the dynamics. In what follows, we introduce such a guess, which we refer to as the \textit{steady-state assumption}.

\subsubsection{Steady-State Assumption} \label{Subsubsec:SteadyState}

The way to simplify the equations is to assume that the memory effect in equation \eqref{eq:DMFT1W} disappears at long times, i.e., the weight of the integral over $[0,t]$ concentrates on times $t'$ that are close to $t$. In this case we say that the dynamics only possesses a short-term memory. We formalize this idea through the following assumption. 

\begin{assumption} \label{Assumption3}

\begin{assumpitem}[Response decay.]\label{Assumption3:Response} The response function $R(t,t')$ decays fast enough to zero as $t-t' \to \infty$.
\end{assumpitem} 
\begin{assumpitem}[Constant noise.]\label{Assumption3:Noise} The Gaussian process $\G(t)$ converges fast enough as $t \to \infty$ so that it can be considered constant in the evolution equation for $W(t)$. 
\end{assumpitem}    
\end{assumption}

Although we do not exactly quantify what fast enough means, we assume that these conditions are strong enough to justify the following long-time approximation of the memory term in \eqref{eq:DMFT1W}:
\begin{equation}
    \int_0^t R(t,t') \Big( \G(t') + Z^* - Z(t') \Big) \d t' \underset{t \to \infty}{\approx} r_\infty \Big( \sqrt{\xi} \G + Z^* - Z(t) \Big),  
\end{equation}
where the constants $\xi, r_\infty$ are such that:
\begin{equation}
    r_\infty = \lim_{t \to \infty} \int_0^t R(t,t') \d t', \hspace{1.5cm} \lim_{t \to \infty} \E \, \G_{ij}(t) \G_{i'j'}(t) = \frac{\xi}{d} \big( \delta_{ii'} \delta_{jj'} + \delta_{i'j} \delta_{ij'} \big),
\end{equation}
and $\G \sim \mathrm{GOE}(d)$. Despite being a strong assumption on the structure of the dynamics, it is directly motivated by a line of work on DMFT and generalized Langevin equations similar to~\eqref{eq:DMFT1W}. More precisely, our approach is closely related to the \textit{time-translational invariance} (TTI) approximation, which has been used to analyze the asymptotic behavior of DMFT equations \citep{sompolinsky1982relaxational, sompolinsky1988chaos, bordelon2024dynamical}. In these works, such assumptions are introduced as physically motivated ansätze and later verified through consistency checks and numerical simulations. Also note that related approaches have been studied rigorously in simpler settings \citep{celentano2021high, fan2025dynamical, chen2025learning}.

Our assumption should be understood in the same spirit: it represents a conjectured structural property of the long-time dynamics, and is natural when describing systems that rapidly reach a steady-state regime. Importantly, the consequences of this simplification will be systematically compared to high-dimensional numerical simulations of the gradient flow dynamics~\eqref{eq:GFdynamics}. We refer to \cref{App:Subsec:SteadyState} for a more detailed discussion of this assumption. 

This assumption leads to an effective dynamics, which we adopt as the starting point for the long-time analysis:
\begin{equation} \label{eq:DynamicsSimplified}
    \dot W(t) = 2 r_\infty \Big( \sqrt{\xi} \G + Z^* - W(t) W(t)^\top \Big) W(t) - 2 \lambda W(t).
\end{equation}
In addition, the self-consistent expressions of the covariance of $\G$ and the function $R$ in equations~\eqref{eq:Covariance1G} and \eqref{eq:AveragedQuantities} lead to the equations on $\xi, r_\infty$:
\begin{align}
    \xi &= \frac{1}{2\alpha} \left( \mathrm{MSE} + \frac{\Delta}{2} \right), \label{eq:SCxi} \\
    r_\infty &= 1 - \frac{1}{\alpha d^2} \lim_{t \to \infty} \int_0^t \tr \left( \left. \frac{\partial \, \E \, Z(t)}{\partial H(t')} \right|_{H = 0} \right) \d t', \label{eq:SCr}  
\end{align}
where $H$ is a perturbation entering additively in the drift term multiplying $W(t)$ in equation~\eqref{eq:DynamicsSimplified}. Then, one can express the solution of the dynamics~\eqref{eq:DynamicsSimplified} as a function of the variables $\xi, r_\infty$. Using the expression of $\xi$ and the definition of $r_\infty$, one can deduce self-consistent equations on these two variables. 

The dynamics in equation~\eqref{eq:DynamicsSimplified} is known as an Oja flow, a nonlinear matrix flow that has been studied in prior works \citep{yan1994global, bodin2023gradient, martin2024impact}. Despite its nonlinearity, this equation admits a closed-form solution, and its convergence properties are well understood. We devote \cref{Sec:OjaFlow} to the study of this flow and provide new results that allow us to derive a closed system of finite-dimensional equations for $r_\infty$ and $\xi$. We present these equations in \cref{Result2} and derive them in \cref{App:LongTimes}. 

\paragraph{A denoising formulation.} Interestingly, the dynamics~\eqref{eq:DynamicsSimplified} can be viewed as a denoising problem of the matrix $Z^*$ corrupted by the Gaussian noise $\G$, solved through regularized gradient flow. This interpretation is similar to the one in the replica calculation done by \citet{maillard2024bayes} for the same model, where the inference problem was mapped onto a matrix denoising formulation in the Bayes-optimal setting. 

Remark that this denoising formulation is very similar to the population dynamics given in equation \eqref{eq:PopulationW} when considering $\Omega$ to be the $\ell_2$-regularization and the gradient flow setting ($\beta = \infty$). Equation \eqref{eq:DynamicsSimplified} introduces two additional parameters: an effective noise variance $\xi$ and a time reparameterization $r_\infty$. Interestingly, the expression for $\xi$ in equation \eqref{eq:SCxi} suggests that the noise arises from the finite number of training samples (through the parameter $\alpha$), and the label noise $\Delta$, which prevents a clean observation of the teacher labels.

At long times, the simplified dynamics reveal that, near the point of convergence, the landscape of the regularized empirical loss~\eqref{eq:Loss} exhibits the same landscape structure as a regularized matrix denoising problem. In this regime, the quantity $r_\infty$ plays a central role: it represents the local curvature of the landscape around the limiting point and thus quantifies the sharpness of the minimum and the associated convergence rate.

\subsubsection{Set of Equations at Long Times} \label{Subsubsec:EquationsLongTimes}

In the following, we give the set of equations resulting from the steady-state assumption (see \cref{App:LongTimes} for their derivation). Before doing so, we define the following operator: for a symmetric matrix $A \in \mathcal{S}_d(\R)$ with spectral decomposition $A = U \mathrm{diag}(\lambda_1, \dots, \lambda_d) U^\top$ and $\lambda_1 \geq \cdots \geq \lambda_d$, we define for $m \le d$:
\begin{equation} \label{eq:defPositivePart}
A_{(m)}^{+} = U \mathrm{diag}\big(\lambda_1^+, \dots, \lambda_m^+, 0,\dots, 0 \big) U^\top,
\hspace{1.2cm} \lambda^+ = \max(\lambda, 0).
\end{equation}
The matrix $A_{(m)}^{+}$ selects the $m$ largest positive eigenvalues of $A$, and is known to be the best rank-$m$ positive semidefinite approximation of $A$ for the Frobenius norm. 

As a consequence of the simplifications introduced in \cref{Subsubsec:SteadyState}, the long-time behavior of the dynamics can be characterized as follows.
    
\begin{claim} \label{Result2}
    Consider the variables $\xi, r_\infty$ defined in equation~\eqref{eq:SCxi}, \eqref{eq:SCr}, and set $q = \lambda / r_\infty$. Define $\mu_\xi$ to be the asymptotic spectral distribution of the random matrix $Z^* + \sqrt{\xi} \G$, where $\G \sim \mathrm{GOE}(d)$. Then, under \cref{Assumption1}, \ref{Assumption2} and \ref{Assumption3} in the $d \to \infty$ limit, $\xi, q$ solve the equations:
    \begin{subequations} \label{eq:SystemResult2}
    \begin{align}
        \min(\kappa, 1) &= \int_\omega \d \mu_\xi(x), \label{eq:ResultLongTimesOmega} \\
        1 &= \frac{\lambda}{q} + \frac{1}{\alpha} \int_{\max(q, \omega)} (x-q) h_\xi(x) \d \mu_\xi(x), \label{eq:ResultLongTimesIntegral} \\
        2\alpha \xi - \frac{\Delta}{2} &= Q_* + \int_{\max(q,\omega)} (q^2 - x^2) \d \mu_\xi(x) + 4 \xi \int_{\max(q,\omega)} (x-q) h_\xi(x) \d \mu_\xi(x), \label{eq:ResultLongTimesMSE} 
    \end{align}
    \end{subequations}
    where $h_\xi$ is the Hilbert transform of $\mu_\xi$ (see \cref{Def:HilbertStieltjes}) and:
    \begin{equation}
        Q_* = \lim_{d \to \infty} \frac{1}{d} \E \, \tr(Z^{*2}) = \int x^2 \d \mu^*(x). 
    \end{equation}
    Moreover, for almost all initializations, the limit of the dynamics~\eqref{eq:DMFT1W} is given by:
    \begin{equation} \label{eq:ZinfinityReg}
        Z_\infty = \Big( Z^* + \sqrt{\xi} \G - q I_d \Big)_{(m)}^+, 
    \end{equation}
    with $\G \sim \mathrm{GOE}(d)$. The MSE and the training loss are given by:
    \begin{equation} \label{eq:LossValue}
        \mathrm{MSE} = 2 \alpha \xi - \frac{\Delta}{2}, \hspace{1.5cm} \mathrm{Loss}_\mathrm{train} = \frac{\lambda^2 \alpha \xi}{q^2}.   
    \end{equation}
\end{claim}

More precisely, the measure $\mu_\xi$ corresponds to the free additive convolution between $\mu^*$, the asymptotic spectral measure of the teacher $Z^*$ and a semicircular density of variance $\xi$ (see \cite{biane1997free} and \cref{App:Subsubsec:FreeConvolution} for more details). 

In contrast with the high-dimensional formulation of \cref{Result1}, the system~\eqref{eq:SystemResult2} is now finite-dimensional and involves only scalar quantities. The ambient dimension no longer appears explicitly, and the long-time limit of the dynamics is described by a small number of order parameters. 

One of the key results is the limit of the flow found in equation~\eqref{eq:ZinfinityReg}. It shows that, asymptotically, the student matrix selects the $m$ largest positive eigenvalues of a noisy version of the teacher matrix, with an eigenvalue shift that is characteristic of the regularized dynamics. In a similar fashion, in the system of equations~\eqref{eq:SystemResult2}, the variable $\omega$ selects a mass $\kappa$ of the measure~$\mu_\xi$.

\begin{figure}[ht]
    \centering
    \includegraphics[width=\linewidth]{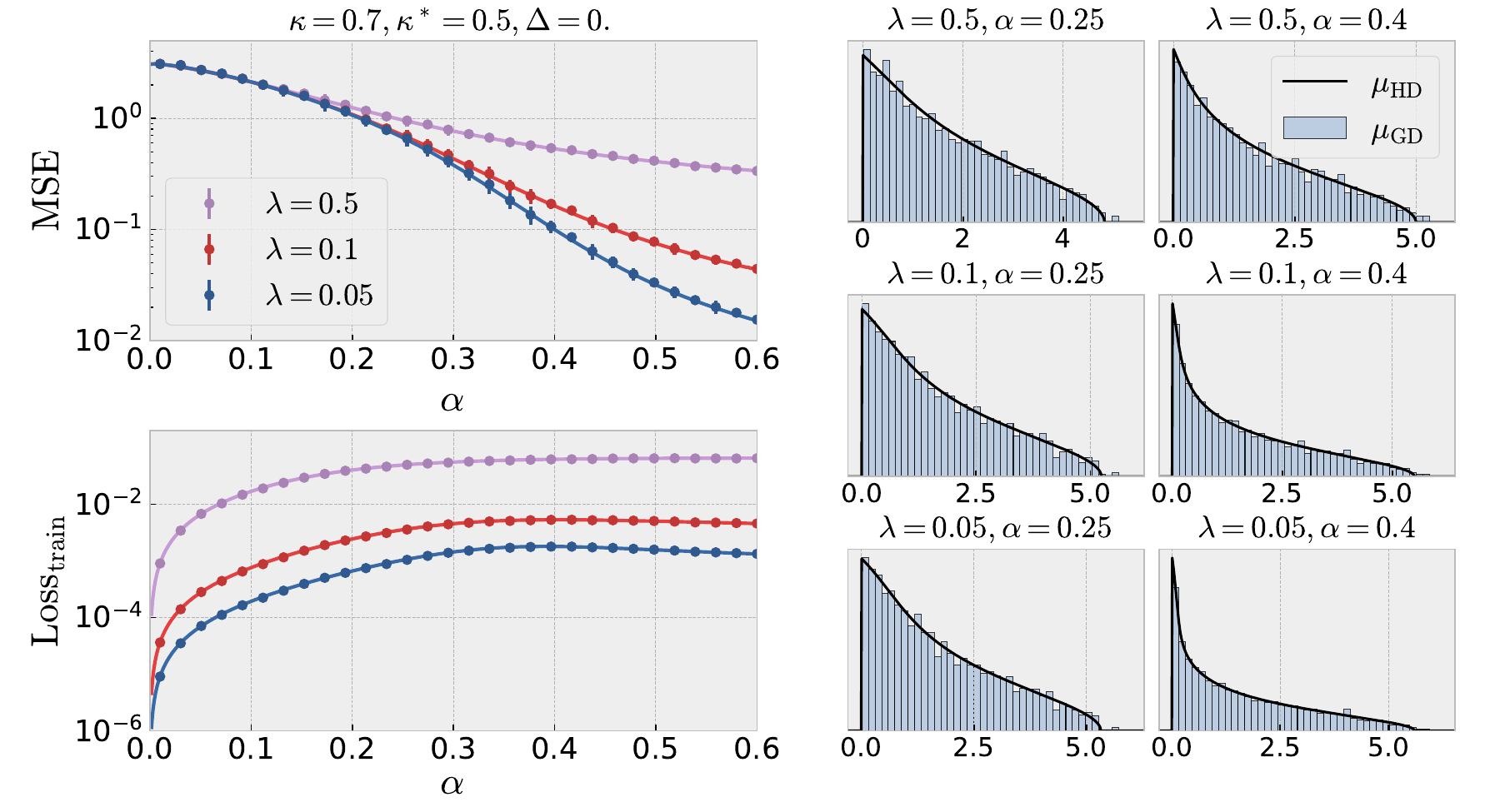}
    \vspace*{-0.6cm}
    \caption{Comparison between simulations of gradient descent, defined in equation \eqref{eq:GDdynamics}, and numerical integration of the system of equations~\eqref{eq:SystemResult2}, for $\kappa = 0.7, \kappa^* = 0.5$ and zero label noise. Gradient descent results are averaged over 10 (left) and 50 (right) realizations of the initialization, teacher and data. Left: MSE and empirical loss value as a function of $\alpha$, for different values of $\lambda$. Dots correspond to GD simulations and full lines to the solution of equations \eqref{eq:SystemResult2}. Right: eigenvalue distribution of $Z = WW^\top$ reached by GD (blue), restricted to its nonzero eigenvalues, and associated density computed with the asymptotic spectral distribution of the matrix in equation~\eqref{eq:ZinfinityReg} (black line) for three values of $\lambda$ and two values of $\alpha$.}
    \label{fig:FigReg1}
\end{figure}

\cref{Result2} is based on assumptions regarding the long-time behavior of the gradient flow dynamics: it is still an open question to prove rigorously that these assumptions are verified. However, as a result of a large number of numerical simulations, we believe that this result holds for any value in our set of parameters $\kappa, \kappa^*, \alpha, \lambda, \Delta$, as soon as the regularization strength remains positive. In addition to \cref{fig:FigReg1,fig:FigReg2} that compare the equations of \cref{Result2} with the results of gradient descent simulations and show excellent agreement, we provide additional numerical evidence in \cref{App:Subsubsec:SimusLearningCurves}. Further details on how to simulate the system of equations~\eqref{eq:SystemResult2} are given in \cref{App:Subsubsec:SimulationSystem}. 

Additionally, let us remark that the assumption we made on the dynamics may not be specific to the choice of our setting: $\ell_2$-regularization, Gaussian label noise and quadratic cost. In the general gradient flow case, the dynamics obtained in \cref{Result0} can also be approximated by a similar dynamics. This should lead to a comparable, although more complicated, set of equations as the one we present here. A sizable challenge would then be to validate numerically or theoretically these approximations in the general case.  

For completeness, and in a spirit of coherence with the results of \cref{subsec:DMFT}, in which the high-dimensional limit is first taken before the long-time limit, we also show that as soon as the dynamics is approximated by equation~\eqref{eq:DynamicsSimplified}, the set of equations we derive is the same no matter in which order the limits are taken. This result also ensures robustness regarding the behavior of the dynamics~\eqref{eq:DynamicsSimplified}: the only relevant timescale for the dynamics is the one that we study in this section. More details on this can be found in \cref{App:Subsec:ResponseInfiniteDim}.

\paragraph{Universality over the teacher's distribution.} 
Interestingly, it appears that the system of equations \eqref{eq:SystemResult2} holds no matter the choice of the teacher distribution, provided that its spectral density converges as $d \to \infty$. However, there are several settings of interest that this result does not include but could be potentially generalized to:
\begin{itemize}
    \item \textbf{Power-law teacher.} In the case where the spectrum of $Z^*$ exhibits a power-law behavior, one has to take into account finite-dimensional corrections to obtain a contribution from the large eigenvalues of the teacher. This generalization has successfully been applied by \citet{defilippis2025scaling} for the empirical risk minimization problem in the same setting as ours. 
    \item \textbf{Finite-width teacher.} Although we consider an extensive-width teacher ($m^* \sim \kappa^* d$), \cref{Result2} should remain valid when $m^*$ remains of order one, in which case the teacher's spectral distribution $\mu^*$ would collapse onto a Dirac mass at zero. As it has been shown by \citet{sarao2020optimization,bonnaire2025role} in the case $m^* = 1$, this setting only requires a number of observations proportional to $d$ (and not $d^2$ as in our case). Therefore, we conjecture that our setting is unable to capture the finite-width case. 
\end{itemize}

\subsubsection{Overparameterization and Global Optimality} \label{Subsubsec:Regions}

We shall now give some remarks on the impact of the overparameterization of the student network (which is controlled by the parameter $\kappa$) on the performance of the gradient flow estimator.

First of all, remark that our system of equations~\eqref{eq:SystemResult2} only depends on $\kappa$ through the threshold~$\omega$ that selects the $m$ largest positive eigenvalues of the noisy teacher $Z^* + \sqrt{\xi} \G$. Due to the dependence of equations~\eqref{eq:ResultLongTimesIntegral}, \eqref{eq:ResultLongTimesMSE} on $\omega$, it can be shown that for a given set of parameters $\kappa^*, \lambda, \alpha, \Delta$, there exists a value $\kappa_{\min}$ such that, as soon as $\kappa \geq \kappa_{\min}$, the solution of the system~\eqref{eq:SystemResult2} does not depend on $\kappa$ and is the same in the case $\kappa \geq 1$. In this region, gradient flow is able to converge to the global minimizer of the training loss over the set of PSD matrices. This solution has rank $\sim \kappa_{\min} d$, which can be expressed:
\begin{equation} \label{eq:kappamin}
    \kappa_{\min} = \int_q \d \mu_\xi(x).
\end{equation}
In this equation, the parameters $q$ and $\xi$ are obtained by solving equations~\eqref{eq:ResultLongTimesIntegral}, \eqref{eq:ResultLongTimesMSE} at $\kappa = 1$, for fixed values of the parameters $\alpha, \kappa^*, \lambda, \Delta$. Therefore, only a mild overparameterization ($\kappa \geq \kappa_{\min}$, compared to $\kappa \geq 1$) is necessary to reach the global minimizer of the regularized loss. This conclusion is non-trivial: for general functions of the quadratic form $WW^\top$, it is not guaranteed that the student matrix $W$ finds the global minimizer of the loss over all PSD matrices, even if it has a high-enough rank to recover it. We give more details on this result in \cref{App:Subsec:AnalysisEquations}.

On the other hand, for $\kappa \leq \kappa_{\min}$, the rank of the student is too small to recover the global minimizer (in terms of $Z = WW^\top$). Then, gradient flow seems to converge to a solution that depends on $\kappa$ (and has maximal rank). Although we cannot directly conclude that this solution corresponds to a global minimizer of the regularized loss (now expressed in terms of $W$), the analysis of the Langevin dynamics in \cref{Subsubsec:Langevin} in the low temperature regime suggests that this is indeed the case. We refer to this section for more details. 

\begin{figure}[ht]
    \centering
    \includegraphics[width=\linewidth]{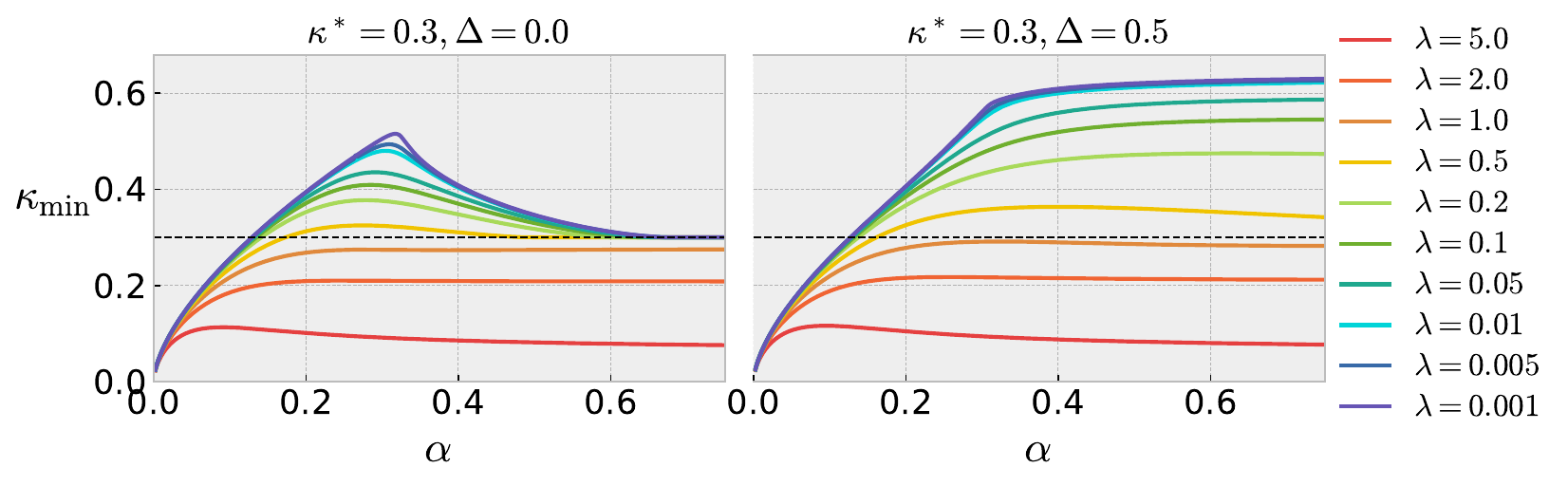}
    \vspace*{-0.7cm}
    \caption{Rank of the solution $\kappa_{\min}$ obtained with $\kappa = 1$, as a function of $\alpha$ and for different values of $\lambda$. $\kappa^* = 0.3$ (horizontal dashed line), $\Delta = 0$ (left) and $\Delta = 0.5$ (right). Curves obtained by simulating the system~\eqref{eq:SystemResult2} and using equation \eqref{eq:kappamin} to compute $\kappa_{\min}$. Above these curves, the solution found by gradient flow does not depend on $\kappa$.}
    \label{fig:FigRank}
\end{figure}

In \cref{fig:FigRank} we plot the threshold $\kappa_{\min}$ as a function of $\alpha$, for a wide range of $\lambda$, and with $\kappa^* = 0.3$. The first observation is that this function decreases when increasing $\lambda$: a stronger regularization leads to a lower-rank global minimizer. This behavior is not a surprise, since this minimizer is obtained by optimizing the regularized loss over all positive semidefinite matrices $Z$. In this case, the $\ell_2$-regularization on $W$ translates into a nuclear norm penalty for $Z = WW^\top$, which is known to favor low-rank solutions \citep{fazel2001rank, recht2010guaranteed}. In addition, in the presence of label noise (right panel), the rank of the global minimizer tends to increase, especially for larger values of $\alpha$: in this region, the model needs more degrees of freedom to compensate for the variability induced by the noise.

\paragraph{Link with empirical risk minimization.} \citet{erba2025nuclear} studied the global minimizer of the regularized empirical loss, in the case $\kappa \geq 1$. In this regime, it is known that the gradient flow always converges to the global minimizer of the loss (over all PSD matrices $Z = WW^\top$). In \cref{App:Subsec:ERM}, we show that the system of equations we derive in \cref{Result2} matches theirs. This agreement, although expected, is of interest as the two sets of equations were derived using different approaches. In their work, the result follows from an exact analysis of the fixed point equations associated with an approximate message passing (AMP) iteration. In addition, the authors provide a study of the stability of the AMP fixed point and derive a condition that coincides with ours for the steady-state assumption. We show that whenever $\kappa \geq 1$, the stability condition is met. 

\subsubsection{Stability of the Steady-State Solution} \label{Subsubsec:Stability}

In \cref{App:Stability}, we provide theoretical insights to assess whether the steady-state assumption holds. To do so, a common approach is to study the response operator (also known as the susceptibility) associated with a perturbation of the steady-state solution \citep[see for instance][]{mezard1987spin}. This operator characterizes the robustness and convergence of the dynamics under a perturbation, which is particularly relevant since the steady-state dynamics~\eqref{eq:DynamicsSimplified} was derived in a perturbative way from the high-dimensional system in \cref{Result1}. However, this approach only covers the steady-state dynamics, not its stability with respect to the high-dimensional system of equations of \cref{Result1}. 

In our case, the susceptibility operator is defined as:
\begin{equation}
    \mathcal{X} = \left. \frac{\partial Z_\infty}{\partial H} \right|_{H = 0}, \hspace{1.5cm} Z_\infty = \Big( Z^* + \sqrt{\xi} \G - q I_d + H \Big)_{(m)}^+. 
\end{equation} 
$Z_\infty$ is the limit of the steady-state dynamics obtained under a perturbation $H$, and the susceptibility can simply be interpreted as the differential of the map $H \in \mathcal{S}_d(\R) \mapsto Z_\infty \in \mathcal{S}_d(\R)$. In \cref{App:Subsec:Susceptibility}, we analyze both the spectrum and the normalized Frobenius norm of $\mathcal{X}$, that allows to investigate stability with respect to average and worst-case perturbations. Overall stability is guaranteed as soon as the spectrum of $\mathcal{X}$ (or its Frobenius norm) remains bounded. We derive the following results:
\begin{itemize}
    \item In the region where $\kappa \geq \kappa_{\min}$, the spectrum of the susceptibility remains in $[0,1]$, and it can be shown that the steady-state dynamics remains stable both at finite and infinite dimension.
    \item When $\kappa \leq \kappa_{\min}$, there exists a small proportion of unstable modes, with susceptibility eigenvalues diverging with the dimension. Overall, this leads to an average susceptibility of order $\log d$. This suggests that at finite $d$, the steady-state dynamics remains stable, but with a potential instability occurring in the high-dimensional limit. 
\end{itemize}
However, this does not imply that the earlier approximation fails. Stability should instead be evaluated with respect to the original high-dimensional dynamical equations, rather than the reduced steady-state equations alone. A more accurate analysis would require treating perturbatively the system of \cref{Result1}, and linearize the dynamics around the steady-state solution. In \cref{App:Subsec:DynamicalStability}, we explain how such a calculation can be carried out. In this setting, the specific structure of the high-dimensional perturbation may regularize the unstable modes observed in the steady-state solution. 

Based on extensive numerical simulations (see \cref{App:Subsec:AdditionalExperiments}), we believe that this is indeed what happens in practice. This conclusion is supported by systematic comparisons between the theoretical predictions and empirical learning curves for averaged quantities, such as MSE and training loss, as well as eigenvalue distributions of the gradient flow predictors. Across a wide range of parameter values, we observe an excellent agreement. 

To support this claim, we compare in \cref{fig:FigReg2} the behavior of gradient descent at convergence with the numerical integration of the system~\eqref{eq:SystemResult2}. The choice of $\kappa = 0.35$ and $\kappa^* = 0.3$ ensures that most of the values of $\alpha$ fall in the regime $\kappa \leq \kappa_{\min}$ (this can be checked in \cref{fig:FigRank}). As shown in \cref{fig:FigReg2}, the agreement between theory and simulations is excellent, both for averaged quantities (MSE and empirical loss, left panel) and spectral distribution (right panel). In addition, for large values of $\alpha$, the spectral distribution of the student matrix develops a spike away from zero. This effect can be understood from equation~\eqref{eq:ZinfinityReg}: when $\xi$ is small (corresponding to large $\alpha$ and near-zero MSE), the spectrum of $Z^* + \sqrt{\xi} \G$ splits into two bulks. Since $\kappa$ is slightly larger than $\kappa^*$, the student matrix recovers the bulk associated with $Z^*$, along with a small fraction of the second, corresponding to a Gaussian matrix with small variance, producing the observed spike. 

\begin{figure}[ht]
    \centering
    \includegraphics[width=\linewidth]{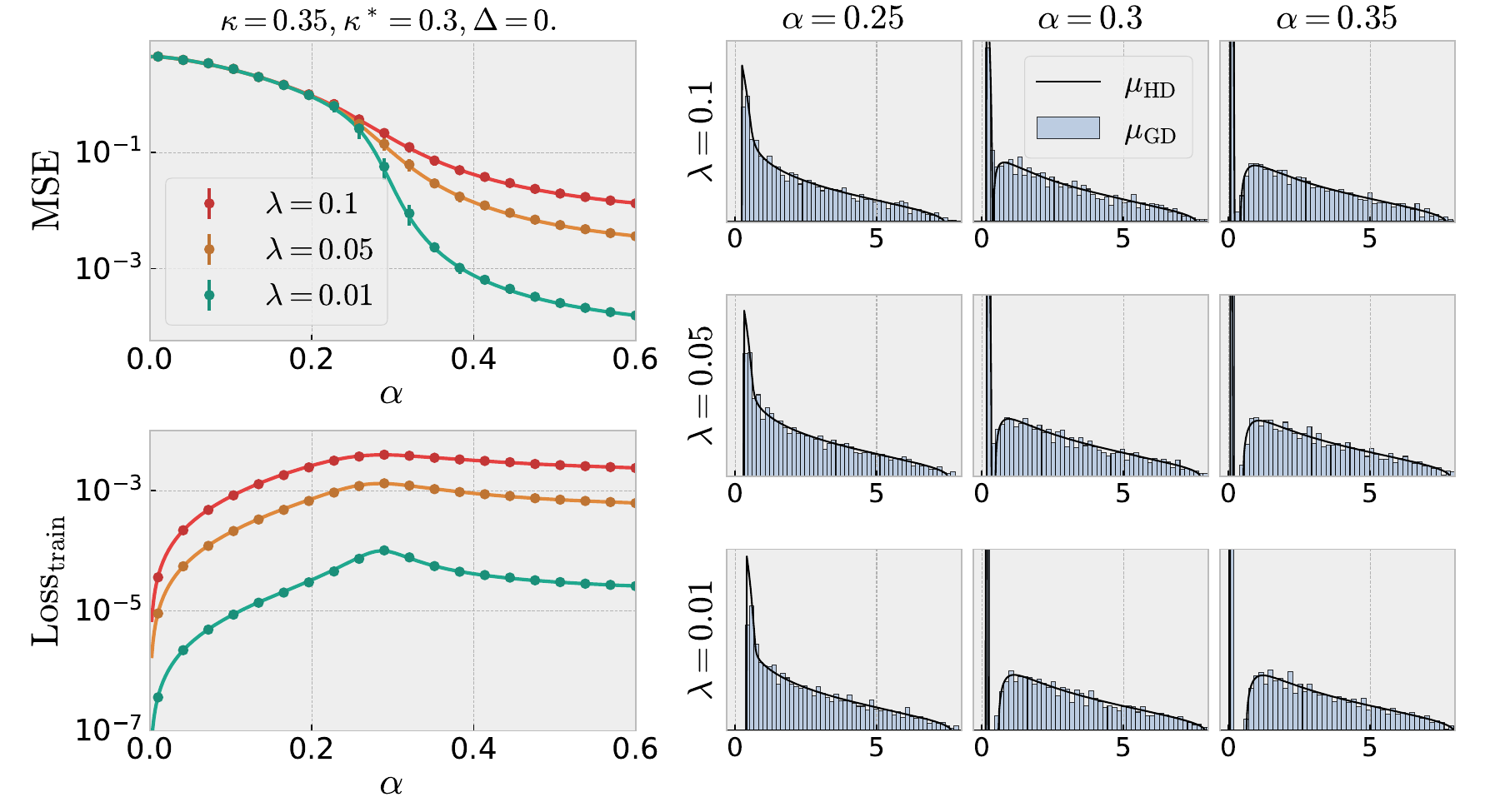}
    \vspace*{-0.7cm}
    \caption{Comparison between simulations of gradient descent, defined in equation \eqref{eq:GDdynamics}, and numerical integration of the system of equations~\eqref{eq:SystemResult2}, for $\kappa = 0.35, \kappa^* = 0.3$ and zero label noise. Gradient descent results are averaged over 10 (left) and 50 (right) realizations of the initialization, teacher and data. Left: MSE and empirical loss value as a function of $\alpha$ for different values of $\lambda$. Dots correspond to GD simulations and full lines to the solution of the system~\eqref{eq:SystemResult2}. Right: eigenvalue distribution of $Z = WW^\top$ reached by GD (blue), restricted to its nonzero eigenvalues, and associated density computed with the asymptotic spectral distribution of the matrix in equation~\eqref{eq:ZinfinityReg} (black line) for three values of $\lambda$ and $\alpha$.}
    \label{fig:FigReg2}
\end{figure}

Finally, a complementary approach to assess the validity of the steady-state approximation is to check \textit{a posteriori} whether the assumptions made in \cref{Subsubsec:SteadyState} are satisfied. In \cref{App:Subsec:SteadyState}, we qualitatively relate \cref{Assumption3} to the fast convergence of the matrix $Z(t)$. This motivates the study of the convergence rates associated with the steady-state solution. In finite dimension, we show in \cref{App:Subsubsec:FiniteDimRates} that the convergence is exponentially fast, but that there exists a few directions with a relaxation time diverging with the dimension. As a consequence, in \cref{App:Subsubsec:InfiniteDimRates}, we study these convergence rates after taking the high-dimensional limit and show that in this case they are degraded into a power-law decay. More precisely, we show that:
\begin{equation}
    \lim_{d \to \infty} \frac{1}{d} \big\| Z(t) - Z_\infty \big\|_F^2 \underset{t \to \infty}{=} \left\{ \begin{array}{cc}
        \Theta(t^{-3}), & \text{if} \ \kappa > \kappa_{\min}, \\
        \Theta(t^{-1}), & \text{if} \ \kappa < \kappa_{\min}.
    \end{array} \right. 
\end{equation}
In the underparameterized region, the convergence is much slower, hence unveiling the existence of a new dynamical regime. To the best of our knowledge, these asymptotics were not known before and come as new instances of scaling laws for optimization dynamics. 

\paragraph{Beyond the steady-state ansatz.} In a complementary approach to investigate the validity of our assumptions, we propose in \cref{App:Aging} a more general approximation of the dynamics, which is often referred to as \textit{aging} in the statistical physics literature \citep[see for instance][]{cugliandolo1993analytical, arous2001aging, sarao2019afraid, altieri2020dynamical}. The idea is to decompose the dynamics between a steady-state part and another regime which is very slow. Assuming a separation of timescales as well as a quasi-static equilibrium for the slow dynamics, we derive a more general set of self-consistent equations than the one in \cref{Result2}. However, these equations involve a matrix-valued distribution whose analysis in the high-dimensional limit is not tractable in general. While we do not analyze these equations further, the excellent agreement between our numerical simulations and the steady-state solution suggests that this more general solution coincides with the steady-state one (physically, one would say that aging is absent). Making this identification explicit from the aging equations is left for future work. 

\subsubsection{Overfitting and Double Descent} \label{Subsubsec:Overfitting}

Several numerical simulations suggest the presence of overfitting during the dynamics, that is, a positive gap:
\begin{equation}
    \delta_\text{MSE} = \mathrm{MSE}_\infty - \inf_{t \geq 0} \mathrm{MSE}(t).
\end{equation}
Our simulations of the gradient descent algorithm \eqref{eq:GDdynamics} show that this phenomenon already appears in the noiseless setting but is amplified in the presence of label noise ($\Delta > 0$). \cref{fig:FigOverfitting1} features the MSE as a function of time for two different label noises $\Delta > 0$, and for different values of $\alpha$, revealing the overfitting phenomenon. It shows that for small values of~$\alpha$, when the student learns with few data, the model exhibits mild overfitting. As the sample complexity increases, the gap between the final MSE and its time-optimal value grows, revealing a progressively stronger overfitting.

\begin{figure}[ht]
    \centering
    \includegraphics[width=\linewidth]{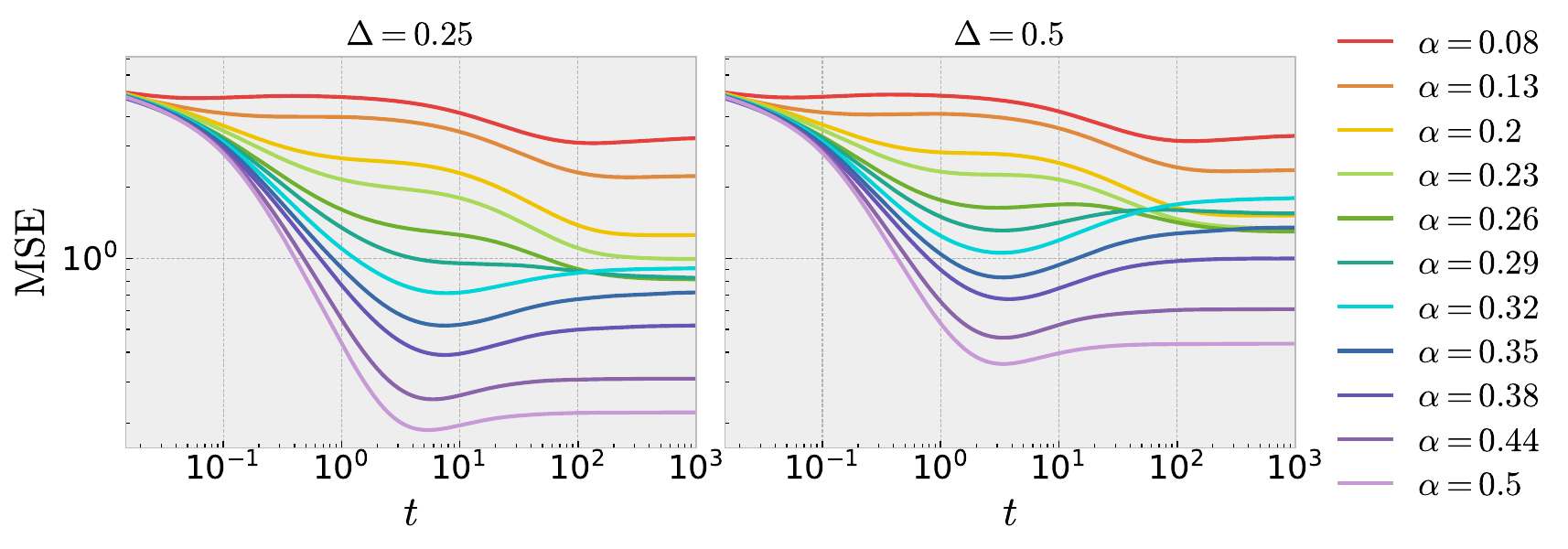}
    \vspace*{-0.7cm}
    \caption{MSE as a function of time for gradient descent trajectories (see equation \ref{eq:GDdynamics}) averaged over 10 realizations of the initialization, teacher and data. Parameters $\kappa = 0.4, \kappa^* = 0.3, \lambda = 0.01$, $\Delta = 0.25$ (left) and $0.5$ (right). Overfitting is characterized by portions where the MSE increases with time.}
    \label{fig:FigOverfitting1}
\end{figure}

In addition, for large values of $\Delta$ (right panel in \cref{fig:FigOverfitting1}), the monotonicity of the MSE with $\alpha$ breaks. This is characteristic of the double descent phenomenon \citep{belkin2019reconciling,nakkiran2021deep}: as the number of observations increases (up to a certain point), the estimator fits all the data points, leading to poor generalization. For larger values of $\alpha$, fitting is not possible anymore, and the student starts to learn the latent structure of the labels. In this double descent regime, \cref{fig:FigOverfitting2} features the dependence of the MSE on $\alpha$ for different values of label noise $\Delta$ (left panel) and regularization strength $\lambda$ (right panel). While such double descent curves are often plotted as a function of the number of parameters, we use here the sample complexity $\alpha$ to remain consistent with the rest of the paper. The figure also highlights the interpolation peak in the limit $\lambda \to 0^+$ (vertical dashed line), that we derive in \cref{subsec:SmallReg}.

Interestingly, overfitting is present even when regularizing the dynamics, but is reduced when increasing the regularization strength $\lambda$, within the range of values considered in \cref{fig:FigOverfitting2}, right panel. This observation is consistent with known results in linear problems \citep{krogh1991simple, nakkiran2020optimal, mei2022generalization} and modern neural networks \citep{nakkiran2021deep, zhang2019three, d2024we}.

\begin{figure}[ht]
    \centering
    \includegraphics[width=\linewidth]{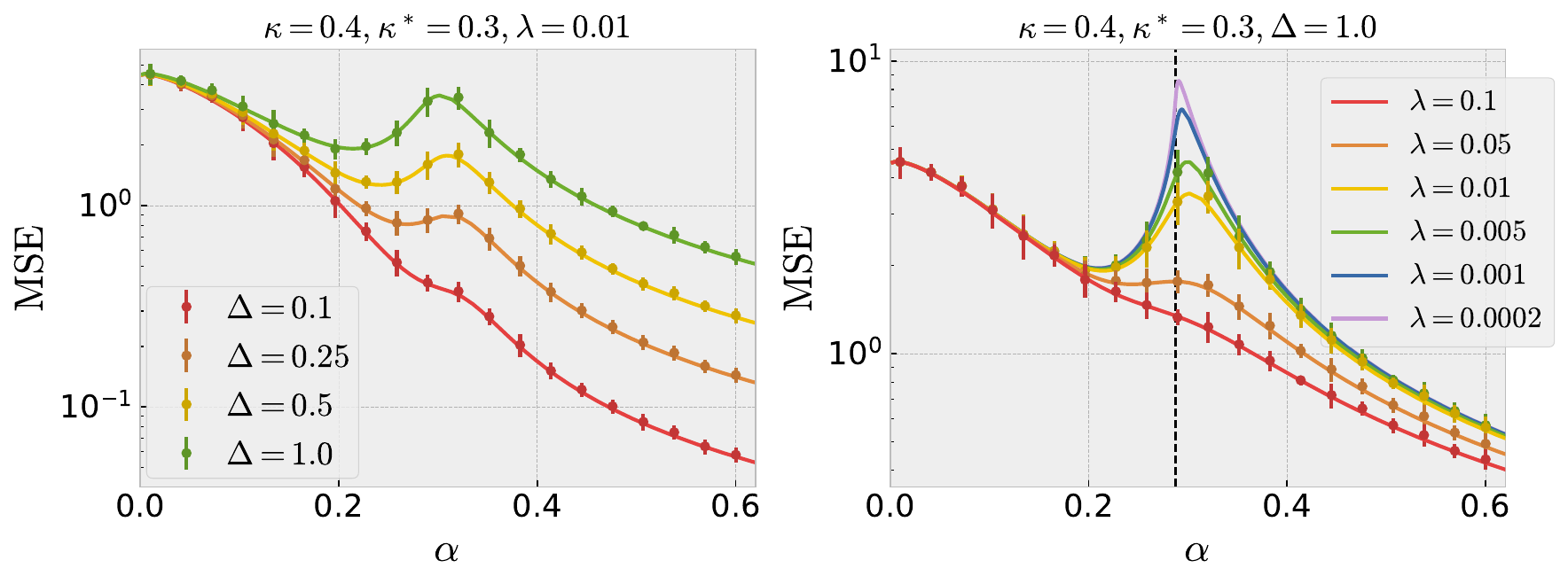}
    \vspace*{-0.6cm}
    \caption{MSE as a function of $\alpha$ for different values of label noise $\Delta$ (left) and regularization strength $\lambda$ (right). Dots: simulations of gradient descent, defined in equation \eqref{eq:GDdynamics}, averaged over 10 realizations of the initialization, teacher and data. Full lines: numerical integration of the system of equations~\eqref{eq:SystemResult2}. Vertical dashed line: interpolation threshold (see \cref{Subsubsec:InterpolationThreshold}).}
    \label{fig:FigOverfitting2}
\end{figure}

Unfortunately, we are not able to precisely characterize the parameter regime in which double descent occurs. For instance, the left panel of \cref{fig:FigOverfitting2} shows that small values of $\Delta$ do not lead to this phenomenon. This suggests that the emergence of double descent may depend on a $\Delta$-dependent scale in the regularization strength $\lambda$. However, the system of equations given in \cref{Result2} matches the empirical curves almost exactly, indicating that double descent is already encoded in the simplified dynamics of equation \eqref{eq:DynamicsSimplified}. Therefore, a more quantitative understanding of overfitting and double descent could, in principle, be obtained by analyzing the generalization properties of these denoising dynamics. 

Overall, our results provide a characterization of double descent in a genuinely nonlinear model. Previous theoretical studies have analyzed this phenomenon in linear regression and least-squares settings \citep{nakkiran2020optimal, belkin2020two, wu2020optimal, derezinski2020exact, hastie2022surprises, bach2024high} and in random feature models \citep{d2020double, d2020triple, gerace2020generalisation, adlam2020neural, adlam2020understanding, mei2022generalization}. Closest to our setting is the recent work of \citet{erba2025nuclear}, who study overparameterized quadratic networks and also observe a double descent phenomenon. However, their analysis builds on a convex relaxation, which allows the problem to be treated within a linear setting.
Our work instead goes beyond linearized models and provides, to the best of our knowledge, the first explicit dynamical characterization of double descent in a nonlinear high-dimensional model.

To conclude this part, let us compare our results with the recent and closely related work of \citet{montanari2025dynamical}, who report that learning and overfitting take place on two different timescales. In
our case, both phenomena arise at the same timescale. A key distinction lies in the choice of activation function. We consider a centered quadratic activation with information exponent two, whereas their analysis focuses on activations with information exponent one. As a consequence, this difference leads the two models to exhibit different learning dynamics. Additional differences include the teacher structure (rank one in their case versus extensive in our setting) and the use of a sequential extensive-width limit, where the number of hidden units is taken to infinity after the dimension and sample size.

\subsubsection{Population Limit}

Let us now consider the population limit, i.e., the regime $\alpha \to \infty$, corresponding to a number of observations $n \gg d^2$. In this case, we can derive the associated limit of the system of equations in \cref{Result2}.

\begin{proposition} \label{Prop:PopulationLimit}
    Consider the case $\kappa \geq \min(\kappa^*, 1)$ and the variables $q, \xi$, solutions of the system of equations~\eqref{eq:SystemResult2}. Then, as $\alpha \to \infty$, we have $q \to \lambda$ and $\xi = \Theta(\alpha^{-1})$. The limit of the gradient flow is then given by:
    \begin{equation}
        Z_\infty = \big( Z^* - \lambda I_d \big)^+,
    \end{equation}
    and the MSE and the training loss write:
    \begin{equation}
        \mathrm{MSE} = \int \min(x, \lambda)^2 \d \mu^*(x), \hspace{1.5cm} \mathrm{Loss}_\mathrm{train} = \frac{1}{2} \mathrm{MSE} + \frac{\Delta}{4}.  
    \end{equation}
\end{proposition}

We prove this result in \cref{App:Subsubsec:PopulationLongTimes}. Interestingly, in the regime where the student has enough parameters to recover the teacher $\kappa \geq \min(\kappa^*, 1)$, these population equations do not depend on $\kappa$ anymore. In addition, since $\mu^*$ is by assumption supported on $\R^+$ (due to the fact that $Z^*$ is positive semidefinite), the MSE vanishes as $\lambda \to 0$: provided that the student has enough rank, it always recovers the teacher in this limit. This result is coherent with prior works on the population loss \citep{sarao2020optimization, martin2024impact}. 

Interestingly, one recovers the same result when taking the limit $n \to \infty$ in the finite-dimensional expression of the loss~\eqref{eq:Loss}, with our choice of quadratic cost, and Gaussian label noise. Indeed, we have:
\begin{equation}
    \lim_{n \to \infty} \L_n(W) = \frac{1}{2d} \big\| WW^\top - Z^* \big\|_F^2 + \frac{\Delta}{4}. 
\end{equation}
Then, one can study the gradient flow associated with this loss, leading to a study of the Oja flow (see \cref{Sec:OjaFlow} for relevant results on this dynamics). Finally, we show that taking the high-dimensional limit in this setting recovers \cref{Prop:PopulationLimit}, leading to the equivalence between the $n, d \to \infty$ limits taken jointly before taking $\alpha = n/d^2 \to \infty$, and the sequential limit $n \to \infty$ then $d \to \infty$. More details can be found in \cref{App:Subsec:Population}.

\subsubsection{Langevin Dynamics} \label{Subsubsec:Langevin}

While the previous results focused on the gradient flow setting, it is natural to ask whether the same analysis can be performed for Langevin dynamics (in the case $\beta < \infty$). In this setting, we are interested in the stationary measure of the stochastic dynamics \eqref{eq:DMFT1W}: due to the Brownian motion, the matrix $W(t)$ does not settle at long times but keeps fluctuating. 

As a consequence, \cref{Assumption3} is not suited for the study of Langevin dynamics, since it assumes that the Gaussian noise $\G(t)$ converges as $t \to \infty$. Instead, we rely on standard assumptions inspired by statistical physics and generalized Langevin equations: time-translational invariance and fluctuation--dissipation. 
\begin{itemize}
    \item \textbf{Time-translational invariance} (TTI) ensures that the correlation of the Gaussian process~$\G$ and the memory kernel $R$ (in equation \ref{eq:DMFT1W}) depend only on time differences. It reflects that the dynamics has reached a stationary regime in which statistical properties no longer drift with time. 
    \item \textbf{Fluctuation--dissipation relation} links the covariance of the Gaussian process $\G$ to the memory kernel $R$ in a physically consistent manner, enforcing a balance between random forcing and dissipation that enables relaxation toward equilibrium. 
\end{itemize}
Together, these conditions ensure the existence of a stationary structure for the dynamics and allow one to show that it is driven toward equilibrium. Such assumptions have been commonly used, and confirmed, in spin-glass models \citep{sompolinsky1982relaxational, altieri2020dynamical} and in high-dimensional learning problems \citep{chen2025learning, fan2025dynamical}. Our precise assumption can be found in \cref{Assumption4}. 

\begin{claim} \label{Result:Langevin}
    Consider the stochastic differential equation \eqref{eq:DMFT1W}, along with the self-consistent equations on the covariance of $\G$ in \eqref{eq:Covariance1G} and the memory kernel $R$ in \eqref{eq:AveragedQuantities}. Under \cref{Assumption4}, the stationary measure of \eqref{eq:DMFT1W} is given by:
    \begin{equation} \label{eq:ResultInvariantMeasure}
        \P_\beta(W) \propto \exp \left( -r \beta d \Big\| WW^\top - Z^* - \sqrt{\xi} \G \Big\|_F^2 -2 \beta d \, \Omega(W) \right),
    \end{equation}
    where $\G \sim \mathrm{GOE}(d)$ and the variables $r, \xi$ are self-consistently computed from $\P_\beta$:
    \begin{align}
        \xi &= \frac{1}{2\alpha} \left( \frac{1}{d} \E_{\G, Z^*} \Big\| \E_\beta \big[WW^\top \big] - Z^* \Big\|_F^2 + \frac{\Delta}{2} \right), \label{eq:XiLangevin} \\
        r &= \frac{\alpha}{\alpha + 2 \beta V_\beta}, \label{eq:rLangevin} \\
        V_\beta &= \frac{1}{d} \E_{\G, Z^*} \E_\beta \Big\| \E_\beta \big[WW^\top \big] - WW^\top \Big\|_F^2,
    \end{align}
    where $\E_\beta$ denotes the expectation with respect to $\P_\beta$. 
\end{claim}

This stationary measure is computed in \cref{App:Subsec:StationaryMeasure}. The main technical tool involved is the mapping of the dynamics \eqref{eq:DMFT1W} onto a coupling with a Gaussian auxiliary matrix process that allows to use standard results for the stationary measure of Langevin dynamics. The self-consistent equations for the variables $\xi, r$ are a direct consequence of the self-consistency of the covariance of $\G(t)$ and $R(t,t')$ in \cref{Result1}. These equations are derived in \cref{App:Subsubsec:SelfConsistentLangevin}.

It is interesting to note that this result exhibits strong similarities with the one in the gradient flow setting. First of all, the expression of the noise variance $\xi$ in equation \eqref{eq:XiLangevin} is almost the same as in equation \eqref{eq:SCxi}. The only difference is that at positive temperature, the Langevin predictor is stochastic and in this case, the MSE is computed as the one of the mean of $WW^\top$ under $\P_\beta$. In addition, the variable $r$ is defined the same way as $r_\infty$ in equation \eqref{eq:SCr}, and it precisely plays the same role as in the approximate dynamics \eqref{eq:DynamicsSimplified}. 

Additionally, we derive in \cref{App:Subsec:LabelLangevin} the stationary measure of the typical label, whose dynamics is given in equation \eqref{eq:DMFT1Y}. As remarked, the typical label remains Gaussian at all times. In the long-time limit, we obtain self-consistent expressions of its mean and covariance, summarized in \cref{App:Subsubsec:LabelLangevinSummary}. Together with the result of \cref{Result:Langevin}, this provides a set of self-consistent equations that allows to compute averaged quantities of the Langevin dynamics in the long-time limit. 

The reasoning carried out in \cref{Subsubsec:SteadyState} allowed to obtain the set of low-dimensional equations of \cref{Result2}, through the use of random matrix theory. Even though the set of equations of \cref{Result:Langevin} is still high-dimensional, we believe that it could be reduced to a system of scalar equations, as it was done by \citet{maillard2024bayes} in the Bayes-optimal setting. This would require to understand the distribution $\P_\beta$, and imply the use of results on HCIZ integrals \citep{harish1957differential, itzykson1980planar,guionnet2002large}. Two cases can be directly analyzed already. 

\paragraph{Zero-temperature limit.} When considering $\ell_2$-regularization $\Omega(W) = \lambda \tr\big( WW^\top \big)$ and taking the $\beta \to \infty$ limit in \cref{Result:Langevin}, we show in \cref{App:Subsec:ZeroTemperature} that we recover the same result as in \cref{Subsubsec:EquationsLongTimes}. Interestingly, as it is known that the stationary measure of Langevin dynamics concentrates on the set of the global minimizers of its associated potential, this implies that the gradient flow dynamics converges to a global minimizer of the regularized empirical loss. While we already claimed this was the case for large values of $\kappa$ in \cref{Subsubsec:Regions}, this general conclusion was still an open question. 

This claim is valid under the assumptions of the section. Indeed, in both the gradient flow and Langevin settings, these assumptions allowed to interpret the high-dimensional dynamics as an effective gradient system with the potential:
\begin{equation}
    U_\text{eff}(W) = \frac{r}{2d} \Big\| WW^\top - Z^* - \sqrt{\xi} \G \Big\|_F^2 + \frac{\lambda}{d} \tr \big( WW^\top \big). 
\end{equation}
As we show in \cref{Subsec:OjaKnownResults}, all local minimizers of this potential are global. Therefore, in this setting, it is no surprise to observe the match between gradient flow dynamics at long times and the zero-temperature limit of the stationary measure of Langevin dynamics.  

\paragraph{Bayes-optimal learning.} In addition, following an appropriate choice of the inverse temperature $\beta$ in the Langevin dynamics, we show that the stationary measure of \cref{Result:Langevin} matches with the Bayes-optimal posterior distribution derived by \citet{maillard2024bayes} for the same problem. This establishes a direct connection between the dynamical formulation considered here and the Bayes-optimal analysis. In addition, it reveals the equivalence between the replica computation performed by \citet{maillard2024bayes} under the replica-symmetric ansatz, and the dynamical mean-field analysis presented here in the TTI regime. 

\subsection{Small Regularization Limit} \label{subsec:SmallReg}

In this section we study the small regularization limit ($\lambda \to 0^+$) of the system of equations given in \cref{Result2}. Note that we do not expect the associated estimator to coincide with the one obtained when running gradient flow on the unregularized empirical loss. Indeed, in the following, we first consider the long-time $(t \to \infty)$ limit before studying the small regularization limit. 

The results derived in this section share several similarities with those of \citet{erba2025nuclear}. As mentioned earlier, when $\kappa \geq 1$, the system of equations presented in \cref{Result2} coincides exactly with theirs. However, our equations (and the conclusions we draw from them) remain valid for all values of $\kappa$, including the underparameterized regime $\kappa < 1$. It is precisely in this regime that the system~\eqref{eq:SystemResult2} depends explicitly on $\kappa$. This dependence allows us to study how this parameter, and thus the amount of overparameterization, affects the performance of the gradient flow estimator.

In the following, we will focus on the case $\kappa \geq \min(\kappa^*, 1)$, when the student possesses enough parameters to be able to recover the teacher. In particular, this setting will allow us to study perfect recovery in the small regularization limit (when $\Delta = 0$).

To capture the dependence on the teacher's effective rank $\kappa^*$, we recall \cref{Assumption1:Teacher} and decompose the teacher's asymptotic spectral measure $\mu^*$ as:
\begin{equation}
    \mu^* = (1 - \min(\kappa^*, 1)) \delta + \min(\kappa^*, 1) \nu^*,
\end{equation}
where $\nu^*$ is a probability measure with support on $\R^+$ that we assume in addition to be bounded away from zero. Some of our proofs will also require $\nu^*$ to admit a smooth density, but we believe that our results can be easily extended to more general settings. 

\subsubsection{Interpolation Threshold} \label{Subsubsec:InterpolationThreshold}

In a similar fashion as in the work of \citet{erba2025nuclear}, we are able to derive an interpolation threshold $\alpha_\text{inter}(\kappa, \kappa^*, \Delta)$ such that:
\begin{itemize}
    \item For $\alpha > \alpha_\text{inter}(\kappa, \kappa^*, \Delta)$, the performance of the gradient flow estimator in the limit $\lambda \to 0^+$ is found by taking $q = 0$ in the system~\eqref{eq:SystemResult2} while $\lambda / q$ remains of order one. In this case the gradient flow estimator behaves as if it was minimizing the unregularized loss. We will detail this correspondence in \cref{Subsubsec:BeyondInterpolation}.
    \item For $\alpha < \alpha_\text{inter}(\kappa, \kappa^*, \Delta)$, the system of equations at $\lambda = 0^+$ is obtained by plugging $\lambda = 0$ into the system~\eqref{eq:SystemResult2}. In this case the effect of the regularization remains in the limit $\lambda \to 0^+$ and the performance differs from the one of the unregularized dynamics. 
\end{itemize}
In \cref{Subsubsec:EquationsSmallReg}, we show that this threshold coincides with the largest sample complexity for which the training labels can still be fitted exactly, in the presence of label noise. This interpretation is consistent with the standard notion of the interpolation threshold in noisy learning settings.

As a consequence of the previous characterization, the interpolation threshold is given by the smallest value of $\alpha$ for which the system \eqref{eq:SystemResult2} admits a solution in the limit where $q$ vanishes proportionally to $\lambda$. For the following, we define the function:
\begin{equation} \label{eq:defI}
    I_\omega(\xi) = \int_{\max(0,\omega)} x h_\xi(x) \d \mu_\xi(x).
\end{equation}
Recall that $h_\xi$ is the Hilbert transform of the measure $\mu_\xi$. In the end, we have the following characterization of the interpolation threshold:
\begin{proposition} \label{Prop:InterpolationThreshold}
    For $\kappa \geq \min(\kappa^*, 1)$, consider $\omega, \xi$ to be the solution of the system:
    \begin{equation}
    \begin{aligned}
        \min(\kappa, 1) &= \int_\omega \d \mu_\xi(x), \\
        2 \xi I_\omega(\xi) &= \int_{\max(0, \omega)} x^2 \d \mu_\xi(x) - \frac{\Delta}{2} - Q_*.    
    \end{aligned}
    \end{equation}
    The interpolation threshold is then given by $\alpha_\mathrm{inter}(\kappa, \kappa^*, \Delta) = I_\omega(\xi)$. 
\end{proposition}

We emphasize that the dependence of these equations on $\kappa^*$ enters only through the teacher’s asymptotic spectral distribution $\mu^*$, which appears in $\mu_\xi$, the free additive convolution of $\mu^*$ with a semicircular distribution of variance $\xi$.

When $\kappa \geq 1$, our system of equations coincides with the one of \citet{erba2025nuclear} (Result~1). In their work, the interpolation threshold is defined as the smallest value of $\alpha$ for which the empirical loss admits a unique global minimizer. We conjecture that a similar interpretation holds for general values of $\kappa$. Indeed, the region $\alpha > \alpha_{\mathrm{inter}}$ is characterized by the fact that the ratio $\lambda/q$ remains positive in the limit of vanishing regularization. As discussed in \cref{Subsubsec:SteadyState}, this quantity coincides with the long-time integrated response $r_\infty$, defined in equation \eqref{eq:SCr}. Within the dynamical approximation introduced in equation~\eqref{eq:DynamicsSimplified}, this scalar variable is directly related to the curvature of the loss landscape near the point of convergence. This leads to two possibilities:
\begin{itemize}
    \item When $r_\infty = 0$, the landscape becomes flat in a neighborhood of the limit point, suggesting that the set of global minimizers of the empirical loss (in terms of $W$) forms a manifold of positive dimension.
    \item When $r_\infty > 0$, the convergence shares the same properties as in the case $\lambda > 0$, and we expect the global minimizer reached by gradient flow to be isolated (up to the representation $W \mapsto WW^\top$).
\end{itemize}
This dynamical interpretation only provides local information about the landscape close to the point reached by gradient flow. When $\kappa \geq 1$, the optimization problem is convex and so is the set of global minimizers. In this regime, the interpolation threshold $\alpha_{\mathrm{inter}}$ describes the threshold at which the loss admits a single global minimizer. This observation was already made by \citet{erba2025nuclear}. However, when $\kappa < 1$, convexity breaks, and the set of global minimizers
may be more complicated, potentially splitting into multiple components.

Interestingly, as illustrated in \cref{fig:FigOverfitting2}, the interpolation threshold appears to coincide with the location of the peak of the MSE in the double descent regime. This correspondence is not specific to our setting and has previously been observed in linear models \citep{belkin2019reconciling, hastie2022surprises} as well as in random feature models \citep{d2020double, mei2022generalization}. More generally, this behavior is consistent with the idea that, in the presence of label noise ($\Delta > 0$), the interpolation threshold marks the largest sample complexity for which the model can exactly fit the training data, leading to vanishing training error but potentially poor generalization. In \cref{subsec:Unregularized}, we further investigate this regime by analyzing the unregularized dynamics and providing additional numerical evidence.

The following corollary gives a more explicit characterization of the interpolation threshold in the noiseless case $\Delta = 0$ and for $\kappa \geq \min(\kappa^*,1)$, that is, when the student can recover the teacher at a finite value of $\alpha$.

\begin{corollary} \label{Prop:InterpolationThreshold1}
    Consider the case where $\kappa \geq \min(\kappa^*, 1)$ and $\Delta = 0$, and let $\sigma$ denote the semicircular distribution (defined in \cref{App:Subsubsec:MatrixEnsembles}). Then the solution $\omega, \xi$ of the system of equations of \cref{Prop:InterpolationThreshold} are reached at $\xi = 0, \omega = 0$. The value of the interpolation threshold is then, for $\kappa^* \leq 1$:
    \begin{equation} \label{eq:InterThreshold}
        \alpha_\mathrm{inter}(\kappa, \kappa^*) = \kappa^* - \frac{\kappa^{*2}}{2} + \frac{(1-\kappa^*)^2}{2} \int_{\max(0, \tilde \omega)} x^2 \d \sigma(x),
    \end{equation} 
    where $\tilde \omega$ solves the equation:
    \begin{equation}
        \frac{\min(\kappa, 1) - \kappa^*}{1 - \kappa^*} = \int_{\tilde \omega} \d \sigma(x).
    \end{equation}
    In addition, if $\kappa, \kappa^* \geq 1$, then $\alpha_\mathrm{inter}(\kappa, \kappa^*) = 1/2$. 
\end{corollary}
The quantity $\tilde{\omega}$ is a rescaled version of the threshold $\omega$. Recall that $\omega$ selects a fraction $\kappa$ of the mass of the measure $\mu_\xi$. When $\kappa^* \leq 1$ and $\kappa > \kappa^*$, this selection involves a part of the semicircular component of $\mu_\xi$, whose support has width of order $\sqrt{\xi}$. As a result, $\omega$ must scale with $\xi$ and vanish as $\xi \to 0$. The rescaled variable $\tilde{\omega}$ captures how the semicircular density should be cut in order to select a mass $\kappa$. The corresponding expression for the interpolation threshold is derived in \cref{App:Subsec:InterpolationThreshold}.

As a first remark, this threshold does not depend on the teacher's distribution $\mu^*$, unlike for positive $\Delta$ (see \cref{Prop:InterpolationThreshold}). The semicircular distribution $\sigma$ appears as a universal object in these equations: it is directly linked to the Gaussian noise added to the teacher to form the matrix $Z^* + \sqrt{\xi} \G$ studied earlier. Indeed, it is well known \citep{wigner1955characteristic} that the semicircular distribution corresponds to the high-dimensional limit of the spectral density of a GOE matrix. 

\subsubsection{Set of Equations in the Small Regularization Limit} \label{Subsubsec:EquationsSmallReg}

Considering the previous observations, we shall now formulate the set of equations that describes the performance and spectral properties of the gradient flow predictor in the small regularization limit. 
\begin{proposition} \label{Prop:EquationsSmallReg}
    Recall the definition of the function $I_\omega$ in equation~\eqref{eq:defI}. Assume that given the values $\alpha, \kappa, \kappa^*, \Delta$ and $\lambda >0$, the system of equations~\eqref{eq:SystemResult2} has a unique solution $(\xi, q)$. Then, as $\lambda \to 0^+$, we have:
    \begin{description}
        \item[1.] If $\alpha \leq \alpha_\mathrm{inter}(\kappa, \kappa^*, \Delta)$, the gradient flow predictor remains:
        \begin{equation} \label{eq:SystemSmallReg1}
            Z_\infty = \Big( Z^* + \sqrt{\xi} \G - q I_d \Big)_{(m)}^+,
        \end{equation}
        with values of MSE and training loss:
        \begin{equation}
            \mathrm{MSE} = 2 \alpha \xi - \frac{\Delta}{2}, \hspace{1.5cm} \mathrm{Loss}_\mathrm{train} = 0,
        \end{equation}
        where $\xi, q$ are solutions of the system of equations~\eqref{eq:SystemResult2} with $\lambda = 0$.
        \item[2.] If $\alpha \geq \alpha_\mathrm{inter}(\kappa, \kappa^*, \Delta)$, the system of equations reduces to the variables $\omega, \xi$, solution of:
        \begin{subequations} \label{eq:SystemSmallReg2}
        \begin{align}
            \min(\kappa, 1) &= \int_\omega \d \mu_\xi(x), \\
            2 \alpha \xi - \frac{\Delta}{2} &= Q_* - \int_{\max(0, \omega)} x^2 \d \mu_\xi(x) + 4\xi I_\omega(\xi).
        \end{align}
        \end{subequations}
        In this case the gradient flow predictor is given by:
        \begin{equation}
            Z_\infty = \Big( Z^* + \sqrt{\xi} \G \Big)_{(m)}^+,
        \end{equation}
        and the MSE and the loss write:
        \begin{equation}
            \mathrm{MSE} = 2 \alpha \xi - \frac{\Delta}{2}, \hspace{1.5cm} \mathrm{Loss}_\mathrm{train} = \alpha \xi \left( 1 - \frac{I_\omega(\xi)}{\alpha} \right)^2. 
        \end{equation}
    \end{description}
\end{proposition}

Even in the presence of label noise, for $\alpha < \alpha_\text{inter}$ the empirical loss vanishes. In the small regularization limit, this implies that gradient flow converges to a predictor that exactly fits all the training labels. When $\alpha \geq \alpha_\text{inter}$, we consider two cases:
\begin{itemize}
    \item If $\Delta = 0$, interpolation occurs after perfect recovery. In this case, the loss remains zero, and so does the MSE.
    \item In the presence of label noise, beyond the interpolation threshold, the student can no longer fit all the observed labels, and the training loss becomes positive. In this regime, we have the characterization of the interpolation threshold:
    \begin{equation}
        \alpha \leq \alpha_\text{inter} \Longleftrightarrow \mathrm{Loss}_\mathrm{train} = 0,
    \end{equation}
    meaning that $\alpha_\text{inter}$ is the largest sample complexity for which exact fitting of the training data is still possible. In noisy settings, this criterion is often taken as the definition of the interpolation threshold.
\end{itemize}

In addition, one can remark that the system of equations obtained for $\alpha > \alpha_\text{inter}$ is precisely the same as we would get when performing the same dynamical assumptions as in \cref{Subsec:LongTimes}, but for the unregularized dynamics. This observation motivates the conjecture formulated in \cref{subsec:Unregularized}: in this regime, the limit $\lambda \to 0^+$ coincides with the unregularized dynamics.

\cref{fig:MSE_Smallreg} illustrates \cref{Prop:EquationsSmallReg} in the noiseless case. It features both gradient descent simulations at $\lambda > 0$, and the solution of the system of equations of \cref{Prop:EquationsSmallReg} at $\lambda = 0^+$. In this limit, the training loss is always zero and the MSE vanishes at a finite value of $\alpha$. In the next part, we derive the perfect recovery threshold, corresponding to this critical value. 

\begin{figure}[ht]
    \centering
    \includegraphics[width=\linewidth]{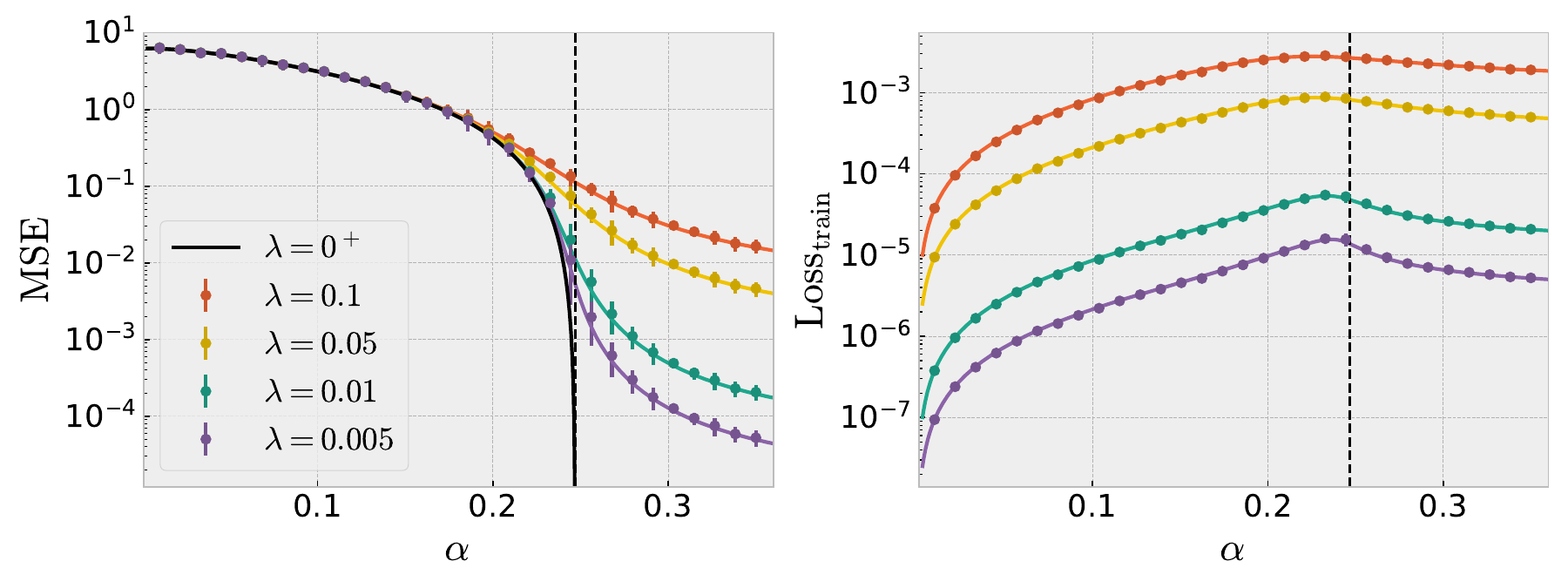}
    \vspace*{-0.7cm}
    \caption{MSE (left) and training loss (right) as a function of $\alpha$, for $\kappa = 0.3, \kappa^* = 0.2, \Delta = 0$ and several values of regularization strength $\lambda$. Dots: simulations of gradient descent, defined in equation \eqref{eq:GDdynamics}, for $\lambda > 0$ and averaged over 10 realizations of the initialization, teacher and data. Full lines: numerical integration of the system of equations~\eqref{eq:SystemResult2}. The black line corresponds to simulations of the system of equations in \cref{Prop:EquationsSmallReg}. Vertical dashed line: perfect recovery threshold, defined in \cref{Prop:PRThresholdRegularization}.}
    \label{fig:MSE_Smallreg}        
\end{figure}

\subsubsection{Perfect Recovery Threshold} \label{Subsubsec:PRThresholdReg}

From these results, we are now able to access the perfect recovery (PR) threshold associated with the system of equations in the small regularization limit, that is the value of $\alpha$ (that we denote $\alpha_\text{PR}^+$) for which the MSE is zero as soon as  $\alpha \geq \alpha_\text{PR}^+$. Note that for $\lambda > 0$ the MSE always remains positive, so that the PR transition can only happen in the small regularization limit. 

\begin{proposition} \label{Prop:PRThresholdRegularization}
Consider the setting $\Delta = 0$ and $\kappa \geq \min(\kappa^*, 1)$. Let $\sigma$ denote the semicircular distribution (defined in \cref{App:Subsubsec:MatrixEnsembles}). Then, if $\kappa^* < 1$, the perfect recovery threshold is given by:
    \begin{equation} \label{eq:PRThresholdReg}
        \alpha_\mathrm{PR}^+(\kappa, \kappa^*) = \kappa^* - \frac{\kappa^{*2}}{2} + \frac{(1-\kappa^*)^2}{2} \int_{\max(h, \tilde \omega)} x(x-h) \d \sigma(x),
    \end{equation}
    where $h, \tilde \omega$ solve the equations:
    \begin{equation}
    \begin{aligned}
        \frac{\min(\kappa, 1) - \kappa^*}{1 - \kappa^*} &= \int_{\tilde \omega} \d \sigma(x), \\
        \frac{\kappa^*}{1 - \kappa^*} &= \frac{1}{h} \int_{\max(h, \tilde \omega)} (x-h) \d \sigma(x). 
    \end{aligned}
    \end{equation}
In addition, if $\kappa^* \geq 1$, $\alpha_\mathrm{PR}^+(\kappa, \kappa^*) = \dfrac{1}{2}$ for any $\kappa \geq 1$. 
\end{proposition}

For $\kappa \geq 1$, this quantity is exactly the same as the one identified by \citet{erba2025nuclear}. This threshold can be derived by sending the MSE to zero, which amounts to taking the $\xi \to 0$ limit. Then, the result is obtained by choosing the scaling $q \propto \sqrt{\xi}$. As previously, the variable $\tilde \omega$ plays the role of the rank constraint. We refer to \cref{App:Subsec:PRThreshold} for the derivation of this threshold. 

\subsubsection{Comparison of the Thresholds}

We shall now compare the interpolation and the PR thresholds derived in \cref{Prop:InterpolationThreshold1} and \cref{Prop:PRThresholdRegularization}. 

\begin{figure}[ht]
    \begin{minipage}{0.6\textwidth}
    \centering
    \includegraphics[width=\linewidth]{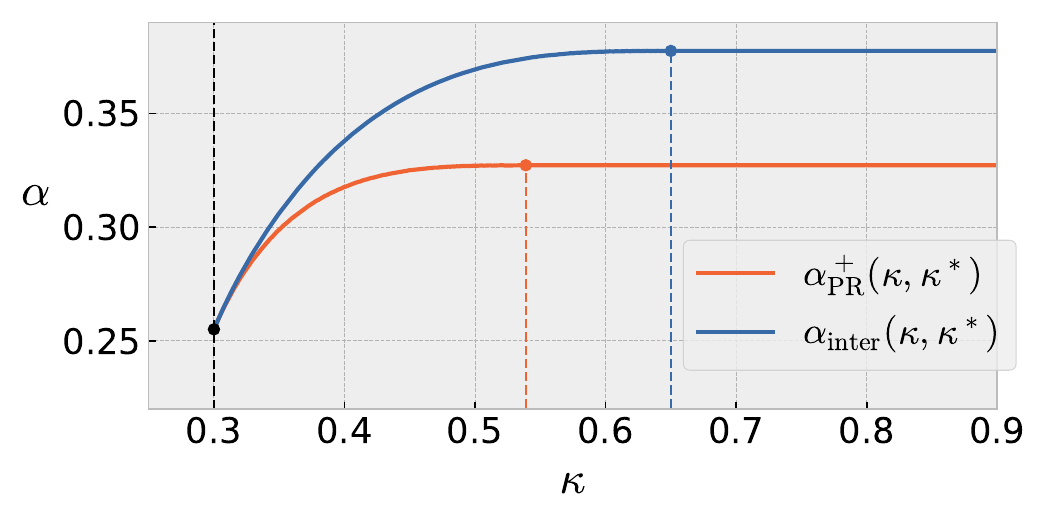}
    \end{minipage}
    \hfill
    \begin{minipage}{0.38\textwidth}
        \caption{Interpolation and PR thresholds as a function of $\kappa$ for $\kappa^* = 0.3$, computed from \cref{Prop:InterpolationThreshold1} and \cref{Prop:PRThresholdRegularization}. The black dot indicates the common value of the thresholds $\kappa^* - \kappa^{*2}/2$ at $\kappa = \kappa^*$. Vertical colored lines indicate the value of $\kappa$ at which each threshold becomes constant.}
        \label{fig:ThresholdsReg}
    \end{minipage}
\end{figure}

A first observation follows directly from equations~\eqref{eq:InterThreshold} and~\eqref{eq:PRThresholdReg}: for almost all values of $\kappa$ and $\kappa^*$, the perfect recovery threshold is strictly smaller than the interpolation threshold. This implies that, in the small regularization limit, perfect recovery typically occurs in a regime where the empirical loss still admits multiple global minimizers. In this sense, the solution selected in the small regularization limit has better generalization properties than a generic interpolator.

The two thresholds only coincide in specific situations: either when the student and teacher ranks match, $\kappa = \kappa^*$, or when both $\kappa$ and $\kappa^*$ are larger than one. In these cases, we have the expressions:
\begin{equation} \label{eq:limsthresholds}
    \alpha_\text{inter}(\kappa, \kappa^*) = \alpha_\text{PR}^+(\kappa, \kappa^*) =
    \begin{cases}
        \kappa^* - \dfrac{\kappa^{*2}}{2}, & \text{if } \kappa = \kappa^* \leq 1, \\
        \dfrac{1}{2}, & \text{if } \kappa, \kappa^* \geq 1.
    \end{cases}
\end{equation}
In the matched-rank case, $\kappa = \kappa^*$, this expression matches with the one previously obtained by \citet{maillard2024bayes}. The regime $\kappa, \kappa^* \geq 1$ also recovers known results for this problem, consistent with earlier works such as \citet{donoho2013phase, gamarnik2019stationary}.

Moreover, it is clear that these thresholds only depend on the effective rank $\kappa$ close to $\kappa^*$. More precisely, one can show that for each threshold there exists a critical value of $\kappa$ beyond which the threshold no longer depends on $\kappa$. For $\kappa^* < 1$, a direct calculation gives the expressions:
\begin{align}
    \alpha_\text{inter}(\kappa, \kappa^*) &= \frac{1}{2} \left( \frac{1}{2} + \kappa^* - \frac{\kappa^{*2}}{2} \right)& &\Longleftrightarrow \kappa \geq \frac{1+\kappa^*}{2}, \label{eq:ConstantInterThreshold} \\ 
    \alpha_\text{PR}^+(\kappa, \kappa^*) &= \kappa^* - \frac{\kappa^{*2}}{2} + \frac{(1-\kappa^*)^2}{2} \int_h x(x-h) \d \sigma(x)& &\Longleftrightarrow \kappa \geq \kappa^* + (1-\kappa^*) \int_h \d \sigma(x), \label{eq:ConstantPRThreshold} 
\end{align}
where $h > 0$ is solution of:
\begin{equation}
    \frac{\kappa^*}{1 - \kappa^*} = \frac{1}{h} \int_h (x-h) \d \sigma(x). 
\end{equation}
These two values of the interpolation and perfect recovery thresholds match with the ones obtained by \citet{erba2025nuclear}, for $\kappa \geq 1$.  Moreover, the critical value of $\kappa$ at which $\alpha_\text{PR}^+$ becomes constant in equation~\eqref{eq:ConstantPRThreshold} exactly matches the quantity $\kappa_{\min}$ defined in~\eqref{eq:kappamin}, in the small regularization limit. Recall that $\kappa_{\min}$ corresponds to the (normalized) rank of the global minimizer of the regularized loss over the set of PSD matrices. As discussed in \cref{Subsubsec:Regions}, this quantity separates two regimes at positive regularization, depending on whether the student has enough parameters to reach this global minimizer. This distinction persists in the small regularization limit. For small values of $\kappa$, the student converges to a suboptimal predictor in terms of the empirical loss (when viewed as a function of $Z = WW^\top$). Nevertheless, the low-rank structure of the solution allows the student to recover the teacher using fewer samples.

This analysis reveals an interesting conclusion about the role of overparameterization, quantified by the network normalized width $\kappa$. As illustrated in \cref{fig:ThresholdsReg}, increasing $\kappa$ requires more observations to reach perfect recovery, but only up to a point. For larger values of $\kappa$ (while still $\kappa < 1$), the perfect recovery threshold becomes independent of the number of parameters. This behavior is driven by the low-rank structure of the teacher, that influences $\kappa_{\min}$, the rank of the global minimizer of the empirical loss over all PSD matrices. Once $\kappa$ exceeds this value, further overparameterization does not increase the sample complexity required for perfect recovery.

To conclude this part, we insist on the generality of the previous results: the expressions of $\alpha_\mathrm{PR}^+$ and $\alpha_\mathrm{inter}$ only depend on the variables $\kappa, \kappa^*$, but not on the distribution of the teacher or the student at initialization. Although the derivation of equations~\eqref{eq:InterThreshold} and~\eqref{eq:PRThresholdReg} requires some mild assumptions (that we detail in \cref{App:SmallReg}), these remain very general. 

\paragraph{Small-width teacher.} One can wonder about the behavior of the thresholds $\alpha_\mathrm{inter}$ and $\alpha_\mathrm{PR}^+$ in the limit $\kappa^* \to 0$, i.e., when the teacher matrix possesses a sub-extensive rank $m^* \ll d$. In this case, as a consequence of \cref{Prop:InterpolationThreshold1} and \cref{Prop:PRThresholdRegularization}, we have the equations:
\begin{equation}
\begin{aligned}
    \alpha_\mathrm{inter}(\kappa, 0^+) &= \frac{1}{2} \int_{\max(0, \tilde \omega)} x^2 \d \sigma(x), \hspace{1.5cm} \min(\kappa, 1) = \int_{\tilde \omega} \d \sigma(x), \\
    \alpha_\mathrm{PR}^+(\kappa, 0^+) &= 0. 
\end{aligned}
\end{equation}
Therefore, even an extensive-width student (with $\kappa > 0$) is able to recover this small teacher with a number of observations $n \ll d^2$, but interpolation only arises at $n = \Theta(d^2)$. In addition, we have the values:
\begin{equation}
    \alpha_\mathrm{inter}(\kappa, 0^+) \xrightarrow[\kappa \to 0^+]{} 0, \hspace{1.5cm} \alpha_\mathrm{inter}(\kappa, 0^+) = \frac{1}{4} \Longleftrightarrow \kappa \geq \frac{1}{2}. 
\end{equation}
This value of $1/4$ is already present in the work of \citet{erba2025nuclear}. In comparison, \citet{sarao2020optimization} derived an interpolation threshold $n = 2d$ in the case of a rank-one teacher. As we do not recover this result when taking the $\kappa^* \to 0$ limit, this suggests the presence of different regimes depending on the scaling between $m^*$ and the dimension. 

\subsubsection{Minimal Regularization Estimator} \label{Subsubsec:MinRegInterpolator}

The previous observations reveal that in the small regularization limit, the gradient flow dynamics is selecting a specific global minimizer of the empirical loss that exhibits favorable generalization properties. In this part we relate this observation to the notion of minimal regularization interpolator. 

We start by mentioning a standard proposition that gives a general characterization of the behavior of the global minimizer of a regularized loss in the small regularization limit.
\begin{proposition} \label{Prop:MinRegularizationInterpolator}
    Consider a loss of the form $\L_\lambda = \L + \lambda \Omega$, where $\L, \Omega$ are both continuous and take positive values. Assume moreover that the regularization $\Omega$ is coercive:
    \begin{equation*}
        \big\| \Omega(W) \big\| \xrightarrow[\|W\| \to \infty]{} \infty.
    \end{equation*}
    Define $S^* = \mathrm{argmin}(\L)$ and assume $S^*$ is non-empty. Then, given a family $(W_\lambda)_{\lambda > 0}$ such that $W_\lambda$ is a global minimizer of $\L_\lambda$ for all $\lambda > 0$, every cluster point of $(W_\lambda)$ as $\lambda \to 0$ belongs to $S^*$ and minimizes $\Omega$ over $S^*$. In addition, if the minimizer of $\Omega$ over $S^*$ is unique, then $W_\lambda$ converges toward it. 
\end{proposition}

Linking this proposition back to our setting, a natural candidate for the predictor selected in the small regularization limit is the minimal $\ell_2$-norm interpolator:
\begin{equation}
    \underset{W \in S^*}{\mathrm{argmin}}\, \|W\|_F^2, 
    \hspace{1.5cm} 
    S^* = \underset{W \in \R^{d \times m}}{\mathrm{argmin}} \left[\frac{1}{n} \sum_{k=1}^n \Big( \tr(X_k WW^\top) - z_k \Big)^2 \right],
\end{equation}
corresponding to the minimal $\ell_2$-norm interpolator. However, convergence toward this predictor is only guaranteed when gradient flow reaches a global minimizer of the regularized loss. As emphasized earlier, our results allow us to establish this property only in the regime $\kappa \geq \kappa_{\min}$, where gradient flow converges to the global minimizer over all PSD matrices. In this case, the link with the minimal-norm interpolator has already been made precise by \citet{erba2025nuclear}. Moreover, rewriting the optimization problem in terms of $Z = WW^\top$ shows that minimizing $\|W\|_F^2$ is equivalent to minimizing $\tr(Z)$, yielding the minimal nuclear-norm interpolator.

In the underparameterized regime $\kappa < \kappa_{\min}$, we cannot directly guarantee that the solution reached by gradient flow at positive $\lambda$ corresponds to a global minimizer of the regularized loss. Nevertheless, as discussed in \cref{Subsubsec:Langevin}, the analysis of low-temperature Langevin dynamics suggests that this may still hold, provided that our asymptotic simplifications on the dynamics (namely time-translational invariance and fluctuation--dissipation) are valid. If confirmed, this would imply that, for all values of $\kappa$, \cref{Prop:EquationsSmallReg} characterizes the performance of the minimal $\ell_2$-norm interpolator.

\subsection{Unregularized Dynamics} \label{subsec:Unregularized}

In this section, we investigate the gradient flow dynamics~\eqref{eq:GFdynamics} in the absence of regularization. First of all, note that the dynamical result of \cref{Result1} still holds in the absence of regularization, but the dynamics exhibit quite different behaviors than in the regularized case. Indeed, numerical simulations suggest that the dynamics at long times strongly depends on the initial condition. For instance, we show in \cref{fig:FigInit1} the MSE reached by the unregularized dynamics initialized with Gaussian weights with different variances $\gamma$ and remark that the curves we obtain as a function of $\alpha$ strongly depend on this parameter. In contrast, it is known that the presence of regularization tends to make the landscape more convex \citep[see for instance][]{hoerl1970ridge, du2018power, kobayashi2024weight}. It also reduces the dependence on initialization by shrinking the weight components orthogonal to the data, which sometimes leads to the independence of the long-time dynamics with respect to the initial condition.

\begin{figure}[ht]
\centering
\begin{minipage}{0.6\textwidth}
\includegraphics[width=\linewidth]{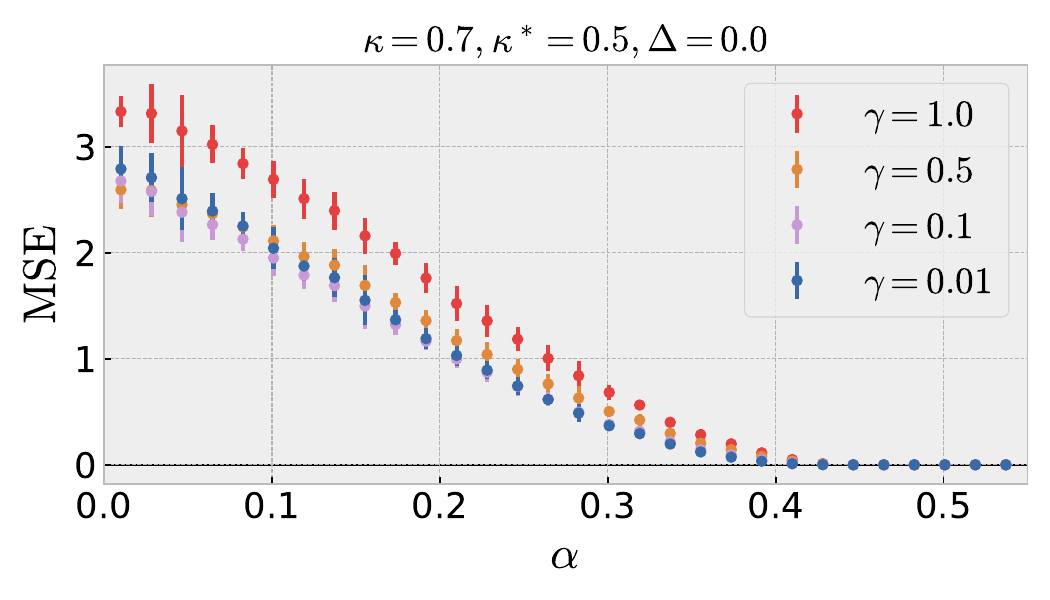}
\end{minipage}
\hfill
\begin{minipage}{0.37\textwidth}
\caption{MSE reached by gradient descent, defined in equation \eqref{eq:GDdynamics}, as a function of $\alpha$ and for different values of initialization variance~$\gamma$. Simulations are averaged over 10 realizations of the initialization, teacher and data. $\kappa = 0.7, \kappa^* = 0.5$ and zero label noise.}
\label{fig:FigInit1}
\end{minipage}
\end{figure}
This observation reveals that the assumption used in \cref{Subsec:LongTimes} is not relevant to describe the unregularized dynamics, since it strongly relies on the hypothesis that the early stages of the dynamics are forgotten at long times.

\subsubsection{Beyond the Interpolation Threshold} \label{Subsubsec:BeyondInterpolation}

Despite these observations, recall that in \cref{Subsubsec:InterpolationThreshold} we introduced the interpolation threshold $\alpha_\text{inter}$, and remarked that for $\alpha > \alpha_\text{inter}$, the statistics of the gradient flow estimator in the small regularization limit are the same as if there were no regularization. In this regime, we therefore conjecture that the equations derived in \cref{Prop:EquationsSmallReg} also apply to the unregularized setting. 
\begin{conjecture} \label{conjecture:unregularized}
    Consider the gradient flow dynamics \eqref{eq:GFdynamics} under \cref{Assumption2}, in the absence of regularization. Then, for $\alpha \geq \alpha_\mathrm{inter}(\kappa, \kappa^*, \Delta)$, the statistics of the gradient flow estimator are given by \textbf{2.} in \cref{Prop:EquationsSmallReg}. 
\end{conjecture}
In \cref{fig:FigUnregularized1}, we present a numerical illustration of this conjecture: the dependence on the initialization strength $\gamma$ appears to be present only for $\alpha < \alpha_{\text{inter}}$. Beyond the interpolation threshold, this dependence vanishes and the performance of the gradient flow estimator matches the one in the small regularization limit. In this figure, we plot the MSE, the empirical loss and the in-sample error, corresponding to the error on the true labels:
\begin{equation} \label{eq:insampleerror}
    \mathrm{Err}_\mathrm{in} = \frac{1}{4n} \sum_{k=1}^n \Big( \tr\big(X_k WW^\top \big) - \tr \big(X_k Z^* \big) \Big)^2,
\end{equation} 
where $W$ is the gradient flow estimator reached at convergence. Similarly to the expressions of the MSE and training loss derived in \cref{Prop:EquationsSmallReg}, we derive an expression for the in-sample error in \cref{App:Subsec:InterpolationThreshold}. 

\cref{fig:FigUnregularized1} reveals that for $\alpha < \alpha_\text{inter}$, the student is able to perfectly fit all the noisy labels, leading to a zero loss, and an in-sample error equal to $\Delta / 4$: in this case there are too few observations for the student to capture the teacher structure behind the label noise. On the contrary, for $\alpha > \alpha_\text{inter}$, the noisy labels cannot all be fitted, the loss becomes positive (and increases with $\alpha$), but the representation learned by the student leads to a smaller error on the labels as well as a smaller MSE. In this region, despite the label noise, the high number of observations leads to an improved generalization. 

\begin{figure}[ht]
    \centering
    \includegraphics[width=\linewidth]{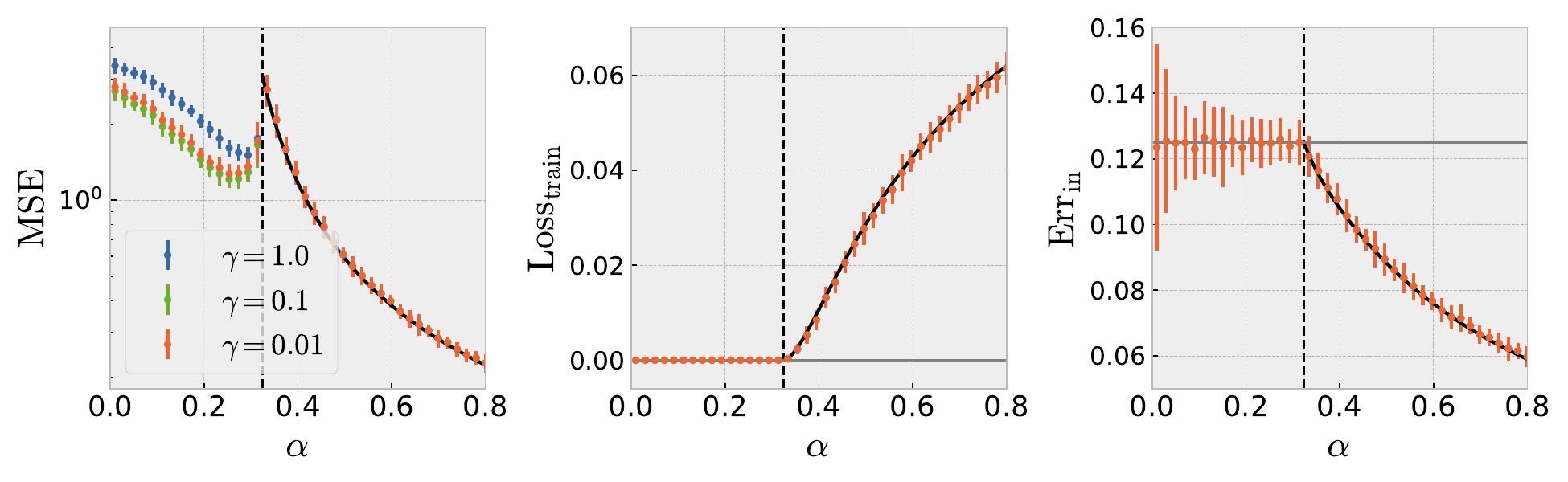}
    \vspace*{-0.8cm}
    \caption{Comparison between simulations of gradient descent, defined in equation \eqref{eq:GDdynamics}, and numerical integration of the system of equations~\eqref{eq:SystemSmallReg2} for $\alpha > \alpha_\text{inter}$ (vertical dashed line), for $\kappa = 0.7, \kappa^* = 0.5$ and $\Delta = 0.5$. MSE (left), empirical loss (middle) and error on the true labels, also known as in-sample error (right), as a function of $\alpha$ for three values of initialization variance~$\gamma$, and with no regularization. The horizontal gray line on the right panel corresponds to the value $\Delta / 4$. Gradient descent simulations are averaged over 20 realizations of the initialization, teacher, and data.}
    \label{fig:FigUnregularized1}
\end{figure}

\subsubsection{Perfect Recovery Threshold} \label{Subsubsec:PRThresholdUnreg}

We shall now investigate the perfect recovery threshold achieved by the gradient flow estimator in the unregularized case. We therefore consider the case $\kappa \geq \min(\kappa^*, 1)$ and $\Delta = 0$, allowing for perfect recovery. As shown previously in \cref{Subsubsec:InterpolationThreshold}, the MSE is already zero in the region $\alpha > \alpha_\text{inter}$. This leads to the question: does the PR threshold match the interpolation threshold, i.e., does gradient flow require the loss to have a single minimizer (at least locally near the point of convergence) in order to recover the teacher, or does the minimizer that is chosen have some particular properties that trigger perfect recovery before interpolation? 

As a result of a large number of simulations, we formulate the following conjecture:
\begin{conjecture} \label{Conjecture:PRThreshold}
Let $\sigma$ denote the semicircular distribution (defined in \cref{App:Subsubsec:MatrixEnsembles}). In the setting where $\Delta = 0$ and $\kappa \geq \min(\kappa^*, 1)$, the value of the perfect recovery threshold is given by:
\begin{equation} \label{eq:PRThresholdUnreg}
\alpha_\mathrm{PR}(\kappa, \kappa^*) = \min \left( \kappa - \frac{\kappa^2}{2}, \alpha_\mathrm{inter}(\kappa, \kappa^*) \right),
\end{equation}
where $\alpha_\mathrm{inter}$ is given in \cref{Prop:InterpolationThreshold1} for $\kappa^* \leq 1$:
\begin{equation}
    \alpha_\mathrm{inter}(\kappa, \kappa^*) = \kappa^* - \frac{\kappa^{*2}}{2} + \frac{(1-\kappa^*)^2}{2} \int_{\max(0, \tilde \omega)} x^2 \d \sigma(x),
\end{equation}
with $\tilde \omega$ being solution of:
\begin{equation}
    \frac{\min(\kappa, 1) - \kappa^*}{1 - \kappa^*} = \int_{\tilde \omega} \d \sigma(x).
\end{equation}
\end{conjecture}

This conjecture gives the expression of the perfect recovery threshold as the minimum of the interpolation threshold $\alpha_\text{inter}$ and a quantity $\alpha_\mathrm{dof}(\kappa) = \kappa - \kappa^2 / 2$ that corresponds to the (normalized) number of degrees of freedom of the set of PSD matrices with rank $\kappa d$:
\begin{equation}
    \kappa - \frac{\kappa^2}{2} = \lim_{d \to \infty} \frac{1}{d^2} \mathrm{dim} \big\{ Z \in \mathcal{S}_d^+(\R), \, \mathrm{rank}(Z) = \lfloor \kappa d \rfloor \big\}.
\end{equation}
Therefore, for $n \geq \alpha_{\mathrm{dof}}(\kappa)\, d^2$, the number of observations exceeds the number of free parameters of the student matrix. Interestingly, this value does not coincide with the interpolation threshold derived in \cref{Subsubsec:InterpolationThreshold}. Indeed, the quantity $\alpha_{\mathrm{dof}}(\kappa)$ is a geometric quantity determined by the symmetries of the predictor, whereas the interpolation threshold is problem-specific and accounts for the structure of the teacher (note that $\alpha_{\text{inter}}$ depends on $\kappa^*$ but not on $\alpha_{\mathrm{dof}}$).

More precisely, for $\kappa^* \leq 1$, an analysis of the function $\alpha_{\mathrm{PR}}(\kappa, \kappa^*)$ in \cref{Conjecture:PRThreshold} reveals the existence of a critical value $\kappa_c \in [\kappa^*, 1]$ such that
\begin{equation}
    \alpha_{\mathrm{PR}}(\kappa, \kappa^*) =
    \begin{cases}
        \kappa - \dfrac{\kappa^2}{2}, & \text{if } \kappa \leq \kappa_c, \\
        \alpha_{\mathrm{inter}}(\kappa, \kappa^*), & \text{if } \kappa \geq \kappa_c.
    \end{cases}
\end{equation}
Therefore, for $\kappa < \kappa_c$, the teacher is perfectly recovered, despite the presence of multiple global minimizers in its vicinity. This conclusion is similar to the case of the small regularization limit \cref{subsec:SmallReg}: this phenomenon is an instance of the implicit bias induced by gradient flow dynamics.

\begin{figure}[ht]
    \centering
    \begin{minipage}{0.58\textwidth}
    \includegraphics[width=\linewidth]{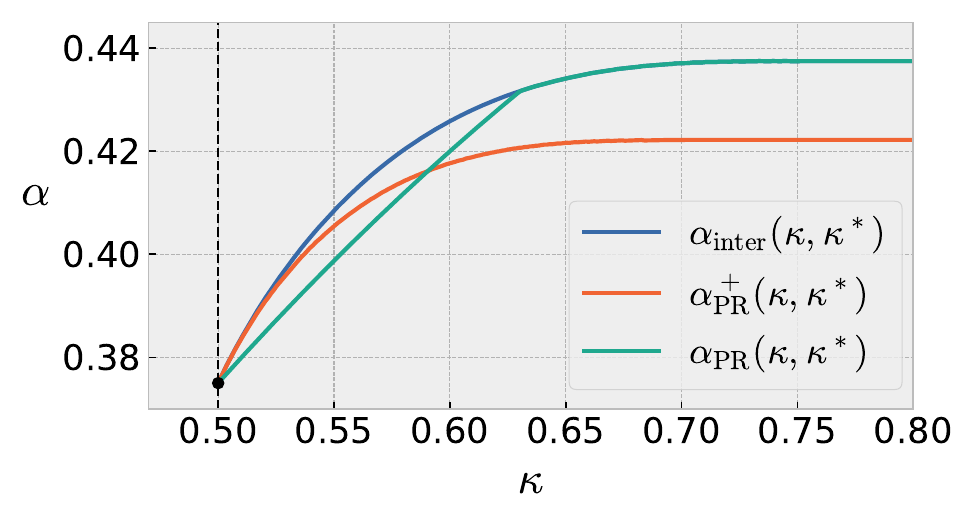}
    \end{minipage}
    \hfill
    \begin{minipage}{0.4\textwidth}
        \caption{Conjectured PR threshold $\alpha_\text{PR}$ (\cref{Conjecture:PRThreshold}), along with the interpolation threshold $\alpha_\mathrm{inter}$ (\cref{Prop:InterpolationThreshold1}) and the PR threshold $\alpha_\text{PR}^+$ (\cref{Prop:PRThresholdRegularization}) in the small regularization limit as a function of $\kappa$ for $\kappa^* = 0.5$ (vertical dashed line).}
        \label{fig:FigThreshold1}
    \end{minipage}
\end{figure}
\cref{fig:FigThreshold1} compares the three thresholds derived in this section and in \cref{subsec:SmallReg}. Interestingly, the figure features a region in which the PR threshold in the $\lambda \to 0^+$ limit is larger than its unregularized counterpart. While this observation is plausible and does not contradict our previous results, the behavior differs from that observed at larger values of $\kappa$. In particular, this suggests that the implicit bias of the unregularized dynamics and the explicit bias induced by vanishing regularization select different solutions. As a result, each mechanism may be advantageous in different regions of the parameter space.

In the unregularized case, the perfect recovery threshold exhibits the same qualitative behavior as in the small regularization limit. At fixed $\kappa^*$, it increases with $\kappa$ up to a certain value and then becomes constant. This occurs before $\kappa = 1$ and indicates that, in this regime, increasing the number of parameters does not require more data to achieve perfect recovery. From the expression~\eqref{eq:PRThresholdUnreg}, one can show that $\alpha_{\mathrm{PR}}$ becomes constant at $\kappa = (1 + \kappa^*) / 2$, provided that $\kappa^* \leq 1$. This value coincides with the interpolation threshold given in \eqref{eq:ConstantInterThreshold}. 

Finally, we emphasize that \cref{Conjecture:PRThreshold} is confirmed by a large number of numerical simulations. More details on these simulations, the results and the measure of the PR threshold can be found in \cref{App:Subsubsec:SimusPRThreshold}. 

\subsection{Gaussian Equivalence and Quadratic Neural Networks} \label{subsec:GaussianUniversality}

As emphasized earlier, the Gaussian structure of the data is a key assumption in our results. It enables to rewrite the dynamical partition function associated with the gradient flow dynamics~\eqref{eq:GFdynamics} by averaging with respect to this distribution and leads to the dynamics given in \cref{Result0}. 

However, we believe that our conclusions should extend beyond the Gaussian setting to a broader class of distributions. The idea that key phenomena or asymptotic behaviors in complex systems remain unchanged regardless of the specific underlying distribution is often referred to as Gaussian universality. This principle has been supported by a long line of work establishing such universality results through asymptotic analyses of high-dimensional inference problems \citep{hu2022universality, montanari2022universality,  dandi2023universality, gerace2024gaussian, bandeira2025exact, xu2025fundamental, wen2025does}. Unlike these works that derive universality results for static problems (either empirical risk minimization or Bayes-optimal learning), our conjecture bears on the behavior of gradient flow trajectories. Such dynamical universality results, although less studied, have been investigated in a few prior works \citep{celentano2021high, goldt2022gaussian, han2025entrywise}. 

In the following, we study the case where the sensing matrices are (centered) rank-one measurements:
\begin{equation} \label{eq:sensingquadratic}
    X_k = \frac{x_kx_k^\top - I_d}{\sqrt{d}}, \hspace{1.5cm} x_1, \dots, x_n \overset{\mathrm{i.i.d.}}{\sim} \N(0, I_d).  
\end{equation}
As underlined in \cref{Sec:Setting}, this setting corresponds to a shallow neural network with quadratic activation function, and fixed output weights. Our conjecture bears on the equivalence between this model and the Gaussian matrix sensing setting for which we derived our results.

\begin{conjecture} \label{Conjecture:GaussianUniversality}
The conclusion of \cref{Result0} still holds if the sensing matrices $X_1, \dots, X_n$ are i.i.d. distributed as in \eqref{eq:sensingquadratic}. 
\end{conjecture}

This conjecture is inspired by some recent works that derived universality results for quadratic neural networks, in the Bayes-optimal setting \citep{maillard2024bayes} and for empirical risk minimization \citep{erba2025nuclear}. Although we formulate it as a conjecture, we build in \cref{App:Subsec:GaussianEquivalence}
an argument that would lead to a proof, but requires a more rigorous treatment.  

\begin{figure}[ht]
    \centering
    \includegraphics[width=\linewidth]{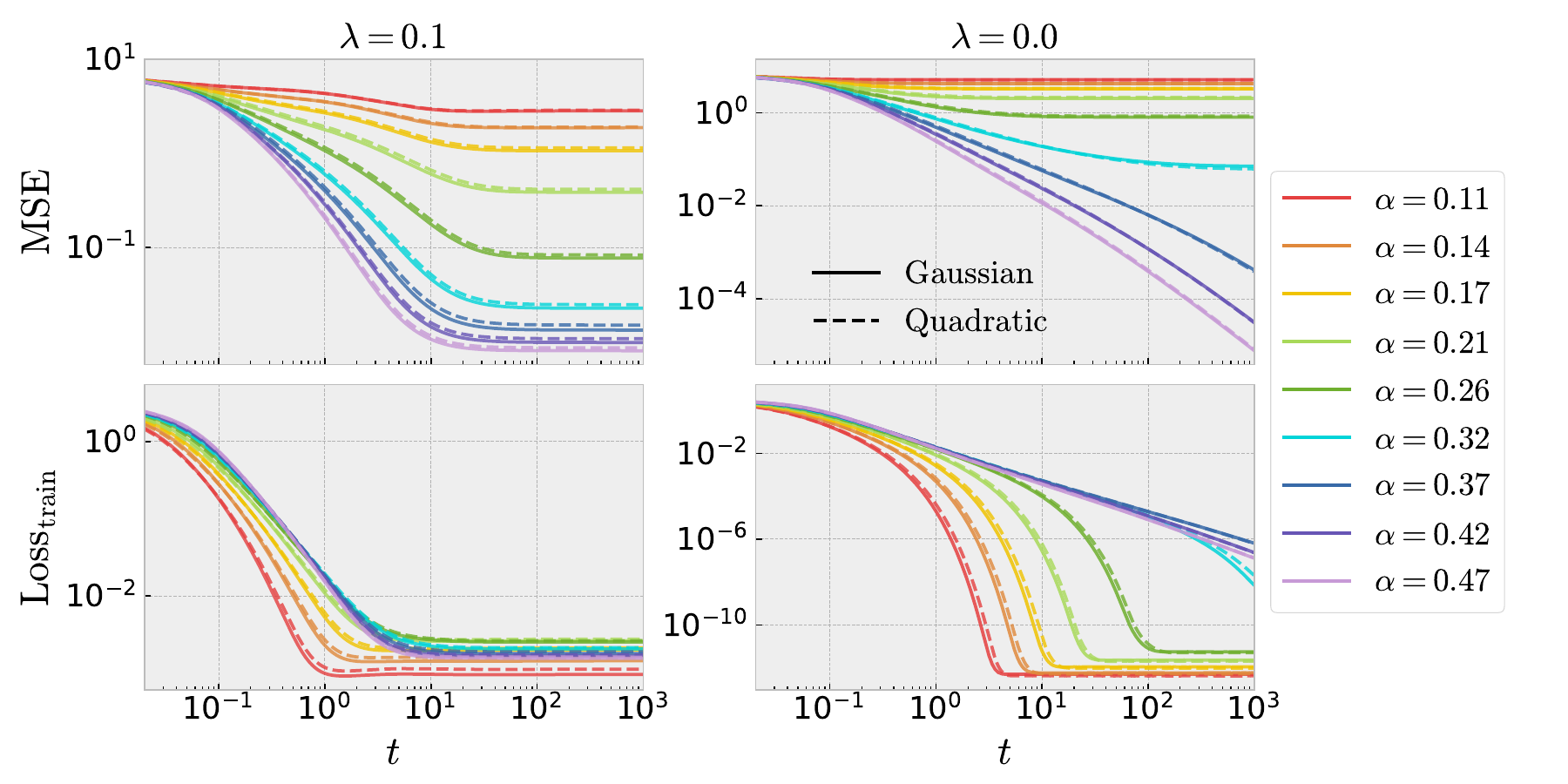}
    \vspace*{-0.7cm}
    \caption{MSE and loss achieved by gradient descent, defined in equation \eqref{eq:GDdynamics}, as a function of time for $d = 150, \kappa = 0.5, \kappa^* = 0.2$, $\Delta = 0$, $\lambda = 0.1$ (left) and $\lambda = 0$ (right), for several values of $\alpha$. Full lines: Gaussian data. Dashed lines: data generated as in equation~\eqref{eq:sensingquadratic}. Gradient descent simulations are averaged over 10 realizations of the initialization, teacher and data.}
    \label{fig:FigQuad1}
\end{figure}

In \cref{fig:FigQuad1}, we numerically justify this claim by comparing time-dependent averaged quantities (MSE, loss) along both gradient descent dynamics (for the Gaussian and quadratic cases). Despite some slight discrepancies that should originate from the finite dimension (here $d = 150$), the trajectories look very similar and provide a first confirmation of this equivalence result.

To conclude this part, one question that would require a deeper clarification is to which extent this Gaussian universality result holds. Indeed, as it is often the case in such results, we believe that \cref{Conjecture:GaussianUniversality} can be extended to a larger class of distributions. Therefore, it would be of interest to understand what would be the necessary requirements on this distribution to guarantee universality in our setting. 

\section{Additional Results on the Oja Flow} \label{Sec:OjaFlow}
In this section, we consider the Oja flow dynamics, corresponding to the approximate dynamics we use in \cref{Subsec:LongTimes} to derive our long-time equations for the gradient flow trajectory. In the following, we derive several results that are then used to prove the main claims of \cref{Subsec:LongTimes}. 

In the following, we let $A \in \mathcal{S}_d(\R)$ be a symmetric matrix, and consider the nonlinear matrix differential equation:
\begin{equation} \label{eq:OjaFlow}
    \dot W(t) = \big(A - W(t)W(t)^\top \big) W(t),
\end{equation}
where $W(t) \in \R^{d \times m}$. Moreover, note that $W(t)$ is solution of the gradient flow associated with the loss:
\begin{equation} \label{eq:PotentialOja}
    U(W) = \frac{1}{4} \big\| WW^\top - A \big\|_F^2. 
\end{equation}
Before presenting some results, we mention several key references on this topic. The Oja flow was introduced in the seminal paper by \citet{oja1982simplified}, and was designed as a continuous-time model for principal component analysis. Subsequent works studied its convergence properties \citep{yan1994global, tsuzuki2025global}. This flow is naturally connected to high-dimensional inference problems, such as low-rank matrix factorization and matrix sensing. Such formulations are central to the Burer–Monteiro approach to semidefinite programming \citep{burer2003nonlinear, boumal2016non} and can also be viewed as part of the broader framework of optimization on matrix manifolds \citep{absil2008optimization}. Finally, related ideas appear in matrix denoising: for instance, \citet{bodin2023gradient} use the structure of the Oja flow to derive a high-dimensional limit for extensive-rank positive semidefinite denoising dynamics.

In the following, we study this dynamics and present new results, that we directly use to study the high-dimensional dynamics for our matrix sensing problem. More precisely:
\begin{itemize}
    \item In \cref{Subsec:OjaKnownResults}, we start by mentioning several existing results on the Oja flow dynamics. 
    \item In \cref{Subsec:OjaCVFinitedim} and \cref{Subsec:OjaCVInfinitedim}, we then derive tight convergence rates in both finite and infinite dimension. We compare those with existing literature on the subject.
    \item In \cref{Subsec:OjaResponse}, we finally derive the linear response of the solution with respect to a time-dependent perturbation. 
\end{itemize}

\subsection{Known Properties of the Oja Flow} \label{Subsec:OjaKnownResults}

We start with some known results on the Oja flow dynamics. The first one we present is a closed-form solution, expressed in terms of the matrix $Z(t) = W(t)W(t)^\top$. 
\begin{proposition} \label{prop:OjaSolution} Let $\big(W(t) \big)_{t \geq 0}$ be the solution of~\eqref{eq:OjaFlow} with initial condition $W_0$. Then, at all times:
\begin{equation} \label{eq:SolOjaFlow}
    W(t)W(t)^\top = e^{tA} W_0 \left( I_m + 2 W_0^\top \int_0^t e^{2sA} \d s \, W_0 \right)^{-1} W_0^\top e^{t A}. 
\end{equation}
\end{proposition}
This result, already present in several previous works \citep{yan1994global, bodin2023gradient, martin2024impact}, can easily be derived by writing the differential equation solved by $Z(t)$ and remarking that~\eqref{eq:SolOjaFlow} solves this equation with the right initial condition. This explicit solution of the Oja flow is very powerful and will be a key result when studying the high-dimensional behavior of the dynamics.

We shall now move on to the convergence properties of the Oja flow. Indeed, this dynamics is very well understood at long times. Originally designed as a dynamical model for principal component recovery, it is no surprise that the dynamics converges to a point that selects the largest eigenvalues of the target matrix $A$. 

\begin{proposition} \label{prop:OjaConvergence}
    Let $W(t)$ be solution of the Oja flow~\eqref{eq:OjaFlow}. Assume that:
    \begin{itemize}
        \item The dynamics is initialized at $t = 0$ with a random matrix whose distribution is absolutely continuous with respect to the Lebesgue measure on $\R^{d \times m}$.
        \item The eigenvalues of $A$ are simple.
    \end{itemize}
    Then, with probability one, $W(t)$ converges as $t \to \infty$ toward some $W_\infty$ satisfying:
    \begin{equation}
        W_\infty W_\infty^\top = A_{(m)}^+, 
    \end{equation}
    where $A_{(m)}^+$ is defined in equation \eqref{eq:defPositivePart} and is obtained spectrally by selecting the $m$ largest positive eigenvalues of $A$.  
\end{proposition}

The proof of this result is standard and is deferred to \cref{App:Subsec:OjaConvergence}. It uses the analysis of the local minimizers associated with the loss~\eqref{eq:PotentialOja}, along with the fact that, under a random initialization, gradient flow almost always converges to a local minimizer. 

\subsection{Finite-Dimensional Convergence Rates} \label{Subsec:OjaCVFinitedim}

In this part, we mention the convergence rates associated with the Oja flow toward a solution of \cref{prop:OjaConvergence}. In the interest of our main results in \cref{Subsec:LongTimes}, we restrict ourselves to the case where the target matrix $A$ is invertible and has simple eigenvalues. Indeed, this assumption allows us to obtain some exponentially fast convergence rates. Regarding prior works on this problem, several works have already quantified such convergence in the presence of a positive-definite \citep{yan1994global} or rank-deficient \citep{martin2024impact} target, and \citet{tsuzuki2025global} analyzed the convergence on the Stiefel manifold. Interestingly, several works \citep{sarao2020optimization, martin2024impact} also identified a power-law convergence in the overparameterized setting. 

In the following, we denote $p$ the number of positive eigenvalues of $A$. Our method is the following: 
\begin{itemize}
    \item For the case $m \leq p$, the gradient flow dynamics converges toward a matrix $W$ which is of full rank. We then prove that in the general case, for functions of $WW^\top$, exponentially fast convergence is guaranteed provided that the Hessian is positive definite when restricted to an appropriate subspace. 
    \item When $m > p$, the gradient flow converges toward a matrix of rank $p < \min(m,d)$, and the previous reasoning does not apply anymore. Using a similar technique to \citet{martin2024impact}, we derive the rates via Grönwall-type bounds.
\end{itemize}

\subsubsection{Convergence to a Full-Rank Matrix} 

We first consider the case $m \leq p$. Then, $A$ has more positive eigenvalues than $m$ and the limit derived in \cref{prop:OjaConvergence} is of rank $m$ (obviously in this case one has $m \leq d$). The first result uses the fact that the loss function~\eqref{eq:PotentialOja} exhibits symmetries: it is a function of $WW^\top$ and associated with this symmetry is a set of what we call \textit{uninformative directions}. These directions are flat for the loss $U$ and should not compromise exponentially fast convergence. Indeed, we show that these directions are not seen through the gradient flow trajectory and that we can restrict the Hessian on the informative directions to study convergence. 

\begin{proposition} \label{prop:CVRates1}
    Let $F(W) = G(WW^\top)$ for $W \in \R^{d \times m}$ where $G \colon \mathcal{S}_d(\R) \to \R$ is $\mathcal{C}^2$. Define:
    \begin{equation}
        \mathcal{H}_W = \Big\{ K \in \R^{d \times m}, \, W^\top K = K^\top W  \Big\},
    \end{equation}
    and consider the gradient flow dynamics $\dot W(t) = - \nabla F(W(t))$. Assume that:
    \begin{itemize}
        \item $W(t) \xrightarrow[t \to \infty]{} W_\infty$ with $\mathrm{rank}(W_\infty) = m \leq d$. 
        \item The restriction of the Hessian of $F$ on $\mathcal{H}_{W_\infty}$ is positive definite with smallest eigenvalue $\varrho > 0$.
    \end{itemize}
    Then, for all $0 < c < \varrho$, we have:
    \begin{equation} \label{eq:ExponentialConvergenceGeneral}
        F\big( W(t) \big) - F(W_\infty) \underset{t \to \infty}{=} o \big( e^{-2ct} \big), \hspace{1.5cm} \big\| Z(t) - Z_\infty \big\| \underset{t \to \infty}{=} o \big( e^{-ct} \big),
    \end{equation} 
    with $Z(t) = W(t)W(t)^\top$ and $Z_\infty = W_\infty W_\infty^\top$. 
\end{proposition}

We prove this proposition in \cref{App:Subsec:CVRates1}. The proof only involves elementary linear algebra and ordinary differential equations tools, but one could interpret the gradient flow dynamics as a projected gradient flow on an appropriate quotient space. Indeed, the subspace $\mathcal{H}_W$ precisely corresponds to the horizontal space associated with the quotient manifold $\big\{ W \in \R^{d \times m}, \, \mathrm{rank}(W) = m \big\}$ with respect to the equivalence relation $V \sim W \Longleftrightarrow VV^\top = WW^\top$. Then, the orthogonal complement of $\mathcal{H}_W$ corresponds to the directions such that the map $W \mapsto WW^\top$ remains constant at first order. Regarding the gradient flow dynamics, these correspond to the uninformative directions, that are inherently flat but do not influence the convergence. More details on the study of this manifold can be found in \citep{massart2020quotient}. 

Therefore, in order to derive the convergence rates, we only need to study the Hessian of the loss~\eqref{eq:PotentialOja}. It can be easily computed:
\begin{equation} \label{eq:HessianOja}
    \d^2 U_W(H) = \big( WW^\top - A \big) H + H W^\top W + WH^\top W.
\end{equation}
We will study this linear map at the point of convergence of the dynamics, i.e., a $W \in \R^{d \times m}$ such that $WW^\top = A_{(m)}^+$. We gather the results obtained in the following proposition:
\begin{proposition} \label{Prop:HessianOja}
    Consider $W \in \R^{d \times m}$ such that $WW^\top = A_{(m)}^+$, and the restriction of $\d^2 U_W$ on the space $\mathcal{H}_W$ (defined in \cref{prop:CVRates1}). Denote $\lambda_1 > \dots > \lambda_d$ the ordered eigenvalues of $A$ and $p$ the number of its positive eigenvalues. Then, for $m \leq p$, on $\mathcal{H}_W$, the Hessian $\d^2 F_W$ has eigenvalues:
\begin{center} \vspace{-0.2cm} \begin{tabular}{c|c|c}
        Eigenvalue & Multiplicity & Indices  \\
        \hline
    $\lambda_i + \lambda_j$ & $1$ &$1 \leq i \leq j \leq m$ \\
    $\lambda_i - \lambda_j$ & $1$ & $i \in \llbracket 1, m \rrbracket$, $j \in \llbracket m+1, d \rrbracket$
\end{tabular} \end{center}
Then, if the flow~\eqref{eq:OjaFlow} converges toward a solution given in \cref{prop:OjaConvergence}, the convergence is exponentially fast, in the sense of equation \eqref{eq:ExponentialConvergenceGeneral}, with rate:
\begin{equation}
    \varrho_\text{CV} = \min \big( 2 \lambda_m, \lambda_m - \lambda_{m+1} \big). 
\end{equation}
\end{proposition}

This proposition is proved in \cref{App:Subsec:HessianOja}. As a remark, we obtain exponentially fast convergence, but in the usual random matrix theory setting, we expect these rates to be of order $d^{-1}$. Of course, given the spectrum of the Hessian, this only concerns a very small number of directions.

\subsubsection{Convergence to a Rank-Deficient Matrix} 

In the case $m > p$, the dynamics converges to a matrix of rank $p$. In this case the result of \cref{prop:CVRates1} does not apply. The convergence rate of the dynamics can be derived by decomposing the dynamics onto the subspaces associated with the positive and negative eigenvalues of $A$. Some standard bounds and Grönwall-type arguments lead to the following result:

\begin{proposition} \label{Prop:CVRates2}
    Consider the case $m > p$ and denote $W(t)$ the solution of~\eqref{eq:OjaFlow}. Assume that $W(t)$ converges to the limit given in \cref{prop:OjaConvergence}. Denote $\lambda_1 > \dots > \lambda_d$ the ordered eigenvalues of $A$ and $p$ the number of its positive eigenvalues. Then, for all $0 < c < \min(2\lambda_p, |\lambda_{p+1}|)$, we have:
    \begin{equation}
        U(W(t)) - U(W_\infty) = o(e^{-2ct}), \hspace{1.5cm} \big\|Z(t) - Z_\infty\big\| = o(e^{-ct}). 
    \end{equation}
\end{proposition}

We prove this proposition in \cref{App:Subsec:CVRates2}. As suggested, the rate depends on the eigenvalues of $A$ near zero, since $\lambda_p$ is the smallest positive eigenvalue of $A$ and $\lambda_{p+1}$ is the largest negative. For instance, in high dimension, if $A$ admits a converging spectral density with mass near zero, this rate is typically of order $d^{-1}$, which is the same scaling as the previous case. This invites us to investigate the convergence rates in the high-dimensional limit, which we do in the next part. 

\subsection{Infinite-Dimensional Convergence Rates}  \label{Subsec:OjaCVInfinitedim}

Building on the previous convergence rates at finite dimension, we derive similar results but by first taking the high-dimensional limit at fixed $t \geq 0$, and then the long-time limit. In the following, we assume that the empirical spectral distribution of $A$ converges to some $\mu_A$ in the high-dimensional limit. Moreover, consistently with the rest of the paper, we consider the limit where $m \sim \kappa d$ as $d \to \infty$. We then define the rescaled number of positive eigenvalues of $A$ in the limit:
\begin{equation}
    \kappa_A = \int \1_{x > 0} \d \mu_A(x). 
\end{equation}
As in \cref{Prop:HessianOja,Prop:CVRates2}, the following results will depend on the value of $\kappa$ with respect to $\kappa_A$. 

For the sake of simplicity, and in order to use random matrix theory results, we will assume that the initialization $W_0$ is a Gaussian matrix. However, since we are only interested in the long-time limit, we believe that our results will hold for more general distributions.

We also insist on the fact that the distribution of $A$ remains arbitrary in our case. Indeed, several similar results have been derived, for instance by \citet{bodin2023gradient}, but for more specific distributions. The result we prove is obtained in two steps:
\begin{itemize}
    \item We exploit the explicit solution of the Oja flow in \cref{prop:OjaSolution} in order to derive an exact high-dimensional limit for the function:
    \begin{equation}
        \frac{1}{d} \big\| Z(t) - \phi(A) \big\|_F^2, \hspace{1.5cm} Z(t) = W(t)W(t)^\top,
    \end{equation}
    for all $t \geq 0$, where $\phi$ is a spectral function of $A$. 
    \item By choosing $\phi$ associated with the limit of the dynamics, we then compute the $t \to \infty$ asymptotic of the previous quantity, depending on the value of $\kappa, \kappa_A$.  
\end{itemize}

\subsubsection{High-Dimensional Limit of the Dynamics}

We first derive the high-dimensional limit associated with the Oja flow.
\begin{proposition} \label{Prop:ConvergenceHighDimOja}
    Let $W(t) \in \R^{d \times m}$ be the solution of the Oja flow~\eqref{eq:OjaFlow} with initial condition $W(t = 0) = W_0$. Assume that:
    \begin{itemize}
        \item $W_0$ is a centered Gaussian matrix with i.i.d. coefficients of variance $1/m$.
        \item $A$ is almost surely invertible, and its empirical spectral distribution converges as $d \to \infty$ to some probability measure $\mu_A$. 
    \end{itemize} 
    Let $\phi$ be a spectral function on $\mathcal{S}_d(\R)$. Then, with $Z(t) = W(t)W(t)^\top$, for all $t \geq 0$, as $d \to \infty$ with $m \sim \kappa d$:
    \begin{equation} \label{eq:ConvergenceHighDimOja}
    \begin{aligned}
        \lim_{d \to \infty} \frac{1}{d} \Big\| Z(t) - \phi(A) \Big\|_F^2 &= \mathfrak{g}(t) \left( \int \left( \frac{x e^{xt}}{q_t(x)} \right)^2 \d \mu_A(x) \right)^2 \left( \kappa + \int \frac{z(e^{2zt} - 1)}{q_t(z)^2} \d \mu_A(z) \right)^{-1} \\
        &\hspace{2cm}+\int \left( \mathfrak{g}(t) \frac{xe^{2xt}}{q_t(x)} - \phi(x) \right)^2 \d \mu_A(x),
    \end{aligned}
    \end{equation}
    where $q_t(x) = (e^{2xt} - 1) \mathfrak{g}(t) + x$, and $\mathfrak{g}(t)$ solves the self-consistent equation:
    \begin{equation}
        \kappa \mathfrak{g}(t) + 1 - \kappa = \int \frac{x}{(e^{2xt}-1) \mathfrak{g}(t) + x} \d \mu_A(x).
\end{equation}
\end{proposition}

This proposition is proved in \cref{App:Subsec:HighDimLimit}. The proof uses the explicit solution derived in \cref{prop:OjaSolution} and interprets this solution using the Stieltjes transform of a Gaussian matrix with time-dependent covariance. Then, the application of some standard random matrix results \citep[that can be found in][]{bun2017cleaning}, as well as some lemmas given in \cref{App:Subsec:HighDimLimit}, allow us to access~\eqref{eq:ConvergenceHighDimOja}. 

To the best of our knowledge, such a result is new and can be more widely applied to understand the Oja flow dynamics in high dimension. 

\subsubsection{Convergence Rates in the High-Dimensional Limit}

We now explain how to derive convergence rates starting from \cref{Prop:ConvergenceHighDimOja}. To obtain the distance to the limit using~\eqref{eq:ConvergenceHighDimOja}, one has to pick the spectral function $\phi$ corresponding to the $m$ largest positive eigenvalues of $A$. This is done by choosing:
\begin{equation}
    \phi(x) = x \1_{x \geq \max(0, \omega)}, \hspace{1.5cm} \kappa = \int \1_{x \geq \omega} \d \mu_A(x). 
\end{equation}
Here $\omega$ selects a fraction $\kappa$ of the measure $\mu_A$ corresponding to the largest eigenvalues and $\phi$ applies a threshold using $\omega$ and 0 to ensure that the selected eigenvalues are positive. Then, one can study the long-time limit of equation~\eqref{eq:ConvergenceHighDimOja} to get the following result:
\begin{proposition} \label{Prop:HighDimCVRates}
    Assume that $\mu_A$ admits a density $\rho_A$. Then, there exists absolute constants $C_1, C_2 > 0$ such that:
    \begin{equation} \label{eq:HighDimCVOja}
    \lim_{d \to \infty}\frac{1}{d} \big\| Z(t)  -Z_\infty \big\|_F^2 \underset{t \to \infty}{\sim} \left\{ \begin{array}{cc} C_1 \dfrac{\rho_A(0)}{t^3}, & \mathrm{if} \ \kappa > \min(\kappa_A, 1) \\
    C_2 \dfrac{\omega^2 \rho_A(\omega)}{t},  & \mathrm{if} \ \kappa < \min(\kappa_A, 1).   
    \end{array}  \right.
    \end{equation} 
\end{proposition}
This result is proved in \cref{App:Subsec:LongTime}. In both cases, the convergence speed is proportional to the value of the density $\rho_A$ near the smallest eigenvalue of $Z_\infty$. Interestingly, if the density is zero in a neighborhood of this value, the convergence remains exponential, which is consistent with the previous observations at finite dimension. On the other hand, when the density is positive at the threshold, because of the directions with vanishing exponential rates identified in \cref{Subsec:OjaCVFinitedim}, the convergence becomes a power-law in the high-dimensional limit. 

The critical case $\kappa = \min(\kappa_A, 1)$ is not covered by \cref{Prop:HighDimCVRates}. We believe this regime to be more technical to analyze, and we leave its study for future work. 

\begin{figure}[ht]
    \centering
    \includegraphics[width=\linewidth]{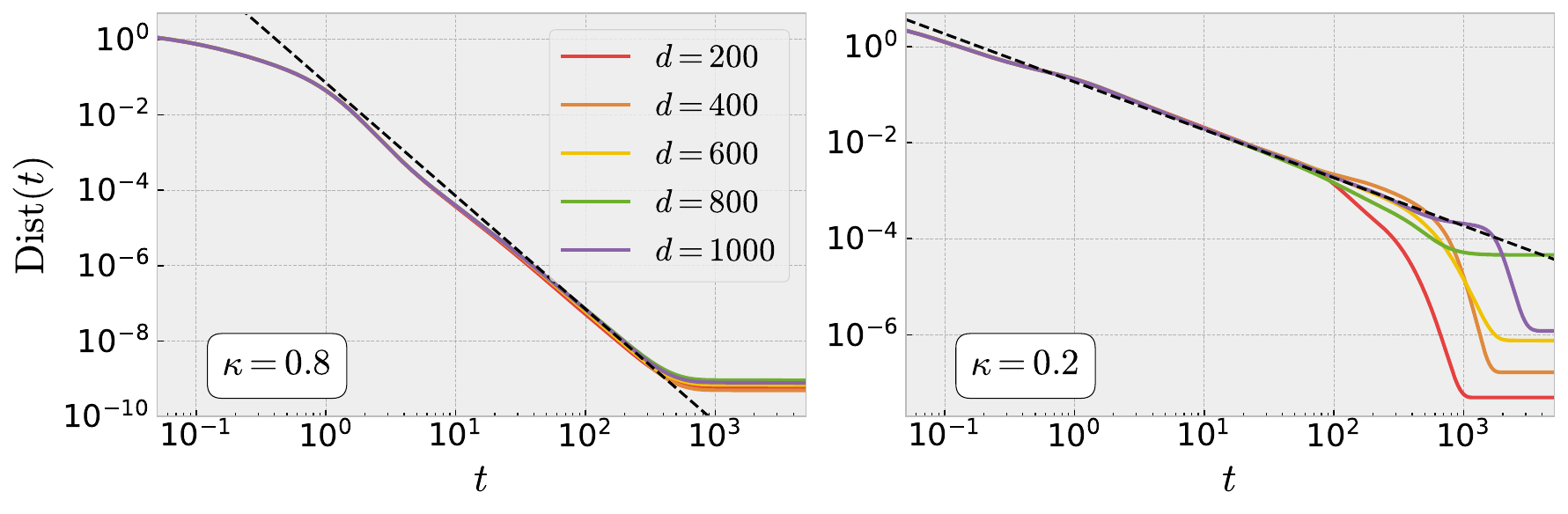}
    \vspace*{-0.7cm}
    \caption{Comparison between the numerical integration of the Oja flow \eqref{eq:OjaFlow} (with stepsize $\eta = 5 \times 10^{-3}$) and the power-law asymptotic of equation \eqref{eq:HighDimCVOja}. Colored lines: the quantity $\mathrm{Dist}(t) = \frac{1}{d} \|Z(t) - Z_\infty\|_F^2$, averaged over 20 repetitions of the initialization and target matrix (here distributed as a GOE), and plotted for several values of the dimension $d$. Black dashed lines: exact power-law asymptotics given by equation \eqref{eq:HighDimCVOja} (the value of the constants can be found in \cref{App:Subsec:LongTime}).}
    \label{fig:CVOja}
\end{figure}

\cref{fig:CVOja} confirms the previous asymptotics. In this figure, we plot the distance to convergence (left-hand side of equation \ref{eq:HighDimCVOja}) as a function of time, and with the choice of a target matrix $A \sim \mathrm{GOE}(d)$ (see \cref{def:GOE} for details), so that in this case $\kappa_A = 1/2$. This function of the dynamics is then compared with the asymptotic prediction of \eqref{eq:HighDimCVOja} (black dashed line): the left panel shows the case $\kappa > \kappa_A$ and the right panel $\kappa < \kappa_A$, leading to the distinct power-law behaviors. The comparison reveals an excellent match, down to the constant factor. For very large values of $t$, the asymptotics seem to break, but we believe that this is due to numerical precision limits in the left panel, and to finite-dimensional effects on the right panel. Indeed, recall that \cref{Subsec:OjaCVFinitedim} derived exponentially fast convergence rates: these still occur when the dimension is large, but only on a timescale of order $d$.  

\subsection{Linear Response} \label{Subsec:OjaResponse}

We shall now study the linear response associated with the Oja flow. This object is central in the high-dimensional equations derived in \cref{subsec:DMFT} and will be necessary to close the system of equations in \cref{Subsec:LongTimes}. We shall now consider the perturbed Oja flow dynamics:
\begin{equation} \label{eq:OjaFlowPerturbed}
    \dot W(t) = \big( A - W(t)W(t)^\top \big) W(t) + H(t) W(t),
\end{equation}
for some $H(t) \in \mathcal{S}_d(\R)$, and we compute the linear response operator:
\begin{equation}
    R(t,t') = \left. \frac{\partial Z(t)}{\partial H(t')} \right|_{H = 0}, \hspace{1.5cm} Z(t) = W(t)W(t)^\top.
\end{equation}
Note that, as emphasized earlier, $R(t,t')$ is a linear map $\mathcal{S}_d(\R) \to \mathcal{S}_d(\R)$. The following proposition gives its expression: 
\begin{proposition} \label{Prop:OjaResponse1}
    The linear response of the Oja flow is given by, for $H \in \mathcal{S}_d(\R)$:
    \begin{equation}
        R(t,t')(H) = \Big( U(t,t') H V(t,t')^\top + V(t,t') H U(t,t')^\top \Big) \1_{t' \leq t}, 
    \end{equation}
    where:
    \begin{equation}
        U(t,t') = P(t)^{-1} P(t'), \hspace{1.2cm} V(t,t') = P(t)^{-1} Z_0 e^{t'A}, \hspace{1.2cm} P(t) = e^{-tA} + Z_0 s_t(A),
    \end{equation}
    and the function $s_t$, defined by $s_t(\lambda) = \lambda^{-1} \big( e^{t \lambda} - e^{-t \lambda} \big)$ if $\lambda \neq 0$ and $s_t(0) = 2t$ is applied spectrally to $A$. In the basis of $\mathcal{S}_d(\R)$:
    \begin{equation*}
        E_{ij} = \frac{e_i e_j^\top + e_je_i^\top}{2},
    \end{equation*}
    for $i \leq j$:
    \begin{align}
        R_{ijkl}(t,t') &\equiv \tr \Big( R(t,t')\big(E_{kl} \big) E_{ij} \Big) \\
        &= \frac{1}{2} \bigg(U_{ik}(t,t') V_{jl}(t,t') + U_{il}(t,t') V_{jk}(t,t') \notag \\
        &\hspace{1cm} + U_{jk}(t,t') V_{il}(t,t') + U_{jl}(t,t') V_{ik}(t,t') \bigg) \1_{t' \leq t} \notag. 
    \end{align}
\end{proposition}

This result is exact and does not require any long-time or high-dimensional assumption. The key point is to decompose $Z(t)$ between the solution at $H = 0$ and a solution of order~$H$, and obtain a differential equation on this solution. Finally, the differential equation can be solved by exploiting the explicit solution at $H = 0$ derived in \cref{prop:OjaSolution}. The proof details are given in \cref{App:Subsec:OjaResponse1}. 

We now compute the high-dimensional limit of the response operator. To do so, we fix some times $t,t'$ and consider the mean diagonal response:
\begin{equation}
    R_\text{diag}(t,t') = \lim_{d \to \infty} \frac{1}{d^2} \tr \left( \left. \frac{\partial Z(t)}{\partial H(t')} \right|_{H = 0} \right), \hspace{1.5cm} r_\text{diag}(t) = \int_0^t R_\text{diag}(t,t') \d t'. 
\end{equation}
These scalar quantities play an essential role in our system of equations, and in the case of the Oja flow they can be exactly computed. In the same fashion as in \cref{Prop:ConvergenceHighDimOja}, we exploit some well-known random matrix results that require the initialization $W_0$ to be a Gaussian matrix. However, we believe that all the results that will be considered in the $t \to \infty$ limit will not depend on this choice of distribution. 

\begin{proposition} \label{Prop:DiagonalResponse}
    Under the same assumptions as in \cref{Prop:ConvergenceHighDimOja}, with $m \sim \kappa d$ in the $d \to \infty$ limit, we have:
    \begin{align}
        R_\mathrm{diag}(t,t') &= \frac{\mathfrak{g}(t)}{2} \iint \frac{1}{q_t(x) q_t(y)} e^{(x+y)t} \Big( y(x - \mathfrak{g}(t)) e^{(y-x)t'} + x(y - \mathfrak{g}(t)) e^{(x-y)t'} \label{eq:Rdiagprop1} \\
        &\hspace{5cm} + (x+y) \mathfrak{g}(t) e^{(x+y)t'} \Big) \d \mu_A(x) \d \mu_A(y), \notag\\
        r_\mathrm{diag}(t) &= \frac{\mathfrak{g}(t)}{2} \iint \frac{1}{y-x} \left[ \frac{y e^{2yt}}{q_t(y)} - \frac{x e^{2xt}}{q_t(x)} \right] \d \mu_A(x) \d \mu_A(y), \label{eq:Rdiagprop2}
    \end{align}
    where $q_t(x) = \mathfrak{g}(t)(e^{2xt} - 1) + x$ and $\mathfrak{g}(t)$ solves the self-consistent equation:
    \begin{equation}
        \kappa \mathfrak{g}(t) + 1 - \kappa = \int \frac{x}{(e^{2xt}-1) \mathfrak{g}(t) + x} \d \mu_A(x).
    \end{equation}
\end{proposition}

We prove this proposition in \cref{App:Subsec:OjaResponse2}. 

Finally, we emphasize that it is technically possible to derive a closed-form expression for higher-order responses for the Oja flow dynamics. In \cref{App:Subsec:DynamicalStability}, we explain how this can be achieved and how these responses can be used to investigate the stability of the long-time approximations made on the high-dimensional equations in \cref{Subsec:LongTimes}.

\section{Conclusion and Perspectives} \label{Sec:Conclusion}
In this work, we studied the high-dimensional training dynamics of a shallow neural network with quadratic activation function in a teacher--student setting. In the extensive-width regime, we derived a high-dimensional description of the gradient flow dynamics on the empirical loss and obtained asymptotic equations governing its long-time behavior. These results clarify how overparameterization affects learning and generalization, and help understand overfitting and double descent phenomena. Overall, our findings contribute to a clearer theoretical picture of learning dynamics in high-dimensional, overparameterized neural networks.

\paragraph{Toward a more rigorous analysis.} Despite the theoretical insights and numerical evidence presented in this paper, several aspects of our analysis call for a more rigorous treatment. 
\begin{itemize}
    \item The derivation of the equivalent dynamics in the high-dimensional limit (\cref{subsec:DMFT}) relies on statistical physics techniques that have been applied to several learning problems. To obtain fully rigorous results, one would need to generalize the mathematical analysis developed for finite width to the extensive width case, in order to characterize the limiting high-dimensional object and to formalize convergence in distribution.
    \item The simplification of the dynamics at long times (\cref{Subsec:LongTimes}) is based on structural assumptions on the dynamics. Confirming this ansatz mathematically is challenging, and comparable rigorous results exist only for simpler models \citep{celentano2021high, fan2025dynamical, chen2025learning}.
    \item Our analysis relies on a Gaussian surrogate for the rank-one sensing matrices associated with the quadratic model. While we provide theoretical arguments and numerical evidence supporting the equivalence between these two models (\cref{subsec:GaussianUniversality}), we believe that this equivalence could be established rigorously and extended to a broader class of distributions. Similar universality results are well understood in the static setting \citep{maillard2024bayes, erba2025nuclear, xu2025fundamental}, but extending them to high-dimensional trajectories remains an open problem.
\end{itemize}

\paragraph{Generalizations.} We expect that the dynamical framework developed in this work can be extended to other settings of interest. As noted by \citet{maillard2024bayes}, studying polynomial activation functions leads to a lifting of the model to a linear one in the space of tensors. This extension introduces new mathematical challenges, in particular the analysis of the resulting nonlinear tensor dynamics under similar assumptions as those introduced in \cref{Subsec:LongTimes}.

Another promising direction concerns the study of learning dynamics in transformer architectures. Simplified attention mechanisms have already been analyzed in the Bayes-optimal and empirical risk minimization settings, including in deep architectures \citep{troiani2025fundamental, erba2025bilinear, boncoraglio2025inductive, boncoraglio2025bayes}. Extending these approaches to a dynamical setting remains challenging and calls for further attention.

\section*{Acknowledgments}
We thank Lenka Zdeborová, Florent Krzakala, Emanuele Troiani, Vittorio Erba, Antoine Maillard, Bruno Loureiro, Guilhem Semerjian, Gérard Ben Arous, Lénaïc Chizat, Louis-Pierre Chaintron, Nicolas Boumal and Andreea-Alexandra Mu\c{s}at for useful remarks and discussions. We also thank the anonymous reviewers for their helpful comments and suggestions.

This work has received support from the French government, managed by the National Research Agency, under the France 2030 program with the reference “PR[AI]RIE-PSAI”
(ANR-23-IACL-0008).

\bibliography{Biblio}

\newpage
\appendix

\section{Guide to the Appendices}
The appendices contain several types of technical material. This short guide summarizes the role of each appendix and clarifies its connections to the main text as well as its level of rigor. The latter is indicated in the following table with the color code of the first column: a green cell refers to a section containing rigorous results, orange stands for a physics-based derivation, and purple is used for numerical confirmation.

\begin{center} \small \renewcommand{\arraystretch}{1.25}
\begin{tabular}{
|>{\raggedright\arraybackslash}p{0.1\textwidth}
|>{\raggedright\arraybackslash}p{0.6\textwidth}
|>{\raggedright\arraybackslash}p{0.15\textwidth}|}
\hline
\textbf{Section} & \textbf{Content} & \textbf{Used in} \\ 
\hline 
\cellcolor{rigorcolor} \cref{App:Prerequisite} & Preliminary technical results. &\\ 
\hline
\cellcolor{predcolor} \cref{App:DMFT} & Derivation of the high-dimensional DMFT equations, study of the simplified setting (\cref{Assumption2}), and discussion on Gaussian equivalence. & \cref{subsec:DMFT}, \cref{subsec:GaussianUniversality} \\ 
\hline 
\cellcolor{predcolor} \cref{App:LongTimes} & Derivation of the long-time scalar equations from the steady-state ansatz (\cref{Assumption3}), correspondence with \citet{erba2025nuclear}, impact of overparameterization, population limit. & \cref{Subsec:LongTimes} \\ 
\hline 
\cellcolor{predcolor} \cref{App:Stability} & Stability analysis of the steady-state solution. & \cref{Subsubsec:Stability} \\ 
\hline 
\cellcolor{predcolor} \cref{App:Langevin} & Analysis of Langevin dynamics from time-translational invariance and fluctuation-dissipation (\cref{Assumption4}). & \cref{Subsubsec:Langevin} \\ 
\hline 
\cellcolor{predcolor} \cref{App:Aging} & Derivation of the aging equations under the assumption of slow relaxation to equilibrium. & \cref{Subsubsec:Stability} \\ 
\hline 
\cellcolor{rigorcolor} \cref{App:SmallReg} & Study of the small regularization limit and derivation of the interpolation and perfect recovery thresholds. & \cref{subsec:SmallReg} \\ 
\hline 
\cellcolor{rigorcolor} \cref{App:Oja} & Analysis of the Oja flow: finite and infinite-dimensional convergence rates, linear response. & \cref{Sec:OjaFlow} \\ 
\hline 
\cellcolor{validcolor} \cref{App:Simulations} & Numerical details and additional simulations. & \\ \hline \end{tabular} \end{center}

Additionally, we refer to \cref{fig:Logicalstructure} for the dependencies between the main results and appendices of the paper.

\section{Technical Background} \label{App:Prerequisite}
This section introduces additional definitions, background material, and technical lemmas that will be used throughout the appendix. More precisely:
\begin{itemize}
    \item In \cref{App:Subsec:RandomMatrices}, we recall some standard results on random matrices, introduce Stieltjes and Hilbert transforms of probability measures, and study the case of the free convolution with a semicircular density. 
    \item In \cref{App:Subsec:PSDApprox}, we state useful properties of the best low-rank positive semidefinite approximation of a symmetric matrix.
    \item Finally, in \cref{App:Subsec:PrerequisiteDMFT}, we prove technical lemmas used in the derivation of the high-dimensional equations in \cref{App:Subsec:DMFTDerivation}. 
\end{itemize}

\subsection{Asymptotic Spectral Properties of Random Matrices} \label{App:Subsec:RandomMatrices}

We start by recalling standard random matrix ensembles and their spectral properties in the high-dimensional limit.

\subsubsection{Standard Matrix Ensembles} \label{App:Subsubsec:MatrixEnsembles}

We now proceed to define the GOE and Wishart distributions, and state the well-known convergence results regarding their respective spectral distributions in the large dimension limit. 

\begin{definition} \label{def:GOE}
    We say that $X \in \mathcal{S}_d(\R)$ is distributed according to the GOE($d$) distribution if:
    \begin{equation*}
        \hspace{3cm} X_{ij} \overset{\mathrm{i.i.d.}}{\sim} \N \left(0, \frac{1 + \delta_{ij}}{d} \right), \hspace{2.5cm} \big(1 \leq i \leq j \leq d \big). 
    \end{equation*}
    As a consequence, if $X \sim \mathrm{GOE}(d)$, the mapping $A \in \mathcal{S}_d(\R) \mapsto \tr(AX) \in \R$ defines a centered Gaussian process on $\mathcal{S}_d(\R)$ with covariance function:
    \begin{equation*}
        \E \, \tr(AX) \tr(BX) = \frac{2}{d} \tr(AB). 
    \end{equation*}
\end{definition}
More properties of this distribution can be found, for instance, in \citet[][Chapter 2]{anderson2010introduction}. 

\begin{definition} \label{def:Wishart}
    We say that $Z \in \mathcal{S}_d^+(\R)$ is distributed according to the Wishart distribution $\mathcal{W}(m,d)$ if:
    \begin{equation*}
        Z = \frac{1}{m} WW^\top, 
    \end{equation*}
    where $W \in \R^{d \times m}$ has i.i.d. standard Gaussian entries. 
\end{definition}

These two matrix ensembles are fundamental in random matrix theory and are used as key models for understanding the behavior of large complex systems. They appear in a wide range of fields, from statistical physics to statistics and machine learning, where they model random data and covariance structures. A key aspect of both ensembles is the behavior of their eigenvalue distributions in the high-dimensional limit. 

\begin{theorem} \citep{wigner1955characteristic, marchenko1967distribution}. Let $\G \sim \mathrm{GOE}(d)$ and $Z \sim \mathcal{W}(m,d)$, with $m \sim \kappa d$ as $d \to \infty$. Then, almost surely as $d \to \infty$, the empirical spectral distributions of $\G$ and $Z$ respectively weakly converge to the probability distributions:
\begin{itemize}
    \item The semicircular distribution $\sigma$:
    \begin{equation*}
        \d \sigma(x) = \frac{1}{2\pi} \sqrt{4-x^2} \bm{1}_{|x| \leq 2} \d x,
    \end{equation*}
    \item The Marchenko--Pastur distribution:
    \begin{equation*}
        \d \nu_{\kappa}(x) = (1 - \kappa)^+ \delta(x) + \frac{\kappa}{2\pi x} \sqrt{(\lambda_+ - x)(x - \lambda_-)} \d x, \hspace{1.5cm} \lambda_\pm = \left( 1 \pm \frac{1}{\sqrt{\kappa}} \right)^2. 
    \end{equation*}
\end{itemize}
\end{theorem}

It will be useful to consider the rescaled semicircular distribution with variance $\xi$:
\begin{equation} \label{eq:rescaledsemicircular}
    \d \sigma_\xi(x) = \frac{1}{2\pi \xi} \sqrt{4 \xi - x^2} \1_{|x| \leq 2 \sqrt{\xi}} \d x. 
\end{equation}

\subsubsection{Stieltjes and Hilbert Transforms} \label{App:Subsubsec:StieltjesHilbert}

We shall now define two important transforms of probability measures. More details can be found in \citet{bun2017cleaning}. 

\begin{definition} \label{Def:HilbertStieltjes}
    For a probability measure $\mu$ on $\R$, we define its Stieltjes transform $m_\mu$ and Hilbert transform $h_\mu$ as:
    \begin{equation*}
         m_\mu(z) = \int \frac{\d \mu(y)}{z-y}, \hspace{2cm} h_\mu(x) = \mathrm{P.V.} \int \frac{\d \mu(y)}{x-y},
    \end{equation*}
    for $z \in \mathbb{C}$ and $x \in \R$. The notation P.V. denotes Cauchy's principal value. More precisely:
    \begin{equation} \label{eq:EpsilonHilbert}
        h_\mu(x) = \lim_{\epsilon \to 0} h_\mu^\epsilon(x), \hspace{2cm} h_\mu^\epsilon(x) = \int \1_{|x-y| > \epsilon} \frac{\d \mu(y)}{x-y}. 
    \end{equation}
    Moreover, if $\mu$ admits a density $\rho$, then, for $x \in \mathrm{Supp}(\mu)$:
    \begin{equation*}
        \lim_{\eta \to 0^+} m_\mu(x + i \eta) = h_\mu(x) - i \pi \rho(x). 
    \end{equation*}
\end{definition}

We now state a technical lemma on the limit in equation \eqref{eq:EpsilonHilbert}. 

\begin{lemma} \label{Lemma:Hilbert}
    Let $\mu$ be a compactly supported probability measure with no atoms, and $h_\mu$ its Hilbert transform. For $\epsilon > 0$, consider $h_\mu^\epsilon$ as defined in equation \eqref{eq:EpsilonHilbert}. 
    \begin{description}
        \item[1.] Let $f : \R \to \R$ be either Lipschitz or non-decreasing on the support of $\mu$. Then:
    \begin{equation*}
        \lim_{\epsilon \to 0^+} \int f(x) h_\mu^\epsilon(x) \d \mu(x) = \frac{1}{2} \iint \frac{f(x) - f(y)}{x-y} \d \mu(x) \d \mu(y).
    \end{equation*}
        \item[2.] Moreover, if $\mu$ admits a bounded square-integrable density with respect to the Lebesgue measure, then for any $f \in L^2(\mu)$:
    \begin{equation*}
        \lim_{\epsilon \to 0^+} \int f(x) h_\mu^\epsilon(x) \d \mu(x) = \int f(x) h_\mu(x) \d \mu(x).
    \end{equation*}
    \end{description}
\end{lemma}
\begin{proof}
For $\epsilon > 0$, define:
\begin{equation*}
    I_\epsilon \equiv \frac{1}{2} \iint \frac{f(x) - f(y)}{x-y} \1_{|x-y| > \epsilon} \d \mu(x) \d \mu(y).
\end{equation*}
Then, it is clear, by symmetry, that:
\begin{equation} \label{eq:LemmaHilbert1}
    I_\epsilon = \int f(x) h_\mu^\epsilon(x) \d \mu(x).
\end{equation}
Now, if $f$ is non-decreasing, for any $x, y \in \mathrm{Supp}(\mu)$ such that $x \neq y$, the family:
\begin{equation*}
    \left( \frac{f(x) - f(y)}{x-y} \1_{|x-y| > \epsilon} \right)_{\epsilon > 0},
\end{equation*}
is non-decreasing as $\epsilon \to 0$, and $I_\epsilon$ converges by the monotone convergence theorem. On the other hand, if $f$ is Lipschitz, the previous family is bounded in absolute value by the Lipschitz constant of $f$ and the integral converges toward a finite value due to the dominated convergence theorem. In both cases the limit is given by:
\begin{equation*}
    \lim_{\epsilon \to 0} I_\epsilon = \frac{1}{2} \iint \frac{f(x) - f(y)}{x-y} \1_{x \neq y} \d \mu(x) \d \mu(y).
\end{equation*}
Since $\mu$ has not atoms, the diagonal has zero mass and the identity \eqref{eq:LemmaHilbert1} concludes the first part. 

For the second point, denote $\rho$ the density of $\mu$, and assume that $\rho$ is bounded and in $L^2(\d x)$. Now, it is standard (see for instance \citet{grafakos2008classical}, Theorem 5.1.12) that $h_\mu^\epsilon$ converges to $h_\mu$ in $L^2(\d x)$. Therefore:
\begin{equation*}
\begin{aligned}
    \left| \int f(x) \big( h_\mu(x) - h_\mu^\epsilon(x) \big) \d \mu(x) \right| &\leq \big\| f(h_\mu - h_\mu^\epsilon) \big\|_{L^1(\mu)} \\
    &\leq \|f\|_{L^2(\mu)} \sqrt{\|\rho\|_\infty} \|h_\mu - h_\mu^\epsilon\|_{L^2(\d x)}.
    \end{aligned}
\end{equation*}
By assumptions, this goes to zero as $\epsilon \to 0$. This gives the result. 
\end{proof}

We now formulate a direct corollary of this result.
\begin{lemma} \label{Lemma:Hilbert2}
    Let $\mu$ be a compactly supported probability measure, with no atoms, with a bounded and square-integrable density, and $h_\mu$ its Hilbert transform. Then:
    \begin{equation*}
        \int x h_\mu(x) \d \mu(x) = \frac{1}{2}. 
    \end{equation*}
\end{lemma}
\begin{proof}
    As a consequence of \cref{Lemma:Hilbert}, we have with $f(x) = x$:
    \begin{equation*}
        \int x h_\mu(x) \d \mu(x) = \lim_{\epsilon \to 0^+} \int x h_\mu^\epsilon(x) \d \mu(x) = \frac{1}{2} \iint \frac{x-y}{x-y} \d \mu(x) \d \mu(y) = \frac{1}{2}. 
    \end{equation*}
\end{proof}

\subsubsection{Free Convolution with a Semicircular Density} \label{App:Subsubsec:FreeConvolution}

In this section, we will be interested in the free additive convolution between some given probability distribution $\mu^*$ and $\sigma_\xi$, defined in equation \eqref{eq:rescaledsemicircular}. We will denote this resulting measure as $\mu_\xi = \mu^* \boxplus \sigma_\xi$. This object is of great interest for this work as it describes the asymptotic spectral distribution of a random matrix of the form $Z^* + \sqrt{\xi} \G$, for $\G \sim \mathrm{GOE}(d)$, and where $Z^*$ is a random matrix independent from $\G$, with asymptotic spectral distribution $\mu^*$. We refer to \citet{biane1997free} for several results on this object. 

The first relevant result for us is the so-called subordination identity between the Stieltjes transforms of $\mu_\xi$ and $\mu^*$. 

\begin{lemma} \label{Lemma:SubordinationSemicircular}
    Let $\mu_\xi$ be the free additive convolution between some probability measure $\mu^*$ and a semicircular density with variance $\xi > 0$. Consider $m_\xi$ and $m_*$ the Stieltjes transforms of $\mu_\xi$ and $\mu^*$. We have the equation, for all $z \in \mathbb{C} \setminus \mathrm{Supp}(\mu_\xi)$:
    \begin{equation}
        m_\xi(z) = m_* \big( z - \xi m_\xi(z) \big). 
    \end{equation}
\end{lemma}

A proof of this can be found in \citet{biane1997free}. This relationship between Stieltjes transforms is known to characterize the distribution $\mu_\xi$.

\begin{lemma} \label{lemma:BurgerStieltjes}
    Let $m_\xi$ be the Stieltjes transform of $\mu_\xi$. Then the map $(\xi, z) \mapsto m_\xi(z)$ is $\mathcal{C}^1$ on $\R^+ \times \mathbb{C} \setminus \R$ and analytic in $z$ for fixed $\xi$. Moreover, $m_\xi$ satisfies the complex Burgers' equation:
    \begin{equation}
        \partial_\xi m_\xi(z) + m_\xi(z) \partial_z m_\xi(z) = 0,
    \end{equation}
    whenever $\xi \geq 0$ and $z \in \mathbb{C} \setminus \mathbb{R}$, with initial condition $m_0(z) = m_*(z)$. 
\end{lemma}
A reference for this result can be found in \citet{voiculescu1997derivative}. 

\subsection{Best PSD Low-Rank Approximation} \label{App:Subsec:PSDApprox}

Throughout the main text and the appendix, we often deal with the notation $A_{(m)}^+$ for a matrix $A \in \mathcal{S}_d(\R)$. Although it is already defined in equation \eqref{eq:defPositivePart}, we now provide a formal definition as well as some relevant properties. 

\begin{definition} \label{Def:EckartYoung}
Let $A \in \mathcal{S}_d(\R)$ with eigenvalue decomposition $A = U \mathrm{diag}(\lambda_1, \dots, \lambda_d) U^\top$ and $\lambda_1 \geq \cdots \geq \lambda_d$, we define for $m \leq d$:
\begin{equation*}
A_{(m)}^{+} = U \mathrm{diag}\big(\lambda_1^+, \dots, \lambda_m^+, 0,\dots, 0 \big) U^\top,
\hspace{1.2cm} \lambda^+ = \max(\lambda, 0).
\end{equation*}
By convention, if $m > d$, we define $A_{(m)}^+ = A^+$ to be the positive part of $A$.  
\end{definition}
This matrix is known to be the best PSD approximation of $A$ by a matrix of rank at most $m$, for the Frobenius norm.  
\begin{lemma} \label{Lemma:MinEckartYoung}
    Assume that $A$ has simple eigenvalues. Then:
\begin{equation*}
    \Big\{ A_{(m)}^+ \Big\} = \underset{\substack{S \in \mathcal{S}_d^+(\R) \\ \mathrm{rank}(S) \leq m}}{\mathrm{argmin}} \big\| A - S \big\|_F^2.
\end{equation*}
\end{lemma}
\begin{proof}
This property is standard. To prove it, one can realize the argmin is the image through the map $W \mapsto WW^\top$ of the set:
\begin{equation*}
    \underset{W \in \R^{d \times m}}{\mathrm{argmin}} \big\| A - WW^\top \big\|_F^2.
\end{equation*}
Now, we precisely show in \cref{App:Subsec:OjaConvergence} that whenever $A$ has simple eigenvalues, any local minimizer of the map $W \mapsto \big\|A - WW^\top \big\|_F^2$ verifies $WW^\top = A_{(m)}^+$. Since any element in the previous argmin is a global minimizer, this concludes the proof.
\end{proof}

The following lemma computes the differential of the map $A \mapsto A_{(m)}^+$ on $\mathcal{S}_d(\R)$. 

\begin{lemma} \label{Lemma:DerivativeSpectral}
    Let $A \in \mathcal{S}_d(\R)$ with simple and non-zero eigenvalues $\lambda_1 > \dots > \lambda_d$. Write $A = U \mathrm{diag}(\lambda_1, \dots, \lambda_d) U^\top$. Then, the map $\Phi \colon X \in \mathcal{S}_d(\R) \mapsto X_{(m)}^+ \in \mathcal{S}_d(\R)$ is differentiable at $A$ and has differential:
    \begin{equation*}
        \d \Phi_A(H) = U \Big( D \circ \big( U^\top H U \big) \Big) U^\top,
    \end{equation*}
    where $\circ$ denotes the Hadamard product (pointwise product between matrices), and $D \in \mathcal{S}_d(\R)$ is defined as:
    \begin{equation*}
        D_{ij} = \left\{ \begin{array}{cc}
            \dfrac{\lambda_i \1_{\lambda_i > 0} \1_{i \leq m} - \lambda_j \1_{\lambda_j > 0} \1_{j \leq m}}{\lambda_i - \lambda_j}, & \mathrm{if} \, i \neq j, \\
            \1_{\lambda_i > 0} \1_{i \leq m}, & \mathrm{if} \, i = j.
        \end{array} \right.
    \end{equation*}
    As a consequence, the spectrum of $\d \Phi_A$ is given by:
    \begin{equation*}
        \mathrm{Sp} \big( \d \Phi_A \big) = \big\{ D_{ij}, \, 1 \leq i \leq j \leq d \big\},
    \end{equation*}
    and we have:
    \begin{equation*}
        \tr \, \d \Phi_A = \sum_{1 \leq i \leq j \leq d} D_{ij}, \hspace{1.5cm} \big\| \d \Phi_A \big\|_F^2 = \sum_{1 \leq i \leq j \leq d} D_{ij}^2.
    \end{equation*}
\end{lemma}

\begin{proof}
    As a consequence of Daleckiĭ--Kreĭn theorem \citep[see for instance][Theorem 2.11]{noferini2017formula}, provided that $F \colon \R \to \R$ is continuously differentiable on the spectrum of $A$, and if $F^\text{sp}$ is the associated spectral function on $\mathcal{S}_d(\R)$:
    \begin{equation} \label{eq:spectralfunction}
        F^\text{sp} \Big( U \mathrm{diag}(\alpha_1, \dots, \alpha_d) U^\top \Big) =  U \mathrm{diag}\big( F(\alpha_1), \dots, F(\alpha_d) \big) U^\top,
    \end{equation}
    for any orthogonal matrix $U \in O_d(\R)$ and $\alpha_1, \dots, \alpha_d \in \R^d$, then $F^\text{sp}$ is differentiable (in the Fréchet sense) at $A$ and has differential:
    \begin{equation} \label{eq:DerivativeSpectral1}
        \d F^\text{sp}_A(H) = U \Big( D \circ \big( U^\top H U \big) \Big) U^\top,
    \end{equation}
    where:
    \begin{equation} \label{eq:DerivativeSpectral2}
        D_{ij} = \left\{ \begin{array}{cc}
            \dfrac{F(\alpha_i) - F(\alpha_j)}{\alpha_i - \alpha_j}, & \mathrm{if} \, i \neq j, \\
            F'(\alpha_i), & \mathrm{if} \, i = j.
        \end{array} \right.
    \end{equation}
    Now recall that in our case, we are interested in $\Phi(X) = X_{(m)}^+$, i.e., the truncation of $X$ onto the subspace associated with the $m$ largest positive eigenvalues of $X$. Note that such a map is not properly a spectral function in the sense of equation~\eqref{eq:spectralfunction}. To solve this problem, we introduce some threshold $\omega$ such that $\lambda_{m+1} < \omega < \lambda_m$. Due to the continuity of the spectrum at $A$, that has simple eigenvalues by assumption, there exists a neighborhood $\mathcal{V} \in \mathcal{S}_d(\R)$ of $A$ such that for every $X \in \mathcal{V}$, we also have $\lambda_{m+1}(X) < \omega < \lambda_m(X)$. Therefore, on this neighborhood, $\Phi$ coincides with the spectral function associated with the real function $F(x) = x \1_{x \geq \max(0, \omega)}$. By assumption $F$ is differentiable on the spectrum of $A$, and combining the identities \eqref{eq:DerivativeSpectral1}, \eqref{eq:DerivativeSpectral2} this leads to the desired result.

    Regarding the spectrum of $\d \Phi_A$, note that $H \in \mathcal{S}_d(\R)$ solves $\d \Phi_A(H) = \gamma H$ for some $\gamma \in \R$ if and only if $K = U^\top H U$ solves $D \circ K = \gamma K$. Taking $K$ to be, for $1 \leq i \leq j \leq d$:
    \begin{equation*}
        K = \frac{e_ie_j^\top + e_j e_i^\top}{2},
    \end{equation*}
    leads to the desired identity with $\gamma = D_{ij}$.  

    The trace of $\d \Phi_A$ can now be computed as the sum of its eigenvalues. For the Frobenius norm, let us remark that, for any $H, K \in \mathcal{S}_d(\R)$:
    \begin{equation*}
        \tr \big( \d \Phi_A(H) K \big) = \sum_{i,j=1}^d D_{ij} \big(U^\top H U \big)_{ij} \big( U^\top K U \big)_{ij},
    \end{equation*}
    so that $\d \Phi_A$ is self-adjoint (as a linear map on $\mathcal{S}_d(\R)$) and $\big\| \d \Phi_A \big\|_F^2$ can be computed as the sum of the squared eigenvalues.
\end{proof}

\subsection{Prerequisite for the High-Dimensional Equations} \label{App:Subsec:PrerequisiteDMFT}

We now state a few lemmas that will be used in \cref{App:Subsec:DMFTDerivation} for the derivation of high-dimensional dynamics. We start with a lemma giving the derivative of a Gaussian expectation with respect to its covariance matrix. 

\begin{lemma} \label{Lemma:DerivativeGaussianExpectation}
    Let $X \in \R^N$ be a centered Gaussian vector with covariance $K \in \mathcal{S}_N^{++}(\R)$. Let $F \colon \R^N \to \mathbb{C}$ such that:
    \begin{equation*}
        \E \big[ |F(X)| \big] < \infty, \hspace{1.5cm} \E \Big[ |F(X)| \big\|XX^\top \big\|_F \Big] < \infty. 
    \end{equation*}
    Then, viewing $\E\big[F(X)\big]$ as a function of $K$ on $\mathcal{S}_N^{++}(\R)$:
    \begin{equation*}
        \nabla_K \E \big[F(X) \big] = - \frac{1}{2} \E \big[ F(X) \big] K^{-1} + \frac{1}{2} K^{-1} \, \E \big[ F(X) XX^\top \big] K^{-1}. 
    \end{equation*}
\end{lemma}
\begin{proof}
    This result is a simple consequence of the formula:
    \begin{equation*}
        \E \big[F(X) \big] = \frac{1}{(2\pi)^{N/2}} \big( \det K \big)^{-1/2} \int \exp \left( - \frac{1}{2} x^\top K^{-1} x \right) F(x) \d x. 
    \end{equation*}
    Now, it is well-known that on the space of positive-definite symmetric matrices:
    \begin{equation*}
        \nabla_K \det K = (\det K) K^{-1}, \hspace{1.5cm} \nabla_K (x^\top K^{-1} x) = - K^{-1} xx^\top K^{-1}. 
    \end{equation*}
    Now, the assumption on $F$ guarantees that we can apply dominated convergence and differentiate under the integral. This leads to the desired. 
\end{proof}

In \cref{App:Subsec:DMFTDerivation}, we will use a generalization of this result for a multi-dimensional Gaussian process $\big( X(t) \big)_{t \in [0,T]}$ on $\R^q$, with covariance $K_{ij}(s,t) = \E \, X_i(s) X_j(t)$ for $i,j \in \{1, \dots, q\}$ and $s,t \in [0,T]$. Informally, the result of \cref{Lemma:DerivativeGaussianExpectation} still applies provided that we add contractions over time variables in addition to those over indices:
\begin{align} \label{eq:DerivativeGaussianProcess}
    \frac{\partial \, \E \big[ F(X) \big]}{\partial K_{ij}(s,t)} &= - \frac{1}{2} \E \big[ F(X) \big] \big(K^{-1} \big)_{ij}(s,t) \\
    &\hspace{1cm}+ \frac{1}{2} \sum_{i',j'=1}^q \int_0^T \!\! \int_0^T \big(K^{-1} \big)_{ii'}(s,s') \big(K^{-1} \big)_{jj'}(t,t') \E \Big[ F(X) X_{i'}(s') X_{j'}(t') \Big] \d s' \d t'. \notag 
\end{align}
$K^{-1}$ denotes the functional inverse of $K$ with respect to both time and indices:
\begin{equation} \label{eq:CovarianceFunction}
    \sum_{k=1}^q \int_0^T K_{ik}(t, u) \big(K^{-1} \big)_{kj}(u,s) \d u = \delta(t-s) \delta_{ij}. 
\end{equation}
$\delta$ denotes the Dirac delta distribution supported at zero. See \cref{App:Subsec:DMFTBackground} for more details. 

The following lemma gives the inverse of a certain type of $3 \times 3$ block matrix, that will be useful in \cref{App:Subsec:DMFTDerivation}. 
\begin{lemma} \label{lemma:inversematrix}
    Let $A \in \mathcal{S}_N(\R)$, $B \in \R^{N \times N}$ invertible, $c \in \R^N$ and $\lambda \in \R^*$. Then:
    \begin{equation*}
        \begin{pmatrix} A & B & c \\ B^\top & 0 & 0 \\ c^\top & 0 & \lambda
        \end{pmatrix}^{-1} = \begin{pmatrix} 0 & (B^\top)^{-1} & 0 \\ B^{-1} & M & - \lambda^{-1} B^{-1} c \\ 0 & - \lambda^{-1} c^\top (B^\top)^{-1} & \lambda^{-1}\end{pmatrix},
    \end{equation*}
    with:
    \begin{equation*}
        M = B^{-1} \left( \frac{1}{\lambda} cc^\top - A \right) (B^\top)^{-1}. 
    \end{equation*}
\end{lemma}

Again, this result can be generalized in order to compute the inverse of the covariance function $K$ as in equation \eqref{eq:CovarianceFunction}, with $q = 3$. The specific structure of the matrix in \cref{lemma:inversematrix} would correspond, in continuous-time, to a Gaussian process in $\R^3$ with a time-independent third coordinate. 

\section{Derivation of the High-Dimensional Dynamics} \label{App:DMFT}
In this section, we derive the results presented in \cref{subsec:DMFT}. The plan goes as follows:
\begin{itemize}
    \item In \cref{App:Subsec:DMFTBackground}, we give some background on the methods we use and introduce the relevant objects for the derivation of the high-dimensional equations.
    \item In \cref{App:Subsec:SummaryDMFT}, we give the full set of self-consistent equations mentioned in \cref{Result0}, associated with the gradient flow dynamics \eqref{eq:GFdynamics} for general cost function, noise channel, and regularization. 
    \item In \cref{App:Subsec:DMFTDerivation}, we derive this set of equations. 
    \item In \cref{App:Subsec:DMFTSimplification}, we simplify the set of equations when dealing with the quadratic cost and a Gaussian noise channel. This leads to the system of equations of \cref{Result1}. 
    \item In \cref{App:Subsec:SecondLayer}, we show that our method can be generalized to learning the output weights of the network. We explain how this can be achieved and give the associated set of equations. 
    \item In \cref{App:Subsec:GaussianEquivalence}, we discuss \cref{Conjecture:GaussianUniversality} and the relationship with quadratic networks, and give the steps that would lead to a rigorous proof of this conjecture. 
\end{itemize}

\subsection{Path Integral Formalism} \label{App:Subsec:DMFTBackground}

Before giving and deriving the set of equations of \cref{Result0}, we give some definitions and prior results on the objects involved in our calculation. 

The technique we use to derive the high-dimensional equations is often referred to as the \textit{path integral} method. This formalism provides a powerful way to study the dynamics of high-dimensional disordered systems. It was first introduced in spin-glass dynamics by \citet{de1978dynamics}, who formulated the dynamics using a functional-integral and generating-functional approach. This idea was further developed by \citet{sompolinsky1982relaxational}, who used it to derive the dynamical mean-field equations for spin-glass models. Since then, the same approach has been applied in many contexts, from classical spin models to modern learning and inference problems \citep{crisanti1993spherical, agoritsas2018out, mignacco2020dynamical, bordelon2022self}, to obtain self-consistent dynamical equations describing the collective behavior of complex random systems. While these derivations are non-rigorous, such results have been rigorously obtained on similar problems, for instance by \citet{arous2001aging}, who derives a large deviation principle for spherical spin glasses. From this perspective, the path integral formulation is closely related, as it relies on a saddle-point asymptotic that is similar to the large deviation approach, though in a non-rigorous setting. 

\paragraph{The dynamical partition function.} For simplicity, consider a gradient flow dynamics similar to ours in equation \eqref{eq:GFdynamics}:
\begin{equation*}
    \dot W(t) = - \nabla_W F \big( W(t), X \big),
\end{equation*}
where $F$ depends on some random parameters $X$. In our case this would correspond to the sensing matrices and random labels. To study this dynamics, one can construct a functional integral representation, leading to an object encoding the full trajectory of the system. Formally, one can integrate over the set of all possible trajectories $\big\{ W(t) \big\}_{t \in [0,T]}$ using a measure that we write $\D W$. This measure can be interpreted as the continuous limit of the product measure at discrete time instants:
\begin{equation*}
    \D W = \lim_{N \to \infty} \prod_{p=0}^N \d W_{pT/N}. 
\end{equation*}
The object we then consider is the dynamical partition function:
\begin{equation*}
    \mathcal{Z}_\text{dyn} = \int \mathcal{D} W \,  \delta \Big( \dot W(t) + \nabla_W F \big( W(t), X \big) \Big).
\end{equation*}
$\delta$ denotes a functional Dirac delta that imposes the equation of motion at every instant. Integrating over all possible trajectories simply counts the unique trajectory consistent with a given initial condition. The resulting integral is therefore a constant, independent of the particular realization of the disorder or the function $F$.

Path integrals have been widely used to describe stochastic and dynamical systems, and several works have studied how to define them more rigorously. For readers interested in more systematic or rigorous discussions of these constructions, we refer to \citet{chow2015path},  \citet{cugliandolo2019building}, \citet{de2022path}, and also \citet[][Chapter 1]{dupuis2023field} for a general introduction and physical applications. 

\paragraph{Representation of the functional Dirac.} To make the expression of the partition function easy to work with, it is convenient to rewrite the functional Dirac delta in its exponential form. Similarly to the well-known expression in $\R^p$:
\begin{equation*}
    \delta(y) = \frac{1}{(2\pi)^p} \int_{\R^p} \d \hat y  \, e^{iy^\top \hat y}, 
\end{equation*}
the same identity extends to functional integrals by introducing a time-dependent \emph{conjugate field} $\hat W(t)$. The partition function then rewrites, up to a constant depending only on the dimension:
\begin{equation} \label{eq:DefPartitionFunction}
    \mathcal{Z}_\text{dyn} \propto \int \D W \D \hat W \, \exp \left( i \int_0^T \tr \Big( \left[ \dot W(t)  + \nabla_W F \big( W(t), X \big) \right] \hat W(t)^\top \Big) \d t \right). 
\end{equation}
$\hat W$ plays the role of a Lagrange multiplier that enforces the equation of motion. This exponential form is particularly useful since it allows to manipulate the dynamics using field-theoretic tools. For instance, this formulation will later allow us to average the dynamical partition function with respect to the Gaussian randomness of $X$. 

\paragraph{Link with the generating functional.} The dynamical partition function can be viewed as a central object characterizing the full distribution of a trajectory. In practice, one can introduce external source fields $J(t), \hat J(t)$ that couple linearly to $W(t)$ and $\hat W(t)$:
\begin{equation*}
    \mathcal{Z}_\text{dyn}[J, \hat J] = \int \D W \D \hat W \exp \left( - S(W, \hat W) + \int_0^T \Big[ \tr \big( J(t) W(t)^\top \big) + \tr \big( \hat J(t) \hat W(t)^\top \big) \Big]\d t \right),
\end{equation*}
where $S(W, \hat W)$ is the quantity already appearing inside the exponential in equation \eqref{eq:DefPartitionFunction}. The function $\mathcal{Z}_\text{dyn}[J, \hat J]$, often called the generating functional of the dynamics (the analogue of a generating function for finite-dimensional random variables), contains all statistical information about the process. Correlations of any observable can be obtained by taking functional derivatives of $\mathcal{Z}_\text{dyn}[J, \hat J]$ with respect to $J, \hat J$, and then setting these variables to zero. In other words, the partition function (and its extension to the generating functional) fully characterizes the trajectory distribution of the random process $W$. 

\paragraph{The temporal Dirac delta distribution.} Above, we introduced the Dirac delta functional on trajectories. In what follows, we also use the Dirac delta distribution on a real scalar variable, denoted again by $\delta$, with the usual meaning of a distribution supported at zero. As in \cref{Result1}, equation \eqref{eq:AveragedQuantities}, we sometimes use the abusive notation:
\begin{equation*}
    R(t,t') = \delta(t-t') + G(t,t'),
\end{equation*}
for some function $G \colon \R^+ \times \R^+ \to \R$. The previous expression is to be understood in the sense of distributions: $R$ implicitly represents a linear operator $\overline R$ acting on test functions $\phi \colon \R^+ \to \R$:
\begin{equation*}
    \big(\overline R\phi \big)(t) = \phi(t) + \int_0^t G(t,t') \phi(t') \d t'. 
\end{equation*}

\subsection{Full Set of Equations} \label{App:Subsec:SummaryDMFT}

In this section, we give the general set of high-dimensional equations associated with the dynamics of the matrix $W(t)$ and the typical label $y(t)$. The general structure of the equations was given in \cref{Result0}, but for completeness we give them in detail. 

In the following, we show that in the high-dimensional limit, starting from the dynamics \eqref{eq:GFdynamics}, the evolution of the student matrix and the typical label is equivalent, in distribution, to the following set of equations:
\begin{align}
    \d W(t) &= \left( \mathcal{H}(t) + r(t) Z^* - \int_0^t \Gamma(t,t') Z(t') \d t' \right) W(t) \d t - \nabla \Omega\big( W(t) \big) \d t + \frac{1}{\sqrt{\beta d}} \d  B(t), \label{eq:DMFTStudent} \\
    0 &= \int_0^t R(t,t') y(t') \d t' + \eta(t) - m(t) y^* + \frac{2}{\alpha} \ell' \big( y(t), z \big), \label{eq:DMFTLabel}     
\end{align}
with $Z(t) = W(t)W(t)^\top$, $y^* \sim \N(0, 2Q_*)$ and $z \sim P \big( \cdot \, | \, y^* \big)$. The functions $\mathcal{H}$ and $\eta$ are independent centered Gaussian processes respectively belonging to $\mathcal{S}_d(\R)$ and $\R$, with covariances:
\begin{equation} \label{eq:CovarianceGaussianProcesses}
    \E \, \mathcal{H}_{ij}(t) \mathcal{H}_{i'j'}(t') = \frac{1}{2d} \big( \delta_{ii'} \delta_{jj'} + \delta_{ij'} \delta_{i'j} \big) \mathcal{K}_Z(t,t'), \hspace{1.5cm} \E \, \eta(t) \eta(t') = \mathcal{K}_y(t,t'). 
\end{equation}
We then consider the averaged quantities with respect to the dynamics of $Z(t), y(t)$:
\begin{align}
     C_y(s,t) &= \E \, y(s) y(t), &\hspace{1.cm} m_y(t) &= \E \, y(t) y^*, \hspace{1cm} \\
     C_Z(s,t) &= \frac{1}{d} \E \, \tr \big(Z(s) Z(t) \big), &\hspace{1.cm} m_Z(t) &= \frac{1}{d} \E \, \tr \big( Z(t) Z^* \big), \hspace{1cm} \\
     Q_* &= \frac{1}{d} \E \, \tr(Z^{*2}),
\end{align}
as well as the first-order derivatives:
\begin{equation} \label{eq:SummaryDMFTResponses}
    R_y(s,t) = \left. \frac{\partial \, \E \, y(s)}{\partial h(t)} \right|_{h = 0}, \hspace{1.5cm} R_Z(s,t) = \frac{1}{d^2} \tr \left( \left. \frac{\partial \, \E \, Z(s)}{\partial H(t)} \right|_{H = 0} \right),
\end{equation}
in response to a perturbation $\eta(t) \mapsto \eta(t) - h(t)$ and $\mathcal{H}(t) \mapsto \mathcal{H}(t) + H(t)$ in equations \eqref{eq:DMFTStudent}, \eqref{eq:DMFTLabel}. Then, the scalar deterministic functions $r, \Gamma, \mathcal{K}_Z, \mathcal{K}_y, m, R$ can be self-consistently computed as:
\begin{align}
    r(t) &= \alpha \left( \frac{1}{2Q_*} \int_0^t R(t,t') m_y(t')\d t' - \int_0^t \int_0^{t'} R(t,t') R_y(t',s) m(s) \d s \d t' \right), \\
    \Gamma(s,t) &= \alpha \left( R(s, t) - \int_t^s \int_t^{s'} R(s, s') R_y(s',t') R(t',t) \d t' \d s' \right), \\
    \mathcal{K}_Z(s,t) &= \frac{\alpha}{2} \Bigg(\frac{1}{2} \int_0^s \int_0^t \, R(s,s') C_y(s',t') R(t,t') \d t' \d s' - \int_0^s \int_0^{s'}  R(s,s') R_y(s',t') \mathcal{K}_y(t,t') \d t' \d s' \notag \\
    &\hspace{1.5cm}- m(t) \int_0^s R(s,s') m_y(s') \d s' + \frac{1}{2} \mathcal{K}_y(s,t) + Q_* m(s) m(t) \Bigg) + \mathrm{Sym}, \label{eq:DMFTSummary1} \\
    \mathcal{K}_y(t,t') &= 2 \int_0^t \int_0^s R(s,s') R(t,t') \Big( C_Z(s',t') - Q_*^{-1} m_Z(s') m_Z(t') \Big) \d s' \d t',  \\
    m(s) &= \frac{1}{Q_*} \int_0^s R(s,s') m_Z(s') \d s', \\
    \delta(t-t') &= \int_{t'}^t R(t,s) R_Z(s,t') \d s.
\end{align}
The notation Sym in equation \eqref{eq:DMFTSummary1} indicates the symmetrization with respect to the variables $s,t$. The last equation defines $R$ as the functional inverse of the response function $R_Z$. Although we do not prove that only one function $R$ satisfies this relationship for all instants $t,t'$, we believe that due to the causal structure of the responses, such a result should hold. 

This set of equations is truly self-consistent, in the usual spirit of DMFT equations: the stochastic processes $W(t), y(t)$ are driven by scalar and deterministic functions that are themselves computed as expectations over these processes. 

\subsection{Derivation of the Equations} \label{App:Subsec:DMFTDerivation}

In this section we derive the set of equations presented in the previous section. We start from the dynamics in equation \eqref{eq:GFdynamics}. Recalling the expression of the empirical loss in equation~\eqref{eq:Loss} and computing its gradient, the dynamics then writes:
\begin{equation} \label{eq:DMFTDynamics1}
    \d W(t) = - \frac{d}{n} \sum_{k=1}^n \ell' \Big( \tr \big(X_k W(t)W(t)^\top \big), z_k \Big) X_k W(t) \d t  - \nabla \Omega \big( W(t) \big) \d t + \frac{1}{\sqrt{\beta d}} \d B(t).  
\end{equation}
For simplicity we denote $\ell'$ the derivative of $\ell$ with respect to the first variable. We recall that $z_k$ is generated by the teacher matrix $Z^*$ through the noisy channel $P$:
\begin{equation*}
    z_k \sim P \Big( \cdot \, \big| \, \tr(X_k Z^*)  \Big). 
\end{equation*}
To keep the notation light, we derive the dynamical equations with $\Omega = 0$ and $\beta = \infty$, only keeping the first term in \eqref{eq:DMFTDynamics1}. In fact, since our calculation only transforms the empirical term, any additive term independent of the sensing matrices $(X_k)_{1 \leq k \leq n}$ could be integrated in equation \eqref{eq:DMFTDynamics1} and would appear in the same manner in the resulting dynamics. 

\paragraph{Plan of the derivation.} We derive our system of equations in several steps:
\begin{itemize}
    \item We start by computing the dynamical partition function and average it with respect to the observations (\cref{App:Subsubsec:DMFT1PartitionFunction}). The $d \to \infty$ limit allows to write a saddle-point equation on the covariance function of the typical label (\cref{App:Subsubsec:DMFT2Saddlepoint}). 
    \item We then study the response operator associated with the dynamics to simplify the structure of this covariance function (\cref{App:Subsubsec:DMFT3Covariance}). 
    \item This simplification allows to write stochastic integro-differential equations for the evolution of the typical label (\cref{App:Subsubsec:DMFT4Labels}) and student matrix (\cref{App:Subsubsec:DMFT6EvolutionStudent}). 
    \item We link several averaged quantities to the coefficients driving the dynamics in order to close the equations (\cref{App:Subsubsec:DMFT5Averaged}).   
\end{itemize}

The first step of this derivation is inspired by the replica calculation by \citet{maillard2024bayes}. In our case, the integration over multiple replicas is replaced by one over the time instants of the dynamics. However, note several differences: in the Bayes-optimal setting, replicas are introduced in order to compute the expectation of the logarithm of the partition function. Since our derivation only requires averaging the partition function (rather than its logarithm) with respect to the observations, we avoid the non-rigorous step of sending the number of replicas to zero. In addition, while the work of \citet{maillard2024bayes} relies on a replica-symmetric ansatz (justified in the Bayes-optimal setting) to simplify the structure of the overlaps between replicas, our approach does not require such a simplification for the time-dependent overlaps that appear in the dynamics. 

\subsubsection{The Dynamical Partition Function} \label{App:Subsubsec:DMFT1PartitionFunction}

Following our introduction on path integrals in \cref{App:Subsec:DMFTBackground}, we write the dynamical partition function associated with the gradient flow dynamics on $[0,T]$ for some finite-time horizon $T$:
\begin{equation*}
    \mathcal{Z}_\text{dyn} = \int \mathcal{D} W \, \delta \left( \dot W(t) + \frac{d}{n} \sum_{k=1}^n \ell' \big( y_k(t), z_k \big) X_k W(t) \right),
\end{equation*}
where $y_k(t) = \tr \big(X_k W(t) W(t)^\top\big)$. The Dirac delta notation indicates here that the constraint must be verified for all $t \in [0,T]$. Then, using the Fourier representation of the Dirac delta function (as explained in \cref{App:Subsec:DMFTBackground}), we can rewrite:
\begin{equation*}
    \mathcal{Z}_\text{dyn} \propto \int \D W \D \hat W \, \exp \bigg( - id \int_0^T \tr \big( \dot W(t) \hat W(t)^\top \big) \d t - \frac{id^2}{n} \sum_{k=1}^n \int_0^T \ell'\big( y_k(t), z_k \big) \hat y_k(t) \d t\bigg), 
\end{equation*}
with $\hat y_k(t) = \tr \big( X_k W(t) \hat W(t)^\top \big)$ and $\propto$ indicates proportionality, up to a constant that may depend only on the dimension. The goal is now to compute the average of $\mathcal{Z}_\text{dyn}$ with respect to the i.i.d. observations $X_1, \dots, X_n$. To do so, we write:
\begin{equation*}
    \E_{\substack{X_1, \dots, X_n \\ z_1, \dots, z_n}} \exp \left( - \frac{id^2}{n} \sum_{k=1}^n \int_0^T \ell'\big( y_k(t), z_k \big) \hat y_k(t) \d t \right) = \left[ \E_{X, z} \exp \left( - \frac{i}{\alpha} \int_0^T \ell'\big( y(t), z \big) \hat y(t) \d t\right) \right]^n. 
\end{equation*}
Here $z \sim P \big( \cdot \, \big| \, y^* \big)$, and the variable $y(t)$ corresponds to the typical label introduced in \cref{Result0}. We also used that $n \sim \alpha d^2$ in the large $d$ limit. Since the variables $W(t), \hat W(t), Z^*$ are fixed, and $X \sim \mathrm{GOE}(d)$ (see \cref{def:GOE}), the process $\mathbf{y}(t) = \big(y(t), \hat y(t), y^* \big)$ is Gaussian with zero mean and covariance:
\begin{equation} \label{eq:defcovarianceQ}
    Q(s,t) = \E \, \mathbf{y}(s) \mathbf{y}(t)^\top = \frac{2}{d} \begin{pmatrix}
        \tr \, Z(s) Z(t) & \tr \, Z(s) \hat Z(t) & \tr \, Z(s) Z^* \\
        \tr \, \hat Z(s) Z(t) & \tr \, \hat Z(s) \hat Z(t) & \tr \, \hat Z(s) Z^* \\
        \tr \, Z^*Z(t) & \tr \, Z^* \hat Z(t) & \tr \, Z^{*2}
    \end{pmatrix},
\end{equation}
with $\hat Z(t) = \mathrm{Sym}\big(W(t) \hat W(t)^\top \big)$. Therefore, we can finally write the averaged partition function (with respect to the observations):
\begin{equation*}
    \overline{\mathcal{Z}}_\mathrm{dyn} \propto \int \D W \D \hat W \, \exp \left(- id \int_0^T \tr \big( \dot W(t) \hat W(t)^\top \big) \d t \right) \left[ \E_{\mathbf{y},z} \exp \left( - \frac{i}{\alpha} \int_0^T \ell' \big(y(t), z \big) \hat y(t) \d t\right) \right]^n. 
\end{equation*}
Finally, we introduce the covariance matrix as an integration variable using a delta function, and obtain the identity:
\begin{equation*}
\begin{aligned}
    \overline{\mathcal{Z}}_\mathrm{dyn} &\propto \int \D W \D \hat W \D Q \D \hat Q \exp \left( - id^2 \sum_{a,b=1}^3 \int_0^T\!\! \int_0^T \hat Q_{ab}(s,t) \left( Q_{ab}(s,t) - \frac{2}{d} \tr \, Z_a(s) Z_b(t) \right) \d s \d t \right) \\
    &\hspace{0.5cm} \exp \left(- id \int_0^T \tr \big(\dot W(t) \hat W(t)^\top \big) \d t \right) \left[\E_\mathbf{y} \int \d z \, P \big( z \, \big| \, y^* \big) \exp \left( - \frac{i}{\alpha} \int_0^T  \ell' \big(y(t), z \big) \hat y(t) \d t\right) \right]^n,
\end{aligned}
\end{equation*}
where $Z_1(t) = Z(t)$, $Z_2(t) = \hat Z(t)$ and $Z_3(t) = Z^*$. Rearranging the terms, we finally obtain the expression:
\begin{equation} \label{eq:saddlepoint}
    \overline{\mathcal{Z}}_\mathrm{dyn} = \int \D Q \D \hat Q \exp \left( -id^2 \mathcal{T}(Q, \hat Q) + d^2 \mathcal{F}(\hat Q) + \alpha d^2 \mathcal{F}_\text{out}(Q) \right),
\end{equation}
with:
\begin{align}
    \mathcal{T}(Q, \hat Q) &= \sum_{a,b=1}^3 \int_0^T \!\! \int_0^T \hat Q_{ab}(s,t) Q_{ab}(s,t)  \d s \d t , \\
    \mathcal{F}(\hat Q) &= \frac{1}{d^2} \log \int \D W \D \hat W \, \exp \Bigg(- id \int_0^T \tr \, \dot W(t) \hat W(t)^\top \d t \label{eq:FunctionsSaddle} \\
    &\hspace{5cm}+ 2id  \sum_{a,b=1}^3 \int_0^T \!\! \int_0^T \hat Q_{ab}(s,t) \tr \big( Z_a(s) Z_b(t) \big) \d s \d t \Bigg), \notag \\
    \mathcal{F}_\text{out}(Q) &= \log \E_\mathbf{y} \int \d z \, P \big( z \, \big| \, y^* \big) \exp \left( - \frac{i}{\alpha} \int_0^T  \ell' \big(y(t), z \big) \hat y(t) \d t \right). 
\end{align}

\subsubsection{Saddle-Point Equations} \label{App:Subsubsec:DMFT2Saddlepoint}

Now, we can take the $d \to \infty$ limit in equation \eqref{eq:saddlepoint}. Although $W$ is still a high-dimensional object, the covariance functions $Q, \hat Q$ are finite-dimensional functions, when considered on timescales of order one. Then, the saddle-point asymptotics can be performed, and we obtain the equations:
\begin{equation} \label{eq:Saddle}
    i \hat Q(s,t) = \alpha \frac{\partial \mathcal{F}_\text{out}(Q)}{\partial Q(s,t)}, \hspace{1.5cm} i Q(s,t) = \frac{\partial \mathcal{F}(\hat Q)}{\partial \hat Q(s,t)}. 
\end{equation}
Note that the second identity leads back to the definition of $Q(s,t)$ in equation \eqref{eq:defcovarianceQ}, but averaged with respect to the dynamics.

The goal is now to compute the derivative of $\mathcal{F}_\text{out}$. To do so, recall that $\mathcal{F}_\text{out}$ depends on $Q$ through the expectation with respect to $\mathbf{y}$ which is a Gaussian process with covariance $Q$. Therefore, we can write:
\begin{equation*}
\begin{aligned}
    \mathcal{F}_\text{out}(Q) &= \log \E_{\mathbf{y} \sim \N(0, Q)} \E_{z \sim P( \cdot \, | \, y^*)} \exp \Big( - i \Phi(y, \hat y, z) \Big), \\
    \Phi(y, \hat y, z) &= \frac{1}{\alpha} \int_0^T \ell'\big(y(t), z \big) \hat y(t) \d t. 
\end{aligned}
\end{equation*}
To compute the derivative of $\mathcal{F}_\text{out}$, we shall now use \cref{Lemma:DerivativeGaussianExpectation} and the discussion that follows. As an application of the identity \eqref{eq:DerivativeGaussianProcess}:
\begin{align}
    \frac{\partial \mathcal{F}_\text{out}(Q)}{\partial Q_{ij}(s,t)} &= - \frac{1}{2} (Q^{-1})_{ij}(s,t) \label{eq:DerivativeSaddle} \\
    &\hspace{1cm}+ \frac{1}{2} \sum_{a,b=1}^3 \int_0^T \!\! \int_0^T (Q^{-1})_{ia}(s,s') (Q^{-1})_{jb}(t,t') \overline\E \big[ y_a(s') y_b(t') \big] \d s' \d t' \notag, 
\end{align}
where $y_1(t) = y(t), y_2(t) = \hat y(t), y_3(t) = y^*$, $Q^{-1}$ is the functional inverse of $Q$ and the expectation $\overline \E$ on $\mathbf{y}$ is computed as:
\begin{equation} \label{eq:DistribLabels}
    \overline \E \, f( \mathbf{y}) = \frac{\displaystyle \E_{\mathbf{y} \sim \N(0,Q)} \E_{z \sim P( \cdot \, | \, y^*)} \, f(\mathbf{y})e^{-i\Phi(y, \hat y, z)}}{\displaystyle \E_{\mathbf{y} \sim \N(0,Q)} \E_{z \sim P( \cdot \, | \, y^*)}  \, e^{-i\Phi(y, \hat y, z)}}. 
\end{equation}
This distribution will correspond to the trajectory distribution of the typical label $y(t)$. Now, since this distribution involves the expectation over a Gaussian process with covariance function $Q(s,t)$, one needs to compute $Q^{-1}$. In the following, using the structure of the dynamics, we simplify $Q$ and compute its inverse. 

\subsubsection{Form of the Covariance Function} \label{App:Subsubsec:DMFT3Covariance}

We now give some arguments in order to simplify the form of the covariance $Q(s,t)$. Recall that the saddle-point equations led us to the same expression of $Q$ as in equation \eqref{eq:defcovarianceQ}, but averaged with respect to the distribution of the dynamics. To simplify this covariance, we study the general perturbed gradient flow:
\begin{equation*}
    \dot W(t) = - \nabla F(W(t)) + H(t) W(t), 
\end{equation*}
with $H(t) \in \mathcal{S}_d(\R)$, and $F \colon \R^{d \times m} \to \R$ is some continuously differentiable function. The dynamical partition function for this equation writes:
\begin{equation} \label{eq:PartitionFunctionGeneric}
    \mathcal{Z}_\text{GF} \propto \int \D W \D \hat W \exp \left( -id \int_0^T \tr \Big[ \Big( \dot W(t) + \nabla F(W(t)) - H(t) W(t) \Big) \hat W(t)^\top \Big] \d t \right).
\end{equation}
Therefore, for any scalar function $f$ of the dynamics, we have the identities, when averaging with respect to the distribution associated with $\mathcal{Z}_\text{GF}$:
\begin{equation*} 
    \left. \frac{\partial \, \E \, f(W)}{\partial H(t)} \right|_{H = 0} = id \, \E \big[ f(W) \hat Z(t) \big], \hspace{1.5cm} \left. \frac{\partial^2 \, \E \, f(W)}{\partial H(t) \partial H(t')} \right|_{H = 0} = -d^2 \, \E \Big[ f(W) \big( \hat Z(t) \otimes \hat Z(t') \big) \Big]. 
\end{equation*}
Here we view $\E \, f(W)$ as a function of the perturbation $\big( H(t) \big)_{t \geq 0}$ on $\mathcal{S}_d(\R)$, and we compute the first and second derivatives of this function in the space $\mathcal{S}_d(\R)$ at $H = 0$. To prove these identities, one can write $\E \, f(W)$ as a function of $H$ with the partition function \eqref{eq:PartitionFunctionGeneric} and use that the gradient of the map $H \mapsto \tr(HM)$ on $\mathcal{S}_d(\R)$ is $\mathrm{Sym}(M)$ \citep[see][for the notion of symmetric gradient]{srinivasan2023gradient}. We recall that $\hat Z(t) = \mathrm{Sym}\big(W(t) \hat W(t)^\top \big)$. In the second derivative, the notation $\otimes$ refers to the tensor product. 

Therefore, the matrix $\hat Z(t)$ acts as a derivative operator when averaging quantities with respect to the dynamics. Thus, from the previous identities, we deduce that:
\begin{equation*}
    \E \, \tr \big( Z^* \hat Z(t) \big) = 0, \hspace{1.5cm} \E \, \tr \big( \hat Z(s) \hat Z(t) \big) = 0. 
\end{equation*}
Indeed, the first quantity can be rewritten as the derivative of $Z^*$ in response to a perturbation of the dynamics, and the second one corresponds to the second-order derivative of a constant function. In addition, we obtain:
\begin{equation*}
    \E \, \tr \big( Z(s) \hat Z(t) \big) = - \frac{i}{d} \tr \left( \left. \frac{\partial \, \E \, Z(s)}{\partial H(t)} \right|_{H = 0} \right).
\end{equation*}
Here the response operator can be interpreted as the differential of the function $\mathcal{S}_d(\R) \to \mathcal{S}_d(\R)$ that maps the perturbation $H$ introduced at time $t$ to the perturbed solution $\E \, Z(s)$ at time $s$. This term is zero when $t > s$. Indeed, the dynamics cannot be influenced by a perturbation added at a later time. Finally, we can conclude that the covariance function $Q(s,t)$ is of the form:
\begin{equation*}
    Q(s,t) = 2\begin{pmatrix} C_Z(s,t) & -i R_Z(s,t) & m_Z(s) \\
    -i R_Z(t,s) & 0 & 0 \\
    m_Z(t) & 0 & Q_*
    \end{pmatrix},
\end{equation*}
with:
\begin{equation} \label{eq:SC1}
\begin{aligned}
    C_Z(s,t) &= \frac{1}{d} \E \, \tr \big(Z(s) Z(t) \big), &\hspace{1.5cm} R_Z(s,t) &= \frac{1}{d^2} \tr \left( \left. \frac{\partial \, \E \, Z(s)}{\partial H(t)} \right|_{H = 0} \right), \\ 
    m_Z(t) &= \frac{1}{d} \E \, \tr \big( Z(t) Z^* \big), &\hspace{1.5cm} Q_* &= \frac{1}{d} \E \, \tr(Z^{*2}). 
\end{aligned}
\end{equation}
From this expression of $Q$, we can deduce the structure of its inverse by applying \cref{lemma:inversematrix}. Although this lemma is formulated with discrete variables, it is easily seen that we can adapt the result when dealing with two-time functions. In this case, the matrix inverse is replaced by the functional inverse, and the standard matrix product now corresponds to integration with respect to time variables. Then, we can conclude that the function $Q^{-1}$ appearing in equation~\eqref{eq:DerivativeSaddle} has the form:
\begin{equation*}
    Q^{-1}(s,t) = \frac{1}{2} \begin{pmatrix} 0 & i R(t,s) & 0 \\ i R(s,t) & K(s,t) & -im(s) \\ 0 & -im(t) & Q_*^{-1}
    \end{pmatrix},
\end{equation*}
and we have the relationships, from \cref{lemma:inversematrix}:
\begin{equation} \label{eq:SC2}
\begin{aligned}
    R(s,t) &= R_Z^{-1}(s,t), \\
    K(s,t) &= \int_0^t \int_0^s R(s,s') R(t,t') \Big( C_Z(s',t') - Q_*^{-1} m_Z(s') m_Z(t') \Big) \d s' \d t', \\
    m(s) &= \frac{1}{Q_*} \int_0^s R(s,s') m_Z(s') \d s'. 
\end{aligned}
\end{equation}
$R$ is defined as the inverse of $R_Z$, therefore also has a causal structure, i.e., $R(t,t') = 0$ for $t' > t$, and we can deduce the relationship for all $0 \leq t' \leq t \leq T$:
\begin{equation*}
    \int_{t'}^t R(t,s) R_Z(s,t') \d s = \delta(t-t').
\end{equation*}

\subsubsection{Evolution of the Labels} \label{App:Subsubsec:DMFT4Labels}

Before deriving the self-consistent set of equations, recall that equation \eqref{eq:DerivativeSaddle} involves averages of the labels with respect to a reweighted distribution, defined in equation \eqref{eq:DistribLabels}. We investigate this distribution and derive the evolution equation for the label $y(t)$. To do so, we use the fact that $\mathbf{y}$ is a Gaussian process with covariance $Q$ and we compute:
\begin{equation*}
\begin{aligned}
    \overline \E \, f(y, y^*) &\propto \int \D y \D \hat y \D y^* f(y,y^*) \exp \left( - \frac{1}{2} \int_0^T \!\! \int_0^T \mathbf{y}(s)^\top Q^{-1}(s,t) \mathbf{y}(t) \d t \d s \right) \\
    &\hspace{5cm} \times \E_z \exp \left( - \frac{i}{\alpha} \int_0^T  \ell'\big(y(s), z \big) \hat y(s) \d s \right).
\end{aligned}
\end{equation*}
Using the expression of $Q^{-1}$, we have:
\begin{equation*}
\begin{aligned}
    \int_0^T \!\! \int_0^T \mathbf{y}(s)^\top Q^{-1}(s,t) \mathbf{y}(t) \d t \d s &= \frac{1}{2Q_*} y^{*2} + i \int_0^T \!\! \int_0^s  \, R(s,t) y(t) \hat y(s) \d t \d s \\
    &\hspace{1cm}- i y^* \int_0^T m(s) \hat y(s) \d s + \frac{1}{2} \int_0^T \!\! \int_0^T K(s,t) \hat y(s) \hat y(t) \d t \d s.  
\end{aligned}
\end{equation*}
We now use the identity:
\begin{equation*}
    \exp \left( - \frac{1}{4} \int_0^T \!\! \int_0^T K(s,t) \hat y(s) \hat y(t) \d t \d s \right) = \E_\eta \exp \left( -\frac{i}{2} \int_0^T \eta(s) \hat y(s) \d s \right),
\end{equation*}
where $\eta(t)$ is a Gaussian process with zero mean and covariance $2 K(s,t)$. Finally, we end up with:
\begin{equation} \label{eq:DistribLabels2}
\begin{aligned}
    \overline \E \, f(y, y^*) &\propto \int \D y \D \hat y \D y^* f(y,y^*) \E_{z, \eta} \exp \left( - \frac{1}{4 Q_*} y^{*2} \right) \\
    &\hspace{0.5cm} \exp \left( \frac{i}{2} \int_0^T \left[ - \int_0^s R(s,t) y(t) \d t + m(s) y^* - \eta(s) - \frac{2}{\alpha} \ell'\big(y(s), z \big) \right] \hat y(s) \d s \right). 
\end{aligned}
\end{equation}
Therefore, integrating with respect to the variable $\hat y$, and using the exponential representation of the delta Dirac function (see \cref{App:Subsec:DMFTBackground}), we end up with the equation for $y(t)$:
\begin{equation} \label{eq:SC3}
    \int_0^t R(t,t') y(t') \d t' + \eta(t) - m(t) y^* + \frac{2}{\alpha} \ell'\big(y(t), z \big) = 0, \hspace{1.5cm} \E \, \eta(s) \eta(t) = 2 K(s,t),
\end{equation}
where $z \sim P( \, \cdot \, | \, y^*)$. Finally, we determine the form of the covariances for the $y$ variables that appear in equation \eqref{eq:DerivativeSaddle}. Following the same arguments as for the covariance $Q$, we can show that:
\begin{equation*}
    \overline \E \, \hat y(s) y^* = 0, \hspace{1.5cm} \overline \E \, \hat y(s) \hat y(t) = 0, \hspace{1.5cm} \overline \E \, y(s) \hat y(t) = -2iR_y(s,t), 
\end{equation*}
where $R_y$ is a response function and is non-zero only for $t \in [0,s]$. Similarly to the response identities we derived for $W(t)$, we can prove, using the distribution in equation \eqref{eq:DistribLabels2}, that:
\begin{equation*}
    R_y(s,t) = \left. \frac{\partial \, \overline \E \, y(s)}{\partial h(t)}\right|_{h = 0},
\end{equation*}
in response to a perturbation of the noise $\eta(t) \mapsto \eta(t) - h(t)$ in equation \eqref{eq:SC3}. In addition, equation \eqref{eq:DistribLabels2} shows that $\overline \E \,  y^{*2} = 2 Q_*$. Finally, we denote the covariances:
\begin{equation*}
    C_y(s,t) = \overline \E \, y(s) y(t), \hspace{1.5cm} m_y(t) = \overline \E \, y(t) y^*. 
\end{equation*}
With all of these quantities, we are now ready to go back to equation \eqref{eq:DerivativeSaddle} and derive the set of self-consistent equations. 

\subsubsection{Equations between Averaged Quantities} \label{App:Subsubsec:DMFT5Averaged}

Back to equation \eqref{eq:DerivativeSaddle}, we can compute the derivative of $\mathcal{F}_\text{out}$ with respect to $Q$, and with the saddle-point equation \eqref{eq:Saddle}, show that $\hat Q$ is of the form:
\begin{equation} \label{eq:matrixhatQ}
    \hat Q(s,t) = \frac{1}{4} \begin{pmatrix} 0 & - \Gamma(t,s) & 0 \\ - \Gamma(s,t) & i \mathcal{K}_Z(s,t) & r(t) \\ 0 & r(s) & 0  
    \end{pmatrix}, 
\end{equation}
and that we have the relationships:
\begin{align}
    \Gamma(s,t) &= \alpha \left( R(s, t) - \int_t^s \int_t^{s'} R(s, s') R_y(s',t') R(t',t) \d t' \d s' \right),  \label{eq:SC4} \\
    \mathcal{K}_Z(s,t) &= \alpha \bigg(\frac{1}{4} \int_0^s \int_0^t \, R(s,s') C_y(s',t') R(t,t') \d t' \d s' - \int_0^s \int_0^{s'}  R(s,s') R_y(s',t') K(t,t') \d t' \d s' \notag \\
    &\hspace{1.5cm}- \frac{1}{2} m(t) \int_0^s R(s,s') m_y(s') \d s' + \frac{1}{2} K(s,t) + \frac{1}{2} Q_* m(s) m(t) \bigg) + \mathrm{Sym}, \notag \\
    r(s) &= \alpha \left( \frac{1}{2Q_*} \int_0^s R(s,s') m_y(s')\d s' - \int_0^s \int_0^{s'} R(s,s') R_y(s',t') m(t') \d t' \d s' \right) \notag. 
\end{align}
$\mathrm{Sym}$ indicates the symmetrization of the previous expression with respect to the time variables $s,t$. These equations relate the coefficients of $\hat Q$ to the quantities derived from $Q$ (namely the functions $R, K, m$ that are the coefficients of $Q^{-1}$), and the averaged quantities of the dynamics of $y(t)$, themselves driven by the functions $R, K, m$. 

\subsubsection{Evolution of the Weights} \label{App:Subsubsec:DMFT6EvolutionStudent}

To close the system of self-consistent equations, we finally need another relationship between $\hat Q$ and $Q$. To do so, recall that the $Q$ variables are related to averages with respect to the dynamics of $W$. Then, from the expression of the dynamical partition function in the high-dimensional limit, we can derive equivalent dynamical equations for the weights $W(t)$. To do so, we go back to the expression of $\mathcal{F}$ in equation \eqref{eq:FunctionsSaddle}, along with the expression of $\hat Q$ in equation \eqref{eq:matrixhatQ}:
\begin{equation*}
\begin{aligned}
    \mathcal{F}(\hat Q) &= \frac{1}{d^2} \log \int \D W \D \hat W \exp \Bigg( - id \int_0^T \tr \big(\dot W(t) \hat W(t)^\top \big) \d t - id \int_0^T \!\! \int_0^t \Gamma(t,t') \tr\big( \hat Z(t) Z(t') \big) \d t' \d t \\
    &\hspace{4.2cm} + id \int_0^T r(t) \tr \big( \hat Z(t) Z^* \big) \d t - \frac{d}{2} \int_0^T \!\! \int_0^T \mathcal{K}_Z(s,t) \tr\big( \hat Z(s) \hat Z(t) \big) \d s \d t \Bigg).
\end{aligned}
\end{equation*}
We now use the identity:
\begin{equation*}
    \exp \left( - \frac{d}{2} \int_0^T \!\! \int_0^T \mathcal{K}_Z (s,t) \tr\big( \hat Z(s) \hat Z(t) \big) \d s \d t \right) = \E_V \exp \left(id \int_0^T \tr\big( \hat Z(s) V(s) \big) \d s \right), 
\end{equation*}
where $V(t) \in \R^{d \times d}$ is a centered Gaussian matrix with covariance:
\begin{equation*}
    \E \, V_{ij}(s) V_{i'j'}(t) = \frac{1}{d} \delta_{ii'}\delta_{jj'} \mathcal{K}_Z(s,t). 
\end{equation*}
Therefore, using that $\hat Z(t) = \mathrm{Sym} \big(W(t) \hat W(t)^\top \big)$, we have:
\begin{equation} \label{eq:FunctionDynamicshatQ} 
\begin{aligned}
    \mathcal{F}(\hat Q) &= \frac{1}{d^2} \log \int \D W \D \hat W \exp \bigg( -id \int \d t \, \tr \, \bigg[ \dot W(t) + \int_0^t \Gamma(t,t') Z(t') \d t' \, W(t) \\
    &\hspace{7cm} - r(t) Z^* W(t) - \mathrm{Sym}\big( V(t) \big) W(t) \bigg] \hat W(t)^\top \bigg) \notag. 
\end{aligned}  
\end{equation}
We can integrate with respect to $\hat W(t)$ and use the exponential formulation of the delta functional in \cref{App:Subsec:DMFTBackground}. This leads to the following equation for $W(t)$:
\begin{equation} \label{eq:DynamicsW1}
    \dot W(t) = \left( \mathrm{Sym}\big( V(t) \big) + r(t) Z^* - \int_0^t \Gamma(t,t') Z(t') \d t' \right) W(t). 
\end{equation}
Now setting $\mathcal{H}(t) = \mathrm{Sym}(V(t))$, we get that $\mathcal{H}$ is still a centered Gaussian process with covariance:
\begin{equation*}
    \E \, \mathcal{H}_{ij}(s) \mathcal{H}_{i'j'}(t) = \frac{1}{2d} \Big( \delta_{ii'} \delta_{jj'} + \delta_{ij'} \delta_{i'j} \Big) \mathcal{K}_Z(s,t). 
\end{equation*}
As mentioned earlier, it is easily seen that we can add back the regularization and the Brownian motion from \eqref{eq:DMFTDynamics1} into equation \eqref{eq:DynamicsW1}. Then, gathering equations \eqref{eq:SC1}, \eqref{eq:SC2}, \eqref{eq:SC3}, \eqref{eq:SC4}, with the above dynamics on $W(t)$, and simply writing $\mathcal{K}_y = 2K$, we obtain the equations of \cref{App:Subsec:SummaryDMFT}. 

\subsection{Simplified Set of Equations} \label{App:Subsec:DMFTSimplification}

We now prove the simplified set of equations in \cref{Result1} under \cref{Assumption2}. More precisely, we take $\ell$ to be the quadratic cost and assume that the labels are generated using a Gaussian noisy channel. With these assumptions, we simplify the previous set of equations. The plan is the following:
\begin{itemize}
    \item We identify the label evolution as a Gaussian dynamics and transform the equation so that it is driven by quantities associated with $W(t)$. This allows to write a system of equations on the variables $R_y, m_y$ (\cref{App:Subsubsec:DMFTSimple1Labels}).
    \item We then use these simplifications to eliminate the variables $r, \Gamma, \mathcal{K}_Z$ (\cref{App:Subsubsec:DMFTSimple2Hat}).
    \item We finally rewrite the response function $R_Z$ and obtain the desired set of equations (\cref{App:Subsubsec:DMFTSimple3Response}). 
\end{itemize}

\subsubsection{Label Equation} \label{App:Subsubsec:DMFTSimple1Labels}

We proceed to simplify the label equation \eqref{eq:SC3}. Using \cref{Assumption2:Channel} on the noisy channel, we write $z = y^* + \sqrt{\Delta} \zeta$ where $\zeta \sim \N(0,1)$, we get the equation:
\begin{equation} \label{eq:DynamicsLabel1}
    \int_0^t R(t,t') y(t') \d t' + \eta(t) - m(t)y^* + \frac{2}{\alpha} \big(y(t) - y^* - \sqrt{\Delta} \zeta \big) = 0. 
\end{equation}
Remark now that $y(t)$ is a centered Gaussian process. Before deriving its covariance, we apply the linear time transformation $R_Z$ (equal to $R^{-1}$ from equation \ref{eq:SC2}) to equation \eqref{eq:DynamicsLabel1}, to get:
\begin{equation*}
    y(t) + \underbrace{\int_0^t R_Z(t,t') \eta(t') \d t'}_{\displaystyle \equiv \xi(t)} - \int_0^t R_Z(t,t') m(t') \d t' \,  y^* + \frac{2}{\alpha} \int_0^t R_Z(t,t') \big(y(t') - y^* - \sqrt{\Delta} \zeta \big) \d t' = 0. 
\end{equation*}
Now, from the expression of $m$ in equation \eqref{eq:SC2}, and the fact that the covariance of $\eta$ is $K$, written in equation \eqref{eq:SC2}, we have the identities:
\begin{equation*}
\begin{aligned}
    \int_0^t R_Z(t,t') m(t') \d t' &= \frac{1}{Q_*} m_Z(t), \\
    \E \, \xi(t) \xi(t') &= 2 C_Z(t,t') - \frac{2}{Q_*} m_Z(t) m_Z(t'). 
\end{aligned}
\end{equation*}
This transformation does not change the expression of $y(t)$, nor those of $m_y, C_y$. However, we should recompute the response with respect to a perturbation of the new dynamics. We then consider $R_y^\text{old}$ to be the response associated with equation \eqref{eq:DynamicsLabel1} and $R_y^\text{new}$ the response after the transformation by $R_Z$. In addition, since the responses are defined as additive perturbations of the noise (with a minus sign in this case), we simply have:
\begin{equation} \label{eq:ResponsesOldNEw}
    R_y^\text{old}(t,t') = - \frac{\partial \, \E \, y(t)}{\partial \eta(t')}, \hspace{1.5cm} R_y^\text{new}(t,t') = - \frac{\partial \, \E \, y(t)}{\partial \xi(t')},
\end{equation}
and due to the relationship between $\xi(t), \eta(t)$, we get:
\begin{equation} \label{eq:UpdateResponses}
    R_y^\text{old}(t,t') = \int_{t'}^t R_y^\text{new}(t,t'') R_Z(t'',t') \d t''. 
\end{equation}
From now on, we denote $R_y = R_y^\text{new}$. This leads to the equation on $y$:
\begin{equation} \label{eq:EvolutionLabels}
    y(t) + \xi(t) - \frac{1}{Q_*} m_Z(t) y^* + \frac{2}{\alpha} \int_0^t R_Z(t,t') \big( y(t') - y^* - \sqrt{\Delta} \zeta \big) \d t' = 0.
\end{equation}
From this stochastic evolution for $y(t)$, we can obtain equations for the deterministic functions $R_y, m_y$. This is possible since $y$ is explicitly written as a sum of independent Gaussian processes. Starting with the response, we use the previous identity and differentiate the new equation on $y(t)$ with respect to $\xi(t)$. This leads to:
\begin{equation} \label{eq:ResponseLabel}
    \delta(t-t') = R_y(t,t') + \frac{2}{\alpha} \int_{t'}^t R_Z(t,t'') R_y(t'',t') \d t''. 
\end{equation}
This implies that we have the identity, from equation \eqref{eq:EvolutionLabels}:
\begin{equation} \label{eq:EvolutionLabels2}
    y(t) = y^* + \sqrt{\Delta} \zeta + \int_0^t R_y(t,t') \Big( \xi(t') + \big(\chi_Z(t') - 1 \big) y^* - \sqrt{\Delta} \zeta \Big) \d t',
\end{equation}
with $\chi_Z(t) = m_Z(t) / Q_*$. Here we have transformed $\xi$ to $-\xi$, which leaves unchanged the statistics of this process. However, similarly to equation \eqref{eq:ResponsesOldNEw}, $R_y$ is now computed as the derivative of $\E \, y$ with respect to $\xi$. In order to arrive at equation \eqref{eq:DMFT1Y}, the only step remaining will be to reexpress $R_y$. 

Finally, the expression of $y(t)$ in equation \eqref{eq:EvolutionLabels2} leads to the identity on $m_y$:
\begin{equation} \label{eq:OverlapCorry}
    \frac{1}{2Q_*} m_y(t) = 1 - \int_0^t R_y(t,t') \big( 1 - \chi_Z(t') \big) \d t'.
\end{equation}

\subsubsection{Hat Variables} \label{App:Subsubsec:DMFTSimple2Hat}

We now compute the variables $\Gamma, \mathcal{K}_Z, r$, which are the coefficients of the conjugate matrix $\hat Q$. To simplify notation, we consider one-time and two-time functions as vectors and matrices, and replace integration by matrix and vector products. We start from equation \eqref{eq:SC4}, and recall that in these equations the response corresponds to $R_y^\text{old}$, that we should replace using equation~\eqref{eq:UpdateResponses}. Therefore, the system of equations \eqref{eq:SC4} writes, in compact notations:
\begin{align}
    \Gamma &= \alpha (R - RR_y), \notag \\
    \mathcal{K}_Z &= \alpha \left( \frac{1}{2} RC_y R^\top - RR_yR_ZK - KR_Z^\top R_y^\top R^\top - \frac{1}{2} Rm_y m^\top - \frac{1}{2} mm_y^\top R^\top + K + Q_* mm^\top \right), \notag \\
    r &= \alpha \left( \frac{1}{2Q_*} Rm_y - RR_y R_Z m \right). \label{eq:SimplificationsHats}
\end{align}
Now, from equation \eqref{eq:ResponseLabel} on $R_y$, we simply have $\Gamma = 2 R_y$. Now using the expression of $m_y$ in equation \eqref{eq:OverlapCorry}, as well as the expression of $m$ in equation \eqref{eq:SC2}, we easily obtain that $r = \Gamma \1$, that is:
\begin{equation*}
    r(t) = \int_0^t \Gamma(t,t') \d t'. 
\end{equation*}
Remains to compute the covariance $\mathcal{K}_Z$. To do so, let us now write the covariance $C_y$, that we can compute from equation \eqref{eq:EvolutionLabels2}. In this equation, $y$ is written as a sum of the three independent Gaussians $y^*, \zeta, \xi$. Therefore, we have:
\begin{equation*}
\begin{aligned}
    C_y &= 2Q_* \big( \1 + R_y(\chi_Z - \1) \big) \big( \1 + R_y(\chi_Z - \1) \big)^\top + \Delta \big(\1 - R_y \1 \big) \big(\1 - R_y \1 \big)^\top \\
    &\hspace{2cm} + 2 R_y \big( C_Z - Q_*^{-1} m_Z m_Z^\top \big) R_y^\top. 
\end{aligned}
\end{equation*}
Since we have the identities, from equations \eqref{eq:OverlapCorry}, \eqref{eq:ResponseLabel} and the definition of $K$ in equation~\eqref{eq:SC2}:
\begin{equation*}
\begin{aligned}
    \1 + R_y (\chi_Z - \1) &= \frac{1}{2Q_*} m_y, \\
    \1 - R_y \1 &= \frac{2}{\alpha} R_ZR_y \1, \\
    C_Z - Q_*^{-1} m_Zm_Z^\top &= R_Z K R_Z^\top, 
\end{aligned}
\end{equation*}
we get the expression:
\begin{equation*}
    R C_y R^\top = \frac{1}{2Q_*} R m_y m_y^\top R^\top + \frac{4\Delta}{\alpha^2} R_y \1 \1^\top R_y^\top + 2 R R_y R_Z K R_Z^\top R_y^\top R^\top. 
\end{equation*}
Then, using the expression of $m$ in equation \eqref{eq:SC2} and rearranging the expression of $\mathcal{K}_Z$ in equation~\eqref{eq:SimplificationsHats}:
\begin{equation} \label{eq:SimpleKZ1}
\begin{aligned}
    \frac{1}{\alpha} \mathcal{K}_Z &= \big( I - RR_y R_Z \big) K \big( I - R_Z^\top R_y^\top R^\top \big) \\
    &\hspace{2cm}+ \frac{1}{4Q_*} R ( 2m_Z - m_y)(2 m_Z - m_y)^\top R^\top + \frac{2\Delta}{\alpha^2} R_y \1 \1^\top R_y^\top. 
\end{aligned}
\end{equation}
Now, with the fact that $RR_Z = I$, we get by multiplying equation \eqref{eq:ResponseLabel} by $R$ on the left and $R_Z$ on the right:
\begin{equation*}
    I - RR_yR_Z = \frac{2}{\alpha} R_y R_Z. 
\end{equation*}
Therefore, back to the definition of $K$ in equation \eqref{eq:SC2}:
\begin{equation*}
    \big( I - RR_y R_Z \big) K \big( I - R_Z^\top R_y^\top R^\top \big) = \frac{4}{\alpha^2} R_y \big( C_Z - Q_*^{-1} m_Z m_Z^\top) R_y^\top. 
\end{equation*}
This gives the expression of the first term of equation \eqref{eq:SimpleKZ1}. For the second, we use the expression of $m_y$ in equation \eqref{eq:OverlapCorry} to get:
\begin{equation*}
    2 m_Z - m_y = 2 (I - R_y) (m_Z - Q_* \1). 
\end{equation*}
Therefore, since $\alpha(R - RR_y) = 2 R_y$, we get:
\begin{equation*}
    \frac{1}{4Q_*} R ( 2m_Z - m_y)(2 m_Z - m_y)^\top R^\top = \frac{4}{\alpha^2 Q_*} R_y (m_Z - Q_* \1)(m_Z - Q_* \1)^\top R_y^\top. 
\end{equation*}
Putting everything together, we arrive at the expression for $\mathcal{K}_Z$:
\begin{equation*}
    \mathcal{K}_Z = \frac{4}{\alpha} R_y \left( C_Z - \1 m_Z^\top - m_Z \1^\top + Q_* \1 \1^\top + \frac{\Delta}{2} \1 \1^\top \right) R_y^\top. 
\end{equation*}
Due to the expression of $C_Z, m_Z, Q_*$, we get the expression:
\begin{equation*}
    \mathcal{K}_Z(s, t) = \frac{4}{\alpha} \int_0^s \int_0^t R_y(s,s') R_y(t,t') \left( \frac{1}{d} \E \, \tr \Big[ \big( Z(s') - Z^* \big) \big( Z(t') - Z^* \big) \Big] + \frac{\Delta}{2} \right) \d t' \d s'.  
\end{equation*}
Now recall that $\mathcal{K}_Z$ is linked to the covariance of the Gaussian process $\mathcal{H}$ in equation \eqref{eq:CovarianceGaussianProcesses}. Due to the integral structure of $\mathcal{K}_Z$, we can write:
\begin{equation*}
    \mathcal{H}(t) = 2 \int_0^t R_y(t,t') \mathcal{G}(t') \d t', 
\end{equation*}
where $\mathcal{G}$ is also a centered Gaussian process taking values in the space of symmetric matrices, and whose covariance is given by:
\begin{equation*}
    \E \, \G_{ij}(t) \G_{i'j'}(t') = \frac{1}{2\alpha d} \big( \delta_{ii'} \delta_{jj'}+ \delta_{ij'} \delta_{i'j} \big) \left(  \frac{1}{d} \E \, \tr \Big[ \big( Z(t) - Z^* \big) \big( Z(t') - Z^* \big) \Big] + \frac{\Delta}{2} \right). 
\end{equation*}
Now, this is precisely the same covariance as in equation \eqref{eq:Covariance1G}. Putting everything together, equation \eqref{eq:DMFTStudent} can be rewritten as:
\begin{equation} \label{eq:CCLWt}
    \d W(t) = 2\left( \int_0^t R_y(t,t') \Big(\G(t') + Z^* - Z(t') \Big) \d t' \right) W(t) \d t - \nabla \Omega \big( W(t) \big) \d t + \frac{1}{\sqrt{\beta d}} \d B(t).
\end{equation}

\subsubsection{Response Function} \label{App:Subsubsec:DMFTSimple3Response}

The last step of the simplification is to rewrite the responses $R_y, R_Z$. At the moment we have the relationships:
\begin{align}
    R_Z(t,t') &= \frac{1}{d^2} \tr \left( \left. \frac{\partial \, \E \, Z(t)}{\partial H(t')} \right|_{H = 0} \right), \\
    \delta(t-t') &= R_y(t,t') + \frac{2}{\alpha} \int_{t'}^t \d t'' \, R_Z(t,t'') R_y(t'',t'), \label{eq:RelationshipResponses}
\end{align}
and $H(t) \in \mathcal{S}_d(\R)$ perturbs the equation for $W(t)$ as:
\begin{equation*}
    \d W(t) = 2\left( \int_0^t R_y(t,t') \Big(\G(t') + Z^* - Z(t') \Big) \d t' \right) W(t) \d t + H(t) W(t) \d t + \dots,
\end{equation*}
where the dots include the other terms of equation \eqref{eq:CCLWt}. We now consider the following perturbation:
\begin{equation*}
    H(t) = 2 \int_0^t R_y(t,t') \tilde H(t') \d t',
\end{equation*}
so that $\tilde H(t')$ appears as an additive perturbation of the noise $\G(t)$ into the equation for $W$. The response
associated with this perturbation writes:
\begin{equation*}
    \frac{1}{d^2} \tr \left( \frac{\partial Z(t)}{\partial \tilde H(t')} \right) = \frac{2}{d^2} \int_0^t \tr \left( \frac{\partial Z(t)}{\partial H(t'')} \right) R_y(t'',t') \bm{1}_{t'' \geq t'} \d t''. 
\end{equation*}
Therefore, when averaging, we obtain the response associated with $\tilde H$:
\begin{equation*}
\begin{aligned}
    \tilde R_Z(t,t') &\equiv \frac{1}{d^2} \tr \left( \left. \frac{\partial \, \E \, Z(t)}{\partial \tilde H(t')} \right|_{H = 0}  \right) \\
    &= 2 \int_{t'}^t \d t'' \, R_Z(t,t'') R_y(t'',t') \\
    &= \alpha \big( \delta(t-t') - R_y(t,t') \big),
\end{aligned}
\end{equation*}
where we used the relationship between $R_Z, R_y$ in equation \eqref{eq:RelationshipResponses}. This means that we can replace $R_y$ using the previous equation, which decouples the equation for $W(t)$ from the one on $y(t)$. We finally arrive at our system of equations by setting $R(t,t') \equiv R_y(t,t') = \delta(t-t') - \frac{1}{\alpha} \tilde R_Z(t,t')$ into equations \eqref{eq:EvolutionLabels2} and \eqref{eq:CCLWt}. 

\subsection{Learning the Second Layer} \label{App:Subsec:SecondLayer}

In this section, following from \cref{Subsubsec:SecondLayer}, we derive a similar set of equations when also optimizing the output weights of the neural network. 

\subsubsection{Gradient Flow Dynamics}

With the notations of \cref{Sec:Setting}, we are now interested in the predictor:
\begin{equation*}
    X \mapsto \tr \big( X W D_a W^\top \big) = \frac{1}{m} \sum_{i=1}^m a_i \tr \big( X w_i w_i^\top \big), \hspace{1.5cm} W = \frac{1}{\sqrt{m}} \Big( w_1 \, \big| \dots \big| \, w_m \Big),
\end{equation*}
with $a = \big(a_1, \dots, a_m \big)^\top \in \R^m$ and $D_a$ is the diagonal matrix with the same coefficients as $a$. We then consider the loss function:
\begin{equation*}
    \L(a, W) = \frac{1}{2n} \sum_{k=1}^n \ell \big( \tr \big( X_k W D_a W^\top \big), z_k \big),
\end{equation*}
where $z_k$ is still drawn from the distribution $z_k \sim P \big( \cdot \, \big| \, \tr(X_k Z^*) \big)$. For this derivation, the teacher matrix can remain as is, but we can think of it as $Z^* = W^* D_{a^*} W^{*\top}$ for $W^* \in \R^{d \times m^*}$ and $a^* \in \R^{m^*}$. We are then interested in the gradient flow dynamics:
\begin{equation*}
\begin{aligned}
    \dot a(t) &= - \vartheta d \, \nabla_a \L \big( a(t), W(t) \big)& &= - \frac{\vartheta d}{2n} \sum_{k=1}^n \ell' \big( y_k(t), z_k \big) \mathrm{diag} \big( W(t)^\top X_k W(t) \big), \\
    \dot W(t) &= - d \, \nabla_W \L \big( a(t), W(t) \big)& &= - \frac{d}{n} \sum_{k=1}^n \ell' \big( y_k(t), z_k \big) X_k W(t) D_{a(t)},
\end{aligned}
\end{equation*}
where $y_k(t) = \tr \big( X_k W(t) D_{a(t)} W(t)^\top \big)$, $\mathrm{diag}(A)$ is the vector composed of the diagonal elements of $A$, and $\vartheta > 0$ is a parameter that allows for different learning speed for $a(t)$ and $W(t)$. Note that the dynamics considered in \cref{Subsubsec:SecondLayer} includes  regularization and thermal noise terms, but those remain unchanged in the resulting dynamics as it is still formulated in terms of the variables $a(t), W(t)$.    

\subsubsection{Derivation of the Equations}

We shall now derive the set of equations given in \cref{Subsubsec:SecondLayer}. 

\paragraph{The dynamical partition function.} Writing the dynamical partition function similarly to what was done in \cref{App:Subsec:DMFTDerivation}, we get by introducing conjugate variables $\hat W(t), \hat a(t)$:
\begin{align}
    &\mathcal{Z}_\text{dyn} \propto \int \D W \D \hat W \D a \D \hat a \exp \left( -id \int_0^T \Big( \tr \big( \dot W(t) \hat W(t)^\top \big) + \dot a(t)^\top \hat a(t) \Big) \d t \right) \label{eq:PartitionFunctionSecondLayer} \\
    &\hspace{0.3cm}\exp \left( - \frac{id^2}{n} \sum_{k=1}^n \int_0^T \ell' \big( y_k(t), z_k \big) \left[ \tr \big( X_k W(t) D_{a(t)} \hat W(t)^\top \big) + \frac{\vartheta}{2} \tr \big( X_k W(t) D_{\hat a(t)} W(t)^\top \big) \right] \d t\right) \notag.
\end{align}
We then define:
\begin{equation} \label{eq:hatZtwolayers}
    \hat Z(t) = \mathrm{Sym} \left( W(t) D_{a(t)} \hat W(t)^\top + \frac{\vartheta}{2} W(t) D_{\hat a(t)} W(t)^\top \right), \hspace{1.2cm} \hat y_k(t) = \tr \big( X_k \hat Z(t) \big), 
\end{equation}
and find ourselves in the same setup as the is \cref{App:Subsec:DMFTDerivation} when considering this $\hat Z(t)$ and $Z(t) = W(t) D_{a(t)} W(t)^\top$. The only difference is the extra term involving $\dot a(t)$ in the dynamical partition function.

\paragraph{Response structure.} In \cref{App:Subsec:DMFTDerivation}, we studied the structure of the overlap matrix when considering perturbed dynamics solely on $W(t)$ (see equation \ref{eq:PartitionFunctionGeneric}). Here, we study a similar generic dynamics taking the form of a perturbed gradient flow associated with a function of $Z = W D_a W^\top$:
\begin{equation} \label{eq:GenericDynamics2Layers}
\begin{aligned}
    \dot a(t) &= - \frac{\vartheta}{2} \mathrm{diag} \Big( W(t)^\top \Big(\nabla F \big( Z(t) \big) + H(t) \Big) W(t) \Big), \\
    \dot W(t) &= -\Big( \nabla F \big( Z(t) \big) + H(t) \Big) W(t) D_{a(t)},
\end{aligned}
\end{equation}
where $H(t) \in \mathcal{S}_d(\R)$ is to be considered as a perturbation of the dynamics. Then, we can write the dynamical partition function in a similar fashion to equation \eqref{eq:PartitionFunctionGeneric}. Only showing the terms associated with the perturbation, we have:
\begin{equation*}
\begin{aligned}
    \mathcal{Z}_\text{GF} &\propto \int \D W \D \hat W \D a \D \hat a \exp \Bigg( id \int_0^T \bigg[\tr \big( H(t) W(t) D_{a(t)} \hat W(t)^\top \big) \\
    &\hspace{6cm} + \frac{\vartheta}{2} \hat a(t)^\top \mathrm{diag} \big( W(t)^\top H(t) W(t) \big) \bigg] \d t \Bigg) \\
    &= \int \D W \D \hat W \D a \D \hat a \exp \left( id \int_0^T \tr \big(H(t) \hat Z(t) \big) \d t \right),
\end{aligned}
\end{equation*}
where $\hat Z(t)$ is defined in equation \eqref{eq:hatZtwolayers}. Then, as previously done, we can show that differentiating averages of a function of the dynamics \eqref{eq:GenericDynamics2Layers} with respect to $H(t')$ is akin to multiplying by $\hat Z(t')$ and take the average. In particular, we find the same structure of the overlap matrix as in \cref{App:Subsec:DMFTDerivation} and we have the response identity:
\begin{equation*}
    \E \, \tr\big( Z(s) \hat Z(t) \big) = - \frac{i}{d} \tr \left( \left. \frac{\partial \, \E \, Z(s)}{\partial H(t)} \right|_{H = 0} \right),
\end{equation*}
where $Z(t) = W(t) D_{a(t)} W(t)^\top$ and $H(t)$ is introduced as in the dynamics \eqref{eq:GenericDynamics2Layers}. 

\paragraph{End of the calculation.} Once the response is computed, the exact same calculation can be carried out as in \cref{App:Subsec:DMFTDerivation}. In the end, to derive the dynamical equations on $a(t), W(t)$, we obtain a function that is similar to the one in equation \eqref{eq:FunctionsSaddle}:
\begin{equation*}
\begin{aligned}
    \mathcal{F} = \frac{1}{d^2} \log \int \D W \D \hat W \D a \D \hat a \exp \Bigg( - id \int_0^T \bigg[& \tr \big(\dot W(t) \hat W(t)^\top \big) + \dot a(t)^\top \hat a(t) \\
    &+ \int_0^t \Gamma(t,t') \tr \big( Z(t') \hat Z(t) \big) \d t' \\
    & - r(t) \tr \big( Z^* \hat Z(t) \big) - \tr \Big(\mathrm{Sym}\big( V(t) \big) \hat Z(t) \Big) \bigg] \d t \Bigg).
\end{aligned}
\end{equation*}
Expanding $\hat Z(t)$ using equation \eqref{eq:hatZtwolayers} and integrating with respect to $\hat W, \hat a$, we get the dynamical equations:
\begin{equation} \label{eq:Dynamics2Layers}
\begin{aligned}
    \dot W(t) &= \left( \mathcal{H}(t) + r(t) Z^* - \int_0^t \Gamma(t,t') Z(t') \d t' \right) W(t) D_{a(t)}, \\
    \dot a(t) &= \frac{\vartheta}{2} \mathrm{diag} \left( W(t)^\top \left[ \mathcal{H}(t) + r(t) Z^* - \int_0^t \Gamma(t,t') Z(t') \d t' \right] W(t) \right),
\end{aligned}
\end{equation}
where $Z(t) = W(t) D_{a(t)} W(t)^\top$. Then, in order to get the set of self-consistent equations, one can simply consider the ones given in \cref{App:Subsec:SummaryDMFT} while replacing the evolution of $W(t)$ in \eqref{eq:DMFTStudent} by the joint dynamics for $W(t), a(t)$ in equation \eqref{eq:Dynamics2Layers}. As underlined previously, the response function $R_Z$ in equation \eqref{eq:SummaryDMFTResponses} is defined under a perturbation of the noise $\mathcal{H}(t) \mapsto \mathcal{H}(t) + H(t)$, in both the dynamics for $a(t)$ and $W(t)$. In the end, the regularization and thermal noise introduced in \cref{Subsubsec:SecondLayer} can be added back into the dynamics \eqref{eq:Dynamics2Layers} and lead to the dynamical equations~\eqref{eq:DMFT2LayersA}, \eqref{eq:DMFT2LayersW}. 

To conclude, note the similarities of the dynamical structure between equations \eqref{eq:GenericDynamics2Layers} (when $H = 0$) and \eqref{eq:Dynamics2Layers}. In these dynamical equations, the structure of the optimization problem (gradient flow associated with a function of $Z = W D_a W^\top$) remains, and the gradient of the optimized function is replaced by a sum of nonlinear terms involving a high-dimensional Gaussian process, the teacher matrix and a non-local memory contribution. Interestingly, as it is the case in equation \eqref{eq:GenericDynamics2Layers}, these terms only depend on $Z(t)$, and not directly on $W(t), a(t)$. 

\subsection{Gaussian Equivalence} \label{App:Subsec:GaussianEquivalence}

In this section we give some intuition regarding \cref{Conjecture:GaussianUniversality} on the equivalence between the Gaussian matrix sensing model and the shallow quadratic networks setting. The dynamical equations of \cref{App:Subsec:SummaryDMFT} were derived under the first model, with the sensing matrices $X_1, \dots, X_n$ being i.i.d. drawn from the GOE distribution. From now on we consider the second one, i.e., we assume the $(X_k)_{1 \leq k \leq n}$ to be distributed as:
\begin{equation} \label{eq:SensingMatricesQuadratic}
    X_k = \frac{x_kx_k^\top - I_d}{\sqrt{d}}, \hspace{1.5cm} x_1, \dots, x_n \overset{\mathrm{i.i.d.}}{\sim} \N(0, I_d). 
\end{equation}
As a first remark, one can easily see that $\E \, X_k = 0$ and that the covariance of $X_k$ is given by:
\begin{equation*}
    \E \, (X_k)_{ij} (X_k)_{i'j'} = \frac{1}{d} \big( \delta_{ii'} \delta_{jj'} + \delta_{ij'} \delta_{i'j} \big),
\end{equation*}
which is the same as for the GOE. This is the first requirement for equivalence to hold: one needs to match the first and second moments. 

In \cref{App:Subsec:DMFTDerivation}, we exploited the Gaussian structure of the observations to derive self-consistent equations for the process $\mathbf{y}$ (see equation \ref{eq:defcovarianceQ}). For more general distributions of $X_1, \dots, X_n$, this process is \emph{a priori} non-Gaussian. However, when the matrices are drawn as in equation~\eqref{eq:SensingMatricesQuadratic}, we argue that $\mathbf{y}$ becomes Gaussian in the high-dimensional limit. In what follows, without giving a rigorous proof, we analyze the cumulants of this process and explain how we can show that they match the ones of a Gaussian process in the limit. 

\subsubsection{Cumulants of the Quadratic Networks Distribution}

Let us start by defining the notion of cumulants. Consider some real random variables $Y_1, \dots, Y_r$ with finite moments. Their joint cumulant is defined as:
\begin{equation} \label{eq:DefCumulants}
    K_r \big( Y_1, \dots, Y_r \big) = \left. \frac{\partial^r}{\partial s_1 \dots \partial s_r} \log \E \left[ \exp \left( \sum_{i=1}^r s_i Y_i \right) \right] \right|_{s_1, \dots, s_r = 0}. 
\end{equation}
It is easily seen that $K_1(X) = \E[X]$ and $K_2(X,Y) = \mathrm{Cov}(X,Y)$. In addition, we have the following characterization of the multivariate Gaussian distribution in terms of cumulants: a random vector $Y \in \R^p$  is Gaussian if and only if:
\begin{equation*}
    K_r \big( Y_{i_1}, \dots, Y_{i_r} \big) = 0,
\end{equation*}
for any $r > 2$ and indices $i_1, \dots, i_r \in \{1, \dots, p\}$. In addition, this property still holds for Gaussian processes, since such processes are characterized by the Gaussianity of their finite-dimensional marginals. 

In our case, we replace the Gaussian observations by the quadratic Gaussian matrices in equation \eqref{eq:SensingMatricesQuadratic}. As suggested by the following lemma, the structure of the cumulants is then more complex, but still can be understood:
\begin{lemma} \label{Lemma:CumulantsQuadratic}
    Let $x \sim \N(0, I_d)$ and $A_1, \dots A_r \in \mathcal{S}_d(\R)$. Then, for $r \geq 1$:
    \begin{equation*}
        K_r \big( x^\top A_1 x, \dots, x^\top A_r x \big) = \frac{2^{r-1}}{r} \sum_{\sigma \in \mathfrak{S}_r} \tr \big( A_{\sigma(1)} \dots A_{\sigma(r)} \big), 
    \end{equation*}
    where $\mathfrak{S}_r$ denotes the set of permutations of $r$ elements.
\end{lemma}

\begin{proof}
Start from the definition of the multidimensional cumulants in equation \eqref{eq:DefCumulants}, and plug in $Y_i = x^\top A_i x$. Then:
\begin{equation*}
    K_r \big( Y_1, \dots, Y_r \big) = \left. \frac{\partial^r}{\partial s_1 \dots \partial s_r} \log \E \Big[ e^{x^\top M_s x} \Big] \right|_{s_1, \dots, s_r = 0}, 
\end{equation*}
with:
\begin{equation*}
    M_s = \sum_{i = 1}^r s_i A_i. 
\end{equation*}
Then, choosing $s_1, \dots, s_r$ close enough to zero, $M_s$ has a small enough spectral radius so that the previous expectation is finite and reads:
\begin{equation*}
     \log \E \Big[ e^{x^\top M_s x} \Big] = - \frac{1}{2} \log \det \big( I_d - 2M_s \big) = \frac{1}{2} \sum_{p=1}^\infty \frac{1}{p} 2^p \, \tr \big( M_s^p \big).
\end{equation*}
Then developing $\tr(M_s^p)$, we obtain:
\begin{equation*}
    \log \E \Big[ e^{x^\top M_s x} \Big] = \sum_{p=1}^\infty \frac{2^{p-1}}{p} \sum_{1 \leq i_1, \dots, i_p \leq r} s_{i_1} \dots s_{i_p} \tr \big( A_{i_1} \dots A_{i_p} \big). 
\end{equation*}
Since we are taking the partial derivative with respect to $s_1, \dots, s_r$ and then setting these variables to zero, all the terms $p \neq r$ vanish. In addition, it is clear that $i_1, \dots, i_r$ has to correspond to a permutation of $\{1, \dots, r\}$, otherwise one of the $s_i$ would not be represented and lead to a zero derivative. Therefore we get the desired.
\end{proof}
As a consequence of this lemma, if we let:
\begin{equation*}
    y_i = \tr \left( A_i \frac{xx^\top - I_d}{\sqrt{d}} \right) = \frac{x^\top A_i x - \tr(A_i)}{\sqrt{d}},
\end{equation*}
due to the homogeneity of the cumulants and the invariance with respect to a constant shift (for $r \geq 2$), we have the expression:
\begin{equation*}
    K_r(y_1, \dots, y_r) = \frac{1}{d^{r/2}} \frac{2^{r-1}}{r} \sum_{\sigma \in \mathfrak{S}_r} \tr \big( A_{\sigma(1)} \dots A_{\sigma(r)} \big). 
\end{equation*}
Then, if all traces of products involving the matrices $A_1, \dots, A_r$ are of order $d$ (as it will be the case for us), we have that in the $d \to \infty$ limit, $K_r(y_1, \dots, y_r) \longrightarrow 0$ as soon as $r > 2$, leading to the asymptotic Gaussianity of the random vector $(y_1, \dots, y_r)$ in the high-dimensional limit. 

\subsubsection{Equivalence for the Dynamical Partition Function}

Let us apply this result in our dynamical setting. Recall the expression of the partition function in \cref{App:Subsec:DMFTDerivation}:
\begin{equation*}
\begin{aligned}
    \overline{\mathcal{Z}}_\mathrm{dyn} &\propto \int \D W \D \hat W \, \exp \left(- id \int_0^T \tr \, \dot W(t) \hat W(t)^\top \d t \right) \left[ \E_{\mathbf{y},z} \exp \left( - \frac{i}{\alpha} \int_0^T \ell' \big(y(t), z \big) \hat y(t) \d t\right) \right]^n, 
\end{aligned}
\end{equation*}
where:
\begin{equation*}
    \mathbf{y}(t) = \Big( \tr(XZ(t)), \tr(X \hat Z(t)), \tr(XZ^*) \Big). 
\end{equation*}
This expression of the partition function holds no matter the distribution of $X$. When $X$ was drawn from the GOE, the next steps were to introduce the covariance function $Q$ of the Gaussian process $\mathbf{y}$, and to write saddle-point equations on this covariance. Similarly here, we could perform the saddle-point with respect to the higher-order cumulants of the process $\mathbf{y}(t)$:
\begin{equation} \label{eq:CumulantDynamics}
    Q_{i_1, \dots, i_r} \big( t_1, \dots, t_r \big) = K_r \big( \mathbf{y}_{i_1}(t_1), \dots, \mathbf{y}_{i_r}(t_r) \big),  
\end{equation}
with $i_1, \dots, i_r \in \{1, 2, 3 \}$ (corresponding to $Z, \hat Z$ or $Z^*$) and $t_1, \dots, t_r \in [0,T]$. As it was already shown for the Gaussian case, in the high-dimensional limit, these cumulants would concentrate around the ones when averaging with respect to the true dynamics, and the order-$r$ cumulant would then be expressed as a finite sum of terms of the form:
\begin{equation} \label{eq:Tracehighorder}
    \frac{1}{d^{r/2}} \E \,  \tr \Big( Z_{i_1}(t_1) \dots Z_{i_r}(t_r) \Big),
\end{equation}
with $Z_1 = Z, Z_2 = \hat Z, Z_3 = Z^*$. We now claim the following: as a consequence of the scaling chosen for the initialization, teacher, and the dynamics, all of these traces (at least in expectation) will remain of order $d$. To be more precise, as we showed in \cref{App:Subsubsec:DMFT3Covariance} when considering a general gradient flow dynamics \eqref{eq:PartitionFunctionGeneric}, the contraction with the matrix $\hat Z$ acts as a derivative when perturbing the dynamics for $Z(t)$. Therefore, any trace as in equation \eqref{eq:Tracehighorder} can in principle be expressed as a derivative of a contraction between the matrix $Z(t)$ (at different instants) and the teacher $Z^*$. Although we do not prove rigorously that these traces all remain of order $d$, we believe that such a property holds. 

In the end, this would allow to treat $\mathbf{y}$ as a Gaussian process, and therefore would lead to the same result derived for the Gaussian case. 

\paragraph{Some remarks.} We give some remarks regarding our previous arguments:
\begin{itemize}
    \item The previous argument, if made rigorous, would only guarantee the convergence in distribution of the finite-dimensional marginals of $\mathbf{y}$. To prove the convergence in distribution (in the space of continuous functions on $[0,T]$) of $\mathbf{y}$ toward a Gaussian process, one would require a tightness argument \citep[see for instance][Section 7]{billingsley2013convergence}. 
    \item Regarding the generality of our result, our calculation heavily rests on the quadratic Gaussian structure of the sensing matrices \eqref{eq:SensingMatricesQuadratic}: it enables us to compute exactly the cumulants. As a consequence, it is not clear whether this cumulant method could be extended to more general distributions. 
\end{itemize}

\section{Derivation of the Long-Time Equations} \label{App:LongTimes}
This section is devoted to the long-time analysis of the system of equations stated in \cref{Result1}. In particular, we derive the system of equations of \cref{Result2} and provide several proofs and insights supporting the results presented in \cref{Subsec:LongTimes}. The plan of the section is as follows:
\begin{itemize}
    \item In \cref{App:Subsec:SteadyState}, we discuss the steady-state assumption formulated in \cref{Assumption3} and relate it to the fast convergence of the dynamics. 
    \item In \cref{App:Subsec:DerivationLongTimes}, we build on this assumption to derive the set of equations at long times given in \cref{Result2}.
    \item In \cref{App:Subsec:ResponseInfiniteDim}, we show that the same set of equations can be obtained by taking the high-dimensional limit before the long-time limit, reinforcing our claim.
    \item In \cref{App:Subsec:ERM}, we show that in the overparameterized case $\kappa \geq 1$, our system of equations exactly matches the one derived by \citet{erba2025nuclear} in the empirical risk minimization setting.
    \item In \cref{App:Subsec:AnalysisEquations}, we analyze the system of equations at long times and derive some additional claims made in \cref{Subsec:LongTimes}.
    \item In \cref{App:Subsec:Population}, we study the population limit and prove the claims made in \cref{Subsubsec:Population} and \cref{Prop:PopulationLimit}. 
\end{itemize}

\subsection{Steady-State Assumption} \label{App:Subsec:SteadyState}

In this section we discuss the steady-state assumption formulated in \cref{Assumption3}. We start by recalling the set of equations from \cref{Result1}. With the choice of the $\ell_2$-regularization, the student matrix $W(t)$ solves the dynamics:
\begin{equation} \label{eq:GFdynamics2}
    \dot W(t) = 2 \int_0^t R(t,t') \Big( \mathcal{G}(t') + Z^* - W(t')W(t')^\top \Big) \d t' \, W(t) - 2 \lambda W(t).
\end{equation}
From the equations in \cref{Result1}, the covariance of the centered Gaussian process $\G(t)$ and the kernel $R(t,t')$ can be computed as averages with respect to equation \eqref{eq:GFdynamics2}:
\begin{align}
    \E \, \G_{ij}(t) \G_{i'j'}(t') &= \frac{1}{2\alpha d} \big( \delta_{ii'} \delta_{jj'} + \delta_{ij'} \delta_{i'j} \big) \left(\frac{1}{d} \E \, \tr \Big( (Z(t) - Z^*) (Z(t') - Z^*) \Big) + \frac{\Delta}{2} \right), \label{eq:LTCovariance} \\
    R(t,t') &= \delta(t-t') - \frac{1}{\alpha d^2} \tr\left( \left. \frac{\partial \, \E \, Z(t)}{\partial H(t')} \right|_{H = 0} \right), \label{eq:LTResponse}
\end{align}
with $Z(t) = W(t) W(t)^\top$, and the response is defined in terms of a perturbation $\G(t) \to \G(t) + H(t)$ in equation \eqref{eq:GFdynamics2}. 

The steady-state assumption bears on the fast convergence of the memory kernel $R(t,t')$ as $t-t' \to \infty$, and the one of the Gaussian process $\G(t)$ toward its final value, that we denote $\G_\infty$. As mentioned in \cref{Subsubsec:SteadyState}, we formulate this assumption in such a way that we can approximate the dynamics \eqref{eq:GFdynamics2} by:
\begin{equation} \label{eq:SimplifiedDynamics2}
    \dot W(t) = 2 r_\infty \Big( \G_\infty + Z^* - W(t)W(t)^\top \Big) W(t) - 2 \lambda W(t),
\end{equation}
where:
\begin{equation*}
    r_\infty = \lim_{t \to \infty} \int_0^t R(t,t') \d t'. 
\end{equation*}
We now explain how the steady-state assumption can be interpreted as one on the fast convergence of the matrix $Z(t) = W(t)W(t)^\top$. 

\paragraph{Gaussian noise.}
We first examine the Gaussian process $\mathcal{G}(t)$. From the expression of its covariance in equation~\eqref{eq:LTCovariance}, we obtain:
\begin{equation*}
    \frac{1}{d} \E \Big[ \big\| \mathcal{G}(t) - \mathcal{G}(t') \big\|_F^2 \Big]
    = \frac{1}{2\alpha d} \left( 1 + \frac{1}{d} \right)
    \E \Big[ \big\| Z(t) - Z(t') \big\|_F^2 \Big].
\end{equation*}
Therefore, the convergence to $Z(t)$ at long times implies the one of $\G(t)$ toward a Gaussian matrix $\G_\infty$. In addition, the convergence rate of $\G(t)$ as $t \to \infty$ is directly given by the one of $Z(t)$. Then, \cref{Assumption3:Noise} on the noise is not an independent dynamical property, but simply reflects the fast convergence of $Z(t)$.

\paragraph{Response.}
We now consider the response function defined in equation~\eqref{eq:LTResponse}. By definition, $R(t,t')$ measures the effect at time $t$ of a perturbation applied at an earlier time $t'$ to the trajectory $Z(t)$. As the gradient flow converges at long times, the influence of the initial stages of the dynamics progressively vanishes. As a consequence, perturbations applied far in the past have little impact on the state of the system at long times, and the response $R(t,t')$ decays as $t - t'$ increases.  Therefore, a faster relaxation of the dynamics leads to a faster decay of $R$.

Once the dynamics is close to convergence, the response is therefore localized near $t' = t$. This motivates approximations in which nonlocal memory terms vanish. In particular, one expects:
\begin{equation*}
    \int_0^t R(t,t') \, \phi(t') \, \d t'
    \underset{t \to \infty}{\approx}
    \left( \int_0^t R(t,t') \, \d t' \right) \phi(t),
\end{equation*}
with an error controlled by the relaxation rate of the dynamics. As for the noise, the decay of the response is not an independent assumption, but follows from the loss of memory along the gradient flow.

\subsection{Derivation of the Long-Time Equations} \label{App:Subsec:DerivationLongTimes}

The steady-state assumption allows to simplify the dynamics at long times. In addition to the dynamics \eqref{eq:SimplifiedDynamics2}, the variables $r_\infty$ and the covariance of $\G_\infty$ are self-consistently computed from $W(t)$ at long times. These self-consistent relations are given in equations \eqref{eq:SCxi}, \eqref{eq:SCr}. 

In this section we start from these equations and derive the set of self-consistent scalar equations of \cref{Result2}. The section is organized as follows:
\begin{itemize}
    \item In \cref{App:Subsubsec:LongTimesLimitDynamics}, we derive the long-time limit of the dynamics given in equation \eqref{eq:ZinfinityReg}. 
    \item In \cref{App:Subsubsec:LongTimesResponse}, we use the simplified dynamics to compute the associated response function, which allows to close the set of equations as $t \to \infty$, and obtain equation \eqref{eq:ResultLongTimesIntegral}.
    \item In \cref{App:Subsubsec:LongTimesMSE}, we compute the high-dimensional expression of the MSE using the expression of $Z_\infty$ as a function of the variables $\xi, q, \omega$. This leads to equation \eqref{eq:ResultLongTimesMSE}.
\end{itemize}

\subsubsection{Limit of the Dynamics} \label{App:Subsubsec:LongTimesLimitDynamics}

Let us now derive the limit of the dynamics under the steady-state assumption. 

\paragraph{Limit of the student matrix.} Given the simplified equation \eqref{eq:SimplifiedDynamics2}, one can interpret the dynamics as an Oja flow (see \cref{Sec:OjaFlow}), whose limit is derived in \cref{prop:OjaConvergence}. More precisely, for almost all initializations, the dynamics~\eqref{eq:SimplifiedDynamics2} converges to a point $W_\infty \in \R^{d \times m}$ such that:
\begin{equation} \label{eq:LTZinf}
    W_\infty W_\infty^\top = \left( Z^* + \G_\infty - \frac{\lambda}{r_\infty} I_d \right)_{(m)}^+. 
\end{equation}
We recall that $X \in \mathcal{S}_d(\R) \mapsto X_{(m)}^+ \in \mathcal{S}_d(\R)$ is the spectral map selecting the $m$ largest positive eigenvalues (see \cref{Def:EckartYoung}). Then, setting $q = \lambda / r_\infty$ and $\G_\infty = \sqrt{\xi} \G$ with $\G \sim \mathrm{GOE}(d)$, we precisely get the limit of the dynamics given in equation \eqref{eq:ZinfinityReg}. 

Now, taking the limit $t,t' \to \infty$ in the expression of the covariance of $\G$ in equation \eqref{eq:LTCovariance} leads to the self-consistent equation between $\xi$ and $Z_\infty$:
\begin{equation} \label{eq:xiMSE1}
    \xi = \frac{1}{2\alpha} \left( \frac{1}{d} \E \, \big\| Z_\infty - Z^* \big\|_F^2 + \frac{\Delta}{2} \right), 
\end{equation}
where $Z_\infty$ is the matrix in equation \eqref{eq:LTZinf}. This leads to the same expression of $\xi$ as in \eqref{eq:SCxi}, and therefore the expression of the MSE in equation \eqref{eq:LossValue} by rearranging. 

\paragraph{Limit of the labels.} Let us now apply the steady-state assumption to derive the limit of the label $y(t)$. To do so, recall the expression of $y(t)$ in equation \eqref{eq:DMFT1Y}. 

Similarly to \cref{Assumption3:Noise}, we assume that the Gaussian process $\xi(t)$ appearing in the dynamics \eqref{eq:DMFT1Y} converges fast enough as $t \to \infty$. Indeed, we have the covariance:
\begin{equation*}
    \E \, \xi(t) \xi(t') = 2 C_Z(t,t') - \frac{2}{Q_*} m_Z(t) m_Z(t'),
\end{equation*}
where $C_Z, m_Z, Q_*$ are defined in equation \eqref{eq:AveragedQuantities}. Therefore, as it was discussed in \cref{App:Subsec:SteadyState} for the case of the Gaussian process $\G(t)$, the fast convergence of the process $\xi(t)$ is a direct consequence of the one of $Z(t)$. 

Then, using the steady-state assumption in equation \eqref{eq:DMFT1Y}, we obtain the following expression of the label as $t \to \infty$:
\begin{equation} \label{eq:LimitLabel}
    y_\infty = \left( 1 - r_\infty + r_\infty\frac{m_Z^\infty}{Q_*} \right) y^* + r_\infty \xi_\infty + \sqrt{\Delta} (1 - r_\infty) \zeta,
\end{equation}
where $y^* \sim \N(0, 2Q_*), \zeta \sim \N(0,1)$ and $\xi_\infty$ are three independent centered Gaussian variables, and $\xi_\infty$ has variance:
\begin{equation*}
    \E \, \xi_\infty^2 = 2 C_Z^\infty - \frac{2}{Q_*} \big( m_Z^\infty \big)^2,
\end{equation*}
and we have:
\begin{equation*}
    C_Z^\infty = \frac{1}{d} \E \, \tr \big( Z_\infty^2 \big), \hspace{1.5cm} m_Z^\infty = \frac{1}{d} \E \, \tr \big( Z_\infty Z^* \big). 
\end{equation*}
This gives access to the expression of the training loss. Indeed, recall that we have:
\begin{equation*}
    \mathrm{Loss}_\mathrm{train} = \frac{1}{4} \E \, (y_\infty - z)^2.
\end{equation*}
In our case, since the noisy label writes $z = y^* + \sqrt{\Delta} \zeta$, we obtain:
\begin{equation*}
    \mathrm{Loss}_\mathrm{train} = \frac{r_\infty^2}{2} \left(  \mathrm{MSE} + \frac{\Delta}{2} \right). 
\end{equation*}
Now using the relationship between the MSE and $\xi$ in equation \eqref{eq:SCxi} and the fact that $r_\infty = \lambda / q$, we get the expression of the training loss \eqref{eq:LossValue}. 

\subsubsection{Response Equation} \label{App:Subsubsec:LongTimesResponse}

The goal of this section is to derive equation \eqref{eq:ResultLongTimesIntegral}. To do so, we use the fact that the variable $r_\infty$ is itself expressed as a function of the dynamics, through the response function:
\begin{equation} \label{eq:rinf1}
    R_Z(t,t') = \frac{1}{d^2} \tr\left( \left. \frac{\partial \, \E \, Z(t)}{\partial H(t')} \right|_{H = 0} \right), \hspace{1.5cm} r_\infty = 1 - \frac{1}{\alpha} \lim_{t \to \infty} \int_0^t R_Z(t,t') \d t'. 
\end{equation}
$H(t)$ is a perturbation of the high-dimensional dynamics \eqref{eq:GFdynamics2}:
\begin{equation} \label{eq:GFdynamicsPerturbed}
    \dot W(t) = 2 \int_0^t R(t,t') \Big( \mathcal{G}(t') + Z^* - Z(t') + H(t') \Big) \d t' \, W(t) - 2 \lambda W(t).
\end{equation}
We start by showing that we can compute the quantity $r_\infty$ only using the long-time limit of the dynamics. This enables to use the result of the previous section. Recall that the response operator at times $t,t'$ quantifies the change of $Z(t)$ in response to a perturbation introduced at $t'$. Therefore, the integrated response operator:
\begin{equation*}
    \int_0^t \left. \frac{\partial \, \E \, Z(t)}{\partial H(t')} \right|_{H = 0} \d t',
\end{equation*}
quantifies the response to a constant perturbation $H$ (present since $t = 0$). Therefore, we have the identity:
\begin{equation} \label{eq:AveragedResponse}
    \lim_{t \to \infty} \int_0^t R_Z(t,t') \d t' = \frac{1}{d^2} \tr \left( \left. \frac{\partial \, \E \, Z_\infty}{\partial H} \right|_{H = 0} \right),
\end{equation}
where $Z_\infty$ is the limit of the perturbed dynamics \eqref{eq:GFdynamicsPerturbed}. Now, similarly to the previous section, with the additional constant matrix $H$, we can deduce the limit from \cref{prop:OjaConvergence}:
\begin{equation} \label{eq:ZinfPerturbed}
    Z_\infty = \left( Z^* + \sqrt{\xi} \G - \frac{\lambda}{r_\infty} I_d + H \right)_{(m)}^+,
\end{equation}
and we only need to compute the derivative of this matrix with respect to $H$. Here, we wrote the limit $\G(t) \xrightarrow[t \to \infty]{} \sqrt{\xi} \G$, with $\G$ being a GOE matrix. 

Now the averaged response in equation \eqref{eq:AveragedResponse} can be computed using \cref{Lemma:DerivativeSpectral}. Indeed, equation \eqref{eq:ZinfPerturbed} shows that we can write $Z_\infty = A_{(m)}^+$ at $H = 0$, with $A$ having simple and non-zero eigenvalues, with probability one with respect to the Gaussian matrix $\G$. Therefore:
\begin{equation*}
    \frac{1}{d^2} \tr \left( \left. \frac{\partial \, \E \, Z_\infty}{\partial H} \right|_{H = 0} \right) = \frac{1}{d^2} \E \sum_{1 \leq i < j \leq d} \frac{\lambda_i \1_{\lambda_i > 0} \1_{i \leq m} - \lambda_j \1_{\lambda_j > 0} \1_{j \leq m}}{\lambda_i - \lambda_j} + \frac{1}{d^2} \E \sum_{i=1}^d \1_{\lambda_i > 0} \1_{i \leq m},
\end{equation*}
where $\lambda_1 > \dots > \lambda_d$ are the ordered eigenvalues of $A$. In the high-dimensional limit, the second term vanishes, and using the expression of $r_\infty$ in equation \eqref{eq:rinf1}, we are left with:
\begin{equation*}
    r_\infty = 1 - \frac{1}{2\alpha} \iint \frac{x \1_{x \geq \max(0, \omega)} - y \1_{y \geq \max(0, \omega)}}{x-y} \d \mu_A(x) \d \mu_A(y),
\end{equation*}
where $\mu_A$ is the limiting spectral density of $A$, and $\omega$ selects a fraction $\kappa$ of this distribution, and solves the equation:
\begin{equation*}
    \min(\kappa, 1) = \int \1_{x \geq \omega} \d \mu_A(x). 
\end{equation*}
Going back to the expression of $Z_\infty$ in equation \eqref{eq:ZinfPerturbed} with $H = 0$, we define $q = \lambda / r_\infty$. We then set $\mu_\xi$ to be the asymptotic spectral distribution of $Z^* + \sqrt{\xi} \G$. Therefore, $\omega$ now selects the $m$ largest eigenvalues of $\mu_\xi$ above $q > 0$. This means that in the integral in the expression of $r_\infty$ we should replace the lower bound of the integral by $\max(\omega, q)$. In addition, since the eigenvalues of $A$ are simply those of $Z^* + \sqrt{\xi} \G$ shifted by $q$, we get the equation:
\begin{align}
    r_\infty &= 1 - \frac{1}{2\alpha} \iint \frac{(x-q) \1_{x \geq \max(q, \omega)} - (y-q) \1_{y \geq \max(q, \omega)}}{x-y} \d \mu_\xi(x) \d \mu_\xi(y) \label{eq:ResponseRinf},\\
    \min(\kappa, 1) &= \int \1_{x \geq \omega} \d \mu_\xi(x).
\end{align}
Then using \cref{Lemma:Hilbert} with $\mu = \mu_\xi$ which admits a bounded square-integrable density as soon as $\xi > 0$, we have the identity by choosing $f(x) = (x-q) \1_{x \geq \max(q, \omega)}$:
\begin{equation*}
    \frac{1}{2} \iint \frac{(x-q) \1_{x \geq \max(q, \omega)} - (y-q) \1_{y \geq \max(q, \omega)}}{x-y} \d \mu_\xi(x) \d \mu_\xi(y) = \int_{\max(q, \omega)} (x-q) h_\xi(x) \d \mu_\xi(x),
\end{equation*}
and we obtain equation \eqref{eq:ResultLongTimesIntegral} by replacing $r_\infty = \lambda / q$ in equation \eqref{eq:ResponseRinf}. 

\subsubsection{MSE in the High-Dimensional Limit} \label{App:Subsubsec:LongTimesMSE}

In the following, we compute the high-dimensional limit of the MSE associated with the predictor in equation \eqref{eq:ZinfinityReg}:
\begin{equation*}
    Z_\infty = \Big( Z^* + \sqrt{\xi} \G - q I_d \Big)_{(m)}^+,
\end{equation*}
where $\G \sim \mathrm{GOE}(d)$.
\begin{proposition} \label{Prop:MSEHD}
    For $\xi > 0$ and $q \in \R$, consider:
    \begin{equation*}
        M_d(\xi, q) = \frac{1}{d} \Big\| \Big(Z^* + \sqrt{\xi} \G - q I_d \Big)_{(m)}^+ - Z^* \Big\|_F^2,  
    \end{equation*}
    where $\G \sim \mathrm{GOE}(d)$ and $m \sim \kappa d$ as $d \to \infty$. If the empirical spectral distribution of $Z^*$ converges to some $\mu^*$ as $d \to \infty$, then:
    \begin{equation*}
        \lim_{d \to \infty} M_d(\xi, q) = \int x^2 \d \mu^*(x) + \int_{\max(q, \omega)} (q^2 - x^2) \d \mu_\xi(x) + 4 \xi \int_{\max(q, \omega)} (x-q) h_\xi(x) \d \mu_\xi(x). 
    \end{equation*}
    where $\mu_\xi$ is the free additive convolution between $\mu^*$ and a semicircular distribution of variance~$\xi$, and $h_\xi$ is the Hilbert transform of $\mu_\xi$ (see \cref{Def:HilbertStieltjes}). Finally, $\omega$ verifies:
    \begin{equation*}
        \min(\kappa, 1) = \int_\omega \d \mu_\xi(x). 
    \end{equation*}
\end{proposition}
\begin{proof}
We start by decomposing $M_d(\xi, q)$ at finite dimension, with $Z_\xi^* = Z^* + \sqrt{\xi} \G$:
\begin{equation*}
    M_d(\xi, q) = \frac{1}{d} \|Z^*\|_F^2  - \frac{2}{d} \tr \Big( Z^* \big(Z_\xi^* - q I_d \big)_{(m)}^+ \Big) + \frac{1}{d} \Big\| \big( Z_\xi^* - q I_d \big)_{(m)}^+ \Big\|_F^2.  
\end{equation*}
Then, by assumption, the first term concentrates in the $d \to \infty$ limit:
\begin{equation*}
    \frac{1}{d} \|Z^*\|_F^2 \xrightarrow[d \to \infty]{} \int x^2 \d \mu^*(x).
\end{equation*}
For the two other terms, we denote $\lambda_1, \dots, \lambda_d$ the eigenvalues of $Z^*$ and $u_1, \dots, u_d$ a family of associated eigenvectors. We do the same for $Z_\xi^*$ and write $\lambda_1^\xi, \dots, \lambda_d^\xi$ and $u_1^\xi, \dots, u_d^\xi$ its eigenvalues and eigenvectors. Now, we have:
\begin{equation*}
\begin{aligned}
    \frac{1}{d} \tr \Big( Z^* \big(Z_\xi^* - q I_d \big)_{(m)}^+ \Big) &= \frac{1}{d} \sum_{i=1}^m \sum_{j=1}^d \lambda_j \big(\lambda_i^\xi - q \big)^+ \big(u_j^\top u_i^\xi \big)^2, \\
    \frac{1}{d} \Big\| \big( Z_\xi^* - q I_d \big)_{(m)}^+ \Big\|_F^2 &= \frac{1}{d} \sum_{i=1}^m \big( \lambda_i^\xi - q \big)^{+2}. 
\end{aligned}
\end{equation*}
We can derive the limit of the second term by remarking that the eigenvalues of $Z_\xi^*$ which are selected are the $m$ largest ones that are larger than $q$. Therefore, the second term concentrates around:
\begin{equation*}
     \frac{1}{d} \Big\| \big( Z_\xi^* - q I_d \big)_{(m)}^+ \Big\|_F^2 \xrightarrow[d \to \infty]{} \int_{\max(q, \omega)} (x-q)^2 \d \mu_\xi(x), \hspace{1.5cm} \min(\kappa, 1) = \int_\omega \d \mu_\xi(x). 
\end{equation*}
For the first term, we use the result of \citet{bun2017cleaning} (see Section 4 and Appendix D), and obtain:
\begin{equation*}
    \E \, \big(u_j^\top u_i^\xi \big)^2 \underset{d \to \infty}{\sim} \frac{1}{d} \frac{\xi}{(\lambda_j - \lambda_i^\xi + \xi h_\xi(\lambda_i^\xi))^2 + \pi^2 \xi^2 \rho_\xi(\lambda_i^\xi)^2},
\end{equation*}
where $h_\xi$ is the Hilbert transform of $\mu_\xi$, and $\rho_\xi$ its density. Therefore:
\begin{equation*}
    \frac{1}{d} \tr \Big( Z^* \big(Z_\xi^* - q I_d \big)_{(m)}^+ \Big) \xrightarrow[d \to \infty]{} \int_{\max(q, \omega)} \int \frac{\xi y(x-q)}{(y-x + \xi h_\xi(x))^2 + \pi^2 \xi^2 \rho_\xi(x)^2} \d \mu^*(y) \d \mu_\xi(x).
\end{equation*}
Note that to be fully rigorous, one should prove that the variance of this last term vanishes. Putting everything together, we get:
\begin{equation} \label{eq:MSEHD1}
\begin{aligned}
    \lim_{d \to \infty} M_d(\xi, q) &= \int x^2 \d \mu^*(x) + \int_{\max(q, \omega)} (x-q)^2 \d \mu_\xi(x) \\
    &- 2\xi \int_{\max(q, \omega)}  \int \frac{y(x-q)}{(y-x + \xi h_\xi(x))^2 + \pi^2 \xi^2 \rho_\xi(x)^2} \d \mu^*(y) \d \mu_\xi(x).
\end{aligned}
\end{equation}
Now that we have obtained an expression in the high-dimensional limit, we proceed to simplify it. Let us rewrite this equation as:
\begin{equation} \label{eq:MSELandscape}
    \lim_{d \to \infty} M_d(\xi, q) = Q_* + \int_{\max(q, \omega)} (x-q)^2 \d \mu_\xi(x) - 2 \xi \int_{\max(q, \omega)} (x-q) I_{\mu^*} \big( z_\xi(x) \big) \d \mu_\xi(x), 
\end{equation}
with:
\begin{equation*}
    I_\nu(z) = \int \frac{y}{|y-z|^2} \d \nu(y), \hspace{1.5cm} z_\xi(x) = x - \xi h_\xi(x) + i \pi \xi \rho_\xi(x).
\end{equation*}
Then, using the identity:
\begin{equation*}
    \frac{y}{|y-z|^2} = \frac{1}{z - \bar{z}} \left( \frac{z}{y-z} - \frac{\bar{z}}{y-\bar{z}} \right),
\end{equation*}
we obtain that:
\begin{equation*}
    I_\nu(z) = -\frac{\mathrm{Im} \, zm_\nu(z)}{\mathrm{Im} \, z}, \hspace{2cm} m_\nu(z) = \int \frac{\d \nu(y)}{z-y}. 
\end{equation*}
$m_\nu$ is known as the Stieltjes transform of $\nu$ (see \cref{Def:HilbertStieltjes}). Now, since $\mu_\xi$ is the free additive convolution between $\mu^*$ and a semicircular density with variance $\xi$, as a consequence of \cref{Lemma:SubordinationSemicircular}, we have for all $z \in \mathbb{C} \setminus \mathrm{Supp}(\mu_\xi)$:
\begin{equation*}
    m_\xi(z) = m_* \big(z - \xi m_\xi(z) \big),
\end{equation*}
where $m_\xi$ and $m_*$ are the Stieltjes transforms of $\mu_\xi$ and $\mu^*$. Taking $z = x + i\eta$ and letting $\eta \to 0^+$, we then obtain, using \cref{Def:HilbertStieltjes}:
\begin{equation*}
    h_\xi(x) - i \pi \rho_\xi(x) = m_* \big( z_\xi(x) \big). 
\end{equation*}
Therefore:
\begin{equation*}
    \xi I_{\mu^*} \big( z_\xi(x) \big) = -\xi \frac{\mathrm{Im} \big[z_\xi(x) m_*(z_\xi(x)) \big]}{\mathrm{Im} \, z_\xi(x)} = x - 2 \xi h_\xi(x). 
\end{equation*}
Plugging this into equation \eqref{eq:MSELandscape} leads to:
\begin{equation*}
    \lim_{d \to \infty} M_d(\xi, q) = Q_* + \int_{\max(q, \omega)} \big( q^2 - x^2 \big) \d \mu_\xi(x) + 4 \xi \int_{\max(q, \omega)} (x-q) h_\xi(x) \d \mu_\xi(x), 
\end{equation*}
which is the desired.
\end{proof}

To conclude, using the expression $Z_\infty = \big( Z^* + \sqrt{\xi} \G - q I_d \big)_{(m)}^+$, and the fact that $M_d(\xi, q)$ is precisely the MSE associated with $Z_\infty$, the relationship between the MSE and $\xi$ in equation \eqref{eq:SCxi} can be rewritten as:
\begin{equation*}
    2 \alpha \xi - \frac{\Delta}{2} = Q_* + \int_{\max(q, \omega)} \big( q^2 - x^2 \big) \d \mu_\xi(x) + 4 \xi \int_{\max(q, \omega)} (x-q) h_\xi(x) \d \mu_\xi(x),
\end{equation*}
which is precisely equation \eqref{eq:ResultLongTimesMSE}. 

\subsection{Equations in the High-Dimensional Limit} \label{App:Subsec:ResponseInfiniteDim}

As mentioned earlier, for the sake of consistency with our earlier results, we derive our set of high-dimensional equations \eqref{eq:SystemResult2} in a slightly different fashion. Indeed, the method proposed in \cref{App:Subsec:DerivationLongTimes} first uses the long-time limit of the Oja flow (and its response), before taking the high-dimensional limit. However, we recall that the results derived in \cref{subsec:DMFT} already rely on a large dimension and are only valid for a fixed time horizon. 

Therefore, in order to strengthen our results, we show that we recover the same equations when first taking the high-dimensional limit (associated with the simplified dynamics) before the limit $t \to \infty$. To do so, we use the results of \cref{Sec:OjaFlow} on the Oja flow dynamics. 

\paragraph{Limit of the dynamics.} In \cref{Subsec:OjaCVInfinitedim}, we derive convergence rates for the Oja flow in the high-dimensional limit. In particular, we show that if $W(t)$ solves the dynamics:
\begin{equation*}
    \dot W(t) = \big( A - W(t)W(t)^\top \big) W(t),
\end{equation*}
then:
\begin{equation*}
    \lim_{t \to \infty} \lim_{d \to \infty} \frac{1}{d} \big\| Z(t) - Z_\infty \big\|_F^2 = 0,
\end{equation*}
with $Z(t) = W(t)W(t)^\top$ and $Z_\infty = A_{(m)}^+$. 
Crucially in the previous equation, the high-dimensional limit is taken before the long-time limit. Applying this result to the target matrix $A = Z^* + \sqrt{\xi} \G - q I_d$, we end up with the same limit as the one found in \cref{App:Subsubsec:LongTimesLimitDynamics}. Using again the result of \cref{App:Subsubsec:LongTimesMSE} as well as the relationship \eqref{eq:SCxi}, this precisely leads to equation \eqref{eq:ResultLongTimesMSE}. 

\paragraph{Response equation.} In order to derive the response equation, recall the expression of $R_Z$ and its link to $r_\infty$ in equation \eqref{eq:rinf1}. Using \cref{Prop:DiagonalResponse}, we have the high-dimensional limit:
\begin{equation} \label{eq:IntegratedResponse1}
    \int_0^t R_Z(t,t') \d t' \xrightarrow[d \to \infty]{} \frac{\mathfrak{g}(t)}{2} \iint \frac{1}{y-x} \left[ \frac{y e^{2yt}}{q_t(y)} - \frac{x e^{2xt}}{q_t(x)} \right] \d \mu_A(x) \d \mu_A(y),
\end{equation} 
where $q_t(x) = \mathfrak{g}(t)(e^{2xt} - 1) + x$ and $\mathfrak{g}(t)$ solves the self-consistent equation:
\begin{equation*}
    \kappa \mathfrak{g}(t) + 1 - \kappa = \int \frac{x}{(e^{2xt}-1) \mathfrak{g}(t) + x} \d \mu_A(x). 
\end{equation*}
We have proven this proposition in the case where the initialization of the flow is a Gaussian matrix, but we believe that when taking the $t \to \infty$ limit, the result becomes independent of the initialization (as it is the case for our derivation in \cref{App:Subsubsec:LongTimesResponse}).

Now, in order to take the long-time limit in equation \eqref{eq:IntegratedResponse1}, one can simply use \cref{Lemma:LongtimesStieltjes} that derives the long-time asymptotics of the function $\mathfrak{g}(t)$. As a consequence, we have the asymptotic:
\begin{equation*}
    \mathfrak{g}(t) \frac{x e^{2xt}}{q_t(x)} \xrightarrow[t \to \infty]{} x \1_{x \geq \max(0, \omega)}, \hspace{1.5cm} \kappa = \int \1_{x \geq \omega} \d \mu_A(x). 
\end{equation*}
Now plugging the expression of $A = Z^* + \sqrt{\xi} \G - q I_d$, we arrive at the identity:
\begin{equation*}
    \lim_{t \to \infty} \lim_{d \to \infty} \int_0^t R_Z(t,t') \d t' = \frac{1}{2} \iint \frac{(x-q) \1_{x \geq \max(q, \omega)} - (y-q) \1_{y \geq \max(q, \omega)}}{x-y} \d \mu_\xi(x) \d \mu_\xi(y). 
\end{equation*}
Using \cref{Lemma:Hilbert}, we end up with equation \eqref{eq:ResultLongTimesIntegral}. 

\paragraph{Conclusion.} In conclusion, the equations we obtain are unchanged when first taking the high-dimensional limit. This suggests a robustness of the Oja flow dynamics and reveals that the only relevant timescale is $t = O_d(1)$. 

Indeed, first taking the limit $t \to \infty$ allows the dynamics to explore all timescales, including those that may depend on the dimension. On the other hand, taking the limit $d \to \infty$ restricts the analysis to a timescale of order one. The fact that both limits lead to the same system of equations indicates that no additional dynamical behavior emerges beyond the order-one timescale. 

\subsection{Link with Empirical Risk Minimization} \label{App:Subsec:ERM}

In this part we link the system of equations in \cref{Result2} to the recent results of \citet{erba2025nuclear}. In this work, the authors derive the statistics of the global minimizer of the same regularized loss as ours, in the case $\kappa \geq 1$. 

In the following, we show that our system of equations is the same as theirs after matching our conventions. This result is no surprise for the following reason: it is known that when optimizing a function of $WW^\top$ (as in our case) in the setting where $m \geq d$, the gradient flow always converges to a global minimizer of the loss over the PSD matrices (see for instance \citet{bach2024learning}, Section 12.3.3).

\subsubsection{Correspondence of the Equations} \label{App:Subsubsec:MatchingERM}

The first step when comparing both setups is to match the constants used. By matching the expressions of our respective loss functions, we arrive at the expression of their regularization parameter:
\begin{equation*}
    \lambda_{\text{ERM}} = \frac{4\alpha}{\sqrt{\kappa}} \lambda. 
\end{equation*}
In addition, their results involve the free additive convolution between the teacher spectral distribution and a semicircular density with radius $2\delta$, i.e., with variance $\delta^2$. Then, with our notations, their set of self-consistent equations (Theorem 1) is given by:
\begin{align}
    4 \alpha \delta - \frac{\delta}{\epsilon} &= 2 \delta \, \partial_1 J(\delta^2, 4 \alpha \lambda \epsilon), \label{eq:EqsVittorio1} \\
    Q_* + \frac{\Delta}{2} + 2 \alpha \delta^2 - \frac{\delta^2}{\epsilon} &= \big( 1 - 4 \alpha \lambda \epsilon \partial_2 \big) J(\delta^2, 4 \alpha \lambda \epsilon), \label{eq:EqsVittorio2}
\end{align}
where the unknowns are $\delta, \epsilon$, and:
\begin{equation*}
    J(a, b) = \int_b (x-b)^2 \d \mu_a(x) ,
\end{equation*}
and $\mu_a$ corresponds to the asymptotic spectral density of the matrix $Z^* + \sqrt{a} \G$ with $\G \sim \mathrm{GOE}(d)$. We refer to \cref{App:Subsubsec:FreeConvolution} for more details.

\paragraph{Derivatives of $J$.} We shall now compute the partial derivatives of the function $J$. Let us start with the derivative with respect to $b$. Since the map $b \mapsto (x-b)^2 \1_{x \geq b}$ is $\mathcal{C}^1$, we have, interchanging integration and differentiation:
\begin{equation} \label{eq:DerivativebJ}
    \partial_2 J(a,b) = -2 \int_b (x-b) \d \mu_a(x). 
\end{equation}
Regarding the derivative with respect to $a$, we start from the complex Burgers' equation satisfied by $m_a$, the Stieltjes transform of $\mu_a$ (see \cref{lemma:BurgerStieltjes}):
\begin{equation*}
    \partial_a m_a(z) + m_a(z) \partial_z m_a(z) = 0,
\end{equation*}
Evaluating at $z = x + i \eta$ and taking imaginary parts while $\eta \to 0$, one gets the continuity equation using \cref{Def:HilbertStieltjes}:
\begin{equation*}
    \partial_a \rho_a(x) + \partial_x \big( h_a(x) \rho_a(x) \big) = 0.
\end{equation*}
Note that this equation is only verified in the sense of distributions. In the following, we use this equation non-rigorously and differentiate under the integral:
\begin{equation*}
\begin{aligned}
    \partial_1 J(a,b) &= \int_b (x-b)^2 \partial_a \rho_a(x) \d x \\
    &= - \int_b (x-b)^2 \partial_x \big( h_a(x) \rho_a(x) \big) \d x. 
\end{aligned}
\end{equation*}
Then integrating by parts, we finally get:
\begin{equation} \label{eq:DerivativeaJ}
    \partial_1 J(a,b) = 2 \int_b (x-b) h_a(x) \d \mu_a(x). 
\end{equation}

\paragraph{Equivalence of the systems of equations.} We now go back to the system of equations~\eqref{eq:EqsVittorio1} and \eqref{eq:EqsVittorio2}. We set $\xi = \delta^2$ and $q = 4 \alpha \lambda \epsilon$, and using the derivatives in equations~\eqref{eq:DerivativebJ} and~\eqref{eq:DerivativeaJ}, we obtain that $q, \xi$ solve the equations:
\begin{align}
    1 - \frac{\lambda}{q} &= \frac{1}{\alpha}  \int_q (x-q) h_\xi(x) \d \mu_\xi(x), \label{eq:Systemvittorio1}\\
    Q_* + \frac{\Delta}{2} + 2 \alpha \xi - \frac{4 \alpha \lambda \xi}{q} &= \int_q \big( x^2 - q^2 \big) \d \mu_\xi(x). \label{eq:Systemvittorio2}
\end{align}
Already the first equation is the same as \eqref{eq:ResultLongTimesIntegral} in \cref{Result2}, in the case where $\kappa \geq 1$. In addition, equation \eqref{eq:ResultLongTimesMSE} is directly obtained by replacing $\lambda / q$ in equation~\eqref{eq:Systemvittorio2} using equation \eqref{eq:Systemvittorio1}.

Let us now show that this also leads to the same expression of the MSE and loss. In our case, recall the expressions:
\begin{equation*}
    \mathrm{MSE} = \frac{1}{d} \big\| WW^\top - Z^* \big\|_F^2, \hspace{1.5cm} \mathrm{Loss}_\mathrm{train} = \frac{1}{4n} \sum_{k=1}^n \big( \tr(X_kZ) - z_k \big)^2. 
\end{equation*}
Now, \citet{erba2025nuclear} studied the test error and the loss value, that we respectively denote $e_\text{test}, L$. Taking into account their conventions, we reach the relationships with ours quantities:
\begin{equation} \label{eq:CorrespondenceMSELoss}
    e_\text{test} = \mathrm{MSE}, \hspace{1.5cm} L = 4 \alpha \, \mathrm{Loss}_\mathrm{train} + 4 \alpha \lambda \lim_{d \to \infty} \frac{1}{d} \|W_\infty\|_F^2. 
\end{equation}
Back to the correspondence between our variables $\xi, q$ and their variables $\delta, \epsilon$, their expression of $e_\text{test}$ directly leads to the MSE equation \eqref{eq:LossValue}. Regarding their loss, we use the expression of the derivative of $J$ in equation \eqref{eq:DerivativebJ} and reach the expression of their loss:
\begin{equation*}
    L = \frac{4 \alpha^2 \xi \lambda^2}{q^2} + 4 \alpha \lambda \int_q (x-q) \d \mu_\xi(x). 
\end{equation*}
Now, since we have:
\begin{equation*}
    \int_q (x-q) \d \mu_\xi(x) = \lim_{d \to \infty} \frac{1}{d} \|W_\infty\|_F^2, 
\end{equation*}
we indeed recover the expression of our loss in equation \eqref{eq:LossValue} using equation~\eqref{eq:CorrespondenceMSELoss}.

\subsubsection{Stability Condition}

In the same work, the authors derive a stability condition for the previous set of equations. This criterion is derived from an approximate message passing (AMP) iteration and reads, with our notations:
\begin{equation} \label{eq:StabilityERM}
    \iint \left( \frac{(x-q)^+ - (y-q)^+}{x-y} \d \mu_\xi(x) \d \mu_\xi(y) \right)^2 < 2 \alpha. 
\end{equation}
Interestingly, this criterion matches the one we later derive in \cref{App:Subsubsec:SusceptibilityNorm}, up to the fact that we only require the above quantity to be finite. We now show, using our system of equations, that the above criterion is verified. Recall that when $\kappa \geq 1$, as a consequence of \cref{Lemma:Hilbert}, $q, \xi$ are linked through the equation:
\begin{equation*}
    1  - \frac{\lambda}{q} = \frac{1}{2\alpha} \iint \frac{(x-q)^+ - (y-q)^+}{x-y} \d \mu_\xi(x) \d \mu_\xi(y).
\end{equation*}
Now, since the map $x \mapsto (x-q)^+$ is continuous and 1-Lipschitz, we have the bound, for all $x \neq y$:
\begin{equation*}
    \left( \frac{(x-q)^+ - (y-q)^+}{x-y} \right)^2 \leq \frac{(x-q)^+ - (y-q)^+}{x-y}. 
\end{equation*}
Putting everything together, and since $\lambda > 0$, the criterion \eqref{eq:StabilityERM} is satisfied. Through the work of \citet{erba2025nuclear} on the AMP iteration associated with the same problem as ours, the direct verification of the stability criterion allows to reinforce the validity of the simplifications regarding the dynamics introduced in \cref{Subsec:LongTimes} (at least in the case $\kappa \geq 1$). 

\subsection{Analysis of the Long-Time Equations} \label{App:Subsec:AnalysisEquations}

In this part we analyze the system of equations given in \cref{Result2}. In the following, we give the results:
\begin{itemize}
    \item We derive the existence of the two regions with respect to the variable $\kappa$ claimed in \cref{Subsubsec:Regions}. 
    \item We confirm that in the overparameterized region $\kappa \geq \kappa_{\min}$, the minimum reached by gradient flow is a global minimizer of the loss. 
\end{itemize}

\subsubsection{Overparameterized Region} \label{App:Subsubsec:Regions}

We now prove the claim made in \cref{Subsubsec:Regions} regarding the two regimes depending on the value of $\kappa$. Indeed, as we claimed, there exists a value $\kappa_{\min}$, depending on $\alpha, \lambda, \kappa^*, \Delta$, such that the set of equations is independent of $\kappa$ for $\kappa > \kappa_{\min}$. We recall the system of equations~\eqref{eq:SystemResult2}:
\begin{subequations} \label{eq:SystemResult2Proof}
\begin{align}
    \min(\kappa, 1) &= \int_\omega \d \mu_\xi(x), \label{eq:SystemProof0}\\
    1 &= \frac{\lambda}{q} + \frac{1}{\alpha} \int_{\max(q, \omega)} (x-q) h_\xi(x) \d \mu_\xi(x), \label{eq:SystemProof1}\\
    2\alpha \xi - \frac{\Delta}{2} &= Q_* + \int_{\max(q,\omega)} (q^2 - x^2) \d \mu_\xi(x) + 4 \xi \int_{\max(q,\omega)} (x-q) h_\xi(x) \d \mu_\xi(x). \label{eq:SystemProof2} 
\end{align}    
\end{subequations}
Remark that these equations only depend on $\kappa$ through the variable $\omega$ that selects a mass $\kappa$ of the measure $\mu_\xi$. Also remark that for a triplet $(q, \xi, \omega)$ solution of the system, having $q \geq \omega$ makes the last two equations independent of $\omega$, and therefore $\kappa$. In this case only the pair $(q, \xi)$ matters to solve the system of equations.
\begin{lemma}
    For a fixed set of parameters $\alpha, \lambda, \kappa^* > 0$ and $\Delta \geq 0$, consider a pair $(q^*, \xi^*)$ solution of the system of equations \eqref{eq:SystemResult2Proof} for $\kappa = 1$ (in this case one can choose $\omega = -\infty$). Let:
    \begin{equation*}
        \kappa_{\min} = \int_{q^*} \d \mu_{\xi^*}(x), 
    \end{equation*}
    then, for $\kappa \geq \kappa_{\min}$, one can choose $\omega \leq q^*$ so that the triplet $(q^*, \xi^*, \omega)$ is solution of the system of equations \eqref{eq:SystemResult2Proof}. 
\end{lemma}
This lemma shows that for $\kappa$ larger than $\kappa_{\min}$ we can pick the same solution $q^*, \xi^*$ as when solving with $\kappa = 1$. Provided that for given values of our parameters $\alpha, \lambda, \kappa, \kappa^*, \Delta$, there is a unique solution in terms of the variables $(q, \xi)$, this guarantees that for $\kappa \geq \kappa_{\min}$, the solution of the system \eqref{eq:SystemResult2} does not depend on $\kappa$. 
\begin{proof}
    The key point is that the system of equations~\eqref{eq:SystemResult2Proof} only depends on $\kappa$ through the variable $\omega$. Let us consider these equations for some $\kappa \geq \kappa_{\min}$. We consider $\omega$ to be solution of:
    \begin{equation*}
        \min(\kappa, 1) = \int_{\omega} \d \mu_{\xi^*}(x).
    \end{equation*}
    Remark that we used $\xi^*$ to define $\omega$. Since $\min(\kappa, 1) \geq \kappa_{\min}$ and due to the definition of $\kappa_{\min}$ we immediately get that $\omega \leq q^*$. We now plug the triplet $(q^*, \xi^*, \omega)$ into the system of equations \eqref{eq:SystemResult2Proof}, and remark that:
    \begin{itemize}
        \item Equation \eqref{eq:SystemProof0} is verified due to the definition of $\omega$. 
        \item Since $\omega \leq q^*$, equations \eqref{eq:SystemProof1}, \eqref{eq:SystemProof2} are exactly the same as in the $\kappa = 1$ case. Therefore they are solved by $(\xi^*, q^*)$. 
    \end{itemize}
    As a conclusion, as soon as $\kappa \geq \kappa_{\min}$, we can find some value of $\omega$ such that $(q^*, \xi^*, \omega)$ indeed solves the system \eqref{eq:SystemResult2Proof}.     
\end{proof}

\subsubsection{Global Minimizer in the Overparameterized Region}

We now briefly explain why in the region $\kappa \geq \kappa_{\min}$, the gradient flow estimator corresponds to a global minimizer of the loss over all PSD matrices. The first evidence is the calculation of \cref{App:Subsubsec:MatchingERM} where it was proved earlier that the system of equations in this region matches the one of \citet{erba2025nuclear}, who worked in the empirical risk minimization setting.

The deeper reason behind this correspondence can be clarified by the following fact, which we do not prove in detail:
\begin{proposition}
    Let $L \colon \R^{d \times m} \to \R$ be such that $L(W) = G(WW^\top)$ for $G \colon \mathcal{S}_d(\R) \to \R$ convex and real analytic. Let $\big( W(t) \big)_{t \geq 0}$ to be solution of the gradient flow:
    \begin{equation*}
        \dot W(t) = - \nabla L \big( W(t) \big).
    \end{equation*}
    Then, if $m \geq d$, for almost all initializations, $W(t)$ converges to a point $W_\infty$ such that $W_\infty W_\infty^\top$ is a global minimizer of $G$ over $\mathcal{S}_d^+(\R)$. 
\end{proposition}
The proof of this can be carried out as follows:
\begin{itemize}
    \item Standard results on the convergence of gradient flow to local minimizers \citep[see for instance][]{lee2016gradient, panageas2016gradient} guarantee that for almost all initializations,
    $W(t)$ converges to a local minimizer of $L$.
    \item The convexity of $G$ allows to conclude that any local minimizer of $L$ translates into a global minimizer of $G$ over the set of PSD matrices. For instance, this result can be found in \citet[][Exercise 12.8]{bach2024learning}.
\end{itemize}
Finally, in our case of interest in \cref{Subsubsec:Regions}, the loss is indeed expressed as a convex function of $WW^\top$:
\begin{equation*}
    \L(W) = \frac{1}{4n} \sum_{k=1}^n \Big( \tr \big(X_k WW^\top \big) - z_k \Big)^2+ \frac{\lambda}{d} \tr \big( WW^\top \big). 
\end{equation*}
All of these ingredients allow us to identify an overparameterized region, where the set of equations does not depend on $\kappa$ anymore, and where the gradient flow converges to a point that minimizes the loss over all PSD matrices. As mentioned earlier, this minimizer has rank $\sim \kappa_{\min} d$, and for $\kappa < \kappa_{\min}$, the flow is unable to converge to such a point, due to the rank constraint.

\subsection{Population Limit} \label{App:Subsec:Population}

In this section, we derive the population limit of the dynamics~\eqref{eq:GFdynamics}, corresponding to the regime where the student has access to an infinite number of observations. In line with our main results in this paper, we study this limit under \cref{Assumption2}, when using the quadratic cost for the loss and generating the labels with a noisy Gaussian channel. More precisely, we derive the following results:
\begin{itemize}
    \item We start by taking the $n \to \infty$ limit (at fixed dimension) in the expression of the loss~\eqref{eq:Loss} and study the associated dynamics. 
    \item Then, we take the limit $\alpha \to \infty$ (corresponding to the regime $n \gg d^2$) in the system of equations of \cref{Result1} and show that we obtain the same dynamics as in the previous step. 
    \item Finally, we take the $\alpha \to \infty$ limit in the system of equations at long times derived in \cref{Result2}, and show that the equations we obtained are coherent with the previous calculations. This provides a proof of \cref{Prop:PopulationLimit}. 
\end{itemize}

Studying this limit from different angles leads to the following conclusions: first of all, taking the sequential limit $n \to \infty$ and then $d \to \infty$ leads to the same result as taking the joint limit $n \sim \alpha d^2$ and then sending $\alpha \to \infty$. This means that there is no intermediate scaling of $n$ with the dimension that produces a different dynamics. Secondly, the same conclusion holds between the $\alpha \to \infty$ limit and the long times, meaning that the population limit preserves the relevant timescale for the dynamics. 

\subsubsection{Sequential Limit}

Let us start with the expression of the loss~\eqref{eq:Loss} under \cref{Assumption2}:
\begin{equation} \label{eq:PopulationLoss}
    \L(W) = \frac{1}{4n} \sum_{k=1}^n \Big( \tr\big( X_k WW^\top \big) - \tr \big( X_k Z^* \big) - \sqrt{\Delta} \zeta_k \Big)^2,
\end{equation}
where $\zeta_1, \dots, \zeta_n$ are i.i.d. standard Gaussian variables, independent from all the other random variables of the problem. Taking the $n \to \infty$ limit allows to replace the empirical average over the $n$ samples by the expectation over their distribution. Since $X_k \sim \mathrm{GOE}(d)$, we obtain that:
\begin{equation} \label{eq:PopulationLoss1}
    \L_\text{pop}(W) = \frac{1}{2d} \big\| WW^\top - Z^* \big\|_F^2 + \frac{\Delta}{4}. 
\end{equation}
In this case the loss is directly related to the distance between the teacher and the student: in the population limit, we are simply optimizing the MSE. In this case, we can write the Langevin dynamics \eqref{eq:GFdynamics} with regularization $\Omega$ and inverse temperature $\beta$ as:
\begin{equation} \label{eq:DynamicsPopulationW1}
    \d W(t) = 2 \Big( Z^* - W(t) W(t)^\top \Big) W(t) \d t - \nabla \Omega \big( W(t) \big) \d t + \frac{1}{\sqrt{\beta d}} \d B(t). 
\end{equation}
Already remark that the first term is very similar to the Oja flow dynamics studied in \cref{Sec:OjaFlow}. 

\subsubsection{Limit of the Dynamical Equations}

We now consider the $\alpha \to \infty$ limit in the system of equations in \cref{Result1}, and show that it corresponds to the dynamics \eqref{eq:DynamicsPopulationW1}. To see this, first note that the evolution of $W(t)$ and the typical label in equations \eqref{eq:DMFT1W}, \eqref{eq:DMFT1Y} do not explicitly depend on $\alpha$, and this dependence enters only through the covariance of the Gaussian process $\G$ and the response $R$ that drive their evolution. We recall that:
\begin{align}
    \E \, \G_{ij}(t) \G_{i'j'}(t') &= \frac{1}{2 \alpha d} \big( \delta_{ii'} \delta_{jj'} + \delta_{ij'} \delta_{i'j} \big) \left(\frac{1}{d} \E \, \tr \Big[ \big(Z(t) - Z^* \big) \big(Z(t') - Z^* \big) \Big] + \frac{\Delta}{2} \right), \label{eq:PopulationCov1} \\
    R(t,t') &= \delta(t-t') - \frac{1}{\alpha d^2} \tr\left( \left. \frac{\partial \, \E \, Z(t)}{\partial H(t')} \right|_{H = 0} \right). \label{eq:PopulationResponse} 
\end{align}
Now, as $\alpha \to \infty$, the covariance function of $Z(t)$ and its average response should remain of order one, so that from equations \eqref{eq:PopulationCov1}, \eqref{eq:PopulationResponse} we get $\G(t) = 0$ almost surely, and $R(t,t') = \delta(t-t')$. Therefore, from equation \eqref{eq:DMFT1W}, we get that $W$ exactly solves equation \eqref{eq:DynamicsPopulationW1}. This means we have recovered the sequential limit by taking the $\alpha \to \infty$ limit. Regarding the evolution of the label, we simply get the expression from equation \eqref{eq:DMFT1Y}:
\begin{equation} \label{eq:LabelPopulation}
    y(t) = \frac{m_Z(t)}{Q_*} y^* + \xi(t), \hspace{1.5cm} y^* \sim \N(0, 2Q_*),
\end{equation}
where $\xi$ is a centered Gaussian process with covariance given in equation \eqref{eq:Covariance1xi}. We will now show that this evolution of the labels corresponds to a random Gaussian projection of the student. Indeed, consider $X \sim \mathrm{GOE}(d)$ independent of all other random variables. Then, conditionally on $Z(t), Z^*$, the random projections $\tilde y(t) = \tr \big( XZ(t) \big)$ and $\tilde y^* = \tr \big( XZ^* \big)$ are Gaussian with zero mean and statistics:
\begin{equation*}
    \E \, \tilde y(t) \tilde y(t') = \frac{2}{d} \tr \big (Z(t) Z(t') \big), \hspace{1.2cm} \E \, \tilde y(t) \tilde y^* = \frac{2}{d} \tr \big( Z(t) Z^* \big), \hspace{1.2cm} \E \, \tilde y^{*2} = \frac{2}{d} \tr \big( Z^{*2} \big). 
\end{equation*}
Now, due to the covariance of $\xi(t)$ in equation \eqref{eq:Covariance1xi}, the couple $y(t), y^*$ in equation~\eqref{eq:LabelPopulation} has precisely the same statistics as $\tilde y(t), \tilde y^*$. This conclusion is intuitive: with a finite number of observations, the student remains correlated with the training examples, leading to the evolution of the typical label in \cref{Result1}. In the population limit, the student becomes independent of the examples, and the evolution of the typical label is the same as for a label associated with a previously unseen sample.

\subsubsection{Long-Time Analysis of the Population Equations}

Let us now go back to equation \eqref{eq:DynamicsPopulationW1}, and similarly to what was done in \cref{Subsec:LongTimes}, consider the gradient flow setting ($\beta = \infty$), and the $\ell_2$-regularization $\Omega(W) = \lambda \tr(WW^\top)$. Then, the population dynamics in equation \eqref{eq:DynamicsPopulationW1} writes:
\begin{equation*}
    \dot W(t) = 2 \Big( Z^* - W(t)W(t)^\top \Big) W(t) - 2 \lambda W(t). 
\end{equation*}
This now precisely corresponds to an Oja flow (see \cref{Sec:OjaFlow}) with the target matrix $Z^* - \lambda I_d$. Under a random initialization, we can apply \cref{prop:OjaConvergence} to get:
\begin{equation*}
    W(t) W(t)^\top \xrightarrow[t \to \infty]{} \big( Z^* - \lambda I_d \big)_{(m)}^+. 
\end{equation*}
Recall that the operator $A \mapsto A_{(m)}^+$ selects the $m$ largest positive eigenvalues (see \cref{Def:EckartYoung}). In the case $\kappa \geq \min(\kappa^*, 1)$, this simply writes $\big(Z^* - \lambda I_d \big)^+$, since $Z^* - \lambda I_d$ cannot have more than $m$ positive eigenvalues. Then, denoting $\mu_1, \dots, \mu_d$ the eigenvalues of $Z^*$, the MSE in the high-dimensional limit writes:
\begin{equation*}
\begin{aligned}
    \mathrm{MSE} = \lim_{d \to \infty} \frac{1}{d} \sum_{k=1}^d \Big( (\mu_k - \lambda)^+ - \mu_k \Big)^2.
\end{aligned}
\end{equation*}
Using that $u^+ = \max(u, 0)$ and the convergence of the empirical spectral distribution of $Z^*$, we get:
\begin{equation*}
    \mathrm{MSE} = \lim_{d \to \infty} \frac{1}{d} \sum_{k=1}^d \min(\lambda, \mu_k)^2 = \int \min(\lambda, x)^2 \d \mu^*(x). 
\end{equation*}
This is precisely the expression of the MSE in \cref{Prop:PopulationLimit}. The expression of the loss is a simple consequence of equation~\eqref{eq:PopulationLoss1}. 

\subsubsection{Population Limit in the Long-Time Equations} \label{App:Subsubsec:PopulationLongTimes}

In the previous steps, we proved \cref{Prop:PopulationLimit} starting from the dynamical equations in the population limit, and then studied the long times. For completeness, we show that the same result holds when starting from the long-time result \cref{Result2} and then taking the $\alpha \to \infty$ limit. We recall the system:
\begin{equation*}
\begin{aligned}
    \min(\kappa, 1) &= \int_\omega \d \mu_\xi(x), \\
    1 &= \frac{\lambda}{q} + \frac{1}{\alpha} \int_{\max(q, \omega)} (x-q) h_\xi(x) \d \mu_\xi(x), \\
    2\alpha \xi - \frac{\Delta}{2} &= Q_* + \int_{\max(q,\omega)} (q^2 - x^2) \d \mu_\xi(x) + 4 \xi \int_{\max(q,\omega)} (x-q) h_\xi(x) \d \mu_\xi(x),
\end{aligned}
\end{equation*}
where $q, \xi$ are the unknowns. Moreover, we have:
\begin{equation} \label{eq:MSELossPopulation}
    \mathrm{MSE} = 2 \alpha \xi - \frac{\Delta}{2}, \hspace{1.5cm} \mathrm{Loss}_\mathrm{train} = \frac{\lambda^2 \alpha \xi}{q^2}, \hspace{1.5cm} Z_\infty = \Big( Z^* + \sqrt{\xi} \G - q I_d \Big)_{(m)}^+. 
\end{equation}
Now, as $\alpha \to \infty$, the finiteness of the MSE and the loss require that $\xi = \Theta(\alpha^{-1})$ and $q = \Theta(1)$. Therefore, the quantity:
\begin{equation*}
    \int_{\max(q, \omega)} (x-q) h_\xi(x) \d \mu_\xi(x),
\end{equation*}
remains of order one. Since $\alpha \to \infty$ and $\xi \to 0$, one gets the equations:
\begin{equation*}
    q = \lambda, \hspace{1.5cm} 2 \lim_{\substack{\alpha \to \infty \\ \xi \to 0}} \alpha \xi= \frac{\Delta}{2} + Q_* + \lim_{\xi \to 0} \int_{\max(q,\omega)} (q^2 - x^2) \d \mu_\xi(x). 
\end{equation*}
This means that from equation \eqref{eq:MSELossPopulation}, we again obtain $Z_\infty = \big( Z^* - \lambda I_d \big)_{(m)}^+$, and we can drop the $m$ largest eigenvalues selection when $\kappa \geq \min(\kappa^*, 1)$. As a consequence, we can pick $\omega = 0$ in the previous equations. Let us now compute the $\xi \to 0$ limit. In \cref{App:Subsec:IntegralAsymptotics}, we compute the small $\xi$ asymptotics of several integrals involving $\mu_\xi$. As a consequence of \cref{Lemma:SmallXiFunctions}, one can show that with $\omega = 0$ and $q = \lambda$, we have at leading order:
\begin{equation*}
\begin{aligned}
    \lim_{\xi \to 0} \int_{\max(q,\omega)} (q^2 - x^2) \d \mu_\xi(x) &= \lim_{\xi \to 0} \int_\lambda \big(\lambda^2 - x^2 \big) \d \mu_\xi(x) \\
    &= \int_\lambda \big(\lambda^2 - x^2 \big) \d \mu^*(x). 
\end{aligned}
\end{equation*}
Now recalling that:
\begin{equation*}
    Q_* = \int x^2 \d \mu^*(x),
\end{equation*}
one finally has the expression, using equation~\eqref{eq:MSELossPopulation}:
\begin{equation*}
\begin{aligned}
    \mathrm{MSE} &= \int^\lambda x^2 \d \mu^*(x) + \lambda^2 \int_\lambda \d \mu^*(x) \\
    &= \int \min(x, \lambda)^2 \d \mu^*(x). 
\end{aligned}
\end{equation*}
Finally, the expression of the loss can be also deduced from equation~\eqref{eq:MSELossPopulation} since $\mathrm{Loss}_\mathrm{train} = \alpha \xi$ and $\mathrm{MSE} = 2 \alpha \xi - \Delta / 2$. Again, this leads to the result of \cref{Prop:PopulationLimit}. As already mentioned, the fact that we recover the same equations starting from the dynamical equations and the long-time one reveal that both the population ($\alpha \to \infty$) and long-time ($t \to \infty$) limits commute and that the population limit does not introduce a new timescale for the dynamics. 

\section{Stability of the Steady-State Solution} \label{App:Stability}
We devote this section to the analysis of the stability of the steady-state equations presented in \cref{Subsubsec:Stability}. More precisely, we consider a perturbation of the simplified dynamics introduced in \cref{Subsubsec:SteadyState}, and investigate if this perturbation may grow and lead to instability. The plan goes as follows:
\begin{itemize}
    \item In \cref{App:Subsec:Susceptibility}, we first introduce the susceptibility operator associated with the long-time steady-state solution and compute its Frobenius norm. We identify a regime where this quantity slowly diverges as the dimension grows, leading to a potential instability, but only in the high-dimensional limit.
    \item In \cref{App:Subsec:CVRates}, we then corroborate our previous observations with a study of the convergence rates associated with the steady-state approximation, in both finite and infinite dimension. We identify some weak directions associated with a vanishing curvature of the Hessian in high dimension. 
    \item In \cref{App:Subsec:DynamicalStability}, we give the steps necessary for a thorough analysis of the dynamical stability. This could be, in principle, achieved by linearizing the true dynamical equations of \cref{Result1} around the steady-state solution. This would lead to an exact stability analysis of the high-dimensional dynamical system.
\end{itemize}

\subsection{Susceptibility Operator} \label{App:Subsec:Susceptibility}

We now consider the susceptibility operator associated with the long-time limit of the dynamics:
\begin{equation*}
    \mathcal{X} = \left. \frac{\partial Z_\infty}{\partial H} \right|_{H = 0}. 
\end{equation*}
This operator can be viewed as the differential of the map that returns, for a given $H \in \mathcal{S}_d(\R)$, the limit $Z_\infty$ of the approximated dynamics with a perturbation $H$. Following the results of \cref{App:Subsubsec:LongTimesLimitDynamics}:
\begin{equation*}
    Z_\infty = \Big( Z^* + \sqrt{\xi} \G - q I_d + H \Big)_{(m)}^+.
\end{equation*}
The susceptibility operator encodes how the system responds to perturbations. We can either analyze its spectrum to investigate the stability of each direction of the system, or its normalized Frobenius norm that gives access to a measure of the averaged susceptibility over all directions. Then, the stability criterion we consider is the finiteness of the normalized Frobenius norm. 

To compute the spectrum and the Frobenius norm of the susceptibility, one can apply the result of \cref{Lemma:DerivativeSpectral}. To do so, denote $\lambda_1, \dots, \lambda_d$ the eigenvalues of $Z^* + \sqrt{\xi} \G$. Then, the spectrum of $\mathcal{X}$, viewed as a linear map $\mathcal{S}_d(\R) \to \mathcal{S}_d(\R)$ is given by:
\begin{equation} \label{eq:SepctrumSusceptibility}
    \mathrm{Sp}(\mathcal{X}) = \left\{ \frac{\phi(\lambda_i) - \phi(\lambda_j)}{\lambda_i - \lambda_j}, 1 \leq i \leq j \leq d \right\},
\end{equation}
with:
\begin{equation} \label{eq:PhiSusceptibility}
    \phi(x) = (x-q) \1_{x \geq \max(q, \omega)},
\end{equation}
and $\omega$ is a threshold selecting the $m$ largest eigenvalues of $Z^* + \sqrt{\xi} \G$:
\begin{equation*}
    \min(m, d) = \sum_{k=1}^d \1_{\lambda_k \geq \omega}. 
\end{equation*}
For $i = j$, the corresponding eigenvalue of $\mathcal{X}$ is understood as $\phi'(\lambda_i) = \1_{\lambda_i \geq \max(q, \omega)}$. Note that \cref{Lemma:DerivativeSpectral} only applies if the matrix of interest has simple and non-zero eigenvalues. In our case this applies with probability one under the randomness of the Gaussian matrix $\G$, whenever $\xi > 0$. 

\subsubsection{Norm of the Susceptibility} \label{App:Subsubsec:SusceptibilityNorm}

We start by investigating the Frobenius norm of this susceptibility operator. As a consequence of \cref{Lemma:DerivativeSpectral}, we have:
\begin{equation} \label{eq:SusceptibilityFinited}
    \mathcal{R}_d = \frac{1}{d^2} \big\| \mathcal{X} \big\|_F^2 = \frac{1}{d^2} \sum_{1 \leq i \leq j \leq d} \left( \frac{\phi(\lambda_i) - \phi(\lambda_j)}{\lambda_i - \lambda_j} \right)^2. 
\end{equation}
In finite dimension, this quantity is almost surely finite. Let us study it in the high-dimensional limit. We have the deterministic limit:
\begin{equation*}
    \mathcal{R}_d \xrightarrow[d \to \infty]{} \frac{1}{2} \iint \left( \frac{(x-q) \1_{x \geq \max(q, \omega)} - (y-q) \1_{y \geq \max(q, \omega)}}{x-y} \right)^2 \d \mu_\xi(x) \d \mu_\xi(y) \equiv \mathcal{R}_\infty. 
\end{equation*}
Again, $\omega$ selects the $m$ largest eigenvalues:
\begin{equation*}
    \min(\kappa, 1) = \int_{\omega} \d \mu_\xi(x). 
\end{equation*}
Let us now consider two cases: 
\begin{itemize}
    \item When $q \geq \omega$, the matrix $Z^* + \sqrt{\xi} \G$ has less than $m$ eigenvalues larger than $q$. In this case, the function $\phi$ in equation \eqref{eq:PhiSusceptibility} is continuous and 1-Lipschitz, which guarantees the finiteness of $\mathcal{R}_\infty$. 
    \item On the other hand, if $q < \omega$, $\phi$ is discontinuous at $x = \omega$. Near $\omega$, it jumps from 0 to $\omega - q > 0$. To investigate the behavior of $\mathcal{R}_\infty$, we focus near $x = y = \omega$, where the jump happens. We have, changing variables $x = \omega - u$ and $y = \omega + v$:
    \begin{equation*}
        \mathcal{R}_\infty \geq \frac{1}{2} \iint \rho_\xi(\omega - u) \rho_\xi(\omega + v) \1_{u \geq 0} \1_{v \geq 0} \left( \frac{\omega + v - q}{u+v} \right)^2 \d u \d v. 
    \end{equation*}
    $\rho_\xi$ denotes the density of $\mu_\xi$. Near the points $u, v = 0$, the integrand is proportional to $(u+v)^{-2}$ which is not integrable with respect to $u,v$. Therefore, as soon as $\mu_\xi$ has positive mass on both sides of $\omega$, the quantity $\mathcal{R}_\infty$ is infinite. Finally, we claim that this situation is generic: unless the density $\rho_\xi$ splits into two parts of mass $\kappa$ and $1 - \kappa$, $\rho_\xi$ will have positive mass on both sides of $\omega$, leading to the divergence of $\mathcal{R}_\infty$. 
\end{itemize}

Therefore, it is natural to understand the typical value of $\mathcal{R}_d$ as $d \to \infty$. To do so, we start from equation \eqref{eq:SusceptibilityFinited} and focus on the eigenvalues pairs $(\lambda_i, \lambda_j)$ where the divergence happens. This corresponds to eigenvalues close to the threshold $\omega$ on each side of it. At this point, the spacing between two eigenvalues is of order $\epsilon_d \propto d^{-1}$. Taking the continuous approximation with the cutoff $\epsilon_d$, we get that:
\begin{equation*}
\begin{aligned}
    \mathcal{R}_d &\underset{d \to \infty}{\sim} \frac{1}{2} \rho_\xi(\omega)^2 (\omega-q)^2 \iint \1_{u \geq 0} \1_{v \geq 0} \1_{|u-v| \geq \epsilon_d} \frac{\d u \d v}{(u+v)^2} \\
    &\underset{d \to \infty}{\sim} \frac{1}{2} \rho_\xi(\omega)^2 (\omega-q)^2 \log d. 
\end{aligned}
\end{equation*}
Therefore, in the case $q < \omega$, the typical value of $\mathcal{R}_d$ is of order $\log d$ as $d \to \infty$, corresponding to a mild divergence. This invites us to study the spectrum of the susceptibility in order to characterize the directions of potential instability. 

Finally, to relate this conclusion to \cref{Subsubsec:Regions,Subsubsec:Stability}, recall the definition of $\kappa_{\min}$ in equation \eqref{eq:kappamin}. Then, the region $q \geq \omega$ corresponds to the case where the gradient flow dynamics converges to a rank-deficient matrix. In \cref{Subsubsec:Regions}, we have precisely identified this region to be the overparameterized one, that is $\kappa \geq \kappa_{\min}$. 

\subsubsection{Spectrum of the Susceptibility}

Recall that the spectrum of the susceptibility $\mathcal{X}$ is given in equation \eqref{eq:SepctrumSusceptibility}. As discussed earlier, the unstable modes correspond to pairs of eigenvalues $(\lambda_i, \lambda_j)$ close and on each side of the cutoff $\omega$. For such pairs such that $\lambda_i < \omega < \lambda_j$ and $\lambda_j - \lambda_i \propto \epsilon$, the associated eigenvalue for the susceptibility operator is:
\begin{equation*}
    \frac{\lambda_j - q}{\lambda_j - \lambda_i} \propto \frac{1}{\epsilon}.
\end{equation*}
As the dimension increases, there are $\Theta(\epsilon^2 d^2)$ pairs with amplitude larger than $\epsilon^{-1}$. The most unstable directions correspond to the scaling $\epsilon \propto d^{-1}$ with a susceptibility eigenvalue proportional to $d$. However, these directions are only in a finite number. Combining all the scales from $d^{-1}$ to order 1 leads to the divergence proportional to $\log d$. 

\paragraph{Eigenvalue jump.} Recall that the gradient flow selects the $m$ largest positive eigenvalues of the target matrix $A = Z^* + \sqrt{\xi} \G - q I_d$. Letting $p$ be the number of positive eigenvalues of this matrix, we have the dichotomy:
\begin{itemize}
    \item In the stable case, when $m \geq p$, the student has enough rank to select all the positive eigenvalues. In this case any infinitesimal perturbation leads to a response of the same order, indicating stability. 
    \item For $m < p$, the rank constraint imposes that several order one eigenvalues are sent to zero. Therefore, for very close eigenvalues, a perturbation can reorder the eigenvalues and change those that are selected by the dynamics: the potential instability originates from the fact that some previously selected eigenvalues (of order one) become zero, and \emph{vice versa}. 
\end{itemize}

As it was made clear in \cref{App:Subsubsec:SusceptibilityNorm}, these two cases respectively correspond to the overparameterized ($\kappa \geq \kappa_{\min}$) and underparameterized ($\kappa < \kappa_{\min}$) regions identified in \cref{Subsubsec:Regions}. 

\subsubsection{A Refined Analysis}

Despite the previous observations, simulations suggest that the gradient descent algorithm (with small stepsize) always converges toward the solution given by \cref{Result2}, even in what we called the unstable region. As a first attempt to explain this, consider $Z_\infty$ to be perturbed by some matrix $H$, and the associated response:
\begin{equation*}
    \frac{1}{d} \big\| \delta Z_\infty \big\|_F^2 = \frac{1}{d} \sum_{1 \leq i \leq j \leq d} \left( \frac{\phi(\lambda_i) - \phi(\lambda_j)}{\lambda_i - \lambda_j} \right)^2 H_{ij}^2.
\end{equation*}
The key point here is that we do not ask the system to be stable under any perturbation, but specific perturbations originating from the dynamics of \cref{Result1}. For instance, the previous equation is written in the basis that diagonalizes the susceptibility $\mathcal{X}$, which is directly related to the one that diagonalizes the matrix $Z^* + \sqrt{\xi} \G$. Now, it is easily shown that if the perturbation $H$ is generic with respect to this basis, the previous quantity is related to the Frobenius norm of the susceptibility \eqref{eq:SusceptibilityFinited} and may diverge in the unstable region.

However, a perturbation originating from the dynamics of \cref{Result1} should necessarily be correlated with the simplified dynamics. Then, the stability of the system would be guaranteed provided that this correlated perturbation puts slightly less weight on the unstable modes. For instance, one could imagine a high-dimensional behavior:
\begin{equation*}
    \E \, H_{ij}^2 \approx \frac{\epsilon^2}{d} V(\lambda_i, \lambda_j), 
\end{equation*}
where $\epsilon$ is the magnitude of the perturbation. We would then get:
\begin{equation*}
    \frac{1}{d} \big\| \delta Z_\infty \big\|_F^2 \xrightarrow[d \to \infty]{}  \frac{\epsilon^2}{2} \iint V(x,y) \left( \frac{(x-q) \1_{x \geq \max(q, \omega)} - (y-q) \1_{y \geq \max(q, \omega)}}{x-y} \right)^2 \d \mu_\xi(x) \d \mu_\xi(y). 
\end{equation*}
Now, a mild decay of $V$ around the point $(\omega, \omega)$ would guarantee the finiteness of the integral, therefore the stability of the system. On the other side, it is also possible that $H$ puts more weight on the unstable modes and amplifies the perturbation. 

In line with our numerical observations, we conjecture that this specific perturbation tends to regularize the dynamical system, i.e., it attenuates the weak modes associated with a potential instability of the approximate dynamics.

\subsection{Convergence Rates for the Approximate Dynamics} \label{App:Subsec:CVRates}

In this part, we analyze the convergence rates of the approximate dynamics of \cref{Subsubsec:SteadyState}:
\begin{equation} \label{eq:ApproximateDynamics2}
    \dot W(t) = 2 r_\infty \Big(Z^* + \sqrt{\xi} \G - q I_d - W(t) W(t)^\top \Big) W(t).
\end{equation}
The aim is twofold:
\begin{itemize}
    \item Interpret the weak or unstable directions of the susceptibility operator in terms of convergence rates and landscape curvature. 
    \item Examine \cref{Assumption3} in light of the discussion in \cref{App:Subsec:SteadyState}, that relates the steady-state assumption with the fast convergence of the dynamics. 
\end{itemize}

We build on the results for the Oja flow dynamics derived in \cref{Subsec:OjaCVFinitedim} and \cref{Subsec:OjaCVInfinitedim} that derive its convergence rates in both finite and infinite dimension. To match with these results, we define:
\begin{equation} \label{eq:targetOja}
    A = Z^* + \sqrt{\xi} \G - q I_d,
\end{equation}
the target matrix of the approximate dynamics~\eqref{eq:ApproximateDynamics2}. As soon as $\xi > 0$, the presence of the GOE matrix $\G$ guarantees that $A$ is invertible and has simple eigenvalues with probability one. Moreover, we denote by $p$ the number of positive eigenvalues of $A$. Now, the previous section has shown that the stability of the approximate dynamics is guaranteed whenever $m \geq p$ (overparameterized regime), but it is still to be clarified in the small-rank regime $m < p$. 

Finally, note that the presence of the factor $2r_\infty$ in the dynamics acts as a time renormalization and only alters the convergence rates by a constant factor (independent of the dimension). For simplicity, we set this quantity to $1$ in the following. 

\subsubsection{Finite-Dimensional Rates} \label{App:Subsubsec:FiniteDimRates}

As shown in \cref{Prop:HessianOja} and \cref{Prop:CVRates2}, when the dimension remains finite, the Oja flow converges exponentially fast. Denoting $\lambda_1 > \dots > \lambda_d$ the ordered eigenvalues of $A$ in equation~\eqref{eq:targetOja}, the convergence rates of the dynamics \eqref{eq:ApproximateDynamics2} are given by:
\begin{equation*}
    \varrho_\text{CV} = \left\{ \begin{array}{cc}
        \min(2 \lambda_m, \lambda_m - \lambda_{m+1}), & \mathrm{if} \ m \leq p, \\
        \min(\lambda_p, |\lambda_{p+1}|), & \mathrm{if} \ m > p,
    \end{array} \right.
\end{equation*}
in the sense that for all $c < \varrho_\text{CV}$:
\begin{equation*}
    \big\| Z(t) - Z_\infty \big\|_F \underset{t \to \infty}{=} o \big(e^{-ct} \big).
\end{equation*}
Here $Z(t) = W(t)W(t)^\top$ and $Z_\infty$ is the limit of $Z(t)$ as $t \to \infty$. 

We believe that an exponentially fast convergence should validate \cref{Assumption3}. This result also matches with the conclusion of \cref{App:Subsec:Susceptibility} that the steady-state approximation remains stable in finite dimension, no matter the value of $\kappa$. 

It is interesting to notice that in both cases, these rates vanish with the dimension as soon as the asymptotic spectral distribution of $A$ has mass either near zero (when $m > p$) or near its $m^\text{th}$ eigenvalue $\lambda_m$ (when $m \leq p$). In both cases the convergence rates are of order $d^{-1}$ when $d$ grows large. This observation was verified for generic values of our parameters (for instance $\kappa > \kappa^*$) with the values of $q, \xi$ obtained from the numerical integration of the system \eqref{eq:SystemResult2}. In most cases, the spectral density of $A$ possesses non-zero mass around $\lambda_m$ for $m \leq p$ and $0$ for $m > p$. 

Therefore, there exists a few directions in the space $\R^{d \times m}$ associated with a slow relaxation time. In terms of landscape, these directions become flatter as the dimension increases. Interestingly, this phenomenon happens in both the stable and unstable cases derived in \cref{App:Subsec:Susceptibility}. As already mentioned, these two regions precisely correspond to the ones identified in \cref{Subsubsec:Regions,Subsubsec:Stability}. 

\subsubsection{Infinite-Dimensional Rates} \label{App:Subsubsec:InfiniteDimRates}

Following the previous observations in finite dimension, we also study the convergence rates of the Oja flow by first letting the dimension go to infinity. As a consequence of the presence of directions with vanishing curvature, our analysis leads to the conclusion of non-exponential rates. More precisely, in \cref{Prop:HighDimCVRates}, we show that, with the same notations as in the previous section:
\begin{equation} \label{eq:AsymptoticRatesSS}
    \lim_{t \to \infty} \lim_{d \to \infty} \frac{1}{d} \big\| Z(t) - Z_\infty \big\|_F^2 \underset{t \to \infty}{=} \left\{ \begin{array}{cc}
        \Theta \big(t^{-3} \big), & \mathrm{if} \,  \kappa > \kappa_A, \\
        \Theta \big(t^{-1} \big), & \mathrm{if} \, \kappa < \kappa_A,
    \end{array} \right.
\end{equation}
where:
\begin{equation*}
    \kappa_A = \int \1_{x > 0} \d \mu_A(x),
\end{equation*}
is the fraction of positive eigenvalues of $A$. With our expression for $A$, the threshold $\kappa_A$ coincides with $\kappa_{\min}$ identified in \cref{Subsec:LongTimes}. This value separates two regimes with distinct convergence properties. In comparison with our stability result, the region we identified as stable corresponds to the fastest convergence rates, namely $\kappa > \kappa_A$. The potentially unstable region, $\kappa < \kappa_A$ is characterized by the slowest convergence rates. 

When convergence is slow, perturbations decay over longer timescales and therefore may interact more strongly with unstable modes, leading to instability. On the other hand, faster convergence limits this effect by damping perturbations more rapidly. 

Interestingly, these rates can be compared to those obtained from the gradient descent dynamics~\eqref{eq:GDdynamics}. In \cref{fig:CVRates}, we plot the function:
\begin{equation} \label{eq:LossFull}
    \mathrm{Loss}_\lambda = \frac{1}{4n} \sum_{k=1}^n \Big( \tr \big( X_k WW^\top \big) - z_k \Big)^2 + \frac{\lambda}{d} \tr \big( WW^\top \big).
\end{equation}
Since the dynamics \eqref{eq:GFdynamics} is exactly the gradient flow associated with the loss function $\mathrm{Loss}_\lambda$, this quantity is known to decrease over time along the flow. \cref{fig:CVRates} is split into two panels, separating values of $\alpha$ corresponding to the stable region $\kappa > \kappa_{\min}$ and the potentially unstable region $\kappa < \kappa_{\min}$. The observed convergence rates are then compared with the power-law asymptotic $t \mapsto t^{-3}$. Interestingly, in both regions the convergence rates are very close to this asymptotic behavior. Therefore, in terms of convergence rates alone, there is no clear difference between the stable and unstable regions. Since we conjectured in \cref{App:Subsec:SteadyState} that the stability of the steady-state solution is directly related to the convergence rates of the dynamics, this numerical observation provides evidence supporting the stability of the underparameterized region $\kappa < \kappa_{\min}$. 

\begin{figure}[ht]
    \centering
    \includegraphics[width=\linewidth]{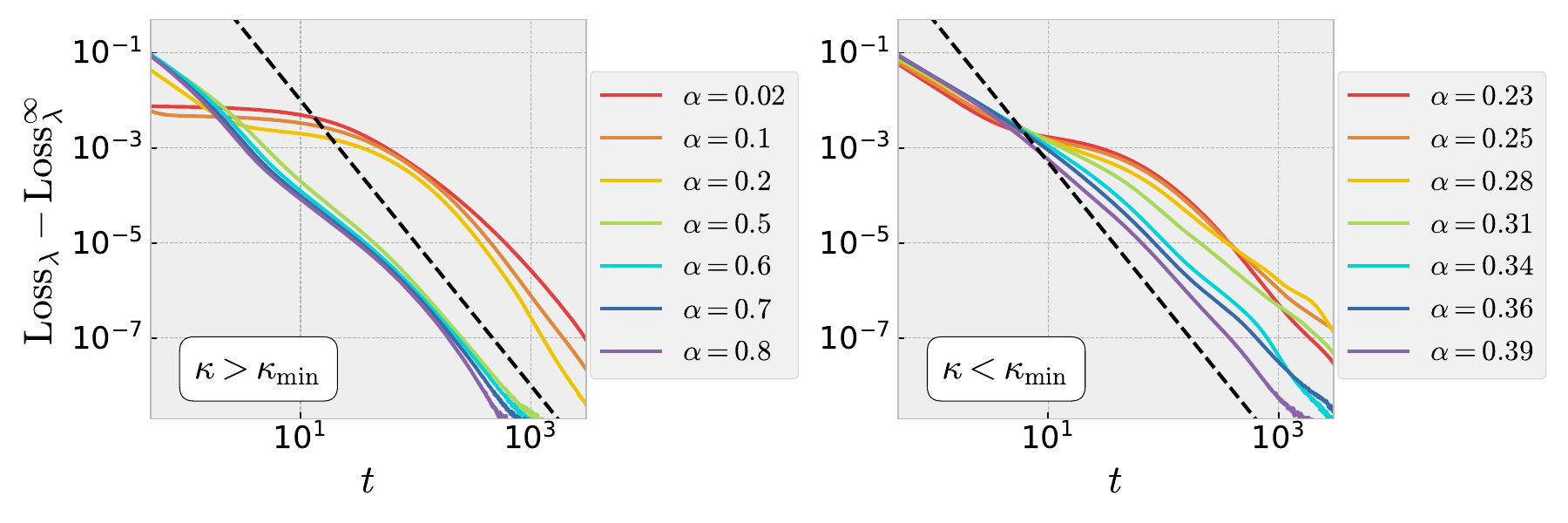}
    \vspace*{-0.7cm}
    \caption{The function $\mathrm{Loss}_\lambda$ (defined in equation \ref{eq:LossFull}), with limiting value subtracted, as a function of time $t$ for gradient descent trajectories (see equation \ref{eq:GDdynamics}). Parameters $\kappa = 0.4, \kappa^* = 0.3, \lambda = 0.01$, $\Delta = 0$. Values of $\alpha$ are chosen so that $\kappa > \kappa_{\min}$ on the left panel and $\kappa < \kappa_{\min}$ on the right panel. Black dashed line: the function $t \mapsto C t^{-3}$, where $C$ is chosen for visual comparison. Gradient descent simulations are averaged over 10 realizations of the initialization, teacher and data.}
    \label{fig:CVRates}
\end{figure}

As a technical remark, we emphasize that the cubic power-law behavior identified in \cref{fig:CVRates} is not directly related to the one in equation \eqref{eq:AsymptoticRatesSS}. Indeed, equation \eqref{eq:AsymptoticRatesSS} concerns the distance to convergence, whereas \cref{fig:CVRates} displays the loss optimized during gradient flow. A more accurate comparison would require extending the convergence rate result of \cref{Prop:HighDimCVRates} to the high-dimensional limit of the loss function. We expect this analysis to be significantly more involved and leave it for future work.

\subsection{Dynamical Stability} \label{App:Subsec:DynamicalStability}

To conclude this section on the stability of the dynamics, we outline an approach that would be worth investigating in order to characterize potential instabilities of the steady-state dynamics. Earlier in this section, we analyzed the stability of the steady-state solution under generic and worst-case perturbations. Here, we aim to go beyond this analysis by taking into account the original set of equations in \cref{Result1}. More precisely, we decompose these equations into a steady-state solution and the exact associated perturbation, and then linearize the dynamics around the steady state for this specific perturbation. This type of approach is not new and has already been applied to simpler settings, such as spin-glass models \citep{sompolinsky1982relaxational, crisanti1993spherical}.

We only sketch the main steps that would lead to such a result. While the calculation can, in principle, be carried out, it quickly becomes technically involved. For this reason, we leave a detailed implementation of this approach for future work.

Recall the high-dimensional equations of \cref{Result1}. With the $\ell_2$-regularization, $W(t)$ is solution of the dynamics:
\begin{equation} \label{eq:DMFTStability}
    \dot W(t) = 2 \left( \int_0^t R(t,t') \Big( \G(t') + Z^* - Z(t') \Big) \d t' \right) W(t) - 2 \lambda W(t).
\end{equation}
The covariance of the Gaussian process $\G$ and the response kernel $R$ are directly related to averages with respect to the dynamics, in equations \eqref{eq:Covariance1G} and \eqref{eq:AveragedQuantities}. From now on, we denote:
\begin{equation} \label{eq:DynamicalStabCovariance}
    \E \, \G_{ij}(t) \G_{i'j'}(t') = \frac{1}{d} \big( \delta_{ii'} \delta_{jj'} + \delta_{ij'} \delta_{i'j} \big) C(t,t'), 
\end{equation}

\subsubsection{Linearization of the Dynamics}

We start by linearizing the dynamical equation \eqref{eq:DMFTStability} around the steady-state solution. To do so, we introduce the functions $\delta R, \delta C$ such that:
\begin{equation*}
    R(t,t') = r_\infty \delta(t-t') + \delta R(t,t'), \hspace{1.5cm} C(t,t') = \xi + \delta C(t,t'). 
\end{equation*}
The first terms $r_\infty \delta(t-t')$ and $\xi = \lim_{t,t' \to \infty} C(t,t')$ correspond to the approximation of the response kernel and the noise covariance in the steady-state approximation. Then, the dynamics on $W(t)$ can be rewritten as:
\begin{equation*}
\begin{aligned}
    \dot W(t) &= 2 r_\infty \left( Z^* + \sqrt{\xi} \G - \frac{\lambda}{r_\infty} I_d - W(t)W(t)^\top \right) W(t) + 2 r_\infty H(t) W(t), \\
    H(t) &= V(t) - \frac{1}{r_\infty} \int_0^t \delta R(t,t') \Big( \sqrt{\xi} \G + Z^* - Z(t') \Big) \d t', 
\end{aligned}
\end{equation*}
where $V$ is a centered Gaussian process taking values in the space of symmetric matrices, with covariance:
\begin{equation*}
    \E \, V_{ij}(t) V_{i'j'}(t') = \frac{1}{d} \big( \delta_{ii'} \delta_{jj'} + \delta_{ij'} \delta_{i'j} \big) \delta C(t,t'). 
\end{equation*}
The key point of our analysis is to consider $H(t)$ as a perturbation and linearize the solution $Z(t) = W(t)W(t)^\top$ around $H(t)$:
\begin{equation} \label{eq:DecompositionSecondOrder}
\begin{aligned}
    Z(t) &= Z_0(t) + 2 r_\infty \int_0^t \left. \frac{\partial Z_0(t)}{\partial H(t')} \right|_{H = 0} \Big( H(t') \Big) \d t' \\
    &\hspace{1cm} + 2 r_\infty^2 \int_0^t \int_0^t  \left. \frac{\partial^2 Z_0(t)}{\partial H(t') \partial H(t'')} \right|_{H = 0} \Big( H(t'), H(t'') \Big) \d t' \d t'' + O \big( \big\| H \big\|^3 \big),
\end{aligned}
\end{equation}
where $Z_0(t) = W_0(t)W_0(t)^\top$ and $W_0$ is solution of the dynamics:
\begin{equation} \label{eq:DynamicsLinearized}
    \dot W_0(t) = 2 r_\infty \left( Z^* + \sqrt{\xi} \G - \frac{\lambda}{r_\infty} I_d - W_0(t)W_0(t)^\top \right) W_0(t).
\end{equation}
Using equations \eqref{eq:Covariance1G} and~\eqref{eq:DynamicalStabCovariance}, we can rewrite the covariance $C(t,t')$:
\begin{equation*}
    C(t,t') = \frac{1}{2\alpha} \left(\frac{1}{d} \E \, \tr \Big( (Z(t) - Z^*) (Z(t') - Z^*) \Big) + \frac{\Delta}{2} \right).
\end{equation*}
Plugging the solution \eqref{eq:DecompositionSecondOrder} into the expression of $C(t,t')$ will allow to obtain a self-consistent expression of $\delta C$. Since $H(t)$ involves the Gaussian process $V(t)$ whose covariance is proportional to $\delta C$, it is required that we study this perturbation up to second order. 

The same can be done for the response $R(t,t')$. Equation \eqref{eq:AveragedQuantities} shows that this function is directly related to the response of $Z$, therefore one can directly differentiate equation \eqref{eq:DecompositionSecondOrder} with respect to an external field perturbing the dynamics.  

Ultimately, it is possible to identify a $2 \times 2$ matrix $K(t,t',s,s')$ such that, after linearizing the dynamics, we have:
\begin{equation} \label{eq:DynamicalStability1}
    \begin{pmatrix} \delta C(t,t') \\ \delta R(t,t') \end{pmatrix} = \begin{pmatrix} C_0(t,t') - \xi \\ (1 - r_\infty) \delta(t-t') - \alpha^{-1} R_Z^0(t,t') \end{pmatrix} + \int \d s \d s' \, K(t,t',s,s') \begin{pmatrix} \delta C(s,s') \\ \delta R(s,s') \end{pmatrix}, 
\end{equation}
where $C_0, R_Z^0$ are the correlation and response computed from the steady-state solution. In addition, the kernel $K$ only involves functions computed from this approximate solution. This is a consequence of the linearization of the dynamics around this solution. More precisely, we can show that $K$ depends on the response kernels of the dynamics \eqref{eq:DynamicsLinearized} up to order three:
\begin{equation*}
    \mathcal{R}_1(t,t') = \left. \frac{\partial Z_0(t)}{\partial H(t')} \right|_{H = 0}, \hspace{1.5cm} \mathcal{R}_2(t,t',t'') = \left. \frac{\partial^2 Z_0(t)}{\partial H(t') \partial H(t'')} \right|_{H = 0}, 
\end{equation*}
and likewise for $\mathcal{R}_3$. 

\subsubsection{Response Kernels of the Oja Flow}

The dynamics that describes our steady-state approximation is nonlinear, but it is possible to compute its response kernels of any order. We briefly explain how this can be done, and refer to \cref{Subsec:OjaResponse} for a complete derivation of the first-order response. We consider the dynamics:
\begin{equation*}
    \dot W(t) = 2 r_\infty \Big( A - W(t) W(t)^\top + \epsilon H(t) \Big) W(t).
\end{equation*}
For simplicity we drop the factor $2r_\infty$, that can be simply recovered by a time reparameterization. We decompose the solution $Z(t) = W(t)W(t)^\top$ in powers of $\epsilon$:
\begin{equation*}
    Z(t) = Z_0(t) + \epsilon Z_1(t) + \dots + \epsilon^n Z_n(t) + \dots.
\end{equation*}
Plugging this into the equation for $W(t)$, we get the equations for each order:
\begin{equation*}
\begin{aligned}
    \dot Z_0(t) &= A Z_0(t) + Z_0(t) A - 2Z_0(t)^2, \\
    \dot Z_n(t) &= B(t) Z_n(t) + Z_n(t) B(t) + \underbrace{H(t) Z_{n-1}(t) + Z_{n-1}(t) H(t) - 2 \sum_{k=1}^{n-1} Z_k(t) Z_{n-k}(t)}_{\equiv M_n(t)},
\end{aligned}
\end{equation*}
with $B(t) = A - 2Z_0(t)$, and with initial conditions $Z_n(0) = 0$ for $n \geq 1$. The dynamics on $Z_n$ is linear but is driven by the time-dependent matrix $B(t)$. To solve this, let us introduce $P(t)$ as the solution of $\dot P(t) = - B(t) P(t)$ with $P(0) = I_d$. Then, for $n \geq 1$, it is easily seen that:
\begin{equation} \label{eq:SolutionZn}
    Z_n(t) = P(t)^{-1} \int_0^t P(s) M_n(s) P(s)^\top \d s \, P(t)^{-\top}. 
\end{equation}
Now, as already been done in \cref{App:Subsec:OjaResponse1}, the matrix $P(t)$ can be computed explicitly and writes: 
\begin{equation*}
    P(t) = e^{-tA} + Z_0 A^{-1} \big( e^{tA} - e^{-tA} \big). 
\end{equation*}
From equation \eqref{eq:SolutionZn}, this leads to an explicit recursive expression of $Z_n(t)$. It is easily shown by recursion that $Z_n(t)$ is a homogeneous polynomial of degree $n$ in $H$, and therefore the $n^{\text{th}}$-order response kernel of the Oja flow can be directly computed from $Z_n$. In \cref{Prop:OjaResponse1}, we give an explicit computation for the case $n = 1$. Then, it is possible to iterate: plug in the solution $Z_1(t)$ into the expression of $Z_2(t)$ and compute the second-order response, and so on. However, the calculations rapidly become heavy and it is not clear if one can practically obtain an expression of the third-order response, which is necessary to compute the kernel $K$. 

\subsubsection{Stability Criterion}

Once the response kernels are computed, it is possible to go back to equation \eqref{eq:DynamicalStability1} and either let the dimension grow to infinity to test the stability in the high-dimensional limit, or leave it fixed. In any case, it is practical to look for a solution of the form:
\begin{equation*}
    X(t,t') \equiv \begin{pmatrix} \delta C(t,t') \\ \delta R(t,t') \end{pmatrix} = e^{\sigma T} \begin{pmatrix} \delta C_0(\tau) \\ \delta R_0(\tau) \end{pmatrix}, \hspace{1.5cm} T = \frac{t + t'}{2}, \hspace{0.2cm} \tau = t-t'. 
\end{equation*}
This separates variables between the time-translational invariant part and a potential exponential growth as $T \to \infty$. Using that the kernel all originates from a steady-state solution (the approximate dynamics), we can integrate with respect to $T$, which should lead to an equation of the form:
\begin{equation*}
    \sigma \begin{pmatrix} \delta C_0(\tau) \\ \delta R_0(\tau) \end{pmatrix} = \sigma \epsilon(\tau) + \int_0^\infty \mathcal{K}(\tau, \tau') \begin{pmatrix} \delta C_0(\tau') \\ \delta R_0(\tau') \end{pmatrix} \d \tau'.   
\end{equation*}
Then, a sizable challenge would be to compute the spectrum of the operator $\mathcal{K}$ (that is computed from $K$), and stability would be guaranteed if its eigenvalues are all with negative real part. However, it is not clear if such a calculation is possible. 

\section{Analysis of Langevin Dynamics} \label{App:Langevin}
In this section we go beyond the gradient flow setting and analyze the Langevin dynamics, with $\beta < \infty$. We start from the dynamical equations in \cref{Result1} and study the stochastic evolution of the matrix $W(t)$. More precisely, under simplifying assumptions on the dynamics at long times, we derive the stationary measure of this process. Here, by stationary measure we mean the probability distribution reached by the process in the long-time limit. The plan of the section goes as follows:
\begin{itemize}
    \item In \cref{App:Subsec:StationaryMeasure}, we consider a stochastic evolution similar to equation \eqref{eq:DMFT1W} and derive its stationary measure.
    \item In \cref{App:Subsec:MappingDynamicsLangevin}, we build on this result to derive the stationary measure of the dynamics on $W(t)$, leading to the equations in \cref{Result:Langevin}. This result requires several assumptions on the dynamics (time-translational invariance and fluctuation--dissipation) that we detail in \cref{Assumption4}. 
    \item In \cref{App:Subsec:LabelLangevin}, under similar assumptions, we derive the stationary measure for the typical label, whose expression is given in equation \eqref{eq:DMFT1Y}. 
    \item In \cref{App:Subsec:ZeroTemperature}, we study the zero-temperature ($\beta \to \infty$) limit of the stationary measure previously derived. At positive regularization, we show that we recover the results obtained in the gradient flow setting. 
    \item In \cref{App:Subsec:BayesOptimal}, we show that an appropriate choice of the inverse temperature $\beta$ leads to the Bayes-optimal equations of \citet{maillard2024bayes}. 
\end{itemize}

\subsection{Derivation of a Stationary Measure} \label{App:Subsec:StationaryMeasure}

In this section, we introduce an auxiliary stochastic differential equation and derive its stationary measure using a Gaussian coupling.

\subsubsection{Summary of the Result} \label{App:Subsubsec:InvariantGeneral}

In the following, we consider the stochastic differential equation on $W(t) \in \R^{d \times m}$:
\begin{equation} \label{eq:SDEGeneral}
\begin{aligned}
    \d W(t) &= 2 \left[ \int_0^t R(t-t') \Big( \mathcal{H}(t') + A - W(t')W(t')^\top \Big) \d t' \right] W(t) \d t \\
    &\hspace{3cm} - \nabla \Omega \big( W(t) \big) \d t + \frac{1}{\sqrt{\beta d}} \d B(t),
\end{aligned}
\end{equation}
where $A \in \mathcal{S}_d(\R)$ is a fixed matrix, $\Omega \colon \R^{d \times m} \to \R^+$ is a continuously differentiable regularization, $B$ is a standard Brownian motion over $\R^{d \times m}$, and $\mathcal{H}$ is a Gaussian process on $\mathcal{S}_d(\R)$, independent of $B$, with zero mean and stationary covariance:
\begin{equation*}
    \E \, \mathcal{H}_{ij}(t) \mathcal{H}_{i'j'}(t') = \frac{1}{d} \big( \delta_{ii'} \delta_{jj'} + \delta_{ij'} \delta_{i'j} \big) C(t-t'). 
\end{equation*}
Moreover, we assume that:
\begin{itemize}
    \item The stationary covariance $C$ can be expressed as:
    \begin{equation*}
        C(t) = \int_0^\infty \phi(\theta) e^{- \theta |t|} \d \theta,
    \end{equation*}
    for some $\phi \colon \R^+ \to \R^+$ and one has $C(t) \xrightarrow[t \to \infty]{} 0$. 
    \item The functions $R, C$ satisfy the fluctuation--dissipation relation:
    \begin{equation} \label{eq:FDRGeneral}
        R(t) - 4 \beta \int_0^t C'(t-t') R(t') \d t' = \delta(t),
    \end{equation}
    where $\delta$ is the Dirac delta distribution supported at zero.  
\end{itemize}
When these assumptions are satisfied, we show that the stochastic equation given in equation~\eqref{eq:SDEGeneral} has stationary measure:
\begin{equation} \label{eq:InvariantMeasureGeneral}
    \P_\beta(W) \propto \exp \left( - \frac{\beta d}{1+4 \beta C(0)} \Big\| WW^\top - A \Big\|_F^2 - 2\beta d \, \Omega(W) \right). 
\end{equation}
For the following, it will be interesting to note that the fluctuation--dissipation relation in equation \eqref{eq:FDRGeneral} implies the equality:
\begin{equation*}
    \int_0^\infty R(t) \d t = \frac{1}{1 + 4 \beta C(0)}. 
\end{equation*}
The following sections are dedicated to the derivation of this result. More precisely, we introduce in the following section a Gaussian coupling that allows to derive the stationary measure in equation \eqref{eq:InvariantMeasureGeneral} from standard results on Langevin dynamics. 

\subsubsection{Gaussian Coupling}

In this part, we consider the following potential on $W \in \R^{d \times m}$ and the symmetric matrices $\Q_1, \dots, \Q_K~\in~\mathcal{S}_d(\R)$:
\begin{equation}
    U_\text{coupling} \Big( W, \big\{ \mathcal{Q}_k \big\} \Big) = \left\| WW^\top - A - \sum_{k=1}^K g_k \mathcal{Q}_k \right\|_F^2 + \sum_{k=1}^K \theta_k \big\| \mathcal{Q}_k \big\|_F^2 + 2 \Omega(W),
\end{equation}
where $A \in \mathcal{S}_d(\R)$ is a fixed matrix and $\Omega$ is a smooth and coercive regularization. While we are mainly interested in $W$, we introduce the matrices $\mathcal Q_1, \dots, \mathcal Q_K$ as auxiliary variables that couple to~$W$.

The goal is to study a Langevin dynamics on the matrices $W, \big\{ \Q_k \big\}$ such that:
\begin{itemize}
    \item The stationary distribution of the dynamics is given by the Boltzmann--Gibbs distribution associated with the potential $U_\text{coupling}$, namely:
    \begin{equation} \label{eq:InvariantMeasure1}
        \P_{\text{coupling}} \big( W, \{\Q_k\} \big) \propto \exp \left( - \beta d \, U_\text{coupling} \Big( W, \big\{ \mathcal{Q}_k \big\} \Big) \right), 
    \end{equation}
    for some inverse temperature $\beta > 0$. Remark that the potential $U_\text{coupling}$ is quadratic in the $\{ \Q_k \}$, so that these variables are Gaussian when drawn from $\P_\text{coupling}$. 
    \item The dynamics for the auxiliary variables $\{ \mathcal Q_k \}$ can be solved explicitly, and leads to an effective stochastic dynamics for $W$ that is equivalent to the dynamics \eqref{eq:SDEGeneral}.
\end{itemize}
To achieve this, we consider the stochastic differential equations:
\begin{equation} \label{eq:CouplingBO}
\begin{aligned}
    \d W(t) &= - \frac{1}{2} \nabla_W U_\text{coupling} \Big( W(t), \big\{ \mathcal{Q}_k(t) \big\} \Big) \d t + \frac{1}{\sqrt{\beta d}} \d B(t), \\
    \d \mathcal{Q}_k(t) &= - \frac{1}{2} \nabla_{\mathcal{Q}_k} U_\text{coupling} \Big( W(t), \big\{ \mathcal{Q}_k(t) \big\} \Big) \d t + \frac{1}{\sqrt{\beta d}} \d \Xi_k(t),
\end{aligned}
\end{equation}
where $B$ is a standard Brownian motion on $\R^{d \times m}$, and $\Xi_1, \dots, \Xi_K$ are independent (and independent of $B$) standard Brownian motions over $\mathcal{S}_d(\R)$. As $U_\text{coupling}$ is confining, it is well known \citep[see for instance][Proposition 4.2]{pavliotis2014stochastic} that the dynamics \eqref{eq:CouplingBO} admits a unique stationary distribution, given by \eqref{eq:InvariantMeasure1}. 

We now analyze the Langevin dynamics \eqref{eq:CouplingBO}. Computing the gradients, we arrive at:
\begin{align}
    \d W(t) &= -2\Big( W(t)W(t)^\top - A - \mathcal{Q}(t) \Big)W(t) \d t - \nabla \Omega \big( W(t) \big) \d t + \frac{1}{\sqrt{\beta d}} \d B(t), \label{eq:BOCouplingW} \\
    \d \mathcal{Q}_k(t) &= - \Big( \theta_k \mathcal{Q}_k(t) + g_k \mathcal{Q}(t) - g_k \Big( W(t)W(t)^\top - Z^* \Big) \Big) \d t + \frac{1}{\sqrt{\beta d}} \d \Xi_k(t), \label{eq:BOCouplingQK} \\
    \mathcal{Q}(t) &= \sum_{k=1}^K g_k \mathcal{Q}_k(t). \label{eq:BOCouplingQ}
\end{align}
Equation \eqref{eq:BOCouplingQK} can be integrated, leading to the expression of $\Q_k(t)$ depending on $\Q(t)$ and $W(t)$:
\begin{equation*}
\begin{aligned}
    \mathcal{Q}_k(t) &= e^{- \theta_k t} \mathcal{Q}_k(0) - g_k \int_0^t e^{-\theta_k(t-t')} \mathcal{Q}(t') \d t' + g_k \int_0^t e^{- \theta_k(t-t')} \Big( W(t') W(t')^\top - A \Big) \d t' \\
    &\hspace{5cm}+ \frac{1}{\sqrt{\beta d}} \int_0^t e^{-\theta_k(t-t')} \d \Xi_k(t'). 
\end{aligned}
\end{equation*}
Plugging this expression into equation \eqref{eq:BOCouplingQ}, we obtain the self-consistent expression of $\mathcal{Q}$:
\begin{equation*}
\begin{aligned}
    \mathcal{Q}(t) &= \sum_{k=1}^K g_k e^{- \theta_k t} \mathcal{Q}_k(0) - \int_0^t \left( \sum_{k=1}^K g_k^2 e^{-\theta_k(t-t')} \right) \Big( \mathcal{Q}(t') + A - W(t')W(t')^\top \Big) \d t' \\
    &\hspace{5cm} + \frac{1}{\sqrt{\beta d}} \sum_{k=1}^K g_k \int_0^t e^{- \theta_k(t-t')} \d \Xi_k(t'). 
\end{aligned}
\end{equation*}
To solve for $\Q$, we introduce the following scalar function and random process:
\begin{align}
    \Gamma(t) &= \sum_{k=1}^K g_k^2 e^{- \theta_k t}, \label{eq:defKernelBO} \\
    \mathcal{H}(t) &= \sum_{k=1}^K g_k e^{- \theta_k t} \Q_k(0) + \frac{1}{\sqrt{\beta d}} \sum_{k=1}^K g_k \int_0^t e^{- \theta_k(t-t')} \d \Xi_k(t'), \label{eq:defNoiseBO}
\end{align}
so that we have the self-consistent expression for $\Q$:
\begin{equation*}
    \Q(t) = -\int_0^t \Gamma(t-t') \Big( \Q(t') + A - W(t') W(t')^\top \Big) \d t' + \mathcal{H}(t). 
\end{equation*}
We can write $\Q$ explicitly by introducing $R$ such that:
\begin{equation} \label{eq:RelationGammaK}
    R(t) + \int_0^t \Gamma(t-t') R(t') \d t' = \delta(t). 
\end{equation}
As a consequence, $\Q$ can be expressed as:
\begin{equation*}
    \Q(t) = W(t) W(t)^\top - A - \int_0^t R(t-t') \Big( W(t')W(t')^\top - A - \mathcal{H}(t') \Big) \d t'. 
\end{equation*}
Then, plugging this expression into the dynamics on $W$ in equation \eqref{eq:BOCouplingW}, we precisely obtain the stochastic equation \eqref{eq:SDEGeneral}. 

\subsubsection{Covariance of the Noise} \label{App:Subsubsec:CovarianceCoupling}

We now study the random process $\mathcal{H}(t)$ defined in equation~\eqref{eq:defNoiseBO}. Provided that the initializations $\{\Q_k^0 \}$ are Gaussian, $\mathcal{H}(t)$ is a matrix-valued Gaussian process. More precisely, we assume that the $\{ \Q_k^0 \}$ are Gaussian matrices with zero mean and a covariance of the form:
\begin{equation*}
    \E \, \big( \Q_k^0 \big)_{ij} \big( \Q_{k'}^0 \big)_{i'j'} = \frac{\xi_k^0}{d} \delta_{kk'} \big( \delta_{ii'} \delta_{jj'} + \delta_{ij'} \delta_{i'j} \big). 
\end{equation*}
Then, since $\Xi_1, \dots, \Xi_K$ are independent Brownian motions over $\mathcal{S}_d(\R)$, one has the expression for the covariance of $\mathcal{H}$:
\begin{equation*}
    \E \, \mathcal{H}_{ij}(t) \mathcal{H}_{i'j'}(t') = \frac{1}{d} \big( \delta_{ii'} \delta_{jj'} + \delta_{ij'} \delta_{i'j} \big) \sum_{k=1}^K g_k^2 e^{ -\theta_k(t+t')} \left( \xi_k^0 + \frac{1}{4 \theta_k \beta} \big( e^{2 \theta_k \min(t,t')} - 1 \big) \right). 
\end{equation*}
This covariance has a transient part that vanishes as $t,t' \to \infty$, and a time-translational invariant part which is only a function of $t-t'$. Since we are only interested in the dynamics of $W(t)$ at long times, and the initialization of the $\big\{ \Q_k \big\}$ is not relevant, we choose the variance:
\begin{equation*}
    \xi_k^0 = \frac{1}{4 \theta_k \beta}. 
\end{equation*}
In this case the covariance of $\mathcal{H}$ is time-translational invariant:
\begin{equation*}
    \E \, \mathcal{H}_{ij}(t) \mathcal{H}_{i'j'}(t') = \frac{1}{4\beta d} \big( \delta_{ii'} \delta_{jj'} + \delta_{ij'} \delta_{i'j} \big) \sum_{k=1}^K \frac{g_k^2}{\theta_k} e^{ -\theta_k|t-t'|}. 
\end{equation*}
We then consider the function:
\begin{equation} \label{eq:CovarianceNoiseLangevin}
    C(\tau) = \frac{1}{4 \beta} \sum_{k=1}^K \frac{g_k^2}{\theta_k} e^{ -\theta_k|\tau|},  
\end{equation}
so that we have the fluctuation--dissipation relation between $C$ and the function $\Gamma$, defined in equation \eqref{eq:defKernelBO}:
\begin{equation*}
    - C'(\tau) = \frac{1}{4\beta} \Gamma(\tau), 
\end{equation*}
for $\tau > 0$. Thanks to the relationship between $R$ and $\Gamma$ in equation \eqref{eq:RelationGammaK}, we arrive at the fluctuation--dissipation relation:
\begin{equation} \label{eq:FDR1}
    R(t) - 4 \beta \int_0^t C'(t-t') R(t') \d t' = \delta(t). 
\end{equation}
This relationship fully characterizes $R$. In addition, equation \eqref{eq:CovarianceNoiseLangevin} can be interpreted as the discretization (at finite $K$) of the continuous representation:
\begin{equation*}
    C(\tau) = \int_0^\infty \phi(\theta) e^{- \theta |\tau|} \d \theta,
\end{equation*}
for $\phi \colon \R^+ \to \R^+$. Combining this remark with the fluctuation--dissipation relation \eqref{eq:FDR1}, this leads to the result given in \cref{App:Subsubsec:InvariantGeneral}. 

\subsubsection{Stationary Measure} \label{App:Subsubsec:StationaryTTI}

We shall now derive the stationary measure given in equation \eqref{eq:InvariantMeasureGeneral}. Recall that we integrated the variables $\{\Q_k\}$ in the Langevin dynamics~\eqref{eq:CouplingBO} to obtain an equation solely on $W(t)$. Now, to access the stationary measure on $W$ itself, we apply the same procedure and average the joint distribution~\eqref{eq:InvariantMeasure1} with respect to the $\{\Q_k\}$:
\begin{equation*}
    \P_\beta(W) \propto \int \exp \left( - \beta d \, U_\text{coupling} \Big( W, \big\{ \mathcal{Q}_k \big\} \Big) \right) \d \Q_1 \dots \d \Q_K. 
\end{equation*}
Due to the expression of the potential $U_\text{coupling}$, we define:
\begin{equation*}
\begin{aligned}
    \mathcal{M} &= \begin{pmatrix} g_1 \big( WW^\top - A \big) \\ \vdots \\ g_K \big( WW^\top - A \big) \end{pmatrix} \in \mathcal{S}_d(\R)^K, \hspace{1.2cm} \mathcal{T} =  \mathrm{diag} \Big( \theta_1 I_d, \dots, \theta_K I_d \Big) + \Big( g_k g_l I_d \Big)_{1 \leq k,l \leq K}, 
\end{aligned}
\end{equation*}
where $\mathcal{T}$ is viewed as a linear map on $\mathcal{S}_d(\R)^K$. Then, denoting $\langle \, , \, \rangle$ the Euclidean inner product on $\mathcal{S}_d(\R)^K$, we have:
\begin{equation} \label{eq:InvariantW1}
\begin{aligned}
    \P_\beta(W) &\propto \exp \left( - \beta d \big\| WW^\top - A \big\|_F^2 - 2 \beta d \, \Omega(W) \right) \\
    &\hspace{3cm}\int_{\mathcal{S}_d(\R)^K} \d \Q \, \exp \Big( 2 \beta d \, \langle \mathcal{M}, \Q \rangle - \beta d \, \langle \Q, \mathcal{T}(\Q) \rangle \Big) \\
    &\propto \exp \left( - \beta d \big\| WW^\top - A \big\|_F^2 - 2 \beta d \, \Omega(W) \right) \exp \Big( \beta d \, \langle \mathcal{M}, \mathcal{T}^{-1}(\mathcal{M}) \rangle \Big).    
\end{aligned}
\end{equation}
We computed the Gaussian integral with respect to $\Q$ and ignored the term proportional to $\det(\mathcal{T})$ that is independent of $W$. In the following, we view $\mathcal{S}_d(\R)^K$ as the tensor product $\R^K \otimes \mathcal{S}_d(\R)$. This induces a tensor structure for the space of linear maps on $\mathcal{S}_d(\R)^K$. With this identification, we write:
\begin{equation*}
    \mathcal{T} = \big( \Theta + gg^\top \big) \otimes \mathrm{Id}, \hspace{1.5cm} \mathcal{M} = g \otimes \big( WW^\top - A \big).
\end{equation*}
where $\mathrm{Id}$ denotes the identity map on the space of symmetric matrices, $g = \big( g_1, \dots, g_K \big)^\top \in \R^K$ and $\Theta = \mathrm{diag}(\theta_1, \dots, \theta_K) \in \R^{K \times K}$. Therefore, we have:
\begin{equation*}
    \langle \mathcal{M}, \mathcal{T}^{-1}(\mathcal{M}) \rangle = g^\top \big( \Theta + gg^\top \big)^{-1} g \, \big\| WW^\top - A \big\|_F^2. 
\end{equation*}
As a consequence of the Sherman--Morrison formula \citep[see][]{hager1989updating}:
\begin{equation*}
    g^\top \big( \Theta + gg^\top \big)^{-1} g = \frac{g^\top \Theta^{-1} g}{1 + g^\top \Theta^{-1} g}. 
\end{equation*}
Computing this last quantity, along with the expression of $\P_\beta(W)$ in equation \eqref{eq:InvariantW1}, we finally obtain:
\begin{equation*}
    \P_\beta(W) \propto \exp \left( - \frac{\beta d}{1+c} \big\|WW^\top - A \big\|_F^2 - 2 \beta d \, \Omega(W) \right),
\end{equation*}
with:
\begin{equation*}
    c = \sum_{k=1}^K \frac{g_k^2}{\theta_k}. 
\end{equation*}
From the expression of the covariance of the noise in equation \eqref{eq:CovarianceNoiseLangevin}, we have that $c = 4 \beta C(0)$. This leads to equation \eqref{eq:InvariantMeasureGeneral}. 

\subsection{Mapping onto the Dynamics} \label{App:Subsec:MappingDynamicsLangevin}

In this section, we apply the result of \cref{App:Subsec:StationaryMeasure} to compute the stationary measure of the dynamics given in \cref{Result1}. We recall that $W(t)$ is solution of the equation:
\begin{equation} \label{eq:DynamicsWLangevin}
    \d W(t) = 2 \left( \int_0^t R(t,t') \Big( \G(t') + Z^* - Z(t')\Big) \d t' \right) W(t) \d t - \nabla \Omega \big( W(t) \big) \d t + \frac{1}{\sqrt{\beta d}} \d B(t),
\end{equation}
where the kernel $R$ and the covariance of $\G$ are self-consistently computed from equation~\eqref{eq:DynamicsWLangevin}:
\begin{align}
    \E \, \G_{ij}(t) \G_{i'j'}(t') &= \frac{1}{2\alpha d} \big( \delta_{ii'} \delta_{jj'} + \delta_{ij'} \delta_{i'j} \big) \left( \frac{1}{d} \E \, \tr \Big( (Z(t) - Z^*)(Z(t') - Z^*) \Big) + \frac{\Delta}{2} \right), \label{eq:CovarianceLangevin} \\
    R(t,t') &= \delta(t-t') - \frac{1}{\alpha d^2}\tr \left( \left. \frac{\partial \, \E \, Z(t)}{\partial H(t')} \right|_{H = 0} \right) \label{eq:ResponseLangevin},
\end{align}
and the derivative in equation \eqref{eq:ResponseLangevin} is defined under a perturbation $\G(t') \mapsto \G(t') + H(t')$. The stochastic differential equation on $W$ is very similar to the one in equation \eqref{eq:SDEGeneral}, with some differences. In the dynamics \eqref{eq:SDEGeneral}, it is required that the noise $\mathcal{H}$ decouples at large times, therefore we decompose $\G(t)$ as a static part and an independent Gaussian noise $\H(t)$:
\begin{equation} \label{eq:DecompositionNoiseStatic}
    \G(t) = \sqrt{\xi} \G + \mathcal{H}(t),
\end{equation}
where $\G \sim \mathrm{GOE}(d)$ and $\H(t)$ is a centered Gaussian process such that:
\begin{equation} \label{eq:Lan gevinCovarianceH}
    \E \, \H_{ij}(t) \H_{i'j'}(t') = \frac{1}{d} \big( \delta_{ii'} \delta_{jj'} + \delta_{ij'} \delta_{i'j} \big) C(t,t'),
\end{equation}
with $C(t,t')$ going to zero as $t-t' \to \infty$. The variable $\xi$ then corresponds to the covariance of the static part of the noise. Therefore, in order to match the setting of \cref{App:Subsec:StationaryMeasure}, we set $A = Z^* + \sqrt{\xi} \G$, so that equation \eqref{eq:DynamicsWLangevin} matches \eqref{eq:SDEGeneral}. 

Additionally, we introduce an assumption regarding the system of equations \eqref{eq:DynamicsWLangevin},~\eqref{eq:CovarianceLangevin},~\eqref{eq:ResponseLangevin} allowing to apply the results obtained in \cref{App:Subsubsec:InvariantGeneral}. 

\begin{assumption} \label{Assumption4}

\begin{assumpitem}[Asymptotic time-translational invariance.]\label{Assumption4:TTI} There exists $C_\text{TTI}, R_\text{TTI} \colon \R^+ \to \R$ such that:
\begin{equation} \label{eq:ApproximateTTI}
\begin{aligned}
    \sup_{t \geq t'} \big| C(t,t') - C_\text{TTI}(t-t') \big| &\xrightarrow[t' \to \infty]{} 0, \\
    \int_0^t \big| R(t,t') - R_\text{TTI}(t-t') \big| \d t' &\xrightarrow[t \to \infty]{} 0. 
\end{aligned}
\end{equation}
\end{assumpitem} 
\begin{assumpitem}[Fluctuation--dissipation.]\label{Assumption4:FDT} $C_\text{TTI}, R_\text{TTI}$ are linked by the fluctuation--dissipation relation:
\begin{equation} \label{eq:FDTAssumption}
    \delta(t) = R_\text{TTI}(t) - 4 \beta \int_0^t C_\text{TTI}'(t-t') R_\text{TTI}(t') \d t'. 
\end{equation}
\end{assumpitem}
\end{assumption}

Indeed, the setting of \cref{App:Subsubsec:InvariantGeneral} requires the covariance of $\mathcal{H}$ and the function $R(t,t')$ to be time-translational invariant, i.e., to only depend on the time difference $t-t'$. This assumption is not satisfied by the dynamics \eqref{eq:DynamicsWLangevin}, for which non-stationary effects may persist at early times.

However, it is standard in dynamical mean-field theory and generalized Langevin equations to assume that  deviations of time-translational invariance (TTI) become negligible in the long-time regime. Following prior works such as those of \citet{fan2025dynamical}, \citet{chen2025learning}, we assume that the correlation $C(t,t')$ and response $R(t,t')$ converge to TTI limits in the sense of equation~\eqref{eq:ApproximateTTI}. In addition, the conditions in equation \eqref{eq:ApproximateTTI} ensure that all the non-TTI contributions vanish in the long-time limit. These subleading corrections do not contribute to the stationary measure, which is solely determined by the asymptotic of the TTI components $C_\text{TTI}$ and $R_\text{TTI}$.

\subsubsection{Fluctuation--Dissipation Relation} \label{App:Subsubsec:FDR}

In this part, we explain how we can recover the fluctuation--dissipation relation given in \cref{Assumption4:FDT} (and obtained in equation \ref{eq:FDRGeneral}) from the system of equations \eqref{eq:DynamicsWLangevin}, \eqref{eq:CovarianceLangevin}, \eqref{eq:ResponseLangevin}. To do so, we consider the perturbed dynamics:
\begin{align}
    \d W(t) &= 2 M(t) W(t) \d t - \nabla \Omega \big( W(t) \big) \d t + 2 \tilde H(t) W(t) \d t + \frac{1}{\sqrt{\beta d}} \d B(t), \label{eq:FDTPerturbed} \\
    M(t) &= \int_0^t R(t,t') \Big( \G(t') + Z^* - Z(t')\Big) \d t'. 
\end{align}
Formally, the additional drift term originating from $\tilde H$ can be interpreted as the gradient of the quadratic function:
\begin{equation*}
    2 \tilde H W = \nabla_W \tr \big( \tilde H WW^\top \big),
\end{equation*}
except that $\tilde H$ is time-dependent. In this case we say that $\tilde H$ is conjugate to the observable $Z(t) = W(t)W(t)^\top$. In the setting of Langevin dynamics driven to equilibrium, it is well known that the linear response to a perturbation conjugate to an observable is related to the correlation function of that observable through the fluctuation--dissipation theorem \citep{kubo1966fluctuation}. 

Assuming that the unperturbed dynamics converges to a stationary state and satisfies the asymptotic time-translational invariance in \cref{Assumption4:TTI}, we therefore assume that the standard fluctuation--dissipation relation holds for the observable $Z(t)$. More precisely:
\begin{equation*}
    \left. \frac{\partial \, \E \, Z_{ij}(t)}{\partial \tilde H_{ij}(t')} \right|_{\tilde H = 0} = 2 \beta d \, \partial_{t'} \E \big[ Z_{ij}(t) Z_{ij}(t') \big]. 
\end{equation*}
This relation expresses that, at equilibrium, the response of $Z_{ij}$ to a small perturbation is governed by its spontaneous fluctuations. The factor $2 \beta d$ is the inverse of the diffusion constant in the dynamics \eqref{eq:FDTPerturbed}. Then, summing over the indices $i,j$, we get the relationship:
\begin{equation} \label{eq:FDT1}
    \frac{1}{d^2} \tr \left( \left. \frac{\partial \, \E \, Z(t)}{\partial \tilde H(t')} \right|_{\tilde H = 0} \right) = 2 \beta \partial_{t'} C_Z(t,t'), 
\end{equation}
with:
\begin{equation*}
    C_Z(t,t') = \frac{1}{d} \E \, \tr \big( Z(t) Z(t') \big). 
\end{equation*}
Now, recall that the fluctuation--dissipation relation in equation \eqref{eq:FDTAssumption} relates the response in equation \eqref{eq:ResponseLangevin} and the covariance of the noise $\mathcal{H}(t)$. Considering the asymptotic TTI regime, we may write $C_Z(t,t') \approx C_Z^\text{TTI}(t-t')$ for large values of $t,t'$ and obtain that:
\begin{equation*}
    C_\text{TTI}(t) = \frac{1}{2 \alpha} \Big( C_Z^\text{TTI}(t) - C_Z^\text{TTI}(\infty) \Big). 
\end{equation*}
This relationship is a consequence of the expression of the covariance of $\G$ in equation \eqref{eq:CovarianceLangevin}, as well as the decomposition of $\G$ in equation \eqref{eq:DecompositionNoiseStatic}. Therefore:
\begin{equation} \label{eq:DerivativeTTI}
    \partial_{t'} C_Z(t,t') = - 2 \alpha C_\text{TTI}'(t-t').
\end{equation}
We now consider the left side of equation \eqref{eq:FDT1}. Recall equations \eqref{eq:FDTPerturbed} and \eqref{eq:ResponseLangevin} that respectively define how the perturbations $\tilde H(t)$ and $H(t)$ appear in the dynamics. Then, one can relate the associated responses by setting: 
\begin{equation*}
    \tilde H(t) = \int_0^t R(t,t') H(t') \d t'. 
\end{equation*}
This leads to the relationship:
\begin{equation} \label{eq:ResponseFDR1}
    \frac{1}{d^2} \tr \left( \left. \frac{\partial \, \E \, Z(t)}{\partial H(t')} \right|_{H = 0} \right) = \int_{t'}^t \frac{1}{d^2} \tr \left( \left. \frac{\partial \, \E \, Z(t)}{\partial \tilde H(t'')} \right|_{\tilde H = 0} \right) R(t'',t') \d t''. 
\end{equation}
Now, the left side of equation \eqref{eq:ResponseFDR1} can be reexpressed using equation \eqref{eq:ResponseLangevin} to obtain:
\begin{align}
    \delta(t-t') &= R(t,t') + \frac{1}{\alpha} \int_{t'}^t \tilde R_Z(t,t'') R(t'',t') \d t'', \\
    \tilde R_Z(t,t') &= \frac{1}{d^2} \tr \left( \left. \frac{\partial \, \E \, Z(t)}{\partial \tilde H(t')} \right|_{\tilde H = 0} \right). \label{eq:LangevinResponse2}
\end{align}
In the asymptotic TTI regime, we precisely get equation \eqref{eq:FDRGeneral} when combining this previous equation, the fluctuation--dissipation relation \eqref{eq:FDT1} and the identity \eqref{eq:DerivativeTTI}.

\subsubsection{Self-Consistent Equations} \label{App:Subsubsec:SelfConsistentLangevin}

Recall that in the dynamics \eqref{eq:DynamicsWLangevin}, the covariance of $\G$ and the response $R$ are computed self-consistently with respect to averaged quantities of $W$. This does not change the stationary measure of equation \eqref{eq:InvariantMeasureGeneral}, with $A = Z^* + \sqrt{\xi} \G$:
\begin{equation} \label{eq:InvariantmeasureDMFT}
    \P_\beta(W) \propto \exp \left( - \frac{\beta d}{1 + 4 \beta C_\text{TTI}(0)} \Big\| WW^\top - Z^* - \sqrt{\xi} \G \Big\|_F^2 - 2 \beta d \, \Omega(W) \right). 
\end{equation}
However, the scalars $\xi, C_\text{TTI}(0)$ need to be computed self-consistently from the distribution $\P_\beta$.  

\paragraph{Self-consistency of the covariances.} From the decomposition of the Gaussian processes in equation \eqref{eq:DecompositionNoiseStatic} and the expression for the covariance of $\G(t)$ in equation \eqref{eq:CovarianceLangevin}, we have the expressions:
\begin{equation} \label{eq:SelfConsistentCovLangevin}
    \xi = \frac{1}{2\alpha} \left( \mathrm{MSE} + \frac{\Delta}{2} \right), \hspace{1.5cm} C_\text{TTI}(0) = \frac{1}{2\alpha} \Big( \overline{\mathrm{MSE}} - \mathrm{MSE} \Big),
\end{equation}
where the quantities $\mathrm{MSE}, \overline{\mathrm{MSE}}$ are computed as expectations over the distribution of $W$ in equation \eqref{eq:InvariantmeasureDMFT}, the GOE matrix $\G$ and the teacher matrix $Z^*$:
\begin{equation} \label{eq:MSEsLangevin}
    \mathrm{MSE} = \frac{1}{d} \E_{\G, Z^*} \Big\| \E_\beta\big[ WW^\top \big] - Z^* \Big\|_F^2, \hspace{1.5cm} \overline{\mathrm{MSE}} = \frac{1}{d} \E_{\G, Z^*} \E_\beta \big\| WW^\top - Z^* \big\|_F^2,
\end{equation}
where $\E_\beta$ designates the average solely with respect to $\P_\beta$ in equation~\eqref{eq:InvariantmeasureDMFT}. 

\paragraph{Self-consistency of the response.} In addition, one can impose the self-consistency of the response in equation \eqref{eq:ResponseLangevin} by integrating with respect to $t'$:
\begin{equation*}
    \lim_{t \to \infty} \int_0^t R(t,t') \d t' = 1 - \frac{1}{\alpha} \lim_{t \to \infty} \int_0^t \frac{1}{d^2}\tr \left( \left. \frac{\partial \, \E \, Z(t)}{\partial H(t')} \right|_{H = 0} \right) \d t'. 
\end{equation*}
Now, it is clear that integrating the response over $\R^+$ is the same as having a time-independent perturbation $H$ (see \cref{App:Subsubsec:LongTimesResponse}). This leads to:
\begin{equation} \label{eq:ResponseInvariantMeasure}
    \lim_{t \to \infty} \int_0^t R(t,t') \d t' = 1 - \frac{1}{\alpha d^2} \tr \left( \left. \frac{\partial \, \E_H \big[WW^\top\big]}{\partial H} \right|_{H = 0} \right),
\end{equation}
when considering the perturbed distribution:
\begin{equation*}
    \P_{\beta, H}(W) \propto \exp \left( - \frac{\beta d}{1 + 4 \beta C_\text{TTI}(0)} \Big\| WW^\top - Z^* - \sqrt{\xi} \G - H \Big\|_F^2 - 2 \beta d \, \Omega(W) \right).
\end{equation*}
Now, it is possible to compute the derivative in equation \eqref{eq:ResponseInvariantMeasure}. By taking into account the fact that the normalization constant in $\P_{\beta, H}$ also depends on $H$, we obtain:
\begin{equation*}
\begin{aligned}
    \frac{1}{d^2} \tr \left( \left. \frac{\partial \, \E_H \big[WW^\top\big]}{\partial H} \right|_{H = 0} \right) = \frac{2 \beta}{1 + 4 \beta C_\text{TTI}(0)} \bigg[& \E_W \, \tr \Big( \big( WW^\top - Z^* - \sqrt{\xi} \G \big) WW^\top \Big) \\
    &- \tr \Big( \E_W\big[WW^\top] \E_W \big[WW^\top - Z^* - \sqrt{\xi} \G \big] \Big) \bigg]. 
\end{aligned}
\end{equation*}
Since the matrices $Z^*, \G$ are fixed when averaging with respect to $W$, we get that:
\begin{equation*}
    \frac{1}{d^2} \E_{\G, Z^*} \tr \left( \left. \frac{\partial \, \E_H \big[WW^\top\big]}{\partial H} \right|_{H = 0} \right) = \frac{2 \beta}{1 + 4 \beta C_\text{TTI}(0)} \Big( \overline{\mathrm{MSE}} - \mathrm{MSE} \Big). 
\end{equation*}
Therefore, with the relationship \eqref{eq:SelfConsistentCovLangevin} between $C_\text{TTI}(0)$ and $\overline{\mathrm{MSE}} - \mathrm{MSE}$, we get the identity:
\begin{equation} \label{eq:IntegratedResponseLangevin}
    \lim_{t \to \infty} \int_0^t R(t,t') \d t' = \frac{1}{1 + 4 \beta C_\text{TTI}(0)}. 
\end{equation}
It is interesting to note that this identity is consistent with the fluctuation--dissipation relation~\eqref{eq:FDTAssumption}. To see this, remark that our assumptions for the asymptotic TTI in equation \eqref{eq:ApproximateTTI} imply that:
\begin{equation*}
    \lim_{t \to \infty} \int_0^t R(t,t') \d t' = \int_0^\infty R_\text{TTI}(t') \d t'. 
\end{equation*}
Then, integrating the fluctuation--dissipation relation \eqref{eq:FDTAssumption} with respect to $t$ and using that $C_\text{TTI}(\infty) = 0$ exactly yields the relation \eqref{eq:IntegratedResponseLangevin}. This correspondence confirms the fluctuation--dissipation relation. 

\paragraph{Conclusion.} To conclude, it is easily seen that we have:
\begin{equation} \label{eq:DefVarianceLangevin}
    \overline{\mathrm{MSE}} - \mathrm{MSE} = \frac{1}{d} \E_{\G, Z^*} \E_\beta \Big\| \E_\beta \big[WW^\top] - WW^\top \Big\|_F^2 \equiv V_\beta. 
\end{equation}
Then, given the expression of the integrated response in equation \eqref{eq:IntegratedResponseLangevin} (that we call $r$ from now on), and replacing the expression of $C_\text{TTI}(0)$ by its expression in equation \eqref{eq:SelfConsistentCovLangevin}, we obtain the self-consistent equations of \cref{Result:Langevin}, along with the stationary measure in equation \eqref{eq:InvariantmeasureDMFT}. 

\subsection{Label Equation} \label{App:Subsec:LabelLangevin}

In this section, we perform a similar calculation to derive the stationary measure for the dynamics of the typical label, given in equation \eqref{eq:DMFT1Y}. To study this dynamics, we rather use the equivalent formulation given in equation \eqref{eq:EvolutionLabels}:
\begin{equation} \label{eq:LangevinY1}
    y(t) + \xi(t) - \chi_Z(t) y^* + \frac{2}{\alpha} \int_0^t \tilde R_Z(t,t') \Big( y(t') - y^* - \sqrt{\Delta} \zeta \Big) \d t' = 0,
\end{equation}
where $\tilde R_Z$ is the response function defined in equation \eqref{eq:LangevinResponse2}, $\zeta \sim \N(0,1)$ and $\xi$ is a Gaussian process with mean and covariance:
\begin{equation*}
    \E \, \xi(t) \xi(t') = 2 C_Z(t,t') - 2Q_* \chi_Z(t)\chi_Z(t') \equiv K_Z(t,t'),
\end{equation*}
and we have the statistics given in \cref{Result1}:
\begin{equation} \label{eq:AveragedLangevin}
\begin{aligned}
    C_Z(t,t') &= \frac{1}{d} \E \, \tr \big( Z(t) Z(t') \big), &\hspace{1.5cm} \chi_Z(t) = \frac{1}{Q_*} \frac{1}{d} \E \, \tr \big( Z(t) Z^* \big), \\
    Q_* &= \frac{1}{d} \E \, \tr(Z^{*2}). 
\end{aligned}
\end{equation}
Moreover, at long times, under time-translational invariance, equation \eqref{eq:FDT1} guarantees the fluctuation--dissipation relation:
\begin{equation} \label{eq:FDRlabels}
    \tilde R_Z(\tau) = - \beta K_Z'(\tau),
\end{equation}
for $\tau > 0$. In the following, we shall introduce a coupling similar to the one of \cref{App:Subsec:StationaryMeasure} that can be mapped onto the dynamics for $y$. 

\subsubsection{Gaussian Coupling}

We introduce auxiliary variables $\mathfrak{q}_1, \dots, \mathfrak{q}_k$ and the potential:
\begin{equation*}
    U_y \big( \{\mathfrak{q}_k\} \big) = \sum_{k=1}^K \theta_k \mathfrak{q}_k^2 + \frac{1}{\alpha} \left( \sum_{k=1}^K g_k \mathfrak{q}_k \right)^2,
\end{equation*}
for some constants $\theta_1, \dots, \theta_K > 0$ and $g_1, \dots, g_K \in \R$. We then consider the Langevin dynamics:
\begin{equation} \label{eq:LangevinQk}
    \d \mathfrak{q}_k(t) = - \frac{1}{2} \frac{\partial U_y}{\partial \mathfrak{q}_k} \Big( \big\{ \mathfrak{q}_k(t) \big\} \Big) \d t + \frac{1}{\sqrt{\beta}} \d B_k(t), 
\end{equation}
where $B_1, \dots, B_K$ are independent Brownian motions. Then, it is clear that the stationary measure for the process $\{\mathfrak{q}_k(t) \}$ is given by the Boltzmann--Gibbs distribution associated with $U_y$. Then, computing the partial derivative of $U_y$, we can integrate the $\mathfrak{q}_k$ to get:
\begin{equation*}
\begin{aligned}
    \mathfrak{q}_k(t) &= e^{-\theta_k t} \mathfrak{q}_k(0) - \frac{1}{\alpha} g_k \int_0^t e^{-\theta_k(t-t')} \mathfrak{q}(t') \d t' + \frac{1}{\sqrt{\beta}} \int_0^t e^{-\theta_k(t-t')} \d B_k(t'), \\
    \mathfrak{q}(t) &= \sum_{k=1}^K g_k \mathfrak{q}_k(t). 
\end{aligned}
\end{equation*}
Then, we have the expression for $\mathfrak{q}(t)$:
\begin{equation} \label{eq:qdefaging}
    \mathfrak{q}(t) = - \frac{1}{\alpha} \sum_{k=1}^K g_k^2 \int_0^t e^{-\theta_k(t-t')} \mathfrak{q}(t') \d t' + \sum_{k=1}^K g_k e^{-\theta_k t} \mathfrak{q}_k(0) + \frac{1}{\sqrt{\beta}} \sum_{k=1}^K g_k \int_0^t e^{-\theta_k(t-t')} \d B_k(t'). 
\end{equation}
We then choose the constants $g_1, \dots, g_K$ and $\theta_1, \dots, \theta_K$ so that:
\begin{equation} \label{eq:ResponseAging1}
    \tilde R_Z(t) = \frac{1}{2} \sum_{k=1}^K g_k^2 e^{-\theta_k t}. 
\end{equation}
Of course it is not obvious that we can decompose this function as a finite sum of exponentials, but as argued in \cref{App:Subsubsec:CovarianceCoupling} one can in principle take the $K \to \infty$ limit once the stationary measure is obtained. We then set:
\begin{equation} \label{eq:CovarianceAging1}
    \xi_0(t) = \sum_{k=1}^K g_k e^{-\theta_k t} \mathfrak{q}_k(0) + \frac{1}{\sqrt{\beta}} \sum_{k=1}^K g_k \int_0^t e^{-\theta_k(t-t')} \d B_k(t').  
\end{equation}
Assuming that the $\mathfrak{q}_k$ are initialized as independent Gaussians with zero mean and variance:
\begin{equation*}
    \E \big[\mathfrak{q}_k(0)^2 \big] = \frac{1}{2 \theta_k \beta},
\end{equation*}
we obtain that the covariance function of $\xi_0$ is time-translational invariant:
\begin{equation*}
    \E \, \xi_0(t) \xi_0(t') = \frac{1}{2\beta} \sum_{k=1}^K \frac{g_k^2}{\theta_k} e^{-\theta_k |t-t'|}. 
\end{equation*}
Now, as a consequence of the fluctuation--dissipation relation in equation \eqref{eq:FDRlabels}, we have for $t > 0$:
\begin{equation*}
    \partial_t \frac{1}{2\beta} \sum_{k=1}^K \frac{g_k^2}{\theta_k} e^{-\theta_k|t|} = - \frac{1}{\beta} \tilde R_Z(t) = K_Z'(t). 
\end{equation*}
Therefore, the covariances of $\xi_0$ and $\xi$ appearing in equation \eqref{eq:LangevinY1} only differ by a constant. We now define $y_\text{eff}(t) = \mathfrak{q}(t) + m$ for some constant $m$. Then, using the expressions of $\mathfrak{q}(t), \tilde R_Z^\text{TTI}(t)$ and $\xi(t)$ in equations \eqref{eq:qdefaging}, \eqref{eq:ResponseAging1}, \eqref{eq:CovarianceAging1}, we have:
\begin{equation*}
    y_\text{eff}(t) + \frac{2}{\alpha} \int_0^t \tilde R_Z(t-t') y_\text{eff}(t') \d t' - \xi_0(t) - \left(1 + \frac{2}{\alpha} r_Z(t) \right) m = 0,
\end{equation*}
where:
\begin{equation*}
    r_Z(t) = \int_0^t \tilde R_Z(s) \d s. 
\end{equation*}
In order to match the noises $\xi$ and $\xi_0$ remark that the covariance of $\xi_0$ between $t,t'$ vanishes when $t-t' \to \infty$. Then, in distribution, we can decompose $\xi(t) = \xi_\infty + \xi_0(t)$, where $\xi_\infty$ is a Gaussian variable with variance:
\begin{equation*}
    \E \, \xi_\infty^2 = \lim_{t \to \infty} K_Z(t). 
\end{equation*}
Then, at long times, we shall replace the one-time functions by their final values, and the dynamics on $y_\text{eff}$ matches equation \eqref{eq:LangevinY1} if we choose:
\begin{equation} \label{eq:meanlabelaging}
    m = \frac{1}{\alpha + 2 r_Z} \Big( \big( \alpha \chi_Z + 2 r_Z \big) y^* + \alpha \xi_\infty + 2 r_Z \sqrt{\Delta} \zeta \Big). 
\end{equation}
with $r_Z, \chi_Z$ being respectively the long-time limits of $r_Z(t), \chi_Z(t)$. Therefore, the previous calculation shows that the dynamics \eqref{eq:LangevinY1} can be mapped onto the Langevin equation \eqref{eq:LangevinQk}. Now, the stationary measure for the variables $\{\mathfrak{q}_k\}$ is given by:
\begin{equation*}
\begin{aligned}
    \P \big( \{\mathfrak{q}_k\} \big) &\propto \exp \left( - \beta \sum_{k=1}^K \theta_k \mathfrak{q}_k^2 - \frac{\beta}{\alpha} \left( \sum_{k=1}^K g_k \mathfrak{q}_k \right)^2 \right) \\
    &= \exp \left( - \beta \mathfrak{q}^\top \Theta \mathfrak{q} - \frac{\beta}{\alpha} (g^\top \mathfrak{q})^2 \right),
\end{aligned}
\end{equation*}
with $g = \big( g_1, \dots, g_K \big)^\top \in \R^K$ and $\Theta = \mathrm{diag}(\theta_1, \dots, \theta_K) \in \R^{K \times K}$. Since $y_\text{eff} = g^\top \mathfrak{q} + m$, it can be shown that $y_\text{eff}$ is Gaussian with mean $m$ and variance:
\begin{equation*}
    \mathrm{Var} \, y_\text{eff} = \frac{1}{2\beta} \frac{\alpha c}{\alpha + c}, \hspace{1.5cm} c = \sum_{k=1}^K \frac{g_k^2}{\theta_k}. 
\end{equation*}
Now, from equations \eqref{eq:FDRlabels}, \eqref{eq:ResponseAging1}, we have the identity:
\begin{equation} \label{eq:coeffsLabels}
    c = 2r_Z = 2 \beta \big(K_Z(0) - K_Z^\infty \big).
\end{equation} 

\subsubsection{Self-Consistent Equations} \label{App:Subsubsec:LabelLangevinSummary}

The previous calculation shows that we cannot write a self-consistent set of equations on the variable $y$ only. Indeed, the distribution of $y$ itself depends on averaged quantities with respect to $W$. Let us now express the quantities involved in the distribution of $y$ as averages over the distribution of $W$ in equation \eqref{eq:ResultInvariantMeasure}:
\begin{align}
    \chi_Z &= \frac{1}{Q_*} \frac{1}{d} \E_{\G, Z^*} \E_\beta \, \tr \big( WW^\top Z^* \big), \\
    K_Z^\infty &= \frac{2}{d} \E_{Z^*, \G} \, \tr \Big( \E_\beta \big[WW^\top \big]^2 \Big) - 2 Q_* \chi_Z^2, \\
    K_Z(0) - K_Z^\infty &= 2 V_\beta,
\end{align}
where $V_\beta$ is defined in equation \eqref{eq:DefVarianceLangevin}. With these notations, the stationary measure for the label dynamics is Gaussian with mean and variance (conditionally on $y^*, \xi, \zeta$):
\begin{align}
    \E \big[y \, \big| \, y^*, \xi, \zeta \big] &= \frac{1}{\alpha + 4 \beta V_\beta} \Big( \big( \alpha \chi_Z + 4 \beta V_\beta \big) y^* + \alpha \sqrt{K_Z^\infty} \xi + 4 \beta V_\beta \sqrt{\Delta} \zeta \Big), \label{eq:LabelLangevinMean} \\
    \mathrm{Var} \big( y \, \big| \, y^*, \xi, \zeta \big) &= \frac{2\alpha V_\beta}{\alpha + 4 \beta V_\beta}, \label{eq:LabelLangevinVar}
\end{align}
where $y^* \sim \N(0, 2Q_*)$ and $\xi, \zeta \sim \N(0,1)$ are independent. 

\subsection{Zero-Temperature Limit} \label{App:Subsec:ZeroTemperature}

In this part, we focus on the stationary distribution for the matrix $W$, and take the $\beta \to \infty$ limit with the choice of the $\ell_2$-regularization $\Omega(W) = \lambda \tr(WW^\top)$. This choice allows to compare these results with the ones obtained in the gradient flow setting, mentioned in \cref{Subsec:LongTimes}. 

We can rewrite the stationary distribution for $W$ in equation \eqref{eq:InvariantmeasureDMFT}:
\begin{equation*}
    \P_\beta(W) \propto \exp \left( - \beta r_\beta d \, \Big\| WW^\top - Z^* - \sqrt{\xi} \G + q I_d \Big\|_F^2 \right),
\end{equation*}
where $q = \lambda / r_\beta$ and $r_\beta$ is the integrated response given in equation \eqref{eq:IntegratedResponseLangevin}:
\begin{equation} \label{eq:IntegratedResponseBeta}
    r_\beta = \int_0^\infty R_\text{TTI}(t') \d t' = \frac{\alpha}{\alpha + 2 \beta V_\beta}.
\end{equation}

\subsubsection{Positive Regularization}

When $\lambda > 0$, the expression of $q$ indicates that the integrated response $r_\beta$ remains of order one. Therefore, the distribution $\P_\beta$ collapses onto the set:
\begin{equation*}
    \underset{W \in \R^{d \times m}}{\mathrm{argmin}} \Big\| WW^\top - Z^* - \sqrt{\xi} \G + q I_d \Big\|_F^2. 
\end{equation*}
As a consequence of \cref{Lemma:MinEckartYoung}, this is precisely the set of the $W \in \R^{d \times m}$ such that:
\begin{equation} \label{eq:LimitDeterministicLangevin}
    WW^\top = \Big( Z^* + \sqrt{\xi} \G - q I_d \Big)_{(m)}^+. 
\end{equation}
The spectral operator $X \mapsto X_{(m)}^+$ selects the $m$ largest positive eigenvalues of $X$. We refer to \cref{Def:EckartYoung} for more details. Since $WW^\top$ is now a deterministic function of the matrices $Z^*, \G$, the two quantities $\mathrm{MSE}$ and $\overline{\mathrm{MSE}}$ (defined in equation \ref{eq:MSEsLangevin}) collapse onto the same value. Combining the definition of $\xi$, the expression of response and the limit of the dynamics, respectively in equations \eqref{eq:SelfConsistentCovLangevin}, \eqref{eq:ResponseInvariantMeasure} and \eqref{eq:LimitDeterministicLangevin}, we arrive at the same system of equations as in \cref{App:Subsubsec:LongTimesLimitDynamics,App:Subsubsec:LongTimesResponse}, that eventually yields \cref{Result2}. 

This suggests that in the regularized case, the zero temperature limit of the long-time equations for the Langevin dynamics matches the gradient flow setting, which is a non-trivial result. Indeed, in the zero-temperature limit, the stationary measure of the Langevin dynamics collapses onto the set of global minimizers of the regularized empirical loss. This agreement implies that, under our dynamical assumptions, the gradient flow also converges to a global minimizer of the loss. 

\subsubsection{Unregularized Dynamics}

We now consider the unregularized case $\lambda = 0$. In this setting, it is now possible for $r_\beta$ to vanish as $\beta \to \infty$. 

Let us start by considering the case where $r_\beta$ remains of order one as $\beta \to \infty$. Then, the same argument as earlier applies: the distribution $\P_\beta$ collapses, and we recover the same equations as in \cref{Result2} with $\lambda, q = 0$. Interestingly, this leads to the same result as in the second part of \cref{Prop:EquationsSmallReg} (corresponding to $\alpha$ larger than the interpolation threshold). As the calculation remains identical as in the gradient flow setting, we can conclude that:
\begin{itemize}
    \item Combining equation \eqref{eq:ResponseInvariantMeasure} with the calculation carried out in \cref{App:Subsubsec:LongTimesResponse} leads to the same set of equations characterizing the interpolation thresholds as in \cref{Prop:InterpolationThreshold}. 
    \item Since $r_\beta$ is precisely the integrated response of equation \eqref{eq:ResponseInvariantMeasure}, we obtain the same characterization of the interpolation threshold as in \cref{Subsubsec:InterpolationThreshold}: it is the smallest value of $\alpha$ for which the integrated response remains positive in the small-regularization limit. 
\end{itemize}
In this regime, it is interesting to note the agreement between the unregularized Langevin dynamics in the zero-temperature limit and the small regularization limit of the gradient flow dynamics. Moreover, as discussed in \cref{conjecture:unregularized}, we expect these results to extend to the unregularized gradient flow itself. The regime of large sample complexity $\alpha$ therefore appears to be associated with a simple landscape, in which the different dynamics we studied converge to the same solution.   

In the case where $r_\beta$ vanishes as $\beta \to \infty$, equation \eqref{eq:IntegratedResponseBeta} yields:
\begin{equation*}
    \beta r_\beta \xrightarrow[\beta \to \infty]{} \frac{\alpha}{2V},
\end{equation*}
where $V$ is the positive limit of the variance $V_\beta$. Note that the intermediate scalings of $r_\beta$ with $\beta$ would also lead to the collapse of $\P_\beta$. When $r_\beta = \Theta(\beta^{-1})$, the distribution $\P_\beta$ does not concentrate:
\begin{equation*}
    \lim_{\beta \to \infty} \P_\beta(W) \propto \exp \left( - \frac{\alpha d}{2V} \Big\| WW^\top - Z^* - \sqrt{\xi} \G \Big\|_F^2 \right). 
\end{equation*}
The scalars $\xi, V$ are self-consistently computed from $W$ thanks to equations \eqref{eq:SelfConsistentCovLangevin}, \eqref{eq:DefVarianceLangevin}. This distribution describes the zero-temperature limit for $\alpha \leq \alpha_\text{inter}$. In this regime we know that the empirical loss \eqref{eq:Loss} exhibits many global minimizers, so the random nature of the Langevin predictor is no surprise. 

This observation is confirmed by the statistics of the typical label in equations \eqref{eq:LabelLangevinMean}, \eqref{eq:LabelLangevinVar}. Taking the $\beta \to \infty$ limit in these equations, with $V_\beta = \Theta(1)$, leads to:
\begin{equation*}
    \E \, y = y^* + \sqrt{\Delta} \zeta, \hspace{1.5cm} \mathrm{Var} \, y = 0. 
\end{equation*}
In this region the noisy labels are perfectly fitted. 

However, this result relies on the asymptotic TTI assumption (see \cref{Assumption4:TTI}), whose validity in this regime remains unclear. In particular, as shown numerically in \cref{fig:FigInit1}, the gradient flow dynamics displays a strong dependence on initialization, even at long times. Such a dependence is incompatible with a TTI structure of the dynamics. Therefore, it remains an open question whether the Langevin dynamics satisfies the asymptotic TTI assumption in the small $\alpha$ regime. Overall, these observations suggest that, in this regime, the zero-temperature limit of the stationary measure of the Langevin dynamics does not coincide with the long-time asymptotics of the gradient flow dynamics.

\subsection{Link with Bayes-Optimal Learning} \label{App:Subsec:BayesOptimal}

In this part, we make the connection with the Bayes-optimal analysis of \citet{maillard2024bayes}. To do so, we consider the stationary measure associated with our Langevin dynamics, given in equation~\eqref{eq:ResultInvariantMeasure}:
\begin{equation*}
    \P_\beta(W) \propto \exp \left( - r\beta d \, \Big\|WW^\top - Z^* - \sqrt{\xi} \G \Big\|_F^2 -2 \beta d \, \Omega(W) \right). 
\end{equation*}
From a Bayesian perspective, this distribution can be interpreted as a posterior measure over $W$, that includes a prior distribution:
\begin{equation*}
    \P_\text{prior}(W) \propto \exp \Big( - 2 \beta d \, \Omega(W) \Big). 
\end{equation*}
In \citet{maillard2024bayes}, the prior on the weights $W$ is chosen to be Gaussian, inducing a Wishart distribution for $WW^\top$. In our framework, this corresponds to a $\ell_2$-regularization:
\begin{equation*}
    \Omega(W) = \frac{\kappa}{4\beta} \tr \big( WW^\top \big). 
\end{equation*}
We recall that $\kappa$ is the width ratio $m/d$ as $d \to \infty$. The Bayes-optimal setting corresponds to the matched case in which the inference model coincides with the true model. In particular, the width of the student must match the one of the teacher, i.e., $\kappa = \kappa^*$. 

In the following, we show two properties:
\begin{itemize}
    \item We prove that a specific choice of the inverse temperature $\beta$ allows to match our setting with the one of \citet{maillard2024bayes}. This leads to the relationship $\beta = \alpha / \Delta$. We recall that $\Delta$ corresponds to the label noise variance when generating the teacher's labels. 
    \item Using some known results in the Bayes-optimal setting, we show that the following relationship should hold:
    \begin{equation} \label{eq:Factor2MSEs}
        \mathrm{MSE} = \frac{1}{2} \overline{\mathrm{MSE}}. 
    \end{equation}
\end{itemize}
We recall that these two quantities are defined in equation \eqref{eq:MSEsLangevin}. Combining these two identities with the ones obtained in \cref{App:Subsubsec:SelfConsistentLangevin}:
\begin{equation*}
    \xi = \frac{1}{2\alpha} \left( \mathrm{MSE} + \frac{\Delta}{2} \right), \hspace{1.5cm} r = \frac{\alpha}{\alpha + 2 \beta \big(\overline{\mathrm{MSE}} - \mathrm{MSE} \big)},
\end{equation*}
we obtain the simple identity:
\begin{equation*}
    r = \frac{1}{4 \beta \xi}.
\end{equation*}
This leads to the posterior distribution:
\begin{equation*}
    \P_\beta(W) \propto \exp \left( - \frac{d}{4\xi} \big\| WW^\top - Z^* - \sqrt{\xi} \G \big\|_F^2 \right) \P_\text{prior}(W). 
\end{equation*}
Therefore, with the choice $\hat q = \xi^{-1}$, we obtain the same posterior distribution for $W$ as in \citet[][equation 46]{maillard2024bayes}. In addition, up to a normalization convention, our self-consistent relation between $\xi$ and the MSE coincides with their equation (7). 

\subsubsection{The Bayes-Optimal Temperature}

We first fix the value of the inverse temperature $\beta$ to match our stationary measure with the Bayes-optimal posterior distribution, at the level of the empirical model. In \citet{maillard2024bayes}, the Bayes posterior distribution over $W$ is given by:
\begin{equation} \label{eq:BO2}
    \P_\text{Bayes}(W) \propto \P_\text{prior}(W) \prod_{k=1}^n P\Big( z_k \, \big| \, \tr(X_k WW^\top) \Big). 
\end{equation}
As we have chosen the square loss in our analysis, we assume this noisy channel to be Gaussian with zero mean and variance $\Delta$:
\begin{equation*}
    P\big( z \, \big| \, y \big) = \frac{1}{\sqrt{2\pi \Delta}} \exp \left( - \frac{(y-z)^2}{2\Delta} \right).
\end{equation*}
Then, the noisy labels can be written as $z_k = \tr(X_k Z^*) + \sqrt{\Delta} \xi_k$, where $\xi_1, \dots, \xi_n \overset{\text{i.i.d.}}{\sim} \N(0,1)$. On our side, the Langevin dynamics in equation \eqref{eq:GFdynamics} can be rewritten as:
\begin{equation} \label{eq:BOLangevin}
    \d W(t) = -d \, \nabla_W\left( \frac{1}{4n} \sum_{k=1}^n \Big( \tr(X_k WW^\top) - z_k \Big)^2 \right) \d t - \nabla \Omega \big( W(t) \big) \d t + \frac{1}{\sqrt{\beta d}} \d B(t).
\end{equation}
Then, the stationary measure for this dynamics writes:
\begin{equation*}
    \P_\beta(W) \propto \exp \left( - \frac{\beta d^2}{2n} \sum_{k=1}^n \Big( \tr(X_k WW^\top) - z_k \Big)^2 - 2 \beta d \, \Omega(W) \right). 
\end{equation*}
The regularization term $\Omega$ has already been chosen to match the prior distribution. Therefore, to identify $\P_\beta$ with the Bayes posterior $\P_\text{Bayes}$ in equation \eqref{eq:BO2}, it suffices to match the empirical terms. This leads to the expression of the inverse temperature:
\begin{equation*}
    \beta = \frac{\alpha}{\Delta}. 
\end{equation*}
With this choice of $\beta$, the stationary measure of the Langevin dynamics coincides with the Bayes-optimal posterior distribution. This provides an \textit{a priori} guarantee that the Langevin dynamics at long times samples from the Bayes-optimal posterior measure. 

\subsubsection{The Nishimori Condition}

In the Bayes-optimal setting, with a planted teacher $Z^*$, the following relationship holds:
\begin{equation} \label{eq:NishimoriCondition}
    \E_{W_1, W_2} \, f \big( W_1W_1^\top, W_2W_2^\top \big) = \E_{W, Z^*} f \big( WW^\top, Z^* \big),
\end{equation}
where $W_1, W_2, W$ are drawn from the Bayes-optimal posterior distribution, with $W_1$ independent of $W_2$. This identity is known as the Nishimori condition, and a general proof of this fact can be found in \citet[][Section I.B.3]{zdeborova2016statistical}.

Let us apply this identity to relate the quantities $\mathrm{MSE}$ and $\overline{\mathrm{MSE}}$. For simplicity, we denote by $\E_W$ the expectation with respect to $\P_\beta$, and $\E$ the joint expectation over the random variables $W, \G, Z^*$. Denoting by $\hat Z$ the posterior mean $\E_W\big[WW^\top\big]$, we have:
\begin{align}
    \overline{\mathrm{MSE}} &= \frac{1}{d} \E \big\| WW^\top - Z^* \big\|_F^2 \notag \\
    &= \frac{1}{d} \E \big\|WW^\top - \hat Z \big\|_F^2 + \frac{1}{d} \E \big\| \hat Z - Z^* \big\|_F^2 - \frac{2}{d} \E_W \, \tr \Big( \big(WW^\top - \hat Z \big) \big( Z^* - \hat Z \big) \Big) \notag \\
    &= \frac{1}{d} \E \big\|WW^\top - \hat Z \big\|_F^2 + \mathrm{MSE}.  \label{eq:MSEsNishimori}
\end{align}
It is clear that the crossed term vanishes. Now, if $W_1, W_2$ are independently drawn from the posterior distribution, one has, inserting again $\hat Z$:
\begin{equation*}
\begin{aligned}
    \frac{1}{d} \E_{W_1, W_2} \big\| W_1W_1^\top - W_2W_2^\top \big\|_F^2 &= \frac{1}{d} \E_{W_1} \big\|W_1W_1^\top - \hat Z \big\|_F^2 + \frac{1}{d} \E_{W_2} \big\|W_2W_2^\top - \hat Z \big\|_F^2 \\
    &\hspace{2cm} - \frac{2}{d} \E_{W_1, W_2} \tr \Big( \big( W_1W_1^\top - \hat Z \big) \big( W_2W_2^\top - \hat Z \big) \Big) \\
    &= \frac{2}{d} \E_W \big\|WW^\top - \hat Z \big\|_F^2. 
\end{aligned}
\end{equation*}
The last term vanishes since $W_1, W_2$ are independent and $\E_{W_1} \, W_1W_1^\top = \E_{W_2} \, W_2W_2^\top = \hat Z$. Then using the Nishimori condition in equation \eqref{eq:NishimoriCondition} with $f(Z_1, Z_2) = \|Z_1 - Z_2\|_F^2$, this yields:
\begin{equation*}
    \frac{1}{d} \E \big\| WW^\top - \hat Z \big\|_F^2 = \frac{1}{2} \overline{\mathrm{MSE}}. 
\end{equation*}
Plugging this identity into equation \eqref{eq:MSEsNishimori} leads to the relationship \eqref{eq:Factor2MSEs}.

\newpage
\section{Aging Ansatz} \label{App:Aging}
In this section, we go beyond the assumptions made in \cref{Subsec:LongTimes} and introduce a more general ansatz on the dynamics, that is usually referred to as aging. This framework is designed to capture situations in which the system does not reach equilibrium on the timescales of interest, and where relaxation becomes increasingly slow as time evolves. In this regime, memory effects remain and the system keeps evolving without settling into a stationary state. Additionally, correlations decay on timescales comparable to the age of the system. This behavior is characteristic of glassy and mean-field systems, and has been extensively studied in the context of dynamical mean-field theory and spin-glass dynamics \citep{bouchaud1992weak, cugliandolo1993analytical, cugliandolo1994out, bouchaud1998out}. Additionally, our calculation follows the same steps than of \citet{altieri2020dynamical} and uses results that can be found in \citet{cugliandolo2000scenario}. 

In what follows, we decompose the stochastic equations of \cref{Result1} into two timescales: a fast converging part that verifies time-translational invariance (for which results have been derived in \cref{App:Langevin}), and an aging component that captures the slow relaxation of the dynamics. Our main assumptions for the section are:
\begin{itemize}
    \item In the TTI regime, the slow variables evolving on aging timescales can be considered as effectively frozen parameters. In this regime, the system is locally at equilibrium at inverse temperature $\beta$, and the corresponding correlations and response functions satisfy a fluctuation--dissipation relation with inverse temperature $\beta$. 
    \item For time separations comparable to the age of the system, correlations and responses enter an aging regime characterized by a violation of time-translational invariance. In this regime, we assume that a generalized fluctuation--dissipation relation holds, with an effective temperature $\beta_e < \beta$. We give more details on this parameter in \cref{App:Subsec:EffectiveTemperature}. 
\end{itemize}

In order to derive the aging equations (stated in \cref{App:Subsec:SummaryAging}), we perform a finite-temperature analysis in \cref{App:Subsec:DerivationAging} under the two-timescale ansatz. The self-consistent equations are then obtained by taking the zero-temperature limit (\cref{App:Subsubsec:ZeroTempAging}). We conclude with a discussion of the effective temperature $\beta_e$ in \cref{App:Subsec:EffectiveTemperature} and whether it can be imposed at a finite value.

\subsection{Summary of the Result} \label{App:Subsec:SummaryAging}

In the following, we derive a set of self-consistent equations describing the long-time limit, under the aging ansatz, of the system given in \cref{Result1}, in the case of $\ell_2$-regularization and zero label noise ($\Delta = 0$). While we derive general equations at finite temperature $(\beta < \infty)$ in \cref{App:Subsec:DerivationAging}, we simplify them in the zero-temperature limit ($\beta \to \infty$) in \cref{App:Subsubsec:ZeroTempAging}. In this limit, we show that the self-consistent equations involve a random matrix $X \in \mathcal{S}_d(\R)$ whose distribution is given by:
\begin{align}
    \P(X) &\propto \exp \left( - d \left( 2r \beta_e + \frac{\alpha}{V_Z} \right) \|X\|_F^2 - \alpha r \beta_e d \, \big\| X_{(m)}^+ \big\|_F^2 + 2 r \beta_e d \, \tr \big( M^* X \big) \right), \label{eq:DistribAging1} \\
    M^* &= \alpha Z^* + \left( 2 + \frac{\alpha}{r \beta_e V_Z} \right) \Big( Z^* + \sigma \G - q I_d.  \Big), \label{eq:TargetMatrixAging}
\end{align}
where $\G \sim \mathrm{GOE}(d)$ and $q = \lambda / r$. The variables $r, V_Z, \sigma$ can be self-consistently computed as:
\begin{align}
    r &= 1 - \frac{1}{\alpha d^2} \E_{Z^*, \G} \E_X \tr \left( \left. \frac{\partial}{\partial H} \big( X + H \big)_{(m)}^+ \right|_{H = 0} \right), \label{eq:EqsAgingResponse} \\ 
    V_Z &= \frac{2}{d} \E_{Z^*, \G} \E_X \Big\| \E_X \big[ X_{(m)}^+ \big] - X_{(m)}^+ \Big\|_F^2, \label{eq:EqsAgingCovariance1} \\
    \sigma^2 &= \frac{\alpha}{2} \frac{1}{(2r \beta_e V_Z + \alpha)^2} \frac{1}{d} \E_{Z^*, \G} \Big\| \E_X \big[ X_{(m)}^+ \big] - Z^* \Big\|_F^2. \label{eq:EqsAgingCovariance2}
\end{align}
The notation $X_{(m)}^+$ is defined in \cref{Def:EckartYoung}, and the partial derivative corresponds to the differential on the space of symmetric matrices. Note that, in the high-dimensional limit, the limit of $r$ can be computed by using similar arguments as in \cref{App:Subsubsec:LongTimesResponse}. However, these equations still involve high-dimensional objects, and it seems challenging to analyze the distribution \eqref{eq:DistribAging1} in the high-dimensional limit. We leave this investigation for future work. 

The only quantity that is not self-consistently determined is the effective inverse temperature $\beta_e$. This parameter is fixed by imposing a marginality condition on the TTI part of the dynamics. We give more details in \cref{App:Subsec:EffectiveTemperature}. 

Finally, the gradient flow predictor is simply given by $Z_\infty = X_{(m)}^+$. This allows, in addition to a set of equations on the typical label that is derived in the following, to access all the relevant averaged quantities of the dynamics in the long-time limit. 

\subsection{Derivation of the Aging Equations} \label{App:Subsec:DerivationAging}

This section is dedicated to the derivation of the system of equations given in \cref{App:Subsec:SummaryAging}. 

\subsubsection{Dynamical Equations}

We start by giving the set of self-consistent equations that we start from in order to derive the aging equations. Instead of directly using those of \cref{Result1}, we go back to \cref{App:Subsec:DMFTSimplification}. In this section, it has been shown that the evolution of the student matrix $W(t)$ and typical label $y(t)$ can be written as:
\begin{align}
    \d W(t) &= 2 \left( \H(t) + \int_0^t R_y(t,t') \big( Z^* - Z(t') \big) \d t' \right) W(t) \d t - \nabla \Omega \big( W(t) \big) \d t + \frac{1}{\sqrt{\beta d}} \d B(t), \label{eq:aging1W} \\
    0 &= y(t) - \chi_Z(t) y^* + \xi(t) + \frac{2}{\alpha} \int_0^t \tilde R_Z(t,t') \big( y(t') - y^* \big) \d t', \label{eq:aging1y}
\end{align}
where $Z(t) = W(t)W(t)^\top$ and $\H(t)$ and $\xi(t)$ are independent centered Gaussian processes with covariances:
\begin{align}
    \E \, \H_{ij}(t) \H_{i'j'}(t') &= \frac{1}{2 \alpha d} \big( \delta_{ii'} \delta_{jj'} + \delta_{ij'} \delta_{i'j} \big) \int_0^t \int_0^{t'} R_y(t,s) R_y(t',s') \mathcal{M}_Z(s,s') \d s' \d s, \label{eq:CovHAging} \\
    \mathcal{M}_Z(s,s') &=  \frac{1}{d} \E \, \tr \Big[ \big( Z(s) - Z^* \big) \big( Z(s') - Z^* \big) \Big], \\
    \E \, \xi(t) \xi(t') &= 2 C_Z(t,t') - 2 Q_* \chi_Z(t) \chi_Z(t'), \label{eq:CovxiAging} 
\end{align}
and:
\begin{equation} \label{eq:AveragedAging}
\begin{aligned}
    C_Z(t,t') &= \frac{1}{d} \E \, \tr \big( Z(t) Z(t') \big), &\hspace{1.5cm} \xi_Z(t) = \frac{1}{Q_*} \frac{1}{d} \E \, \tr \big( Z(t) Z^* \big), \\
    Q_* &= \frac{1}{d} \E \, \tr(Z^{*2}). 
\end{aligned}
\end{equation}
The responses $R_Z, \tilde R_Z$ are defined as:
\begin{equation} \label{eq:ResponsesAging}
    R_Z(t,t') = \frac{1}{d^2} \tr \left( \left. \frac{\partial \, \E \, Z(t)}{\partial H(t')} \right|_{H = 0} \right), \hspace{1.5cm} \tilde R_Z(t,t') = \frac{1}{d^2} \tr \left( \left. \frac{\partial \, \E \, Z(t)}{\partial \tilde H(t')} \right|_{H = 0} \right),
\end{equation}
in response to respective perturbations $H, \tilde H$ in the dynamics \eqref{eq:aging1W}:
\begin{equation*}
    \d W(t) = \dots + 2 \int_0^t R_y(t,t') H(t') \, \d t' W(t) \d t, \hspace{1.5cm} \d W(t) = \dots + \tilde H(t) W(t) \d t. 
\end{equation*}
Likewise, $R_y$ is the response associated with $y$, that can be written:
\begin{equation*}
    R_y(t,t') = - \frac{\partial \, \E \, y(t)}{\partial \xi(t')}. 
\end{equation*}
This leads to the relationships between $R_y, R_Z, \tilde R_Z$:
\begin{equation} \label{eq:RelationResponses}
\begin{aligned}
    \delta(t-t') &= R_y(t,t') + \frac{1}{\alpha} R_Z(t,t'), \\
    \delta(t-t') &= R_y(t,t') + \frac{2}{\alpha} \int_{t'}^t \tilde R_Z(t,t'') R_y(t'',t') \d t''.
\end{aligned}
\end{equation}
Finally, one can derive the following relationship, starting from equation \eqref{eq:aging1y} and using equations \eqref{eq:CovHAging}, \eqref{eq:CovxiAging}, \eqref{eq:AveragedAging}:
\begin{equation} \label{eq:CovarianceHY}
    \E \, \H_{ij}(t) \H_{i'j'}(t') = \frac{1}{4\alpha d} \big( \delta_{ii'} \delta_{jj'} + \delta_{ij'} \delta_{i'j} \big) \E \Big[ \big( y(t) - y^* \big) \big( y(t') - y^* \big) \Big]. 
\end{equation}
In the following we shall denote $C_y$ the covariance function of $y-y^*$. 

\subsubsection{Timescale Decomposition and Fluctuation--Dissipation}

We now start from the previous set of equations and decompose the responses and correlations depending on two timescales, that we call time-translational invariant (TTI) and aging. For the responses, we write:
\begin{equation} \label{eq:TTIAgingResponse}
\begin{aligned}
    R_y(t,t') &= R_y^\text{TTI}(t-t') + R_y^A(t,t'), \\ 
    \tilde R_Z(t,t') &= \tilde R_Z^\text{TTI}(t-t') + \tilde R_Z^A(t,t'),
\end{aligned}
\end{equation}
and likewise for the correlations:
\begin{equation} \label{eq:TTIAgingCovariance}
\begin{aligned}
    C_y(t,t') &= C_y^\text{TTI}(t-t') + C_y^A(t,t'), \\ 
    K_Z(t,t') &= K_Z^\text{TTI}(t-t') + K_Z^A(t,t'),
\end{aligned}
\end{equation}
where $K_Z$ is the covariance of $\xi$ given in equation \eqref{eq:CovxiAging}. As we will be interested in the $t,t' \to \infty$ limit, one-time functions can be replaced by their limit, which yields in particular $K_Z(t,t') = 2 C_Z(t,t') - 2 Q_* (\chi_Z^\infty)^2$. In the previous equations, the TTI contributions vary when $t-t'$ is of order one, whereas the aging regime corresponds to $t-t'$ diverging. In this regime we choose the TTI covariances to be zero, that is:
\begin{equation*}
     C_y^\text{TTI}(\tau) \xrightarrow[\tau \to \infty]{} 0, \hspace{1.5cm} K_Z^\text{TTI}(\tau) \xrightarrow[\tau \to \infty]{} 0. 
\end{equation*}
In addition, equation \eqref{eq:TTIAgingCovariance} on the covariances induces the decomposition of $\xi$ and $\H$ as sums of independent Gaussian processes:
\begin{equation*}
    \xi(t) = \xi_\text{TTI}(t) + \xi_A(t), \hspace{1.5cm} \H(t) = \H_\text{TTI}(t) + \H_A(t),
\end{equation*}
where $\xi_\text{TTI}$ and $\xi_A$ have respective covariances $K_Z^\text{TTI}$ and $K_Z^A$. The same holds for $\H_\text{TTI}$ and $\H_A$, with the additional normalization:
\begin{align} 
    \E \, \big( \H_\text{TTI}(t) \big)_{ij} \big( \H_\text{TTI}(t') \big)_{i'j'} &= \frac{1}{4\alpha d} \big( \delta_{ii'} \delta_{jj'} + \delta_{ij'} \delta_{i'j} \big) C_y^\text{TTI}(t-t'), \label{eq:CovarianceWTTI} \\
    \E \, \big( \H_A(t) \big)_{ij} \big( \H_A(t') \big)_{i'j'} &= \frac{1}{4\alpha d} \big( \delta_{ii'} \delta_{jj'} + \delta_{ij'} \delta_{i'j} \big) C_y^A(t-t'),\label{eq:CovarianceWAging}
\end{align}
where we used the covariance structure of $\H$ in equation \eqref{eq:CovarianceHY}. 

In addition, we assume the fluctuation--dissipation relations in the TTI regime:
\begin{equation} \label{eq:FDRTTIregime}
    R_y^\text{TTI}(\tau) = - \beta \big(C_y^\text{TTI} \big)'(\tau), \hspace{1.5cm} \tilde R_Z^\text{TTI}(\tau) = - \beta \big( K_Z^\text{TTI} \big)'(\tau). 
\end{equation}
While the second expression is a consequence of the calculation carried out in \cref{App:Subsubsec:FDR} (see in particular equation \ref{eq:FDT1}), along with the definition of $\tilde R_Z, C_Z$ and $K_Z$, the first one can be derived in a similar fashion. 

In the aging regime, fluctuation--dissipation does not hold at inverse temperature $\beta$, but with an effective inverse temperature, that we denote $\beta_e$:
\begin{equation} \label{eq:FDRAgingregime}
    R_y^A(t,t') = \beta_e \partial_{t'} C_y^A(t,t'), \hspace{1.5cm} \tilde R_Z^A(t,t') = \beta_e \partial_{t'} K_Z^A(t,t'). 
\end{equation}

\subsubsection{TTI Regime}

Using the previous decompositions, we can write the joint dynamics on $W(t)$ and $y(t)$ given in equations \eqref{eq:aging1W}, \eqref{eq:aging1y}:
\begin{align}
    \d W(t) &= 2 \left( \mathcal{H}_\text{TTI}(t) + \int_0^t R_y^\text{TTI}(t-t') \big( Z^* - Z(t') \big) \d t' \right) W(t) \d t - \nabla \Omega \big( W(t) \big) \d t \label{eq:aging2W} \\
    &\hspace{2cm}+ 2 \Psi(t) W(t) \d t + \frac{1}{\sqrt{\beta d}} \d B(t) \notag, \\
    0 &= y(t) - \chi_Z(t) y^* + \xi_\text{TTI}(t) + \frac{2}{\alpha} \int_0^t \tilde R_Z^\text{TTI}(t - t') \big( y(t') - y^* \big) \d t' + h(t), \label{eq:aging2y}
\end{align}
where the slow fields $\Psi(t), h(t)$ are given by:
\begin{align}
    \Psi(t) &= \H_A(t) + \int_0^t R_y^A(t,t') \big( Z^* - Z(t') \big) \d t', \label{eq:slowfieldW} \\
    h(t) &= \xi_A(t) + \frac{2}{\alpha} \int_0^t \tilde R_Z^A(t,t') \big( y(t') - y^* \big) \d t'. \label{eq:slowfieldy}
\end{align}
The key point is to consider the slow fields $\Psi(t), h(t)$ as frozen in the TTI regime, that corresponds to timescales of order one. On the other hand, we assume that the slow aging variables typically vary on a much larger timescale, comparable to the age of the system. 

\paragraph{Student equation.} Let us start by analyzing the dynamics of $W(t)$ in equation \eqref{eq:aging2W}. In the TTI regime, this equation can be written exactly as the one studied in \cref{App:Subsec:StationaryMeasure}, with an additional gradient term involving $\Psi$. Since the covariance of $\mathcal{H}_\text{TTI}$ between instants $t,t'$ vanishes as $t-t' \to \infty$, this implies the stationary distribution, conditionally on $\Psi(t)$:
\begin{equation} \label{eq:StationaryTTIW}
    \P_W^\text{TTI} \big( W \, | \, \Psi(t) \big) \propto \exp \left( - r \beta d \, \big\| WW^\top - Z^* \big\|_F^2 - 2 \beta d \, \Omega(W) + 2 \beta d \, \tr \Big( \Psi(t) \big( WW^\top - Z^* \big) \Big) \right).     
\end{equation}
In addition, $r$ is given by:
\begin{equation} \label{eq:IntegratedResponse}
    r = \int_0^\infty R_y^\text{TTI}(\tau) \d \tau = \beta C_y^\text{TTI}(0), 
\end{equation}
where we used the fluctuation--dissipation relation for $y$ in equation \eqref{eq:FDRTTIregime}. 

\paragraph{Label equation.} To compute the stationary distribution of the label in the TTI regime, we introduce the same coupling as the one studied in \cref{App:Subsec:LabelLangevin}. We obtain that conditionally on the slow field $h(t)$, the typical label $y$ is Gaussian, with statistics:
\begin{equation*}
\begin{aligned}
    \E \, y &= \frac{1}{\alpha + 2r_Z} \Big( \big( \alpha \chi_Z + 2 r_Z \big) y^* - \alpha h(t) \Big), \\
    \mathrm{Var} \, y &= \frac{1}{\beta} \frac{\alpha r_Z}{\alpha + 2r_Z}, 
\end{aligned}
\end{equation*}
where $\chi_Z$ is the limit of $\chi_Z(t)$ as $t \to \infty$ and we have from equation \eqref{eq:FDRTTIregime}:
\begin{equation*}
    r_Z = \int_0^\infty \tilde R_Z^\text{TTI}(t) \d t = \beta K_Z^\text{TTI}(0). 
\end{equation*}
Therefore, we can write the distribution of $y$ in the TTI regime as:
\begin{equation} \label{eq:StationaryTTIy}
    \P_y^\text{TTI} \big( y \, | \, h(t) \big) \propto \exp \left( - \frac{\beta}{2 \alpha r_Z} \Big( (\alpha + 2r_Z) y^2 - 2(\alpha \chi_Z + 2 r_Z) yy^* + 2 \alpha h(t) (y-y^*) \Big) \right). 
\end{equation}

\subsubsection{An Auxiliary Coupling} \label{App:Subsubsec:CouplingAging}

In order to study the aging regime and derive the stationary distribution of the slow variables $\Psi(t), h(t)$, defined in equations \eqref{eq:slowfieldW}, \eqref{eq:slowfieldy}, we introduce an auxiliary coupling. This coupling is inspired by the one introduced by \citet{cugliandolo2000scenario} in the context of aging in spin-glass models. 

We consider a smooth function $F \colon \mathcal{S}_d(\R) \to \R$, a fixed matrix $A \in \mathcal{S}_d(\R)$, and define the potential on the variables $\Q_1, \dots, \Q_K \in \mathcal{S}_d(\R)$:
\begin{equation*}
    U_\text{aging}\big( \{\Q_k\} \big) = \rho F \left( A + \sum_{k=1}^K g_k \Q_k \right) + \sum_{k=1}^K \theta_k \|\Q_k\|_F^2.
\end{equation*}
We then study the associated Langevin dynamics:
\begin{equation} \label{eq:LangevinCouplingAging}
    \d \Q_k(t) = - \frac{\rho g_k}{2} \nabla F \left( A + \sum_{j=1}^K g_j \Q_j(t) \right) \d t - \theta_k \Q_k(t) \d t + \frac{1}{\sqrt{\beta_e d}} \d \Xi_k(t),
\end{equation}
where $\Xi_1, \dots, \Xi_k$ are independent standard Brownian motions over $\mathcal{S}_d(\R)$. The presence of $\beta_e$ is motivated by the assumptions of the section: in the aging regime, correlations and responses obey a generalized fluctuation--dissipation relation with effective temperature $\beta_e$. The stationary measure associated with the dynamics \eqref{eq:LangevinCouplingAging} is given by:
\begin{equation} \label{eq:stationaryAging1}
    \P_\text{aging} \big( \{\Q_k\} \big)  \propto \exp \Big( - \beta_e d \, U_\text{aging} \big( \{\Q_k\} \big)  \Big). 
\end{equation}
Moreover, we can integrate the dynamics \eqref{eq:LangevinCouplingAging} to get:
\begin{equation*}
    \Q_k(t) = e^{-\theta_k t} \Q_k(0) - \frac{\rho g_k}{2} \int_0^t e^{-\theta_k(t-t')} \nabla F \left( A + \sum_{j=1}^K g_j \Q_j(t') \right) \d t' +  \frac{1}{\sqrt{\beta_e d}} \int_0^t e^{-\theta_k(t-t')} \d \Xi_k(t').
\end{equation*}
We then define:
\begin{equation*}
\begin{aligned}
    \Psi(t) &= A + \sum_{k=1}^K g_k \Q_k(t), \hspace{1.5cm} R_0(t) = \sum_{k=1}^K g_k^2 e^{-\theta_k t}, \\
    \G(t) &= \sum_{k=1}^K g_k e^{-\theta_k t} \Q_k(0) + \frac{1}{\sqrt{\beta_e d}} \sum_{k=1}^K g_k \int_0^t e^{-\theta_k(t-t')} \d \Xi_k(t'),
\end{aligned}
\end{equation*}
so that $\Psi$ is solution of the equation:
\begin{equation} \label{eq:CouplingPsi}
    \Psi(t) = A + \G(t) - \frac{\rho}{2} \int_0^t R_0(t-t') \nabla F \big( \Psi(t') \big) \d t'. 
\end{equation}
Similarly to what was done in \cref{App:Subsubsec:CovarianceCoupling}, we pick the initializations of the $\Q_k$ as independent centered Gaussian matrices, with covariance:
\begin{equation*}
        \E \, \big( \Q_k^0 \big)_{ij} \big( \Q_{k'}^0 \big)_{i'j'} = \frac{1}{4 \theta_k \beta_e d} \delta_{kk'} \big( \delta_{ii'} \delta_{jj'} + \delta_{ij'} \delta_{i'j} \big). 
\end{equation*}
This choice guarantees that the covariance of the Gaussian process $\G$ is time-translational invariant:
\begin{equation} \label{eq:CovarianceCouplingAging}
    \E \, \G_{ij}(t) \G_{i'j'}(t') = \frac{1}{d} \big( \delta_{ii'} \delta_{jj'} + \delta_{ij'} \delta_{i'j} \big) C(t-t'), \hspace{1.5cm} 
    C(t) = \frac{1}{4\beta_e} \sum_{k=1}^K \frac{g_k^2}{\theta_k} e^{-\theta_k |t|}.
\end{equation}
From this it is clear that we have the fluctuation--dissipation relation, for $t > 0$:
\begin{equation} \label{eq:FDRCouplingAging}
    R_0(t) = - 4 \beta_e C'(t). 
\end{equation}
Then, following a similar reasoning as in \cref{App:Subsubsec:StationaryTTI}, we can integrate the stationary distribution in equation \eqref{eq:stationaryAging1} to get the one on $\Psi$:
\begin{equation} \label{eq:stationaryAging2}
    \P_\text{aging}(\Psi) \propto \exp \left( - \rho \beta_e d \, F(\Psi) - \frac{\beta_e d}{c} \big\|\Psi - A \big\|_F^2 \right), \hspace{1.5cm} c = \sum_{k=1}^K \frac{g_k^2}{\theta_k}. 
\end{equation}
With the previous equations, we have:
\begin{equation*}
    c = \int_0^\infty R_0(t) \d t = 4 \beta_e C(0). 
\end{equation*}

\subsubsection{Aging Regime}

We shall now apply the previous result to the aging regime. More precisely, we now derive the stationary measures for the slow fields $\Psi(t), h(t)$. First recall the TTI stationary distributions given in equations \eqref{eq:StationaryTTIW}, \eqref{eq:StationaryTTIy}, and define the free energies:
\begin{align}
    F_W(\Psi) &= \frac{1}{\beta d} \log \int \d W \exp \left( - r \beta d \, \Big\|WW^\top - Z^* \Big\|_F^2 - 2 \beta d \, \Omega(W)  + 2 \beta d \, \tr \Big( \Psi \big( WW^\top - Z^* \big) \Big) \right), \label{eq:FreeenergyW} \\
    F_y(h) &= \frac{r_Z}{\beta} \log \int \d y \exp \left( - \frac{\beta}{2 \alpha r_Z} \Big( (\alpha + 2r_Z) y^2 - 2(\alpha \chi_Z + 2 r_Z) yy^* + 2 \alpha h (y-y^*) \Big) \right). \label{eq:FreeenergyY}
\end{align}
Then, we have:
\begin{equation*}
    \nabla F_W(\Psi) = 2 \, \E \big[ WW^\top - Z^* \big], \hspace{1.5cm} F_y'(h) = - \E \big[ y-y^* \big],
\end{equation*}
where the expectations are computed with respect to the TTI distributions in \eqref{eq:StationaryTTIW}, \eqref{eq:StationaryTTIy}. Now, in the slow aging regime, the variables $W(t)W(t)^\top, y(t)$ can be replaced by their averages with respect to fluctuations on short timescales, that is with respect to the TTI distributions. This means we can rewrite:
\begin{align}
    \Psi(t) &= \H_A(t) - \frac{1}{2} \int_0^t R_y^A(t,t') \nabla F_W \big( \Psi(t') \big) \d t', \label{eq:EquivalentAgingPsi} \\ 
    h(t) &= \xi_A(t) - \frac{2}{\alpha} \int_0^t \tilde R_Z^A(t,t') F_y' \big( h(t') \big) \d t'. \label{eq:EquivalentAgingH}
\end{align}
In order to match with the setting of the previous section, we decompose the Gaussian noises:
\begin{equation*}
    \H_A(t) = \tilde \H_A(t) + \H_A^\infty, \hspace{1.5cm} \xi_A(t) = \tilde \xi_A(t) + \xi_A^\infty,
\end{equation*}
so that $\tilde \H_A, \tilde \xi_A$ have vanishing covariances between instants $t,t'$ whose difference diverges beyond the aging regime. 

Let us now identify the dynamics \eqref{eq:EquivalentAgingPsi} with the one of the auxiliary coupling in equation~\eqref{eq:CouplingPsi}. First of all, we identify the constant Gaussian noise $\H_A^\infty$ with the matrix $A$, as well as $\G(t) = \tilde \H_A(t)$ in equation \eqref{eq:CouplingPsi}, which implies the identity between covariances, using equations \eqref{eq:CovarianceWAging} and \eqref{eq:CovarianceCouplingAging}:
\begin{equation*}
    C = \frac{1}{4 \alpha} \big( C_y^A - C_y^A(\infty) \big). 
\end{equation*}
Now, in order to match the expressions of $\Psi$ in equations \eqref{eq:CouplingPsi} and \eqref{eq:EquivalentAgingPsi}, we choose $F = F_W$ and $\rho R_0 = R_y^A$. Then, we can identify the fluctuation-dissipation relations in equations \eqref{eq:FDRAgingregime} and \eqref{eq:FDRCouplingAging}, that leads to the choice $\rho = \alpha$. These identities show that we can apply the result of the previous section to the dynamics \eqref{eq:EquivalentAgingPsi}, and obtain, thanks to equation \eqref{eq:stationaryAging2}, that $\Psi$ has stationary measure:
\begin{equation} \label{eq:InvariantWAging}
    \P_W^\mathrm{aging}(\Psi) \propto \exp \left( - \alpha \beta_e d \, F_W(\Psi) - \frac{\alpha d}{C_y^A(0) - C_y^A(\infty)} \big\| \Psi - \tilde \H_A^\infty \big\|_F^2 \right). 
\end{equation}
The same identification can be performed for the dynamics of $h(t)$ in equation \eqref{eq:EquivalentAgingH}, in one dimension. In this case, we choose:
\begin{equation*}
    F = F_y, \hspace{1.2cm} A = \xi_A^\infty, \hspace{1.2cm} C = \frac{K_Z - K_Z(\infty)}{2}, \hspace{1.2cm} R_0 = \frac{4}{\alpha \rho} \tilde R_Z^A. 
\end{equation*}
The third equation originates from matching the covariance of $\tilde \xi_A^\infty$ with the one of $\G$ in equation~\eqref{eq:CovarianceCouplingAging}, equal to $2C$ in one dimension. Finally, imposing the equivalence of the fluctuation-dissipation relations in equations \eqref{eq:FDRAgingregime}, \eqref{eq:FDRCouplingAging} leads to the relation $\alpha \rho = 2$. Using equation~\eqref{eq:stationaryAging2}, this shows that $h$ has stationary distribution:
\begin{equation} \label{eq:InvariantyAging}
    \P_y^\text{aging}(h) \propto \exp \left( - \frac{2 \beta_e}{\alpha} F_y(h) - \frac{1}{2} \frac{1}{K_Z^A(0) - K_Z^A(\infty)} \big( h - \tilde \xi_A^\infty \big)^2 \right). 
\end{equation}
Already it is interesting to remark that the distribution for $y$ conditionally to $h$ in equation~\eqref{eq:StationaryTTIy} is Gaussian, and so is the distribution of $h$ itself (this can be shown by computing $F_y$). Therefore, the typical label $y$ is itself Gaussian, which is consistent with the observation made in \cref{Subsubsec:SimplifiedDMFT}. 

\subsubsection{Closing the System of Equations}

The previous calculations in the TTI and aging regimes allow to access the distribution of $W$ and $y$ under the aging ansatz. These are given by:
\begin{equation*}
\begin{aligned}
    \P(W) &\propto \int \P_W^\text{TTI} \big( W \, | \, \Psi \big) \P_W^\text{aging}(\Psi) \d \Psi,  \\
    \P(y) &\propto \int \P_y^\text{TTI} \big( y \, | \, h \big) \P_y^\text{aging}(h) \d h . 
\end{aligned}
\end{equation*}
Then, all the quantities appearing in these four distributions can be expressed as expectations over these distributions. We have:
\begin{align}
    r &= \beta C_y^\text{TTI}(0) = \beta  \, \E_h \E_y \big[ ( \overline y - y)^2 \big],  \label{eq:SCAging1} \\
    r_Z &= \beta K_Z^\text{TTI}(0) = \frac{2\beta}{d} \E_\Psi \E_W \Big[ \big\| \overline Z - WW^\top \big\|_F^2 \Big], \label{eq:SCAging2} \\
    \chi_Z &= \frac{1}{Q_*} \frac{1}{d} \E_\Psi \E_W \Big[ \tr \big( WW^\top Z^* \big) \Big], \label{eq:SCAging3} \\
    C_y^A(0) - C_y^A(\infty) &= \E_h \Big[ \big( \E_h[\overline y] - \overline y \big)^2 \Big] \equiv V_y, \label{eq:SCAging4} \\
    K_Z^A(0) - K_Z^A(\infty) &= \frac{2}{d} \E_\Psi \Big[ \big\| \E_\Psi[\overline Z] - \overline Z \big\|_F^2 \Big] \equiv V_Z, \label{eq:SCAging5}
\end{align}
where $\E_W, \E_y$ denote the expectations with respect to the TTI distributions, conditionally on $\Psi, h$, as well as $\overline y = \E_y[y]$ and $\overline Z = \E_W[WW^\top]$. Moreover, in the following, we denote the Gaussian noises in equations \eqref{eq:InvariantWAging}, \eqref{eq:InvariantyAging} by:
\begin{equation*}
    \tilde \H_A^\infty = \sigma_y \G, \hspace{1.5cm} \tilde \xi_A^\infty = \sigma_W \phi,
\end{equation*}
where $\G \sim \mathrm{GOE}(d), \phi \sim \N(0,1)$, and:
\begin{equation} \label{eq:SigmaAging}
    \sigma_y^2 = \frac{1}{4\alpha} \big( \E_h[\overline y] - y^* \big)^2, \hspace{1.5cm} \sigma_W^2 = \frac{2}{d} \tr \big( \E_\Psi[\overline Z]^2 \big) - 2 Q_* \chi_Z^2. 
\end{equation}
This provides a set of self-consistent equations describing the long-time behavior of the Langevin dynamics under the aging ansatz. In the following, we take the zero-temperature limit ($\beta \to \infty$) in order to simplify these equations. 

\subsection{Zero-Temperature Limit} \label{App:Subsubsec:ZeroTempAging}

In the $\beta \to \infty$ limit, we analyze the previous equations with the choice of the regularization $\Omega(W) = \lambda \tr(WW^\top)$. In this limit, the equations can be simplified and lead to the result given in \cref{App:Subsec:SummaryAging}. However, it is not guaranteed that taking the zero-temperature limit of the stationary equations (i.e., after taking $t \to \infty$) leads to the long-time behavior of gradient flow dynamics. Nevertheless, as shown in \cref{App:Subsec:ZeroTemperature} in the TTI analysis (with positive regularization), both lead to the same limiting behavior, which suggests that they may remain closely related under the aging ansatz. 

\subsubsection{Student Equations}

Let us start by computing the free energy $F_W$ in equation \eqref{eq:FreeenergyW}. The first thing to notice is that the density:
\begin{equation*}
    \exp \left( - r \beta d \, \Big\|WW^\top - Z^* \Big\|_F^2 - 2 \lambda \beta d \, \tr(WW^\top)  + 2 \beta d \, \tr \Big( \Psi \big( WW^\top - Z^* \big) \Big) \right),
\end{equation*}
concentrates as $\beta \to \infty$ on the set:
\begin{equation} \label{eq:ProbaCollapseAging}
    \left\{ W \in \R^{d \times m}, WW^\top = \left( Z^* + \frac{1}{r} \Psi - \frac{\lambda}{r} I_d \right)_{(m)}^+ \right\},
\end{equation}
where the operator $X \mapsto X_{(m)}^+$ is defined in \cref{Def:EckartYoung} and selects the $m$ largest positive eigenvalues of $X$. The concentration of the previous density is a consequence of \cref{Lemma:MinEckartYoung}. Therefore, under the TTI distribution \eqref{eq:StationaryTTIW}, $WW^\top$ is a deterministic function of $\Psi$. 

Now, the free energy defined in equation \eqref{eq:FreeenergyW} has the asymptotic:
\begin{equation*}
    F_W(\Psi) \xrightarrow[\beta \to \infty]{} r \Big\| \big( Z^*_\lambda + r^{-1} \Psi \big)_{(m)}^+ \Big\|_F^2 - 2 \tr \big( \Psi Z^* \big) - r \|Z^*\|_F^2,
\end{equation*}
with $Z_\lambda^* = Z^* - \lambda r^{-1} I_d$. We can conclude the expression of the distribution of $\Psi$ given in equation \eqref{eq:InvariantWAging}:
\begin{equation} \label{eq:InvariantWAging2}
    \P_W^\text{aging}(\Psi) \propto \exp \left( - \alpha r \beta_e d \, \Big\| \big( Z_\lambda^* + r^{-1} \Psi \big)_{(m)}^+ \Big\|_F^2 + 2 \alpha \beta_e d \, \tr \big( \Psi Z^* \big) - \frac{\alpha d}{V_y} \big\| \Psi - \sigma_y \G \big\|_F^2 \right). 
\end{equation}

As a consequence of this concentration, the variance term in equation \eqref{eq:SCAging2} is of order $\beta^{-1}$, leading to $r_Z$ of order one. To compute $r_Z$, we can go back to equation \eqref{eq:RelationResponses} that links the responses $\tilde R_Z, R_y$. In the TTI regime, this shows the relation:
\begin{equation*}
    \alpha + 2r_Z = \frac{\alpha}{r}. 
\end{equation*}
Moreover, we can apply the same argument as in \cref{App:Subsubsec:LongTimesResponse} to compute the integrated response as a derivative over the stationary distribution with respect to a constant perturbation:
\begin{equation} \label{eq:IntegratedResponseAging}
    r = 1 - \frac{1}{\alpha d^2} \tr \left( \left. \frac{\partial}{\partial H} \left( Z^* + \frac{1}{r} \Psi - \frac{\lambda}{r} I_d + H \right)_{(m)}^+ \right|_{H = 0} \right),
\end{equation}
where $H \in \mathcal{S}_d(\R)$ and the trace is to be interpreted as taken over the linear maps on the space of symmetric matrices. 

\subsubsection{Label Equations}

Let us now move on to the label equation. In the $\beta \to \infty$ limit, the distribution of $y$ conditionally on $h$ in equation \eqref{eq:StationaryTTIy} concentrates around:
\begin{equation*}
    y = \frac{\alpha \chi_Z + 2r_Z}{\alpha + 2r_Z} y^* - \frac{\alpha}{\alpha + 2r_Z} h. 
\end{equation*}
Therefore, $y$ is a deterministic function of $h$. In addition, we have the limit for the free energy in equation \eqref{eq:FreeenergyY}:
\begin{equation*}
    F_y(h) \xrightarrow[\beta \to \infty]{} hy^* + \frac{1}{2\alpha(\alpha + 2r_Z)} \Big( (\alpha \chi_Z + 2r_Z) y^* - \alpha h \Big)^2.
\end{equation*}
Plugging this expression into the distribution of $h$ in equation \eqref{eq:InvariantyAging}, we obtain that:
\begin{equation*}
    \P_y^\text{aging}(h) \propto \exp \left( - \frac{2 \beta_e}{\alpha} hy^* - \frac{r \beta_e}{\alpha} \left( \left( \chi_Z + \frac{2r_Z}{\alpha} \right) y^* - h \right)^2 - \frac{1}{2 V_Z} \big( h - \sigma_W \phi \big)^2 \right). 
\end{equation*}
Therefore $h$ is Gaussian, with mean and variance:
\begin{equation*}
\begin{aligned}
    \E \, h &= \frac{1}{2r \beta_e V_Z + \alpha} \Big( - 2r \beta_e V_Z (1 - \chi_Z) y^* + \alpha \sigma_W \phi \Big), \\
    \mathrm{Var} \, h &= \frac{\alpha V_Z}{2r \beta_e V_Z + \alpha}. 
\end{aligned}
\end{equation*}
This leads to the equation on $V_y$ in equation \eqref{eq:SCAging4}:
\begin{equation} \label{eq:VaryAging}
    V_y = r^2 \mathrm{Var} \, h = \frac{\alpha r^2 V_Z}{2r \beta_e V_Z + \alpha}. 
\end{equation}
In addition, the variance of the GOE noise given in equation \eqref{eq:SigmaAging} writes:
\begin{equation} \label{eq:sigma2aging}
    \sigma_y^2 = \frac{\alpha r^2}{2} \frac{1}{(2r \beta_e V_Z + \alpha)^2} \frac{1}{d} \Big\| \E_\Psi \big[WW^\top \big] - Z^* \Big\|_F^2. 
\end{equation}
These expressions guarantee that we can write a closed set of equations on the variable $\Psi$. 

\subsubsection{Set of Equations}

We now explain how to arrive at the system of equations given in \cref{App:Subsec:SummaryAging}. Let us start by defining:
\begin{equation*}
    X = Z^* - \frac{\lambda}{r} I_d + \frac{1}{r} \Psi,
\end{equation*}
so that $X$ is a linear transform of $\Psi$. Then:
\begin{itemize}
    \item Equation \eqref{eq:ProbaCollapseAging} guarantees that the limit of the dynamics is simply $WW^\top = X_{(m)}^+$. 
    \item Equation \eqref{eq:EqsAgingResponse} is easily derived from equation \eqref{eq:IntegratedResponseAging}. 
    \item Equation \eqref{eq:EqsAgingCovariance1} is a consequence of equation \eqref{eq:SCAging5}. 
    \item Letting $\sigma = \sigma_y / r$ in equation~\eqref{eq:sigma2aging} yields equation \eqref{eq:EqsAgingCovariance2}. 
    \item Plugging equation \eqref{eq:VaryAging} and the expression of $X$ into the distribution \eqref{eq:InvariantWAging2} leads to equation~\eqref{eq:DistribAging1}. 
\end{itemize}

\subsection{The Effective Temperature} \label{App:Subsec:EffectiveTemperature}

The previous calculation leads to a set of self-consistent equations relating the distribution of $X$ in \eqref{eq:DistribAging1} to the parameters appearing in this distribution. The only quantity that remains undetermined is the effective inverse temperature $\beta_e$. 

The value of $\beta_e$ cannot be determined from the aging equations alone and must be fixed by an additional condition. To do so, we impose that the time-translational invariant (TTI) solution should be \emph{marginally stable}, meaning that small perturbations around this solution relax on arbitrarily long timescales. This marginality condition selects a unique value of $\beta_e$ and ensures a consistent matching between the TTI and aging regimes. 

To impose the effective temperature, recall that we studied in \cref{App:Subsec:Susceptibility} the susceptibility operator associated with the TTI solution at zero temperature. Since the aging solution is also given by a matrix of the form $Z = X_{(m)}^+$, a similar computation leads to an expression for the Frobenius norm of the susceptibility in the high-dimensional limit:
\begin{equation*}
    \mathcal{R} = \frac{1}{2} \iint \left( \frac{x \1_{x \geq \max(0, \omega)} - y \1_{y \geq \max(0, \omega)}}{x-y} \right)^2 \d \mu_X(x) \d \mu_X(y),
\end{equation*}
where $\mu_X$ is the asymptotic spectral distribution of $X$, and $\omega$ is a threshold selecting a fraction $\kappa$ of the largest eigenvalues of $X$. In \cref{App:Subsec:Susceptibility}, we have seen that this quantity may remain finite because of two mechanisms: if $\omega \leq 0$ or if the measure $\mu_X$ does not have mass in the vicinity of $\omega$. 

Due to the form of the distribution $\P(X)$ in equation \eqref{eq:DistribAging1}, we expect $\mu_X$ to have mass near the threshold $\omega$, at least when it is positive. Indeed, the second term in equation \eqref{eq:DistribAging1} can be interpreted as a penalization of the eigenvalues exceeding $\omega$, which should favor an accumulation of eigenvalues close to this threshold. 

This observation suggests that the marginal stability of the TTI solution cannot be achieved by requiring $\mu_X$ to vanish near the threshold $\omega$. We therefore propose that the marginality condition should be imposed by setting $\omega = 0$, i.e., by constraining the matrix $X$ to have exactly $m$ positive eigenvalues. This condition implicitly defines the effective inverse temperature $\beta_e$.

It is, however, not obvious if this marginality condition can be imposed. Addressing this question requires understanding the effect of the parameter $\beta_e$ on the distribution in equation \eqref{eq:DistribAging1}, and in particular how it controls the number of positive eigenvalues of the matrix $X$. If such a control is not possible, the marginality condition cannot be imposed, and the time-translational invariant solution would remain stable, leading to the absence of aging in the dynamics. We leave this challenging question for further work. 

\paragraph{Perfect recovery thresholds.} A more reasonable analysis could help derive the perfect recovery thresholds from the system of equations in \cref{App:Subsec:SummaryAging}. In this limit, we expect that the distribution of $X$ can be understood using a Laplace method and a perturbative analysis, leading to a simplification of the self-consistent equations. However, it remains open whether the thresholds obtained from this aging system coincide with those of \cref{Prop:PRThresholdRegularization} and \cref{Conjecture:PRThreshold}.

\section{Small Regularization Limit} \label{App:SmallReg}
In this section, we derive the results presented in \cref{subsec:SmallReg} on the small regularization limit. The section is organized as follows:
\begin{itemize}
    \item In \cref{App:Subsec:InterpolationThreshold} and \cref{App:Subsec:PRThreshold}, we compute the interpolation and perfect recovery thresholds respectively introduced in \cref{Subsubsec:InterpolationThreshold} and \cref{Subsubsec:PRThresholdReg}. 
    \item In \cref{App:Subsec:MinRegInterpolator}, for the sake of completeness, we give a proof of \cref{Prop:MinRegularizationInterpolator} regarding the notion of minimal regularization interpolator. 
    \item In \cref{App:Subsec:IntegralAsymptotics}, we derive the asymptotics associated with the functions introduced in \cref{Result2} in the limit of the small noise $\xi$. These technical results allow to compute the interpolation and perfect recovery thresholds in the noiseless case. 
\end{itemize}

In the following, and in agreement with the factorization of the teacher $Z^* = W^* W^{*\top}$, where $W^* \in \R^{d \times m^*}$, we adopt the following decomposition for the limiting distribution of the teacher $\mu^*$:
\begin{equation*}
    \mu^* = (1 - \min(\kappa^*, 1)) \delta + \min(\kappa^*, 1) \nu^*,
\end{equation*}
where $\nu^*$ has support in $\R^+$ bounded away from zero. In the following we will also assume that $\nu^*$ admits a smooth density with respect to the Lebesgue measure. However, since the results we give do not depend on this property, we believe they can be extended in the general case.

This representation will be useful to study our system of equations in the $\xi \to 0$ limit. In this limit, provided that the support of $\nu^*$ is bounded away from zero, the support of the measure $\mu_\xi$ (the free additive convolution between $\mu^*$ and a semicircular density of variance $\xi$) splits into two parts when $\kappa^* < 1$: a semicircular density with radius $\propto \sqrt{\xi}$ and mass $1 - \kappa^*$, and a measure of mass $\kappa^*$ that resembles $\nu^*$. More quantitative details can be found in \cref{App:Subsec:IntegralAsymptotics}, and we also refer to \citet[][Section F]{maillard2024bayes}. 

\subsection{Interpolation Threshold} \label{App:Subsec:InterpolationThreshold}

In this section we prove the expressions of the interpolation threshold given in \cref{Prop:InterpolationThreshold} and \cref{Prop:InterpolationThreshold1}. 

\subsubsection{Derivation of the Interpolation Threshold}

We shall now derive the system of equations that characterizes the interpolation threshold in \cref{Prop:InterpolationThreshold}. As a consequence, we also access the system of equations in the small regularization limit given in \cref{Prop:EquationsSmallReg}. 

Let us recall the system of equations \eqref{eq:SystemResult2}. The equations on the variables $(q, \xi)$ can be written:
\begin{align}
    1 &= \frac{\lambda}{q} + \frac{1}{\alpha} I_\omega(q, \xi), \label{eq:SystemInterp1}\\
    2 \alpha \xi - \frac{\Delta}{2} &= Q_* + \int_{\max(q, \omega)} (q^2 - x^2) \d \mu_\xi(x) + 4 \xi I_\omega(q, \xi), \label{eq:SystemInterp2}
\end{align}
with:
\begin{align}
    I_\omega(q, \xi) &= \int_{\max(q, \omega)} (x-q) h_\xi(x) \d \mu_\xi(x), \label{eq:defIomega} \\
    \min(\kappa, 1) &= \int_\omega \d \mu_\xi(x). 
\end{align}
From the expression of the MSE in equation \eqref{eq:LossValue}, we have $\mathrm{MSE} \leq 2 \alpha \xi$. This implies that when $\alpha$ is smaller than the perfect recovery threshold, $\xi$ remains positive in the small regularization limit. Regarding the behavior of $q$ in this limit, we consider two cases in the following. 

We work with the scaling $q = r\lambda$, with fixed $r > 0$. Here $r$ corresponds to the variable $r_\infty$ introduced in \cref{Subsubsec:SteadyState}. Then, equation \eqref{eq:SystemInterp2} rewrites:
\begin{equation} \label{eq:Interpolation1}
    2 \alpha \xi - \frac{\Delta}{2} = Q_* - \int_{\max(0, \omega)} x^2 \d \mu_\xi(x) + 4 \xi I_\omega(\xi),
\end{equation}
with $I_\omega(\xi) = I_\omega(0, \xi)$. This corresponds to the second part of \cref{Prop:EquationsSmallReg}. In this case the first equation of \eqref{eq:SystemInterp1} simply gives the expression of $r$:
\begin{equation*}
    r = 1 - \frac{1}{\alpha} I_\omega(\xi). 
\end{equation*}
This is only valid if $r > 0$, which leads to the constraint $\alpha > I_\omega(\xi)$. This leads to the definition of the interpolation threshold as the limiting value for which these equations are verified. Combining equation~\eqref{eq:Interpolation1} with the equality $\alpha_\text{inter} = I_\omega(\xi)$ leads to the result of \cref{Prop:InterpolationThreshold}. 

Now considering the regime where $q$ remains of order one as $\lambda \to 0^+$, we obtain the first part of \cref{Prop:EquationsSmallReg} for $\alpha < \alpha_\text{inter}$. Note that one could also consider intermediate regimes for the dependence between $q$ and $\lambda$, but all of these collapse into the equations at $\alpha = \alpha_\text{inter}$. 

We also derive the expression of the in-sample error that is defined in equation \eqref{eq:insampleerror} and plotted in \cref{fig:FigUnregularized1}. To do so, we go back to the expression of the label $y_\infty$ at long times in equation~\eqref{eq:LimitLabel}. In the same way, we computed the empirical loss on the noisy labels, the in-sample error writes:
\begin{equation*}
\begin{aligned}
    \mathrm{Err}_\mathrm{in} &= \frac{1}{4} \E \, \big(y_\infty - y^* \big)^2 \\
    &= \frac{r_\infty^2}{2} \mathrm{MSE} + \frac{\Delta}{4} (1- r_\infty)^2. 
\end{aligned}
\end{equation*}
For $\alpha < \alpha_\text{inter}$, in the small regularization limit, we have $r_\infty = 0$ and therefore $\mathrm{Err}_\mathrm{in} = \Delta / 4$. Above the interpolation threshold, we have $r_\infty = 1 - I_\omega(\xi) / \alpha$, and replacing the value of the MSE, we get:
\begin{equation*}
    \mathrm{Err}_\mathrm{in} = \alpha \xi \left( 1 - \frac{I_\omega(\xi)}{\alpha} \right)^2 + \frac{\Delta}{4}\left( \frac{2 I_\omega(\xi)}{\alpha} - 1 \right). 
\end{equation*}

\subsubsection{The Noiseless Case}

We now compute the interpolation threshold in the case where $\Delta = 0$, proving \cref{Prop:InterpolationThreshold1}. In this case we show that the interpolation threshold is reached at $\xi = 0$, meaning that in the small regularization limit, the interpolation threshold is larger than the perfect recovery one.  

To prove the result, we can combine equation \eqref{eq:Interpolation1} at $\Delta = 0$ with the identity $\alpha_\text{inter} = I_\omega(\xi)$ to obtain a relationship solely on $\xi$:
\begin{equation} \label{eq:ConditionInter}
     I_\omega(\xi) = \frac{1}{2\xi} \left( \int_{\max(0, \omega)} x^2 \d \mu_\xi(x) - Q_* \right).
\end{equation}
Therefore, to prove the result, it is enough to show that this relationship is verified in the limit $\xi \to 0$. As shown in \cref{Lemma:SmallXiOmega}, for $\kappa^* < 1$, we have $\omega \sim \sqrt{(1-\kappa^*) \xi} \tilde \omega$ as $\xi \to 0$, with:
\begin{equation*}
    \frac{\min(\kappa, 1) - \kappa^*}{1 - \kappa^*} = \int_{\tilde \omega} \d \sigma(x). 
\end{equation*}
Now using \cref{Lemma:SmallXiFunctions} with $q, h = 0$, and using the definition of $I_\omega(\xi)$ in equation \eqref{eq:defIomega}, we have:
\begin{align}
    \int_{\max(0, \omega)} x^2 \d \mu_\xi(x) &\underset{\xi \to 0}{=} Q_* + 2 \xi \left( \kappa^* - \frac{\kappa^{*2}}{2} \right) + \xi (1 - \kappa^*)^2 \int_{\max(0, \tilde \omega)} x^2 \d \sigma(x) + o(\xi), \\
    I_\omega(\xi) &\underset{\xi \to 0}{=} \kappa^* - \frac{\kappa^{*2}}{2} + \frac{(1- \kappa^*)^2}{2} \int_{\max(0, \tilde \omega)} x^2 \d \sigma(x) + o(1). 
\end{align}
Rearranging, this shows that equation \eqref{eq:ConditionInter} is verified in the $\xi \to 0$ limit. We then obtain the result of \cref{Prop:InterpolationThreshold1} since $\alpha_\text{inter}$ is now given by $\lim_{\xi \to 0} I_\omega(\xi)$. 

For the case $\kappa, \kappa^* \geq 1$, the result is simply an application of \cref{Lemma:SmallXiFunctionsOverparam}. Indeed, since $\kappa \geq 1$, there is no need for the threshold $\omega$. We then find the desired value $\alpha_\text{inter}(\kappa, \kappa^*) = 1/2$. 

\subsection{Perfect Recovery Threshold} \label{App:Subsec:PRThreshold}

In this section we show \cref{Prop:PRThresholdRegularization} that gives the expression of the perfect recovery threshold in the small regularization limit. To do so, we consider the first part of \cref{Prop:EquationsSmallReg} with $\Delta = 0$. 

To derive the value of the perfect recovery threshold, remark that as $\alpha \to \alpha_\text{PR}^+$, we have $Z_\infty \to Z^*$ and from equation \eqref{eq:SystemSmallReg1} we can conclude that the variables $q, \xi$ vanish. Moreover, when $\kappa^* < 1$, we show in \cref{Lemma:SmallXiOmega} that the cutoff $\omega$ needs to scale with $\xi$ as $\omega \sim \sqrt{(1-\kappa^*)\xi} \tilde \omega$. Finally, we choose the scaling $q \sim \sqrt{(1-\kappa^*) \xi} h$, with $h = \Theta(1)$. Using the system of equations in the small regularization limit, we then have the equations describing the perfect recovery threshold:
\begin{align}
    \alpha_\text{PR}^+ &= \lim_{\xi \to 0} \int_{\max(q, \omega)} (x-q) h_\xi(x) \d \mu_\xi(x), \label{eq:SystemPR1} \\
    \alpha_\text{PR}^+ &= \lim_{\xi \to 0} \frac{1}{2\xi} \left( Q_* + \int_{\max(q, \omega)} (q^2 - x^2) \d \mu_\xi(x) + 4 \xi \int_{\max(q, \omega)} (x-q) h_\xi(x) \d \mu_\xi(x) \right), \label{eq:SystemPR2}
\end{align}
and one has to solve this for the variables $\alpha_\text{PR}^+, q$. Then, taking into account the scalings of $\omega, q$ as $\xi \to 0$, we have the asymptotics from \cref{Lemma:SmallXiFunctions}, for $\kappa^* < 1$:
\begin{equation*}
\begin{aligned}
    \frac{1}{\xi} \left( Q_* + \int_{\max(q, \omega)} (q^2 - x^2) \d \mu_\xi(x) \right) &\underset{\xi \to 0}{=} (1 - \kappa^*) h^2 \left( \kappa^* + (1 - \kappa^*) \int_{\max(h , \tilde \omega)} \d \sigma(x) \right) \\
    &-2 \left( \kappa^* - \frac{\kappa^{*2}}{2} \right)  - (1 - \kappa^*)^2 \int_{\max(h, \tilde \omega)} x^2 \d \sigma(x) + o(1), \\
    \int_{\max(q, \omega)} (x-q) h_\xi(x) \d \mu_\xi(x) &\underset{\xi \to 0}{=} \kappa^* - \frac{\kappa^{*2}}{2} + \frac{(1 - \kappa^*)^2}{2} \int_{\max(h, \tilde \omega)} x(x-h) \d \sigma(x) + o(1). 
\end{aligned}
\end{equation*}
Rearranging, we get the system:
\begin{equation*}
\begin{aligned}
    \alpha_\text{PR}^+ &= \kappa^* - \frac{\kappa^{*2}}{2} + \frac{(1 - \kappa^*)^2}{2} \int_{\max(h, \tilde \omega)} x(x-h) \d \sigma(x), \\
    \alpha_\text{PR}^+ &= \kappa^* - \frac{\kappa^{*2}}{2} + \frac{\kappa^*(1 - \kappa^*)}{2} h^2 + \frac{(1-\kappa^*)^2}{2} \int_{\max(h, \tilde \omega)} (x-h)^2 \d \sigma(x). 
\end{aligned}
\end{equation*}
This leads to the following equation on $h$:
\begin{equation*}
    \kappa^* h^2 = (1 - \kappa^*) h \int_{\max(h, \tilde \omega)} (x-h) \d \sigma(x). 
\end{equation*}
Now, there are two possible solutions to this equation: the one given in \cref{Prop:PRThresholdRegularization}, and $h = 0$, which leads back to the interpolation threshold. Now, one needs to take the solution that leads to the smallest value of $\alpha_\text{PR}^+$, since the perfect recovery threshold is defined as the minimal value of $\alpha$ such that the MSE vanishes. Therefore we need to choose the solution $h > 0$ and obtain the result of \cref{Prop:PRThresholdRegularization}. 

Whenever $\kappa, \kappa^* \geq 1$, we simply apply \cref{Lemma:SmallXiFunctionsOverparam}, in which case there is no need for the threshold $\omega$. Then, equation \eqref{eq:SystemPR1} leads to $\alpha_\text{PR}^+ = 1/2$. 

\subsection{Minimal Regularization Interpolator} \label{App:Subsec:MinRegInterpolator}

In this section, we prove \cref{Prop:MinRegularizationInterpolator}, which concerns the minimal regularization interpolator. This result is already known, but we give a proof for completeness. We recall that the loss we consider here writes $\L_\lambda = \L + \lambda \Omega$. We denote $\mathcal{S}^* = \mathrm{argmin}(\L)$ and assume without loss of generality that $\mathcal{S}^* = \L^{-1}(\{0\})$. We let $(W_\lambda)_{\lambda > 0}$ to be such that $W_\lambda$ is a global minimizer of $\L_\lambda$ for all $\lambda > 0$. 

Due to the optimality of $W_\lambda$ for $\L_\lambda$, we have for any $V \in \mathcal{S}^*$:
\begin{equation} \label{eq:MinReg1}
    \L(W_\lambda) + \lambda \Omega(W_\lambda) \leq \lambda \Omega(V),
\end{equation}
where we used that $\L(V) = 0$. Therefore $\Omega(W_\lambda) \leq \Omega(V)$, i.e., $W_\lambda$ has a smaller regularization than any element of $\mathcal{S}^*$. In addition, since $\Omega$ is coercive, the sublevel set $\big\{ W, \, \Omega(W) \leq \Omega(V) \big\}$ is compact, and so is the family $(W_\lambda)_{\lambda > 0}$. Finally, taking the $\lambda \to 0$ limit in equation \eqref{eq:MinReg1} leads to the fact that $\L_\lambda(W_\lambda) \xrightarrow[\lambda \to 0]{} 0$. 

Now, we let $\overline W$ be a cluster point of $W_\lambda$, and a sequence $\lambda_k \xrightarrow[k \to \infty]{} 0$ such that $W_{\lambda_k} \xrightarrow[k \to \infty]{} \overline W$. By continuity of $\L$ and since $\L(W) \leq \L_\lambda(W)$ for all $W$, we have:
\begin{equation*}
    \L\big( \overline W \big) = \lim_{k \to \infty} \L(W_{\lambda_k}) \leq \lim_{k \to \infty} \L_{\lambda_k} \big( W_{\lambda_k} \big) = 0.
\end{equation*}
Therefore $\overline W \in \mathcal{S}^*$, and $\overline W$ minimizes $\Omega$ on $\mathcal{S}^*$ since $\Omega(W_\lambda) \leq \Omega(V)$ for all $V \in \mathcal{S}^*$ and $\lambda > 0$. Now, if such an element is unique, the family $(W_\lambda)_{\lambda > 0}$ is compact and has a single cluster point, therefore it converges to $\overline W$. 

\subsection{Small Noise Asymptotics} \label{App:Subsec:IntegralAsymptotics}

In the following, we derive the small $\xi$ asymptotics used in the proofs of \cref{Prop:InterpolationThreshold1} and \cref{Prop:PRThresholdRegularization}. These computations rely on the understanding of the measure $\mu_\xi$ at small $\xi$. We recall that we define the measure $\mu_\xi$ and give some basic properties in \cref{App:Subsubsec:FreeConvolution}. The following lemma gives an expansion of the density $\rho_\xi$ and the Hilbert transform of this measure up to the needed order. 
\begin{lemma} \label{Lemma:ExpansionSmallXi}
    Consider $\kappa^* \in (0,1)$. Let $\rho_\xi$ be the density of $\mu_\xi$ and $h_\xi$ its Hilbert transform (see \cref{Def:HilbertStieltjes}). Decompose:
    \begin{equation*}
        \mu^* = (1 - \kappa^*) \delta + \kappa^* \nu^*,
    \end{equation*}
    where $\nu^*$ has a smooth density with a compact support away from zero (that we also denote $\nu^*$ for the sake of simplicity). Then, for all $x \neq 0$, as $\xi \to 0$:
    \begin{equation*}
    \begin{aligned}
        \rho_\xi(x) &= \kappa^* \nu^*(x) + \xi \kappa^*(1 - \kappa^*) \left( \frac{\nu^*(x)}{x^2} - \frac{\partial_x \nu^*(x)}{x} \right) + \frac{\xi \kappa^{*2}}{\pi} \mathrm{Im} \Big( m_{\nu^*}(x) \partial_x m_{\nu^*}(x) \Big) + O(\xi^2), \\
        h_\xi(x) &= \kappa^* h_{\nu^*}(x) + \frac{1 - \kappa^*}{x} + O(\xi),
    \end{aligned}
    \end{equation*}
    where $h_{\nu^*}$ and $m_{\nu^*}$ are respectively the Hilbert and Stieltjes transforms of $\nu^*$ (see \cref{Def:HilbertStieltjes}), and $m_{\nu^*}(x)$ is understood as the limit $m_{\nu^*}(x + i \eta)$ as $\eta \to 0^+$. 
    
    In addition, for any $|y| < 2$, with $s = 1 - \kappa^*$:
    \begin{equation*}
    \begin{aligned}
        \rho_\xi \big( \sqrt{s\xi} y \big) &= \sqrt{\frac{s}{\xi}} \frac{1}{2\pi} \sqrt{4-y^2} + O(1), \\
        h_\xi \big(\sqrt{s \xi} y \big) &= \sqrt{\frac{s}{\xi}}\frac{y}{2} + O(1).
    \end{aligned}
    \end{equation*}
\end{lemma}
\begin{proof}
    Let us start with the subordination equation between the Stieltjes transforms of $\mu^*$ and $\mu_\xi$ (see \cref{Lemma:SubordinationSemicircular}):
    \begin{equation*}
        m_\xi(z) = m_* \big( z - \xi m_\xi(z) \big). 
    \end{equation*}
    Due to the decomposition of $\mu^*$, this also writes:
    \begin{equation} \label{eq:DevelopmentStieltjes}
        m_\xi(z) = \frac{1-\kappa^*}{z - \xi m_\xi(z)} + \kappa^* m_{\nu^*}\big( z - \xi m_\xi(z) \big),
    \end{equation}
    Then decomposing $m_\xi(z) = m_0(z) + \xi m_1(z)$ at leading order, we arrive at:
    \begin{equation*}
    \begin{aligned}
        m_0(z) &= \frac{1 - \kappa^*}{z} + \kappa^* m_{\nu^*}(z), \\
        m_1(z) &= m_0(z) \left( \frac{1-\kappa^*}{z^2} - \kappa^* \partial_z m_{\nu^*}(z) \right). 
    \end{aligned}
    \end{equation*}    
    Taking $z = x + i \eta$ with $\eta$ going to $0^+$, we use the identity (see \cref{Def:HilbertStieltjes}):
    \begin{equation*}
        m_\xi(x + i \eta) \xrightarrow[]{} h_\xi(x) - i \pi \rho_\xi(x). 
    \end{equation*}
    This immediately gives the relationship at first order for $h_\xi$, $\rho_\xi$ (by considering $m_0$). Now at second order for $\rho_\xi$, we have:
    \begin{equation*}
    \begin{aligned}
        \rho_1(x) &= - \frac{1}{\pi} \lim_{\eta \to 0^+} \mathrm{Im} \, m_1(x+ i \eta) \\
        &= \frac{\kappa^*(1-\kappa^*)}{x^2} \nu^*(x) - \frac{\kappa^*(1-\kappa^*)}{x} \partial_x \nu^*(x) + \frac{\kappa^{*2}}{\pi} \mathrm{Im} \Big(  m_{\nu^*}(x) \partial_x m_{\nu^*}(x) \Big).        
    \end{aligned}
    \end{equation*}
    This gives the result knowing that $\rho_\xi(x) = \kappa^* \nu^*(x) + \xi \rho_1(x) + O(\xi^2)$. 
    
    Let us now consider the microscopic part, i.e., when $x = \sqrt{s \xi} y$ for some $y = O(1)$. To do so, we start from equation \eqref{eq:DevelopmentStieltjes} with $z = \sqrt{s \xi} y + i \eta$ and take the limit $\eta \to 0$, to obtain:
    \begin{equation*}
    \begin{aligned}
        h_\xi \big( \sqrt{s\xi} y \big) - i \pi \rho_\xi \big( \sqrt{s\xi} y \big) &= \frac{1-\kappa^*}{\sqrt{s\xi} y - \xi h_\xi \big( \sqrt{s\xi} y \big) + i \pi \xi \rho_\xi \big( \sqrt{s\xi} y \big)} \\
        &\hspace{1.5cm}+ \kappa^* m_{\nu^*} \Big( \sqrt{s\xi} y - \xi h_\xi \big( \sqrt{s\xi} y \big) + i \pi \xi \rho_\xi \big( \sqrt{s\xi} y \big) \Big). 
    \end{aligned}
    \end{equation*}
    Looking at the first term, we should have $h_\xi \big( \sqrt{s\xi} y \big)$ and $\rho_\xi \big( \sqrt{s\xi} y \big)$ of order $\xi^{-1/2}$. In this case the second term remains of order one and we can ignore it at leading order. Let us define:
    \begin{equation} \label{eq:SmallXiTilde}
        \tilde h(y) = \lim_{\xi \to 0} \sqrt{\frac{\xi}{s}} h_\xi \big( \sqrt{s\xi} y \big), \hspace{1.5cm} \tilde \rho(y) = \lim_{\xi \to 0} \sqrt{\frac{\xi}{s}} \rho_\xi \big( \sqrt{s\xi} y \big).
    \end{equation}
    Then, we get the equation:
    \begin{equation*}
        \tilde h(y) - i \pi \tilde \rho(y) = \frac{1}{y - \tilde h(y) + i\pi \tilde \rho(y)}. 
    \end{equation*}
    Solving for $\tilde m(y) = \tilde h(y) - i\pi \tilde \rho(y)$, we get that $\tilde m(y)^2 - y \tilde m(y) + 1 = 0$, and solving for $|y| \leq 2$, we obtain:
    \begin{equation*}
        \tilde h(y) = \frac{y}{2}, \hspace{1.5cm} \tilde \rho(y) = \frac{1}{2\pi} \sqrt{4-y^2}. 
    \end{equation*}
    This leads to the result with the definition of $\tilde h, \tilde \rho$ in equation \eqref{eq:SmallXiTilde}. Note that we have recovered the density and Hilbert transform of the semicircular density.
\end{proof}

Interestingly, such a result holds for arbitrary measure $\nu^*$, provided that it admits a smooth density. Before computing the asymptotics of the relevant quantities in our high-dimensional equations, we recall the definition of the semicircular distribution:
\begin{equation*}
    \d \sigma(x) = \frac{1}{2\pi} \sqrt{4-x^2} \1_{|x| \leq 2} \d x. 
\end{equation*}
As underlined by the previous lemma, this corresponds to the limiting measure of $\mu_\xi$ (with an appropriate rescaling) in the microscopic part $x = \Theta \big( \sqrt{\xi} \big)$. 

\begin{lemma} \label{Lemma:SmallXiOmega}
    Assume $0 < \kappa^* < \kappa < 1$ and consider $\omega$ to be solution of the equation:
    \begin{equation*}
        \kappa = \int_{\omega} \d \mu_\xi(x). 
    \end{equation*}
    Then, if $\mu^*$ is decomposed as $\mu^* = (1 - \kappa^*) \delta + \kappa^* \nu^*$, we have as $\xi \to 0$, $\omega \sim \sqrt{(1-\kappa^*)\xi} \tilde \omega$, where $\tilde \omega$ is solution of the equation:
    \begin{equation*}
        \frac{\kappa - \kappa^*}{1 - \kappa^*} = \int_{\tilde \omega} \d \sigma(x). 
    \end{equation*}
\end{lemma}
\begin{proof}
For $\kappa > \kappa^*$, it is clear that $\omega$ goes to zero with $\xi$, since $\omega$ selects a proportion $\kappa$ of the mass of $\mu_\xi$. If $\omega$ were to remain of order 1, it would select either a proportion $\kappa^*$ corresponding to the mass of $\nu^*$ (if $\omega > 0$) or a mass of one if $\omega < 0$. Therefore, we let $\omega \sim \sqrt{s\xi} \tilde \omega$, with $s = 1 - \kappa^*$.

Then, using \cref{Lemma:ExpansionSmallXi}, we decompose the integral between $x = O(1)$ and $x = O(\sqrt{\xi})$ and obtain the asymptotics:
\begin{equation*}
\begin{aligned}
    \int \1_{x \geq \omega} \d \mu_\xi(x) &\underset{\xi \to 0}{=} \kappa^* \int \d \nu^*(x) +\sqrt{s \xi}  \int_{|y| \leq 2} \1_{y \geq \tilde \omega}\rho_\xi(\sqrt{s \xi} y) \d y + o(1) \\
    &\xrightarrow[\xi \to 0]{} \kappa^* + (1 - \kappa^*) \int \1_{x \geq \tilde \omega} \d \sigma(x). 
\end{aligned}
\end{equation*}
Since the integral on the left side is equal to $\kappa$, we get the derived self-consistent equation on $\tilde \omega$. When $\tilde \omega$ ranges from $-2$ to $2$ (the edges of $\sigma$), $\kappa$ goes from $1$, in which case it selects the whole spectrum, to $\kappa^*$ where it only retains the positive mass of $\nu^*$. 
\end{proof}

Although we assumed $\kappa^* < \kappa < 1$, the previous result still applies to the limit cases:
\begin{itemize}
    \item $\kappa = \kappa^* < 1$. In this case there is no need to cut through the semicircular density, and we can pick any scaling $\omega \sim \sqrt{s \xi} \tilde \omega$ with $\tilde \omega > 2$. 
    \item $\kappa^* < 1 \leq \kappa$. Now $\omega$ needs to select the whole spectrum of $\mu_\xi$, so we can take the same scaling $\omega \sim \sqrt{s \xi} \tilde \omega$ with $\tilde \omega < -2$. 
\end{itemize}

With this understanding of the threshold $\omega$ in the small $\xi$ limit, we can now compute the following asymptotics:
\begin{lemma} \label{Lemma:SmallXiFunctions}
    Consider $\kappa^* \in (0,1)$ and $\kappa \geq \kappa^*$. Then, we have, with $\omega \sim \sqrt{s \xi} \tilde \omega$ and $q \sim \sqrt{s \xi} h$:
    \begin{align}
        \int_{\max(q, \omega)} \d \mu_\xi(x) &\underset{\xi \to 0}{=} \kappa^* + (1 - \kappa^*) \int_{\max(h , \tilde \omega)} \d \sigma(x) + o(1), \label{eq:SmallXi1} \\
        \int_{\max(q, \omega)} x^2 \d \mu_\xi(x)  &\underset{\xi \to 0}{=} Q_* + 2\xi \left( \kappa^* - \frac{\kappa^{*2}}{2} \right) + \xi (1 - \kappa^*)^2 \int_{\max(h, \tilde \omega)} x^2 \d \sigma(x) + o(\xi), \label{eq:SmallXi2} \\
        \int_{\max(q, \omega)} (x-q) h_\xi(x) \d \mu_\xi(x) &\underset{\xi \to 0}{=} \kappa^* - \frac{\kappa^{*2}}{2} + \frac{(1 - \kappa^*)^2}{2} \int_{\max(h, \tilde \omega)} x(x-h) \d \sigma(x) + o(1). \label{eq:SmallXi3} 
    \end{align}
\end{lemma}
\begin{proof}
The derivation of equation \eqref{eq:SmallXi1} is precisely the same as in \cref{Lemma:SmallXiOmega} with a slightly different threshold. Then, let us show equation \eqref{eq:SmallXi2}. We have, as a consequence of \cref{Lemma:ExpansionSmallXi}:
\begin{equation} \label{eq:PRCalculation1}
\begin{aligned}
    \int_{\max(q, \omega)} x^2 \d \mu_\xi(x)&= \kappa^* \int x^2 \d \nu^*(x) - \xi \kappa^* (1 - \kappa^*) \int x^2 \frac{\d}{\d x} \left( \frac{\nu^*(x)}{x} \right) \d x \\
    &+ \frac{\xi \kappa^{*2}}{\pi} \int x^2 \, \mathrm{Im} \Big( m_{\nu^*}(x) \partial_x m_{\nu^*}(x) \Big) \d x \\
    &+ (1-\kappa^*)^2 \xi \int_{\max(h, \tilde \omega)} x^2 \d \sigma(x) + o(\xi).
\end{aligned}
\end{equation}
In the last expression, the three first term come from the expansion of $\rho_\xi(x)$ for $x$ fixed and non-zero, and the second comes from the microscopic behavior of $\rho_\xi(x)$ with $x = O(\sqrt{\xi})$. Also $m_{\nu^*}(x)$ is understood as the limit of $m_*(x+i\eta)$ as $\eta \to 0^+$. Now, integrating by parts:
\begin{equation*}
    \int x^2 \, \mathrm{Im} \Big( m_{\nu^*}(x) \partial_x m_{\nu^*}(x) \Big) \d x = \left[ \frac{x^2}{2} \mathrm{Im} \big(m_{\nu^*}(x)^2 \big) \right] - \int x  \, \mathrm{Im} \big(m_{\nu^*}(x)^2 \big) \d x. 
 \end{equation*}
 Now, since $m_{\nu^*}(x+i\eta) = h_{\nu^*}(x) -i\pi \nu^*(x)$ as $\eta \to 0^+$, we have $\mathrm{Im} \big(m_{\nu^*}(x)^2 \big) = -2 \pi \nu^*(x) h_{\nu^*}(x)$, so the bracket vanishes. Therefore:
 \begin{equation*}
     \int x^2 \, \mathrm{Im} \Big( m_{\nu^*}(x) \partial_x m_{\nu^*}(x) \Big) \d x = 2 \pi \int x h_{\nu^*}(x) \d \nu^*(x). 
 \end{equation*}
 Now, using \cref{Lemma:Hilbert2}, we have:
\begin{equation*}
     \int x h_{\nu^*}(x) \d \nu^*(x) = \frac{1}{2}.
\end{equation*}
Now going back to equation \eqref{eq:PRCalculation1}, we have, integrating by parts:
\begin{equation*}
    \int x^2 \frac{\d}{\d x} \left( \frac{\nu^*(x)}{x} \right) \d x = -2.
\end{equation*}
Therefore, we have:
\begin{equation*}
    \int_{\max(q, \omega)} x^2 \d \mu_\xi(x) = Q_* + 2\xi \left( \kappa^* - \frac{\kappa^{*2}}{2} \right) + \xi (1-\kappa^*)^2 \int_{\max(h, \tilde \omega)} x^2 \d \sigma(x) + o(\xi),
\end{equation*}
which is precisely equation \eqref{eq:SmallXi2}.

Let us show equation \eqref{eq:SmallXi3}. We know from \cref{Lemma:ExpansionSmallXi} that:
\begin{equation*}
    h_\xi(x) \xrightarrow[\xi \to 0]{} \kappa^* h_{\nu^*}(x) + \frac{1-\kappa^*}{x}, \hspace{1.5cm} h_\xi \big( x \sqrt{s \xi} \big) \underset{\xi \to 0}{\sim} \sqrt{\frac{s}{\xi}} \frac{x}{2}. 
\end{equation*}
Therefore at first order:
\begin{equation*}
\begin{aligned}
    \int_{\max(q, \omega)} (x-q) h_\xi(x) \d \mu_\xi(x) &= \kappa^{*2} \int x h_{\nu^*}(x) \d \nu^*(x) + \kappa^* (1 - \kappa^*) \\
    &\hspace{1cm}+ \frac{(1 - \kappa^*)^2}{2} \int_{\max(h, \tilde \omega)} x(x-h) \d \sigma(x) + o(1).
\end{aligned}
\end{equation*}
As shown before, the first integral is simply $1/2$. Combining the first and second term, we get the expected result. 
\end{proof}

To conclude these computations, we consider the case where $\kappa^* \geq 1$, and derive the same asymptotics as previously.

\begin{lemma} \label{Lemma:SmallXiFunctionsOverparam}
    Consider the case $\kappa, \kappa^* \geq 1$. Then, as $\xi \to 0$, with $q \to 0$:
    \begin{equation*}
    \begin{aligned}
        \int_q \d \mu_\xi(x) &\xrightarrow[\xi \to 0]{} 1, \\
        \int_q x^2 \d \mu_\xi(x)  &\underset{\xi \to 0}{=} Q_* + \xi + o(\xi), \\
        \int_q (x-q) h_\xi(x) \d \mu_\xi(x) &\xrightarrow[\xi \to 0]{} \frac{1}{2}. 
    \end{aligned}
    \end{equation*}
\end{lemma}
\begin{proof}
    In this case, the measure $\mu_\xi$ converges smoothly to $\mu^*$ and there is no semicircular density near zero. Therefore, we can adapt the result of \cref{Lemma:ExpansionSmallXi} and obtain for all $x \geq 0$:
    \begin{equation*}
    \begin{aligned}
        \rho_\xi(x) &= \mu^*(x) + \frac{\xi}{\pi} \mathrm{Im} \Big( m_*(x) \partial_x m_*(x) \Big) + O(\xi^2) \\
        h_\xi(x) &= h^*(x) + o(1),
    \end{aligned}
    \end{equation*}
    with $h^*, m_*$ being the Hilbert and Stieltjes transform of $\mu^*$ (see \cref{Def:HilbertStieltjes}). Therefore, at the necessary orders:
    \begin{equation*}
    \begin{aligned}
        \int_q \d \mu_\xi(x) &\xrightarrow[\xi \to 0]{} \int \d \mu^*(x) = 1, \\
        \int_q x^2 \d \mu_\xi(x) &\underset{\xi \to 0}{=} \int x^2 \d \mu^*(x) + \frac{\xi}{\pi} \int x^2 \, \mathrm{Im} \Big( m_*(x) \partial_x m_*(x) \Big) \d x + o(\xi), \\
        \int_q (x-q) h_\xi(x) \d \mu_\xi(x) &\xrightarrow[\xi \to 0]{} \int x h^*(x) \d \mu^*(x).
    \end{aligned}
    \end{equation*}
    Now we have shown in the proof of \cref{Lemma:SmallXiFunctions} that:
    \begin{equation*}
         \frac{1}{\pi} \int x^2 \, \mathrm{Im} \Big( m_*(x) \partial_x m_*(x) \Big) \d x = 2 \int x h^*(x) \d \mu^*(x). 
    \end{equation*}
    Finally, as a consequence of \cref{Lemma:Hilbert2}, we have:
    \begin{equation*}
        \int x h^*(x) \d \mu^*(x) = \frac{1}{2},
    \end{equation*}
    and putting everything together, this concludes the proof. 
\end{proof}

\section{Proofs of the Results on the Oja Flow} \label{App:Oja}
In this section, we prove the results on the Oja flow presented in \cref{Sec:OjaFlow}. More precisely:
\begin{itemize}
    \item In \cref{App:Subsec:OjaConvergence}, we characterize the limit of the Oja flow by studying the local minimizers of the potential in equation \eqref{eq:PotentialOja}.
    \item We then derive finite-dimensional convergence rates, both in the full-rank setting (\cref{App:Subsec:CVRates1,App:Subsec:HessianOja}) and in the rank-deficient case (\cref{App:Subsec:CVRates2}).
    \item Next, we study the high-dimensional regime by proving the limit of the distance to convergence (\cref{App:Subsec:HighDimLimit}) and the corresponding long-time convergence rates (\cref{App:Subsec:LongTime}).
    \item Finally, we analyze the linear response of the Oja flow, first in finite dimension (\cref{App:Subsec:OjaResponse1}) and then in the high-dimensional limit (\cref{App:Subsec:OjaResponse2}).
\end{itemize}

\proofsubsection{prop:OjaConvergence}
\label{App:Subsec:OjaConvergence}

In the following, we prove \cref{prop:OjaConvergence}, giving the limit of the Oja flow dynamics \eqref{eq:OjaFlow} at long times, under a random initialization. To do so, we study the local minimizers of the associated loss and invoke the stable manifold theorem \citep[see][]{smale1963stable}. We recall that the dynamics we study is the gradient flow associated with the loss defined in equation \eqref{eq:PotentialOja}:
\begin{equation*}
    U(W) = \frac{1}{4} \big\|WW^\top - A \big\|_F^2. 
\end{equation*}
The gradient and Hessian of $U$ are given by:
\begin{align}
    \nabla U(W) &= \big(WW^\top - A \big) W, \label{eq:gradientOja} \\
    \tr \big( \d^2 U_W(K)K^\top \big) &= \tr \Big( \big( WW^\top - A \big) KK^\top \Big) + \tr \big( KW^\top KW^\top \big) + \tr \big(KW^\top WK^\top \big), \label{eq:hessianOja1}
\end{align}
for $K \in \R^{d \times m}$. Now, if $\nabla U(W) = 0$, then with $Z = WW^\top$, we have $Z^2 = AZ$, so that $Z$ and $A$ commute. Thus, we can diagonalize $Z$ in the same basis as $A$. Then, we write the eigenvalues and singular values decompositions for $A$ and $W$:
\begin{equation*}
    A = \sum_{k=1}^d \lambda_k u_k u_k^\top, \hspace{1.5cm} W = \sum_{k=1}^{\min(m,d)} \sqrt{\mu_k} v_k u_k^\top.
\end{equation*}
Plugging these expressions into the equation $\nabla U(W) = 0$ leads to $\mu_k \in \{0, \lambda_k\}$ for all $k \in \{1, \dots, \min(m,d)\}$. Of course if $\lambda_k < 0$ then $\mu_k = 0$ since $Z$ is positive semidefinite. We now assume that $W \in \R^{d \times m}$ is a local minimizer and evaluate equation \eqref{eq:hessianOja1} with $K = v_j u_i^\top$. Under the constraint that $\tr \big( \d^2 U_W(K)K^\top \big) \geq 0$, we get:
\begin{equation*}
    \mu_j - \lambda_j + \mu_i \delta_{ij} + \mu_i \geq 0. 
\end{equation*}
We then choose the indices $i \neq j$ such that $\mu_i = \lambda_i > 0$ and $\lambda_j > \lambda_i$. Then, we obtain that $\mu_j \geq \lambda_j - \lambda_i > 0$ so that necessarily $\mu_j = \lambda_j$. Therefore, as soon as $Z$ matches a positive eigenvalue of $A$, it needs to match the larger eigenvalues of $A$. 

We now show that $Z$ recovers as many eigenvalues as it is possible under the rank constraint. To do so, we assume that $\mathrm{rank}(W) < m$ and introduce $u \in \R^m$ non-zero such that $Wu = 0$. Then evaluating equation \eqref{eq:hessianOja1} with $K = v_j u^\top$ we get $\mu_j \geq \lambda_j$. Therefore if $\lambda_j > 0$, necessarily $\mu_j = \lambda_j$ and $Z$ matches all the positive eigenvalues of $A$. Otherwise $\mathrm{rank}(W) = m$ and the previous reasoning shows that $Z$ recovers the $m$ largest positive eigenvalues of $A$. Therefore, we can conclude the equivalence:
\begin{equation*}
    W \, \text{local minimizer of } U \Longleftrightarrow WW^\top = A_{(m)}^+. 
\end{equation*}
We now use the stable manifold theorem \citep{smale1963stable}: since $U$ is analytic in the coefficients of $W$, and the gradient flow $W(t)$ associated with $U$ is initialized with a random matrix $W_0$ whose distribution is absolutely continuous with respect to the Lebesgue measure on $\R^{d \times m}$, then with probability one, $W(t)$ converges toward some local minimizer of $U$. This leads to the result. 

\proofsubsection{prop:CVRates1} \label{App:Subsec:CVRates1}

We now prove \cref{prop:CVRates1} that gives the exponentially fast convergence of the gradient flow dynamics $\dot W(t) = - \nabla F(W(t))$ when $F$ is of the form $F(W) = G(WW^\top)$, for some twice continuously differentiable $G \colon \mathcal{S}_d(\R) \to \R$. Note that the result we show in the following only applies whenever $W(t)$ converges to a full-rank matrix. We start with a few useful lemmas. 

\begin{lemma} \label{Lemma:CVRates1}
    Consider the family of subspaces indexed by $W \in \R^{d \times m}$:
    \begin{equation} \label{eq:HorizontalSpace}
        \mathcal{H}_W = \Big\{ K \in \R^{d \times m}, \, W^\top K = K^\top W \Big\},
    \end{equation}
    and $\pi_W$ to be the orthogonal projection onto $\mathcal{H}_W$. Then, for $m \leq d$, the map $W \mapsto \pi_W$ is continuous on the open set $\big\{ W \in \R^{d \times m}, \, \mathrm{rank}(W) = m \big\}$. 
\end{lemma}
\begin{proof}
    It is known \citep[see][]{massart2020quotient} that the orthogonal complement $\mathcal{V}_W$ of $\mathcal{H}_W$ is given by:
    \begin{equation} \label{eq:VerticalSpace}
        \mathcal{V}_W = \big\{ W\Omega, \, \Omega \in \mathcal{A}_m(\R) \big\},
    \end{equation}
    where $\mathcal{A}_m(\R)$ is the set of $m \times m$ skew-symmetric matrices. Now decomposing some $K \in \R^{d \times m}$ as $K = H + W\Omega$, where $H \in \mathcal{H}_W$ and $\Omega \in \mathcal{A}_m(\R)$, we obtain:
    \begin{equation*}
        W^\top K - K^\top W = \phi_{W^\top W}(\Omega), \hspace{1.5cm} \phi_M(X) = MX + XM.
    \end{equation*}
    Now, provided that $M$ is invertible, $\phi_M$ is also invertible, and we get the expression:
    \begin{equation*}
        \pi_W(K) = K - W \phi_{W^\top W}^{-1} \big( W^\top K - K^\top W \big). 
    \end{equation*}
    This leads to the result thanks to the continuity of the inverse. 
\end{proof}

Before proving \cref{prop:CVRates1}, we give the structure of the Hessian of the function $F$ at a critical point. 
\begin{lemma} \label{lemma:CVRates2}
    Let $W \in \R^{d \times m}$ such that $\nabla F(W) = 0$ and $\mathrm{rank}(W) = m$. Then the Hessian $\d^2 L_W$ has the following representation:
    \begin{equation*}
    \begin{pNiceMatrix}[first-row, last-row=3, first-col, last-col=3, nullify-dots]
    & \mathcal{H}_W & \mathcal{V}_W &\\
    & * & 0 & \mathcal{H}_W\\
     & 0 & 0 & \mathcal{V}_W\\
    &&&
\end{pNiceMatrix},
\end{equation*}
where the orthogonal subspaces $\mathcal{H}_W$ and $\mathcal{V}_W$ are defined in equations \eqref{eq:HorizontalSpace}, \eqref{eq:VerticalSpace}. 
\end{lemma}
\begin{proof}
    Let us show first of all that for all $H \in \R^{d \times m}$, $\d^2 L_W(H) \in \mathcal{H}_W$, whenever $W$ has full rank and $\nabla F(W) = 0$. First of all, remark that, due to the structure of $F(W) = G(WW^\top)$:
    \begin{equation*}
        \nabla F(W) = 2 \nabla G(WW^\top) W \in \mathcal{H}_W. 
    \end{equation*}
    Now, as $\epsilon \to 0$, since $\nabla F(W) = 0$ and $F$ is $\mathcal{C}^2$:
    \begin{equation} \label{eq:projectiongrad1}
        \nabla F(W + \epsilon H) = \epsilon \d^2 F_W(H) + o(\epsilon). 
    \end{equation}
    Now, let $K \in \mathcal{V}_W$ and denote $W_\epsilon = W + \epsilon H$. We have:
    \begin{equation*}
        K = \pi_W^\perp(K) = \pi_{W_\epsilon}^\perp(K) + \big( \pi_W^\perp - \pi_{W_\epsilon}^\perp \big)(K), 
    \end{equation*}
    where $\pi_W^\perp$ denotes the projection onto $\mathcal{V}_W$. Then, since $\nabla F(W_\epsilon) \in \mathcal{H}_{W_\epsilon}$ is orthogonal to $\mathcal{V}_{W_\epsilon}$:
    \begin{equation*}
        \Big| \tr \big( \nabla F(W_\epsilon) K^\top) \Big| \leq \big\| \pi_W^\perp - \pi_{W_\epsilon}^\perp \big\|_{op} \big\| \nabla F(W_\epsilon) \big\| \big\|K\big\|. 
    \end{equation*}
    Now by \cref{Lemma:CVRates1}, $V \mapsto \pi_V^\perp$ is continuous at $V = W$, and using equation \eqref{eq:projectiongrad1}, we have that the previous quantity is $o(\epsilon)$. Taking the scalar product with $K$ in equation \eqref{eq:projectiongrad1}, we obtain that:
    \begin{equation*}
        \tr \big( \d^2 F_W(H) K^\top \big) = 0. 
    \end{equation*}
    Therefore, for all $H \in \R^{d \times m}$, $\d^2 L_W(H) \in \mathcal{H}_W$. Now, since $\d^2 L_W$ is self-adjoint, for any $H \in \R^{d \times m}$ and $K \in \mathcal{V}_W$:
    \begin{equation*}
        \tr \big(\d^2 L_W(K) H^\top \big) = \tr \big( \d^2 L_W(H) K^\top \big) = 0,
    \end{equation*}
    since $\d^2 L_W(H) \in \mathcal{H}_W$. This shows that $\d^2 L_W$ is zero on $\mathcal{V}_W$ and concludes the proof. 
\end{proof}

With these tools, we are now ready to prove \cref{prop:CVRates1}. In the following, we will denote $\pi_t, \pi_t^\perp, \pi_\infty, \pi_\infty^\perp$ the orthogonal projections onto $\mathcal{H}_{W(t)}, \mathcal{V}_{W(t)}, \mathcal{H}_{W_\infty}$ and $\mathcal{V}_{W_\infty}$. First of all, we start with the standard identity for gradient flow:
\begin{equation} \label{eq:CVRatesProj1}
    F(W(t)) - F(W_\infty) = \int_t^\infty \big\| \nabla F(W(s)) \big\|_F^2 \d s. 
\end{equation}
This can be proven by differentiating the function $t \mapsto F(W(t))$ and using the gradient flow structure. Now, since $\nabla F(W(s)) \in \mathcal{H}_{W(s)}$, and by continuity of the projection derived in \cref{Lemma:CVRates1}, we have:
\begin{equation} \label{eq:CVRatesProj2}
    \big\| \nabla F(W(s)) \big\|_F^2 \underset{s \to \infty}{\sim} \big\| \pi_\infty \big( \nabla F(W(s)) \big) \big\|_F^2. 
\end{equation}
We now define:
\begin{equation*}
    \phi(t) = \frac{1}{2} \big\| \pi_\infty \big( \nabla F(W(t)) \big) \big\|_F^2. 
\end{equation*}
It is easily shown that:
\begin{equation}  \label{eq:Gronwall1}
    \dot \phi(t) = - \tr \Big( \d^2 F_{W(t)} \big( \nabla F(W(t)) \big) \pi_\infty \big( \nabla F(W(t)) \big)^\top \Big). 
\end{equation}
We decompose this term by introducing $\d^2 F_{W_\infty}$. We have, using the Cauchy--Schwarz inequality:
\begin{equation} \label{eq:Gronwall2}
\begin{aligned}
    &\Big| \tr \Big( \big(\d^2 F_{W(t)} - \d^2 F_{W_\infty} \big) \big( \nabla F(W(t)) \big) \pi_\infty \big(\nabla F(W(t)) \big)^\top \Big) \Big| \\
    &\hspace{1.5cm} \leq \big\| \d^2 F_{W(t)} - \d^2 F_{W_\infty} \big\|_{op} \big\| \nabla F(W(t)) \big\|_F \big\|  \pi_\infty \big(\nabla F(W(t)) \big) \big\|_F.
\end{aligned}
\end{equation}
Since $F$ is $\mathcal{C}^2$ and using equation \eqref{eq:CVRatesProj2}, this quantity is of the form $\epsilon(t) \phi(t)$ with $\epsilon(t) \xrightarrow[t \to \infty]{} 0$. Moreover, with the representation of the Hessian in \cref{lemma:CVRates2} and the fact that it is positive definite on $\mathcal{H}_{W_\infty}$ with smallest eigenvalue $\varrho$, we have the lower bound:
\begin{equation} \label{eq:Gronwall3}
    \tr \Big( \d^2 F_{W_\infty} \big( \nabla F(W(t)) \big) \pi_\infty \big( \nabla F(W(t)) \big)^\top \Big) \geq 2 \varrho \phi(t).
\end{equation}
Therefore, gathering equations \eqref{eq:Gronwall1}, \eqref{eq:Gronwall2}, \eqref{eq:Gronwall3}, we have:
\begin{equation*}
    \dot \phi(t) \leq -2 \big( \varrho - \epsilon(t) \big) \phi(t).
\end{equation*}
Now, given some $c < \varrho$, integrating the previous inequality leads to:
\begin{equation*}
    \phi(t) = o(e^{-2ct}),
\end{equation*}
as $t \to \infty$. Combining this with equations \eqref{eq:CVRatesProj1}, \eqref{eq:CVRatesProj2}, we get:
\begin{equation*}
    F(W(t)) - F(W_\infty) \underset{t \to \infty}{\sim} 2 \int_t^\infty \phi(s) \d s = o(e^{-2ct}). 
\end{equation*}
For the distance to convergence, we have due to the gradient flow structure:
\begin{equation*}
    \big\|W(t) - W_\infty\big\|_F \leq \int_t^\infty \big\| \nabla F(W(s)) \big\|_F \d s = o(e^{-ct}). 
\end{equation*}
Therefore, with $Z(t) = W(t)W(t)^\top$ and $Z_\infty = W_\infty W_\infty^\top$:
\begin{equation*}
\begin{aligned}
    \big\| Z(t) - Z_\infty \big\|_F &\leq \big\| (W(t) - W_\infty) W(t)^\top \big\|_F + \big\| W_\infty(W(t) - W_\infty)^\top \big\|_F \\
    &\leq \big( \|W(t)\|_F + \|W_\infty\|_F \big) \big\| W(t) - W_\infty\|_F \\
    &= o(e^{-ct}). 
\end{aligned}
\end{equation*}
This concludes the proof. 

\proofsubsection{Prop:HessianOja} \label{App:Subsec:HessianOja}

Let us now conclude the proof of the convergence rates in the case where the Oja flow dynamics converges to a full-rank local minimizer. To do so, we study the eigenvalues of the Hessian:
\begin{equation*}
    \d^2 U_W(H) = \big( WW^\top - A \big) H + H W^\top W + WH^\top W.
\end{equation*}
Given \cref{prop:OjaConvergence}, we write $A = Q \Lambda Q^\top$ and $W = QDU^\top$, where:
\begin{equation} \label{eq:representationMatrices}
    \Lambda = \begin{pmatrix} \Lambda_1 & 0 \\ 0 & \Lambda_2
    \end{pmatrix}, \hspace{1.5cm} D = \begin{pmatrix}
        \Lambda_1^{1/2} \\ 0
    \end{pmatrix},
\end{equation}
where $\Lambda_1 \in \R^{m \times m}$ is diagonal and gathers the $m$ largest positive eigenvalues of $A$, and $\Lambda_2 \in \R^{(d-m) \times (d-m)}$ contains all the remaining eigenvalues. Moreover, $Q \in O_d(\R)$ and $U \in O_m(\R)$. Then, changing variables $H = QK U^\top$, we get:
\begin{equation*}
    \d^2 U_W(QKU^\top) = Q \Big( \big( DD^\top - \Lambda) K + K D^\top D + DK^\top D \Big) U^\top. 
\end{equation*}
Since the mapping $K \in \R^{d \times m} \mapsto QKU^\top  \in \R^{d \times m}$ is invertible, we can diagonalize the Hessian by considering the representation \eqref{eq:representationMatrices}. We therefore consider the eigenvalue problem:
\begin{equation} \label{eq:EigvalsHessian1}
    \big( DD^\top - \Lambda) K + K D^\top D + DK^\top D  = \mu K. 
\end{equation}
To solve this, we decompose:
\begin{equation*}
    K = \begin{pmatrix}
        K_1 \\ K_2
    \end{pmatrix}, \hspace{1.5cm} K_1 \in \R^{m \times m}, \, \hspace{0.5cm} K_2 \in \R^{(d-m) \times m}.
\end{equation*}
Since we restrict $K \in \mathcal{H}_W$, we have the relationship $\Lambda_1^{1/2} K_1 = K_1^\top \Lambda_1^{1/2}$. Then, equation \eqref{eq:EigvalsHessian1} writes:
\begin{align}
    K_1 \Lambda_1 + \Lambda_1^{1/2} K_1^\top \Lambda_1^{1/2} &= \mu K_1, \label{eq:eigvalsOrtho1} \\
    K_2 \Lambda_1 - \Lambda_2 K_2 &= \mu K_2.  \label{eq:eigvalsOrtho2}
\end{align}
This corresponds to two separate eigenvalue equations that we can solve independently. Starting with equation \eqref{eq:eigvalsOrtho2}, it is easily seen that the solutions are of the form $\lambda_i - \lambda_j$ for $i \in \{1, \dots, m\}$ and $j \in \{m+1, \dots, d\}$, with $\lambda_1 \geq \dots \geq \lambda_d$ being the ordered eigenvalues of $A$. Now considering equation \eqref{eq:eigvalsOrtho1}, we choose, for $1 \leq i \leq j \leq d$:
\begin{equation*}
    K_1 = e_i e_j^\top + \sqrt{\frac{\lambda_i}{\lambda_j}} e_j e_i^\top, 
\end{equation*}
so that:
\begin{equation*}
    \Lambda_1^{1/2} K_1 = \sqrt{\lambda_i} \big( e_i e_j^\top + e_j e_i^\top \big) = K_1^\top \Lambda^{1/2},
\end{equation*}
and:
\begin{equation*}
\begin{aligned}
    K_1 \Lambda_1 + \Lambda_1^{1/2} K_1^\top \Lambda_1^{1/2} &= (\lambda_i + \lambda_j) e_ie_j^\top + \left( \lambda_i \sqrt{\frac{\lambda_i}{\lambda_j}} + \sqrt{\lambda_i \lambda_j} \right) e_je_i^\top \\
    &= (\lambda_i + \lambda_j) K_1. 
\end{aligned}
\end{equation*}
This corresponds to $\dfrac{1}{2}m(m+1)$ eigenvalues, which is the dimension of the admissible $K_1$ under the symmetry constraint $\Lambda_1^{1/2} K_1 = K_1^\top \Lambda_1^{1/2}$. This concludes the proof. 

\proofsubsection{Prop:CVRates2} \label{App:Subsec:CVRates2}

Let us now prove the convergence rates of the Oja flow in the case where it converges toward a rank-deficient matrix. To do so, we decompose $A = Q \Lambda Q^\top$, with:
\begin{equation} \label{eq:LambdaDiagonal}
    \Lambda = \begin{pmatrix} \Lambda_1 & 0 \\ 0 & -\Lambda_2 \end{pmatrix}.
\end{equation}
In this case $\Lambda_1 \in \R^{p \times p}, \Lambda_2 \in \R^{(d-p) \times (d-p)}$ are diagonal and are both positive definite. Indeed, we assumed here that $A$ was invertible. If this is not the case, it is known \citep[see for instance][]{martin2024impact, sarao2020optimization} that the convergence is not exponentially fast anymore. From the gradient flow dynamics:
\begin{equation*}
    \dot W(t) = \big( A - W(t)W(t)^\top \big) W(t),
\end{equation*}
we get the equation on $Z(t) = W(t)W(t)^\top$:
\begin{equation*}
    \dot Z(t) = \big( A - Z(t) \big) Z(t) + Z(t) \big( A - Z(t) \big). 
\end{equation*}
Now, we consider $Y(t) = Q^\top Z(t) Q$ and obtain:
\begin{equation} \label{eq:CVRatesOja1}
    \dot Y(t) = \big( \Lambda - Y(t) \big) Y(t) + Y(t) \big( \Lambda - Y(t) \big). 
\end{equation}
In addition, we have $\|Z(t) - Z_\infty\|_F^2 = \|Y(t) - Y_\infty\|_F^2$, so we restrict ourselves to the study of $Y(t)$. In this case, we have from \cref{prop:OjaConvergence}:
\begin{equation*}
    Y_\infty = \begin{pmatrix} \Lambda_1 & 0 \\ 0 & 0 \end{pmatrix}.
\end{equation*}
We then decompose:
\begin{equation*}
    Y(t) = \begin{pmatrix} Y_1(t) & Y_2(t) \\ Y_2(t)^\top & Y_3(t)   \end{pmatrix},
\end{equation*}
where $Y_1(t) \in \mathcal{S}_p^+(\R), Y_3(t) \in \mathcal{S}_{d-p}^+(\R)$ and $Y_2(t) \in \R^{p \times (d-p)}$. Then, equation \eqref{eq:CVRatesOja1} leads to the dynamics:
\begin{equation*}
\begin{aligned}
    \dot Y_1(t) &= (\Lambda_1 - Y_1(t)) Y_1(t) + Y_1(t) (\Lambda_1 - Y_1(t)) - 2 Y_2(t) Y_2(t)^\top, \\
    \dot Y_2(t) &= (\Lambda_1 - 2 Y_1(t)) Y_2(t) - Y_2(t)( \Lambda_2 + 2 Y_3(t)), \\
    \dot Y_3(t) &= -(\Lambda_2 + Y_3(t)) Y_3(t) - Y_3(t)(\Lambda_2 + Y_3(t)) - 2 Y_2(t)^\top Y_2(t).
\end{aligned}
\end{equation*}
In order to study the distance to $Y_\infty$, we define the functions:
\begin{equation*}
    \phi_1(t) = \big\| Y_1(t)  -\Lambda_1 \big\|_F^2, \hspace{1.2cm} \phi_2(t) = \big\| Y_2(t) \big\|_F^2, \hspace{1.2cm} \phi_3(t) = \big\| Y_3(t) \big\|_F^2.
\end{equation*}
We obtain the dynamics:
\begin{equation*}
\begin{aligned}
    \frac{1}{2} \dot \phi_1(t) &= - 2 \tr \Big( Y_1(t) (Y_1(t) - \Lambda_1)^2 \Big)- 2 \tr \Big( Y_2(t) Y_2(t)^\top (Y_1(t) - \Lambda_1) \Big), \\
    \frac{1}{2} \dot \phi_2(t) &= \tr \Big( Y_2(t) Y_2(t)^\top (\Lambda_1 - 2 Y_1(t)) \Big) -  \tr \Big( Y_2(t)^\top Y_2(t)(\Lambda_2 + 2 Y_3(t)) \Big), \\
    \frac{1}{2} \dot \phi_3(t) &= - 2 \tr \Big( Y_3(t)^2 ( Y_3(t) + \Lambda_2) \Big) - 2 \tr \Big( Y_2(t)^\top Y_2(t) Y_3(t) \Big).
\end{aligned}
\end{equation*}
Since $Y_2^\top Y_2$ and $Y_3$ are both PSD, we have $\tr \big( Y_2^\top Y_2 Y_3) \geq 0$ and $\tr(Y_3^3) \geq 0$, so that:
\begin{equation*}
\begin{aligned}
    \dot \phi_1(t) &\leq -4\lambda_{\min} \big(Y_1(t) \big) \phi_1(t) - 4 \lambda_{\min} \big( Y_1(t) - \Lambda_1 \big) \phi_2(t), \\
    \dot \phi_2(t) &\leq -2 \Big( \lambda_{\min}\big( Y_1(t) - \Lambda_1 \big) + \lambda_{\min}\big( Y_1(t) \big) + \lambda_{\min}(\Lambda_2) \Big) \phi_2(t), \\
    \dot \phi_3(t) &\leq -4 \lambda_{\min}(\Lambda_2) \phi_3(t),
\end{aligned}
\end{equation*}
where we used the inequality:
\begin{equation*}
    \lambda_{\min}(B) \tr(A) \leq \tr(AB),
\end{equation*}
for $B \in \mathcal{S}_d(\R)$ and $A \in \mathcal{S}_d^+(\R)$. Therefore, defining:
\begin{equation*}
\begin{aligned}
    \alpha_1(t) &= \int_0^t \lambda_{\min}\big( Y_1(s) \big) \d s, \\
    \alpha_2(t) &= \alpha_1(t) + \lambda_{\min}(\Lambda_2) t + \int_0^t  \lambda_{\min}\big( Y_1(s) - \Lambda_1 \big) \d s,
\end{aligned}
\end{equation*}
we can integrate the previous bounds and get:
\begin{align}
    \phi_1(t) &\leq e^{-4 \alpha_1(t)} \left( \phi_1(0) - 4 \int_0^t \lambda_{\min}\big( Y_1(s) - \Lambda_1 \big) e^{4 \alpha_1(s)} \phi_2(s) \d s \right), \label{eq:boundconvergence1} \\ 
    \phi_2(t) &\leq e^{-2 \alpha_2(t)} \phi_2(0), \label{eq:boundconvergence2} \\
    \phi_3(t) &\leq e^{-4 \lambda_{\min}(\Lambda_2) t} \phi_3(0). \label{eq:boundconvergence3}
\end{align}
Since $Y_1(t)$ converges to $\Lambda_1$, we have the asymptotics:
\begin{equation*}
\begin{aligned}
    \alpha_1(t) &= \lambda_{\min}(\Lambda_1) t + o(t), \\
    \alpha_2(t) &= \big(  \lambda_{\min}(\Lambda_1) +  \lambda_{\min}(\Lambda_2) \big) t + o(t).   
\end{aligned}
\end{equation*}
For simplicity, let us now write $\rho_1 = \lambda_{\min}(\Lambda_1)$ and $\rho_2 = \lambda_{\min}(\Lambda_2)$. Then, from equations \eqref{eq:boundconvergence2}, \eqref{eq:boundconvergence3}, as $t \to \infty$:
\begin{equation} \label{eq:boundsphi23}
    \phi_2(t) \leq e^{-2(\rho_1 + \rho_2)t} e^{o(t)}, \hspace{1.5cm} \phi_3(t) \leq e^{-4 \rho_2 t} e^{o(t)}.  
\end{equation}
In addition, the integrand in equation \eqref{eq:boundconvergence1} verifies:
\begin{equation*}
    \Big| \lambda_{\min}\big( Y_1(t) - \Lambda_1 \big) e^{4 \alpha_1(t)} \phi_2(t) \Big| \leq e^{2(\rho_1 - \rho_2)t} e^{o(t)}.
\end{equation*}
Therefore, we have as $t \to \infty$:
\begin{equation*}
    \left| \int_0^t \lambda_{\min}\big( Y_1(s) - \Lambda_1 \big) e^{4 \alpha_1(s)} \phi_2(s) \d s \right| \leq \left\{ \begin{array}{cc}
         e^{2(\rho_1 - \rho_2)t} e^{o(t)}, & \mathrm{if} \ \rho_1 > \rho_2, \\
         e^{o(t)}, & \mathrm{otherwise.}
    \end{array} \right. 
\end{equation*}
Plugging this into equation \eqref{eq:boundconvergence1}, we obtain:
\begin{equation} \label{eq:boundsphi1}
    \phi_1(t) \leq e^{-2\rho_1 t} e^{-2 \min(\rho_1, \rho_2) t} e^{o(t)}. 
\end{equation}
Combining the bounds \eqref{eq:boundsphi23} and \eqref{eq:boundsphi1}, we obtain that the three functions $\phi_1(t), \phi_2(t)$ and $\phi_3(t)$ are all bounded by $e^{-4 \min(\rho_1, \rho_2) t} e^{o(t)}$ as $t \to \infty$. To conclude the proof, we use the expressions:
\begin{equation*}
\begin{aligned}
    U(W(t)) - U(W_\infty) &= \frac{1}{4} \Big( \phi_1(t) + 2 \phi_2(t) + \phi_3(t) + 2 \tr \big( Y_3(t) \Lambda_2 \big) \Big), \\
    \big\| Z(t) - Z_\infty \big\|_F^2 &= \frac{1}{4} \Big( \phi_1(t) + 2 \phi_2(t) + \phi_3(t) \Big). 
\end{aligned}
\end{equation*}
Using the Cauchy--Schwarz inequality guarantees that:
\begin{equation*}
     \tr \big( Y_3(t) \Lambda_2 \big) \leq \| \Lambda_2 \|_F \sqrt{\phi_3(t)} \leq e^{-2 \rho_2 t} e^{o(t)}. 
\end{equation*}
Now introducing $c < \min( 2 \rho_1, \rho_2)$, we obtain that as $t \to \infty$:
\begin{equation*}
    U(W(t)) - U(W_\infty) = o(e^{-2ct}), \hspace{1.5cm} \big\| Z(t) - Z_\infty \big\|_F = o(e^{-ct}). 
\end{equation*}
To conclude the proof, it is clear that we have $\rho_1 = \lambda_p$ and $\rho_2 = |\lambda_{p+1}|$, where $\lambda_1 > \dots > \lambda_d$ are the ordered eigenvalues of $A$. 

\proofsubsection{Prop:ConvergenceHighDimOja} \label{App:Subsec:HighDimLimit}

Before proving \cref{Prop:ConvergenceHighDimOja}, we give two essential lemmas for our proof. The first lemma establishes high-dimensional limits for traces of functions of a covariance matrix composed with a Gaussian transformation.

\begin{lemma} \label{Lemma:ConvergenceRMT1}
    Let $W \in \R^{d \times m}$ be a Gaussian matrix with i.i.d. coefficients of variance $1/m$, and $K \in \mathcal{S}_d^+(\R)$ with converging empirical spectral distribution as $d \to \infty$ to some probability measure $\mu$ with compact support. Let $V = K^{1/2} W \in \R^{d \times m}$ and:
    \begin{equation*}
        G = V \big( I_m + V^\top V \big)^{-1} V^\top. 
    \end{equation*}
    Then, given $\phi$ a spectral function on $\mathcal{S}_d(\R)$, we have in the $d \to \infty$ limit, with $m \sim \kappa d$:
    \begin{equation*}
    \begin{aligned}
        \frac{1}{d} \E \, \tr \big( \phi(K) G \big) &\xrightarrow[d \to \infty]{} \mathfrak{g} \int \frac{x \phi(x)}{1 + \mathfrak{g} x} \d \mu(x), \\
        \frac{1}{d} \E \, \tr \big( \phi(K) G \phi(K) G \big) &\xrightarrow[d \to \infty]{} \mathfrak{g}^2 \int  \phi(x)^2 \frac{x^2}{(1+x \mathfrak{g})^2} \d \mu(x) \\
        & \hspace{1.5cm}+ \mathfrak{g} \left( \int \frac{x \phi(x)}{1 +x \mathfrak{g}} \d \mu(x) \right)^2 \left( \kappa + \int \frac{x}{(1+x \mathfrak{g})^2} \d \mu(x) \right)^{-1},
    \end{aligned}
    \end{equation*}
    where $\mathfrak{g}$ is the unique solution of the equation:
    \begin{equation*}
        \kappa \mathfrak{g} + 1 - \kappa = \int \frac{\d \mu(x)}{1 + x \mathfrak{g}}. 
    \end{equation*}
\end{lemma}
\begin{proof}
    In a basis where $K$ is diagonal, we have:
    \begin{align}
        \frac{1}{d} \tr \big( \phi(K) G \big) &= \frac{1}{d} \sum_{i=1} \phi(\lambda_i) G_{ii}, \label{eq:TraceDecompositions1} \\
        \frac{1}{d} \tr \big( \phi(K) G \phi(K) G \big) &= \frac{1}{d} \sum_{i,j=1}^d \phi(\lambda_i) \phi(\lambda_j) G_{ij}^2. \label{eq:TraceDecompositions2}
    \end{align}
    We now compute the statistics of the random variables $G_{ij}$. Let us denote $w_1, \dots, w_d \in \R^m$ the rows of $W$, so that:
    \begin{equation*}
        G_{ij} = \sqrt{\lambda_i \lambda_j} w_i^\top \left( I_m + \sum_{k=1}^d \lambda_k w_k w_k^\top \right)^{-1} w_j. 
    \end{equation*}
    Let us start by taking $i = j$, and define:
    \begin{equation*}
        M = \left(  I_m + \sum_{k = 1}^d \lambda_k w_k w_k^\top \right)^{-1}, \hspace{1.5cm} M_i = \left(  I_m + \sum_{k \neq i} \lambda_k w_k w_k^\top \right)^{-1}.
    \end{equation*}
    The key point is that, due to the covariance structure of the Gaussian matrix $W$, $M_i$ is independent from $w_i$. Now, as a consequence of the Sherman--Morison formula \citep[see for instance][]{hager1989updating}, we have:
    \begin{equation*}
        G_{ii} = \frac{\lambda_i w_i^\top M_i w_i}{1 + \lambda_i w_i^\top M_i w_i}. 
    \end{equation*}
    Since $M_i$ and $w_i$ are independent, we have:
    \begin{equation*}
    \begin{aligned}
        \E \, w_i^\top M_i w_i &= \frac{1}{m} \E \, \tr(M_i), \\
        \mathrm{Var} \, w_i^\top M_i w_i &= \frac{2}{m^2} \E \, \tr(M_i^2) + \frac{1}{m^2} \mathrm{Var} \, \tr(M_i). 
    \end{aligned}
    \end{equation*}
    Now, as $d \to \infty$, one can replace $M_i$ by $M$ in the previous equations since we only removed one index. Now, it is known that \citep[see for instance][]{bun2017cleaning}:
    \begin{equation*}
        \frac{1}{m} \E \, \tr(M) \xrightarrow[d \to \infty]{} \mathfrak{g}, \hspace{1.5cm} \lim_{m \to \infty} \frac{1}{m} \E \, \tr(M^2) < \infty,
    \end{equation*}
    where $\mathfrak{g}$ solves the self-consistent equation:
    \begin{equation} \label{eq:StieltjesSelfConsistent1}
        \kappa \mathfrak{g} + 1 - \kappa = \int \frac{\d \mu(x)}{1 + x \mathfrak{g}}. 
    \end{equation}
    Therefore, $w_i^\top M_i w_i$ concentrates as $d \to \infty$ toward $\mathfrak{g}$, and:
    \begin{equation*}
        \E \, G_{ii} \xrightarrow[d \to \infty]{} \frac{\lambda_i \mathfrak{g}}{1 + \lambda_i \mathfrak{g}}. 
    \end{equation*}
    This shows that, using equation \eqref{eq:TraceDecompositions1}:
    \begin{equation*}
        \frac{1}{d} \E \, \tr \big( \phi(K) G \big) \xrightarrow[d \to \infty]{} \mathfrak{g} \int \frac{x \phi(x)}{1 + \mathfrak{g} x} \d \mu(x). 
    \end{equation*}
    For the off-diagonal case, the same can be done by removing the rows $i,j$. Then using again the Sherman--Morrison formula with the rank-two matrix update $\lambda_i w_iw_i^\top + \lambda_j w_jw_j^\top$, we get:
    \begin{equation*}
    \begin{aligned}
        G_{ij} = \sqrt{\lambda_i \lambda_j} \frac{w_i^\top M_{ij} w_j}{(1 + \lambda_i w_i^\top M_{ij} w_i)(1 + \lambda_j w_j^\top M_{ij} w_j) - \lambda_i \lambda_j (w_i^\top M_{ij} w_j)^2},
    \end{aligned}
    \end{equation*}
    with:
    \begin{equation*}
        M_{ij} = \left(  I_m + \sum_{k \neq i,j} \lambda_k w_k w_k^\top \right)^{-1}.
    \end{equation*}
    Again, $w_i^\top M_{ij} w_i$ and $w_j^\top M_{ij} w_j$ concentrate around $\mathfrak{g}$. Now, we have:
    \begin{equation*}
    \begin{aligned}
        \E \, (w_i^\top M_{ij} w_j)^2 &= \frac{1}{m^2} \E \, \tr(M_{ij}^2), \\
        \mathrm{Var} \, (w_i^\top M_{ij} w_j)^2 &= \frac{3}{m^4} \E\big[\tr(M_{ij}^2)^2 \big] - \frac{1}{m^4} \E \big[ \tr(M_{ij}^2) \big]^2 + \frac{6}{m^4} \E \, \tr(M_{ij}^4). 
    \end{aligned}
    \end{equation*}
    This shows that:
    \begin{equation} \label{eq:AsymptoticOffdiag1}
        d \, \E \, G_{ij}^2 \xrightarrow[d \to \infty]{} \frac{1}{\kappa} \frac{\lambda_i \lambda_j}{(1 + \lambda_i \mathfrak{g})^2(1 + \lambda_j \mathfrak{g})^2} \lim_{d \to \infty} \frac{1}{m} \E \, \tr(M^2). 
    \end{equation}
    However, note that $G_{ij}^2$ has a variance which is of the same order than its expectation. This means that this quantity does not concentrate in the high-dimensional limit. 

    Let us compute the expectation of $\tr(M^2)$ in the limit. From \citet[][Chapter 3]{bun2017cleaning}, if we consider $M(z) = \big( z I_m - V^\top V \big)^{-1}$, then we have the asymptotic:
    \begin{equation*}
        \frac{1}{m} \E \, \tr \, M(z)  \xrightarrow[d \to \infty]{} \mathfrak{g}(z), \hspace{1.5cm} \kappa z \mathfrak{g}(z) + 1 - \kappa = \int \frac{\d \mu(x)}{1 - x \mathfrak{g}(z)}. 
    \end{equation*}
    Then:
    \begin{equation*}
        \frac{1}{m} \E \, \tr(M(z)^2) \xrightarrow[d \to \infty]{} - \mathfrak{g}'(z),
    \end{equation*}
    and we can compute $\mathfrak{g}'(z)$ by differentiating the self-consistent equation for $\mathfrak{g}(z)$:
    \begin{equation*}
        \mathfrak{g}'(z) = - \kappa \mathfrak{g}(z) \left( \kappa z - \int \frac{x \d \mu(x)}{(1-x \mathfrak{g}(z))^2} \right)^{-1}.
    \end{equation*}
    Then, taking $z = -1$, we obtain:
    \begin{equation} \label{eq:AsymptoticOffdiag2}
         \lim_{d \to \infty} \frac{1}{m} \E \, \tr(M^2) = \kappa \mathfrak{g} \left( \kappa + \int \frac{x \d \mu(x)}{(1 + x \mathfrak{g})^2} \right)^{-1},
    \end{equation}
    where $\mathfrak{g}$ is the solution of the self-consistent equation \eqref{eq:StieltjesSelfConsistent1}. Then, back to equation \eqref{eq:TraceDecompositions2}, using that $G_{ii}$ concentrates around its mean as well as equations \eqref{eq:AsymptoticOffdiag1}, \eqref{eq:AsymptoticOffdiag2}:
    \begin{equation*}
    \begin{aligned}
        \frac{1}{d} \E \, \tr \big( \phi(K) G \phi(K) G \big) &= \frac{1}{d} \sum_{i=1}^d \phi(\lambda_i)^2 \E \, G_{ii}^2 + \frac{1}{d} \sum_{i \neq j} \phi(\lambda_i) \phi(\lambda_j) \E \, G_{ij}^2 \\
        &\xrightarrow[d \to \infty]{} \mathfrak{g}^2 \int \phi(x)^2 \frac{x^2}{(1+x \mathfrak{g})^2} \d \mu(x)  \\
        & \hspace{1.5cm}+ \mathfrak{g} \left( \int \frac{x \phi(x)}{(1 +x \mathfrak{g})^2} \d \mu(x) \right)^2 \left( \kappa + \int \frac{x \d \mu(x)}{(1+x \mathfrak{g})^2} \right)^{-1},
    \end{aligned}
    \end{equation*}
    which is the desired. 
\end{proof}

The previous lemma established the convergence of the expectation of our quantities of interest. The following guarantees concentration in the high-dimensional limit.

\begin{lemma} \label{Lemma:ConvergenceRMT2}
    Consider the same setup as in \cref{Lemma:ConvergenceRMT1}. Then the two quantities:
    \begin{equation*}
        \frac{1}{d} \tr \big( \phi(K) G \big), \hspace{1.5cm} \frac{1}{d}\tr \big( \phi(K) G \phi(K) G \big),
    \end{equation*}
    have vanishing variance as $d \to \infty$. 
\end{lemma}
\begin{proof}
    We see both of these quantities as functions of $W$. Following from the Poincaré inequality, we have for a $\mathcal{C}^1$ function $F \colon \R^{d \times m} \to \R$:
    \begin{equation*}
        \mathrm{Var}(F(W)) \leq \frac{1}{m} \E \, \big\| \nabla F(W) \big\|_F^2. 
    \end{equation*}
    Writing $G = K^{1/2} E(W) K^{1/2}$ with $E(W) = W(I_m + W^\top K W)^{-1} W^\top$ we have the identities:
    \begin{equation*}
    \begin{aligned}
        \nabla_W \frac{1}{d} \tr \big( \phi(K) G \big) &= \frac{1}{d} \d E_W^* \big( K \phi(K) \big), \\
         \nabla_W \frac{1}{d} \tr \big( \phi(K) G \phi(K) G\big) &= \frac{2}{d} \d E_W^* \big( K \phi(K) E(W) K \phi(K) \big),        
    \end{aligned}
    \end{equation*}
    where $\d E_W^*$ is the adjoint of the differential of $E$ at $W$. Now, with $M = I_m + W^\top K W$, one has for $U \in \mathcal{S}_d(\R)$:
    \begin{equation*}
        \d E_W^*(U) = 2 U W M^{-1} - 2 K W M^{-1} W^\top U W M^{-1}. 
    \end{equation*}
    Now computing the squared norms of the gradients, one can write in both cases:
    \begin{equation*}
        \mathrm{Var}\big( F(W) \big) \leq \frac{1}{md^2} \E \, \tr \Big( \mathcal{H} \big( K, \phi(K), E, \tilde E \big) \Big), \hspace{1.5cm} \tilde E = W M^{-2} W^\top,
    \end{equation*}
    and $\mathcal{H}$ is a sum of products of its arguments. Since $E, \tilde E$ only depend on $K$ and the Gaussian matrix $W$ whose covariance scales as $d^{-1}$, in the high-dimensional limit, these traces are of order $d$ as $d \to \infty$. This leads to a variance of order $d^{-2}$. 
\end{proof}

We are now ready to prove \cref{Prop:ConvergenceHighDimOja}, that derives the high-dimensional limit associated with a scalar quantity of the Oja flow:
\begin{equation*}
    \frac{1}{d} \big\| Z(t) - \phi(A) \big\|_F^2,
\end{equation*}
where $Z(t)$ is solution of the Oja flow \eqref{eq:OjaFlow} with target matrix $A$, and $\phi$ is a spectral function on $\mathcal{S}_d(\R)$. 

We start by using the explicit solution of the Oja flow in \cref{prop:OjaSolution}. Due to the invertibility of $A$, we have:
\begin{equation*}
    Z(t) = e^{tA} W_0 \left( I_m + W_0^\top A^{-1} \big( e^{2tA} - I_d \big) W_0 \right)^{-1} W_0^\top e^{tA}. 
\end{equation*}
Let us now define:
\begin{equation*}
    \Sigma(t) = A^{-1} \big( e^{2tA} - I_d \big) \succeq 0, \hspace{1.5cm} R(t) = e^{tA} \Sigma(t)^{-1/2},
\end{equation*}
which are both spectral functions of $A$. Now, with $V(t) = \Sigma(t)^{1/2} W_0$, we have:
\begin{equation*}
    Z(t) = R(t) \underbrace{V(t) \Big( I_m + V(t)^\top V(t) \Big)^{-1} V(t)^\top}_{\displaystyle \equiv G(t)} R(t).
\end{equation*}
Then, given $\phi$ a spectral function of $A$:
\begin{equation*}
    \frac{1}{d} \big\| Z(t) - \phi(A) \big\|_F^2 = \frac{1}{d} \tr \Big( R(t)^2 G(t) R(t)^2 G(t) \Big) - \frac{2}{d} \tr \Big( \phi(A) R(t)^2 G(t) \Big) + \frac{1}{d} \tr \big( \phi(A)^2 \big). 
\end{equation*}
Since $\Sigma(t), R(t)$ are both spectral functions of $A$, we can apply \cref{Lemma:ConvergenceRMT1} and \cref{Lemma:ConvergenceRMT2} to get that the first two terms concentrate in the $d \to \infty$ limit, for fixed $t \geq 0$:
\begin{equation*}
\begin{aligned}
    \frac{1}{d} \tr \Big( \phi(A) R(t)^2 G(t) \Big) &\xrightarrow[d \to \infty]{} \mathfrak{g}(t) \int \phi(x) \frac{x e^{2tx}}{(e^{2xt}-1)\mathfrak{g}(t) + x} \d \mu_A(x), \\
    \frac{1}{d} \tr \Big( R(t)^2 G(t) R(t)^2 G(t) \Big) &\xrightarrow[d \to \infty]{} \mathfrak{g}(t)^2 \int \left( \frac{x e^{2tx}}{(e^{2xt}-1)\mathfrak{g}(t) + x} \right)^2 \d \mu_A(x) \\
    &+ \mathfrak{g}(t) \left( \int \left( \frac{x e^{tx}}{(e^{2xt}-1)\mathfrak{g}(t) + x} \right)^2 \d \mu_A(x) \right)^2 \\
    &\hspace{1cm} \times \left( \kappa + \int \frac{x(e^{2xt} - 1)}{((e^{2xt}-1)\mathfrak{g}(t) + x)^2} \d \mu_A(x) \right)^{-1},
\end{aligned}
\end{equation*}
where we used that the eigenvalues of $\Sigma(t)$ are of the form:
\begin{equation*}
    \frac{e^{2\lambda t} - 1}{\lambda},
\end{equation*}
for $\lambda \in \mathrm{Sp}(A)$. Moreover, $\mathfrak{g}(t)$ is solution of the self-consistent equation:
\begin{equation} \label{eq:selfconsistentStieltjes}
    \kappa \mathfrak{g}(t) + 1 - \kappa = \int \frac{x}{(e^{2xt}-1) \mathfrak{g}(t) + x} \d \mu_A(x).
\end{equation}
Finally, since we have:
\begin{equation*}
    \frac{1}{d} \tr(\phi(A)^2) \xrightarrow[d \to \infty]{} \int \phi(x)^2 \d \mu_A(x),
\end{equation*}
we get the desired by rearranging the terms.

\proofsubsection{Prop:HighDimCVRates} \label{App:Subsec:LongTime}

Let us now show \cref{Prop:HighDimCVRates}, that gives access to the high-dimensional convergence rates for the Oja flow. To do so, we apply the result of \cref{Prop:ConvergenceHighDimOja} to the spectral function:
\begin{equation*}
    \phi(x) = x \1_{x \geq \max(0, \omega)}, \hspace{1.5cm} \kappa = \int \1_{x \geq \omega} \d \mu_A(x). 
\end{equation*}
The following lemma describes the behavior of the function $\mathfrak{g}(t)$, solution of the self-consistent equation \eqref{eq:selfconsistentStieltjes}.

\begin{lemma} \label{Lemma:LongtimesStieltjes}
    Consider $\mathfrak{g}(t)$ to be the unique solution of equation \eqref{eq:selfconsistentStieltjes}. Define:
    \begin{equation*}
        \kappa_A = \int \1_{x > 0} \d \mu_A(x). 
    \end{equation*}
    \begin{itemize}
        \item If $\kappa > \kappa_A$, $\mathfrak{g}(t) \xrightarrow[t \to \infty]{} \mathfrak{g}_\infty \in (0,1)$, such that:
        \begin{equation} \label{eq:AsymptoticStieltjes1}
            \kappa \mathfrak{g}_\infty + 1 - \kappa = \int_{x < 0} \frac{x}{x - \mathfrak{g}_\infty} \d \mu_A(x). 
        \end{equation}
        \item If $\kappa < \kappa_A$, $\mathfrak{g}(t) \underset{t \to \infty}{\sim} \omega e^{-2\omega t}$, where $\omega > 0$ is such that:
        \begin{equation} \label{eq:AsymptoticStieltjes2}
            \kappa = \int \1_{x \geq \omega} \d \mu_A(x). 
        \end{equation}
    \end{itemize}
\end{lemma}
\begin{proof}
For the first case, we simply start to assume that $\mathfrak{g}(t)$ converges to a non-zero value, and take the pointwise limit in the integral of equation \eqref{eq:selfconsistentStieltjes}. We directly obtain equation \eqref{eq:AsymptoticStieltjes1}. 

Now, the positivity of $\mathfrak{g}_\infty$ imposes the inequalities:
\begin{equation*}
    1 - \kappa < \kappa \mathfrak{g}_\infty + 1 - \kappa = \int_{x < 0} \frac{x}{x - \mathfrak{g}_\infty} \d \mu_A(x) < 1 - \kappa_A.
\end{equation*}
Therefore this asymptotic behavior is only suitable for $\kappa > \kappa_A$. 

Now considering the case $\kappa < \kappa_A$, we assume an asymptotic of the form $\mathfrak{g}(t) = h(t) e^{-2\omega t}$ for some $\omega > 0$ and a sub-exponential function $h$, that is:
\begin{equation*}
    \frac{1}{t} \log h(t) \xrightarrow[t \to \infty]{} 0. 
\end{equation*}
Then, going back to the self-consistent equation \eqref{eq:selfconsistentStieltjes}, we get:
\begin{equation*}
    1 - \kappa = \int^{\omega} \d \mu_A(x).
\end{equation*}
We then obtain equation \eqref{eq:AsymptoticStieltjes2}. Since $\kappa < \kappa_A$, we have that $\omega > 0$. Let us now compute the asymptotic of $h(t)$. We have, replacing the expression of $\mathfrak{g}(t)$ in equation \eqref{eq:selfconsistentStieltjes} and using equation \eqref{eq:AsymptoticStieltjes2}:
\begin{equation*}
    \kappa h(t) e^{-2\omega t} = \int \left( \frac{x}{(e^{2xt} - 1) h(t) e^{-2\omega t} + x} - \1_{x < \omega} \right) \d \mu_A(x).
\end{equation*}
Changing variables $x = \omega + u/t$ and getting rid of the subleading terms, we get:
\begin{equation*}
    \kappa h(t) e^{-2wt} \underset{t \to \infty}{\sim} \frac{1}{t} \rho_A(\omega) \int_{-\infty}^{+\infty} \left( \frac{\omega}{e^{2u} h(t) + \omega} - \1_{u < 0} \right) \d u.  
\end{equation*}
Now, since $h(t)$ is sub-exponential, the integral converges to zero so that the LHS can indeed decay exponentially fast. Now, one can rearrange this integral to get:
\begin{equation*}
    0 = \lim_{t \to \infty} \int_{-\infty}^{+\infty} 
    \left( \frac{\omega}{e^{2u} h(t) + \omega} - \1_{u < 0} \right) \d u  = \lim_{t \to \infty} \left[\theta \left( \frac{\omega}{h(t)} \right) - \theta \left( \frac{h(t)}{\omega} \right) \right],
\end{equation*}
with:
\begin{equation*}
    \theta(x) = \int_0^{+\infty} \frac{x}{e^{2u} + x} \d u = \frac{1}{2} \log(x+1). 
\end{equation*}
Therefore, necessarily $\lim_{t \to \infty} h(t) = \omega$. This concludes the proof. 
\end{proof}

Let us now prove the asymptotics given in \cref{Prop:HighDimCVRates}. We start by the case $\kappa > \kappa_A$, in which case $\phi(x) = x \1_{x > 0}$, since the Oja flow selects all the positive eigenvalues of $A$. We start with the first term of equation \eqref{eq:ConvergenceHighDimOja}. We have, setting $u = tx$:
\begin{equation*}
\begin{aligned}
    \int \left( \frac{xe^{xt}}{q_t(x)} \right)^2 \d \mu_A(x) &= \frac{1}{t^3} \int \rho_A \left( \frac{u}{t} \right) u^2 e^{2u} \left( \big( e^{2u} - 1 \big) \mathfrak{g}(t) + \frac{u}{t} \right)^{-2} \d u \\
    &\underset{t \to \infty}{\sim} \frac{\rho_A(0)}{\mathfrak{g}_\infty^2 t^3} \int_{-\infty}^{+\infty} \frac{u^2 e^{2u}}{(e^{2u} - 1)^2} \d u. 
\end{aligned}
\end{equation*}
Now, since the integral with respect to $z$ in equation \eqref{eq:ConvergenceHighDimOja} remains positive, we obtain for the first term: 
\begin{equation*}
    \mathfrak{g}(t) \left( \int \left( \frac{x e^{xt}}{q_t(x)} \right)^2 \d \mu_A(x) \right)^2 \left( \kappa + \int \frac{z(e^{2zt} - 1)}{q_t(z)^2} \d \mu_A(z) \right)^{-1} = O \left( \frac{1}{t^6} \right). 
\end{equation*}
Now, for the second term, we perform the same change of variables:
\begin{equation*}
\begin{aligned}
    \int \left( \mathfrak{g}(t) \frac{xe^{2xt}}{q_t(x)} - \phi(x) \right)^2 \d \mu_A(x) &= \frac{1}{t^3} \int  \rho_A \left( \frac{u}{t} \right) u^2 \left( \mathfrak{g}(t) e^{2u} \left( \big( e^{2u} - 1 \big) \mathfrak{g}(t) + \frac{u}{t} \right)^{-1} - \1_{u > 0} \right)^2 \d u \\
    &\underset{t \to \infty}{\sim} \frac{\rho_A(0)}{t^3} \int_{-\infty}^{+\infty} u^2 \left( \frac{e^{2u}}{e^{2u} - 1} - \1_{u > 0} \right)^2 \d u. 
\end{aligned}
\end{equation*}
This gives the asymptotic for $\kappa > \kappa_A$ since the integral is finite. 

Regarding the underparameterized region $\kappa < \kappa_A$, we have $\phi(x) = x \1_{x \geq \omega}$, and we split the expression \eqref{eq:ConvergenceHighDimOja} as:
\begin{equation} \label{eq:asymptoticIntegralDecomp}
    \lim_{d \to \infty} \frac{1}{d} \big\| Z(t) - Z_\infty \big\|_F^2 = \mathfrak{g}(t) \frac{I_1(t)^2}{\kappa + I_2(t)} + I_3(t),
\end{equation}
with:
\begin{equation*}
\begin{aligned}
    I_1(t) &= \int \frac{x^2 e^{2xt}}{((e^{2xt}-1) \mathfrak{g}(t) +x)^2} \d \mu_A(x), \\
    I_2(t) &= \int \frac{x(e^{2xt} - 1)}{((e^{2xt}-1) \mathfrak{g}(t) +x)^2} \d \mu_A(x), \\
    I_3(t) &= \int \left( \mathfrak{g}(t) \frac{xe^{2xt}}{(e^{2xt}-1) \mathfrak{g}(t) +x} - x \1_{x \geq \omega} \right)^2 \d \mu_A(x).     
\end{aligned}
\end{equation*}
Now using the asymptotic of $\mathfrak{g}(t)$ and changing variables $x = \omega + u/t$, we get the asymptotics:
\begin{equation*}
\begin{aligned}
    I_1(t) &\underset{t \to \infty}{\sim} \frac{e^{2\omega t}}{t} \rho_A(\omega) \int_{-\infty}^{+\infty} \frac{e^{2u}}{(e^{2u} + 1)^2} \d u, \\
    I_2(t) &\underset{t \to \infty}{\sim} \frac{e^{2\omega t}}{t} \frac{\rho_A(\omega)}{\omega} \int_{-\infty}^{+\infty} \frac{e^{2u}}{(e^{2u} + 1)^2} \d u, \\
    I_3(t) &\underset{t \to \infty}{\sim} \frac{1}{t} \rho_A(\omega) \omega^2 \int_{-\infty}^{+\infty} \left( \frac{e^{2u}}{e^{2u} + 1} - \1_{u \geq 0} \right)^2 \d u. 
\end{aligned}
\end{equation*}
Therefore, combining these asymptotics with equation \eqref{eq:asymptoticIntegralDecomp}:
\begin{equation*}
    \lim_{d \to \infty} \frac{1}{d} \big\| Z(t) - Z_\infty \big\|_F^2 \underset{t \to \infty}{\sim} \frac{\omega^2 \rho_A(\omega)}{t} \left( \int_{-\infty}^{+\infty} \frac{e^{2u}}{(e^{2u} + 1)^2} \d u + \int_{-\infty}^{+\infty} \left( \frac{e^{2u}}{e^{2u} + 1} - \1_{u \geq 0} \right)^2 \d u \right). 
\end{equation*}
The result follows from the finiteness of the integrals. 

Since we plot in \cref{fig:CVOja} the convergence rates obtained in \cref{Prop:ConvergenceHighDimOja} with the exact constants, we give the values of the above integrals:
\begin{equation*}
\begin{aligned}
    \int_{-\infty}^{+\infty} u^2 \left( \frac{e^{2u}}{e^{2u} - 1} - \1_{u > 0} \right)^2 \d u &= \frac{\pi^2}{12} - \frac{1}{2} \zeta(3), \\
    \int_{-\infty}^{+\infty} \frac{e^{2u}}{(e^{2u} + 1)^2} \d u &= \frac{1}{2}, \\
     \int_{-\infty}^{+\infty} \left( \frac{e^{2u}}{e^{2u} + 1} - \1_{u \geq 0} \right)^2 \d u &= \log 2 - \frac{1}{2}. 
\end{aligned}
\end{equation*}

\proofsubsection{Prop:OjaResponse1} \label{App:Subsec:OjaResponse1}

In this section we compute the linear response associated with the perturbed Oja flow dynamics:
\begin{equation*}
    \dot W(t) = \Big( A - W(t)W(t)^\top \Big) W(t) + H(t) W(t),
\end{equation*}
where $A \in \mathcal{S}_d(\R)$ and $H(t) \in \mathcal{S}_d(\R)$ is the
perturbation. We are interested in computing:
\begin{equation*}
    R(t,t') = \left. \frac{\partial Z(t)}{\partial H(t')} \right|_{H = 0}, \hspace{1.5cm} Z(t) = W(t)W(t)^\top.
\end{equation*}
$R(t,t')$ is an operator $\mathcal{S}_d(\R) \to \mathcal{S}_d(\R)$ quantifying the change of $Z(t)$ in response to the introduction of $H$ at time $t'$. To compute it, we replace $H$ by $\epsilon H$ and decompose the solution $Z(t) = Z_0(t) + \epsilon Y(t)$ at first order in $\epsilon$. We then get the equations:
\begin{align}
    \dot Z_0(t) &= A Z_0(t) + Z_0(t) A - 2Z_0(t)^2, \label{eq:DynamicsResponseOrder0} \\
    \dot Y(t) &= Y(t) (A - 2 Z_0(t)) + (A - 2Z_0(t)) Y(t) + H(t) Z_0(t) + Z_0(t) H(t), \label{eq:DynamicsResponseOrder1}
\end{align}
with $Y(0) = 0$. Remark that the dynamics on $Y(t)$ is linear, but driven by a time-dependent matrix. This type of equation cannot be solved in general, but remarkably it is possible in our case. First of all, the dynamics \eqref{eq:DynamicsResponseOrder0} is an Oja flow, and following \cref{prop:OjaSolution}, we have the expression for $Z_0(t)$:
\begin{equation} \label{eq:solOja1}
    Z_0(t) = e^{tA} W_0 \left( I_m + 2 W_0^\top \int_0^t e^{2sA} \d s \, W_0 \right)^{-1} W_0^\top e^{tA},
\end{equation}
where $W_0W_0^\top = Z_0(t = 0)$. To solve the equation on $Y(t)$, we set $B(t) = A - 2Z_0(t)$. Then, it is easily seen that $B(t)$ is solution of the dynamics:
\begin{equation} \label{eq:ODEB}
    \dot B(t) = B(t)^2 - 4 A^2.
\end{equation}
Now, from equation \eqref{eq:DynamicsResponseOrder1} on $Y(t)$ and since $Y(0) = 0$, we have the solution:
\begin{equation} \label{eq:integral2}
    Y(t) = P(t)^{-1} \int_0^t P(s) \Big( H(s) Z_0(s) + Z_0(s) H(s) \Big) P(s)^\top \, \d s \, P(t)^{-\top}, 
\end{equation}
where $P(t)$ is such that $\dot P(t) = - P(t) B(t)$ with $P(0) = I_d$. Differentiating this equation together with the differential equation \eqref{eq:ODEB} on $B(t)$, we obtain:
\begin{equation*}
    \ddot P(t) = 4 P(t) A^2. 
\end{equation*}
We can integrate this equation with the initial conditions $P(0) = I_d$ and $\dot P(0) = 4Z_0 - 2A$. We obtain:
\begin{equation} \label{eq:solOja2}
    P(t) = e^{-tA} + Z_0 A^{-1} \big( e^{tA} - e^{-tA}\big). 
\end{equation}
This expression remains well defined when $A$ is not invertible if $s_t(\lambda) = \lambda^{-1} (e^{t\lambda} - e^{-t \lambda})$ is continuously extended to $\lambda = 0$ by setting $s_t(0) = 2t$. Now, in order to compute $R(t,t')$, we differentiate $Y$ with respect to $H$ using equation \eqref{eq:integral2}. We need to take into account that the derivative is taken on the space $\mathcal{S}_d(\R)$. To do so, we define, for $1 \leq i \leq j \leq d$:
\begin{equation} \label{eq:ElementarySym}
    E_{ij} = \frac{e_i e_j^\top + e_j e_i^\top}{2},
\end{equation}
where $(e_1, \dots, e_d)$ is the standard basis of $\R^d$. Then, setting:
\begin{equation*}
    U(t,t') = P(t)^{-1} P(t'), \hspace{1.5cm} V(t,t') = P(t)^{-1} P(t') Z_0(t'),
\end{equation*}
and evaluating equation \eqref{eq:integral2} along $E_{kl}$, we get for $t' \leq t$:
\begin{equation*}
\begin{aligned}
    R_{ijkl}(t,t') &\equiv \tr \left( \left. \frac{\partial Z(t)}{\partial H(t')} \right|_{H = 0} \big( E_{kl} \big) E_{ij} \right) \\
    &= \frac{1}{2} \Big(U_{ik}(t,t') V_{jl}(t,t') + U_{il}(t,t') V_{jk}(t,t') + U_{jk}(t,t') V_{il}(t,t') + U_{jl}(t,t') V_{ik}(t,t') \Big). 
\end{aligned}
\end{equation*}
To conclude, combining equations \eqref{eq:solOja1} and \eqref{eq:solOja2}, we get:
\begin{equation*}
    P(t) Z_0(t) = W_0 W_0^\top e^{tA},
\end{equation*}
and the result of \cref{Prop:OjaResponse1} follows. 

\proofsubsection{Prop:DiagonalResponse} \label{App:Subsec:OjaResponse2}

To prove \cref{Prop:DiagonalResponse}, we compute the functions:
\begin{equation*}
    R_\text{diag}(t,t') = \lim_{d \to \infty} \frac{1}{d^2} \tr \left( \left. \frac{\partial Z(t)}{\partial H(t')} \right|_{H = 0} \right), \hspace{1.5cm} r_\text{diag}(t) = \int_0^t R_\text{diag}(t,t') \d t'. 
\end{equation*}
To compute the trace, we use the orthogonal basis of $\mathcal{S}_d(\R)$ defined in equation \eqref{eq:ElementarySym}. It is easily seen that $\|E_{ij}\|_F^2 = (1 + \delta_{ij})/2$. Then, we have, with the notations of \cref{Prop:OjaResponse1}:
\begin{align}
    \frac{1}{d^2} \tr \left( \left. \frac{\partial Z(t)}{\partial H(t')} \right|_{H = 0} \right) &= \sum_{1 \leq i \leq j \leq d} \frac{1}{\|E_{ij}\|_F^2} \tr \left( \left. \frac{\partial Z(t)}{\partial H(t')} \right|_{H = 0} \big( E_{ij} \big) E_{ij} \right) \notag \\
    &= \frac{1}{d^2} \sum_{i,j=1}^d \tr \left( \left. \frac{\partial Z(t)}{\partial H(t')} \right|_{H = 0} \big( E_{ij} \big) E_{ij} \right) \notag \\
    &=\frac{1}{d^2} \tr \Big( U(t,t') \Big) \tr \Big( V(t,t') \Big) + \frac{1}{d^2} \tr \Big( V(t,t') U(t,t')^\top \Big). \label{eq:Rdiag1}
\end{align}
Now, $U(t,t'), V(t,t')$ are $d$-independent functions of the matrices $Z_0$, $A$. Since the scalings we chose guarantee that both $Z_0$ (as a Wishart matrix) and $A$ have convergent empirical spectral distributions in the high-dimension limit, all traces of scalar (and independent of $d$) functions of these matrices should be of order $d$. This implies that, as $d \to \infty$:
\begin{equation} \label{eq:Rdiag2}
    \frac{1}{d^2} \tr \Big( V(t,t') U(t,t')^\top \Big) = O \left( \frac{1}{d} \right). 
\end{equation}
Therefore, only the first term in equation \eqref{eq:Rdiag1} contributes in the high-dimensional limit. Now using the expressions of $U(t,t'), V(t,t')$, we have, denoting by $\lambda_1 \geq \dots \geq \lambda_d$ the eigenvalues of $A$, and working in a basis where $A$ is diagonal:
\begin{equation} \label{eq:TracesResponse}
\begin{aligned}
    \frac{1}{d} \tr \big( U(t,t') \big) &= \frac{1}{d} \sum_{i=1}^d \big( P(t)^{-1} Z_0 \big)_{ii} \frac{e^{\lambda_i t'} - e^{-\lambda_i t'}}{\lambda_i} + \frac{1}{d}\sum_{i=1}^d \big( P(t)^{-1} \big)_{ii} e^{-\lambda_i t'}, \\
    \frac{1}{d} \tr \big( V(t,t') \big) &= \frac{1}{d} \sum_{i=1}^d \big( P(t)^{-1} Z_0 \big)_{ii} e^{\lambda_i t'}. 
\end{aligned}
\end{equation}
We now compute the limit of these traces as $d \to \infty$ by replacing the averages by integrals with respect to $\mu_A$. The key point is to deal with the coefficients $\big( P(t)^{-1} Z_0 \big)_{ii}$, $\big( P(t)^{-1} \big)_{ii}$, which need to be expressed as functions of the eigenvalues of $A$. To do so, we define:
\begin{equation*}
    \Sigma(t) = A^{-1} \big( e^{2tA} - I_d \big) \succeq 0, \hspace{1.5cm} V(t) = \Sigma(t)^{1/2} W_0. 
\end{equation*}
Then, we can express:
\begin{equation*}
\begin{aligned}
    P(t)^{-1} &= e^{tA} \Big( I_d + W_0 W_0^\top \Sigma(t) \Big)^{-1} \\
    &= e^{tA} \Sigma(t)^{-1/2} \Big( I_d + V(t)V(t)^\top \Big)^{-1} \Sigma(t)^{1/2}.  
\end{aligned}
\end{equation*}
Therefore, since we are working in a basis where $A$ (and therefore $\Sigma(t)$) is diagonal, we have the expressions:
\begin{align}
    \big( P(t)^{-1} Z_0 \big)_{ii} &= \frac{\lambda_i e^{\lambda_i t}}{e^{2\lambda_i t} - 1} e_i^\top \Big( I_d + V(t)V(t)^\top \Big)^{-1} V(t) V(t)^\top e_i, \label{eq:AsymptoticCoeffs1} \\
    \big( P(t)^{-1} \big)_{ii} &= e^{\lambda_i t} e_i^\top\Big( I_d + V(t)V(t)^\top \Big)^{-1}e_i. \label{eq:AsymptoticCoeffs2}
\end{align}
Now, since:
\begin{equation*}
    \big( I_d + V(t) V(t)^\top \big)^{-1} V(t) V(t)^\top = I_d - \big(I_d + V(t)V(t)^\top \big)^{-1},
\end{equation*}
we can rewrite:
\begin{equation} \label{eq:AsymptoticCoeffs3}
    \big( P(t)^{-1} Z_0 \big)_{ii} = \frac{\lambda_i e^{\lambda_i t}}{e^{2\lambda_i t} - 1} \Big( 1 -  e_i^\top \big( I_d + V(t) V(t)^\top \big)^{-1} e_i \Big). 
\end{equation}
Therefore we only need to compute $e_i^\top \big( I_d + V(t) V(t)^\top \big)^{-1} e_i$. We now invoke the result of \citet[][Section 4]{bun2017cleaning}, that yields in our case:
\begin{equation} \label{eq:AsymptoticCoeffsStieltjes}
    e_i^\top \big( I_d + V(t) V(t)^\top \big)^{-1} e_i \longrightarrow \frac{1}{1 + \mu_i(t) \mathfrak{g}(t)}, \hspace{1.5cm} \mu_i(t) = \frac{e^{2\lambda_i t} - 1}{\lambda_i},
\end{equation}
where:
\begin{equation*}
    \mathfrak{g}(t) = \lim_{d \to \infty} \frac{1}{m} \tr \, \Big(I_m + V(t)^\top V(t)\Big)^{-1}, 
\end{equation*}
that solves the self-consistent equation:
\begin{equation*}
    \kappa \mathfrak{g}(t) + 1 - \kappa = \int \frac{x}{(e^{2xt} - 1) \mathfrak{g}(t) + x} \d \mu_A(x).
\end{equation*}
Therefore, in the high-dimensional limit, combining equations \eqref{eq:AsymptoticCoeffs2}, \eqref{eq:AsymptoticCoeffs3}, we have the expression:
\begin{equation} \label{eq:Rdiag3}
\begin{aligned}
    \big( P(t)^{-1} Z_0 \big)_{ii} &\xrightarrow[d \to \infty]{} \mathfrak{g}(t) \frac{\lambda_i e^{\lambda_i t}}{\mathfrak{g}(t)(e^{2\lambda_i t} - 1) + \lambda_i}, \\
    \big( P(t)^{-1} \big)_{ii} &\xrightarrow[d \to \infty]{} \frac{\lambda_i e^{\lambda_i t}}{\mathfrak{g}(t)(e^{2\lambda_i t} - 1) + \lambda_i}.
\end{aligned}
\end{equation}
Finally, using equations \eqref{eq:Rdiag1}, \eqref{eq:Rdiag2}, \eqref{eq:TracesResponse}, \eqref{eq:Rdiag3}, we get the expression of $R_\text{diag}$:
\begin{equation*}
\begin{aligned}
    R_\text{diag}(t,t') &= \mathfrak{g}(t) \iint \frac{1}{q_t(x) q_t(y)} y e^{y(t+t')} \Big( \big(x - \mathfrak{g}(t) \big) e^{x(t-t')} + \mathfrak{g}(t) e^{x(t+t')} \Big) \d \mu_A(x) \d \mu_A(y),
\end{aligned}
\end{equation*}
with $q_t(x) = \mathfrak{g}(t) (e^{2xt} - 1) + x$. Symmetrizing with respect to $x,y$ we get:
\begin{equation*}
\begin{aligned}
    R_\text{diag}(t,t') &= \frac{\mathfrak{g}(t)}{2} \iint \frac{1}{q_t(x) q_t(y)} e^{(x+y)t} \Big( y(x - \mathfrak{g}(t)) e^{(y-x)t'} + x(y - \mathfrak{g}(t)) e^{(x-y)t'} \\
    &\hspace{6cm} + (x+y) \mathfrak{g}(t) e^{(x+y)t'} \Big) \d \mu_A(x) \d \mu_A(y).
\end{aligned}
\end{equation*}
This is precisely equation \eqref{eq:Rdiagprop1} in \cref{Prop:DiagonalResponse}. Let us now compute $r_\text{diag}$:
\begin{equation*}
\begin{aligned}
    r_\text{diag}(t) &= \int_0^t R_\text{diag}(t,t') \d t' \\
    &= \frac{\mathfrak{g}(t)}{2} \iint \frac{1}{q_t(x) q_t(y)} e^{(x+y)t} \Bigg[ \frac{y(x - \mathfrak{g}(t))}{y-x} \Big( e^{(y-x)t} - 1 \Big) + \frac{x(y - \mathfrak{g}(t))}{x-y} \Big( e^{(x-y)t} - 1 \Big) \\
    &\hspace{6cm} + \mathfrak{g}(t) \Big( e^{(x+y)t} - 1 \Big) \Bigg] \d \mu_A(x) \d \mu_A(y). 
\end{aligned}
\end{equation*}
First considering the terms with no exponential in the bracket, we have:
\begin{equation*}
    \frac{y(x - \mathfrak{g}(t))}{y-x} +  \frac{x(y - \mathfrak{g}(t))}{x-y} + \mathfrak{g}(t) = 0. 
\end{equation*}
Now splitting the last term into two, we can write the bracket as:
\begin{equation*}
\begin{aligned}
    &\frac{y}{y-x} e^{yt} \Big( e^{xt} + (x - \mathfrak{g}(t)) e^{-xt} \Big) + \frac{x}{x-y} e^{xt} \Big( e^{yt} + (y - \mathfrak{g}(t)) e^{-yt} \Big) \\
    &=\frac{y}{y-x} e^{(y-x)t} q_t(x) + \frac{x}{x-y} e^{(x-y)t} q_t(y).  
\end{aligned}
\end{equation*}
Putting everything together, this leads to equation \eqref{eq:Rdiagprop2} and concludes the proof.  

\section{Numerical Simulations} \label{App:Simulations}
In this section, we give more details on our numerical simulations and provide additional figures that complement the results presented in the main text.

\subsection{Details on the Numerics}

We now detail our simulation setup, including the gradient descent algorithm and the numerical integration of the system of equations \eqref{eq:SystemResult2}.  

\subsubsection{Gradient Descent Simulations} \label{App:Subsubsec:SimulationsGD}

We refer to \cref{Subsec:optimization} for the details of the gradient descent simulations performed in this paper. All numerical simulations were implemented in \texttt{PyTorch} and executed on NVIDIA RTX6000, RTX8000, and H100 GPUs provided by the CLEPS infrastructure at INRIA Paris. 

The dimension $d$, stepsize $\eta$ and total number of gradient descent steps were selected after preliminary experiments. We found that the chosen dimension captures the high-dimensional regime considered in our theoretical analysis, while the stepsize is sufficiently small for discrete gradient descent to approximate the corresponding gradient flow dynamics. Simulations were run for a time horizon long enough to observe convergence of the gradient descent trajectories. 

Unless otherwise specified, all numerical results are averaged over several independent realizations of the random initialization, teacher matrix, and sensing matrices. Error bars in the figures represent two standard deviations across these realizations. 

\subsubsection{Simulating the System of Equations} \label{App:Subsubsec:SimulationSystem}

Throughout this work, we compared gradient descent simulations with the numerical integration of the system of equations \eqref{eq:SystemResult2}. We now explain how it is possible to simulate this set of equations. 

In this system of equations, the unknowns are the variables $\xi, q, \omega$, and the parameters are $\alpha, \kappa, \kappa^*, \lambda, \Delta$. To numerically integrate the equations, the key point is to consider $\xi$ itself as a parameter. Therefore $\omega$ can be directly computed by solving:
\begin{equation*}
    \min(\kappa, 1) = \int_\omega \d \mu_\xi(x). 
\end{equation*}
Then, we rather consider $\alpha$ as an unknown, and rewrite equation \eqref{eq:ResultLongTimesMSE}:
\begin{equation*}
    \alpha = \frac{1}{2\xi} \left( \frac{\Delta}{2} + Q_* + \int_{\max(q, \omega)} \big( q^2 - x^2 \big) \d \mu_\xi(x) + 4 \xi \int_{\max(q, \omega)} (x-q) h_\xi(x) \d \mu_\xi(x) \right). 
\end{equation*}
Then, this expression can be plugged in equation \eqref{eq:ResultLongTimesIntegral}, and leads to an equation solely on the variable $q$. This equation can then be solved using a numerical root solver (we used \texttt{scipy.brentq}). Then, the previous equation allows to deduce the value of $\alpha$. 

Still, the equation on $q$ involves integrals (of simple functions) with respect to the measure $\mu_\xi$, which is the free additive convolution between the teacher's spectral distribution and a semicircular density with variance $\xi$. For simplicity, we chose the teacher distribution to be the Marchenko--Pastur distribution with parameter $\kappa^*$. In this setting, the density of $\mu_\xi$ can be computed by solving a third-degree polynomial equation. We refer to \citet[][Appendix H.1]{maillard2024bayes} for details on this result. Then, the integration with respect to $\mu_\xi$ can be computed efficiently using \texttt{scipy.quad}. 

\begin{figure}[ht]
    \centering
    \includegraphics[width=\linewidth]{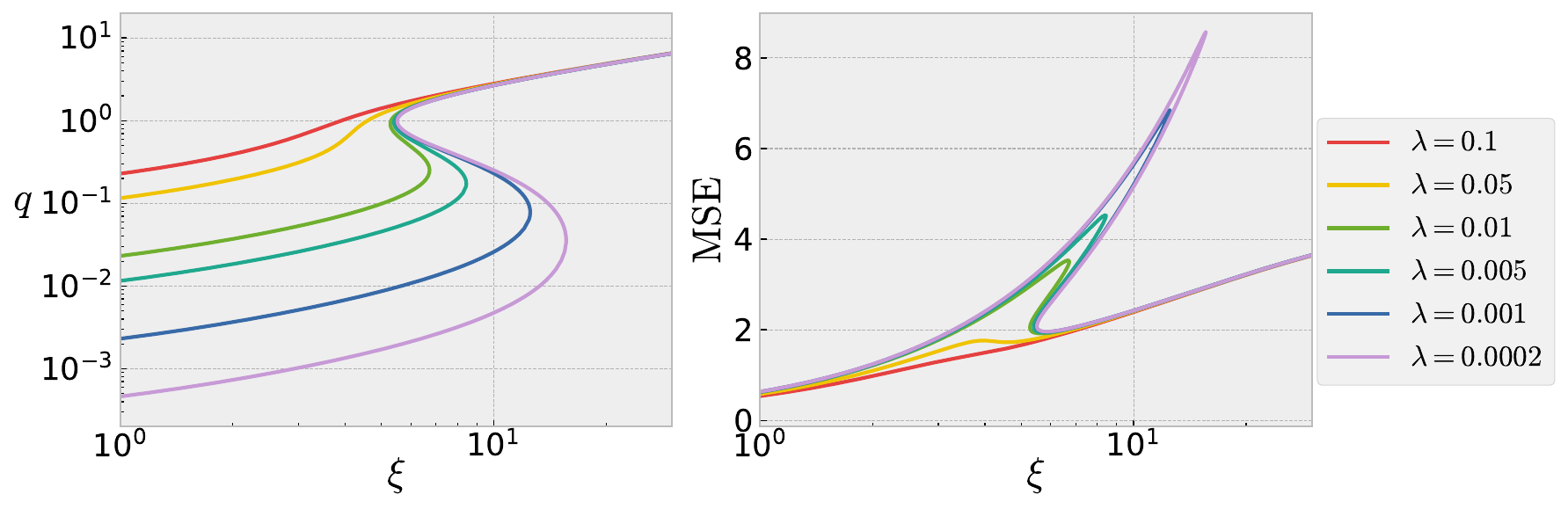}
    \vspace*{-0.6cm}
    \caption{Left: value of $q$ as a function of $\xi$ when simulating the system of equations~\eqref{eq:SystemResult2}. Right: corresponding value of the MSE, for different values of $\lambda$ and with $\kappa = 0.4, \kappa^* = 0.3$ and $\Delta = 1.0$. As $\lambda$ decreases, a single value of $\xi$ may lead to several solutions for $q$.}
    \label{fig:functionsxi}
\end{figure}

\cref{fig:functionsxi} displays $q$ and the MSE as a function of $\xi$ in the noisy setting $\Delta > 0$, for several choices of the regularization strength $\lambda$. We have observed that for most values of the parameters $\kappa, \kappa^*, \lambda, \Delta$, there was a unique solution $q$ for a given $\xi$. However, for small $\lambda$ and large $\Delta$, as underlined by \cref{fig:functionsxi}, several solutions $q$ may exist simultaneously. This makes it more difficult to integrate numerically the system of equations. To circumvent this problem, we used a technique known as arc-length continuation, that allows to compute the curve $q(\xi)$ in the 2D space and therefore capture the whole solution. Finally, we remark that this degeneracy in the solutions reveals the double descent phenomenon in our equations: the non-monotonicity of the MSE as a function of $\alpha$ is linked to the fact that a single value of $\xi$ may allow several solutions~$q$.

\newpage

\subsection{Additional Experiments and Figures} \label{App:Subsec:AdditionalExperiments}

In this section we present additional figures to support claims made in the main text. 

\subsubsection{Learning Curves} \label{App:Subsubsec:SimusLearningCurves}

\cref{Result2} provides expressions for the MSE and the training loss $\mathrm{Loss}_\mathrm{train}$ as a function of the parameters $\alpha, \kappa, \kappa^*, \lambda, \Delta$. These expressions are obtained from the high-dimensional equations of \cref{Result1} using non-rigorous arguments. The figures below provide more numerical evidence supporting the validity of our results. In \cref{fig:Curves1,fig:Curves2,fig:Curves3,fig:CurvesDD}, we plot the MSE, loss, and in-sample error (the latter two coincide in the noiseless case $\Delta = 0$) as a function of $\alpha$ for a wide range of parameter values, and compare gradient descent simulations with the numerical integration of the system of equations of \cref{Result2}. In addition, \cref{fig:Curves1,fig:CurvesDD} feature results for values of $\kappa$ close to $\kappa^* = 0.3$. As discussed in \cref{Subsubsec:Stability}, this corresponds to a region where the steady-state solution may be unstable. Together with the numerical evidence of \cref{fig:FigReg2}, these figures support our theoretical predictions in this regime. 

Moreover, \cref{fig:Curves1} highlights the dependence of the MSE and the training loss on $\kappa$. Increasing $\kappa$ leads to a lower empirical loss, reflecting the increased representational capacity of the model. However, this improvement leads to a higher MSE and a poorer generalization. This effect becomes more pronounced as the regularization strength $\lambda$ decreases. 

\begin{figure}[ht]
    \centering
    \includegraphics[width=\linewidth]{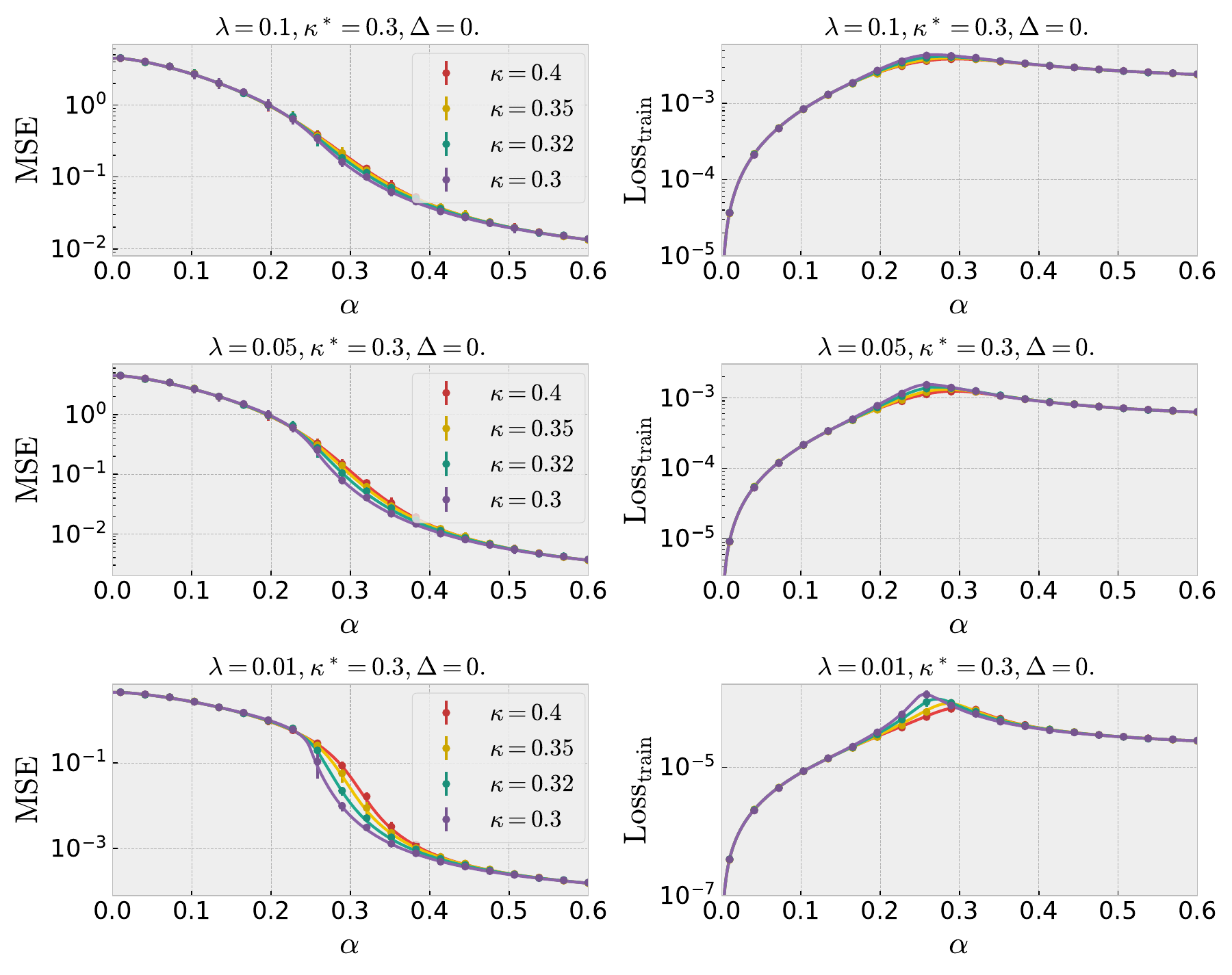}
    \vspace*{-0.7cm}
    \caption{Comparison between simulations of gradient descent, defined in \eqref{eq:GDdynamics}, and numerical integration of the system of equations~\eqref{eq:SystemResult2}, for $\kappa^* = 0.3, \Delta = 0$ and several values of the width $\kappa$ and regularization strength $\lambda$. MSE (left) and training loss (right) as a function of the sample complexity $\alpha$. Gradient descent simulations are averaged over 10 realizations of the initialization, teacher and data.}
    \label{fig:Curves1}
\end{figure}

\cref{fig:Curves2,fig:Curves3} display the same learning curves for several values of $\kappa, \kappa^*$, respectively in the noiseless ($\Delta = 0$) and noisy ($\Delta = 0.5$) settings. These figures not only confirm the strong agreement between our theory and gradient descent simulations, but also highlight the role of regularization on the performance of gradient flow. In the noiseless case (\cref{fig:Curves2}), reducing the regularization leads to simultaneous improvements in both training and generalization performance. This monotonic behavior holds in the range of regularization values considered here ($0.05 \leq \lambda \leq 0.5$) and suggests that in this regime, regularization mainly limits generalization. In the noisy setting (\cref{fig:Curves3}), the dependence on the regularization strength is less clear, suggesting the existence of an optimal level of regularization. We leave this question for future work.

\begin{figure}[ht]
    \centering
    \includegraphics[width=\linewidth]{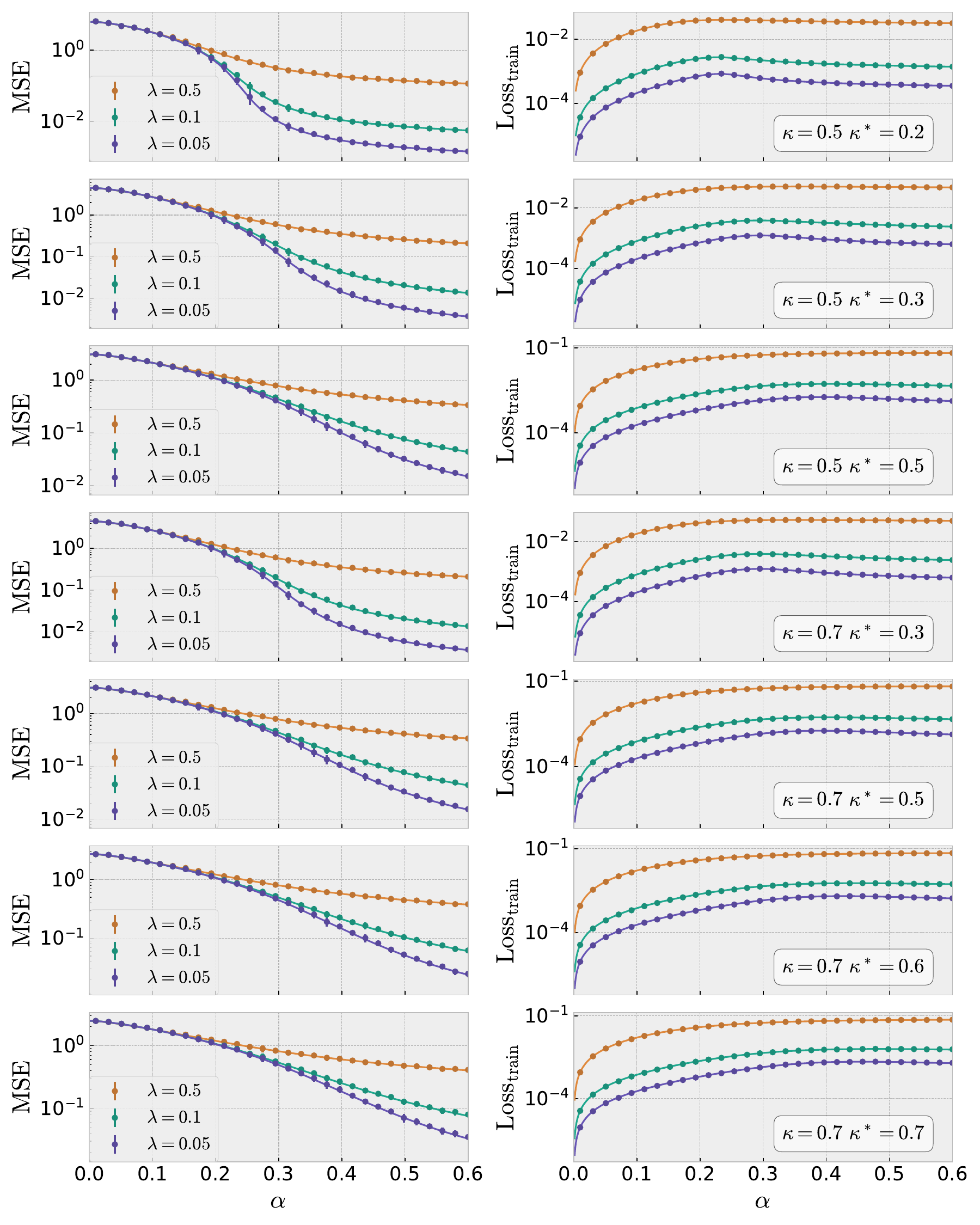}
    \vspace*{-0.3cm}
    \caption{Comparison between simulations of gradient descent, defined in \eqref{eq:GDdynamics}, and numerical integration of the system of equations~\eqref{eq:SystemResult2}, for $\Delta = 0$ and several values of the widths $\kappa, \kappa^*$ and regularization strength $\lambda$. MSE (left) and training loss (right) as a function of the sample complexity $\alpha$. Gradient descent simulations are averaged over 10 realizations of the initialization, teacher and data.}
    \label{fig:Curves2}
\end{figure}

\begin{figure}[ht]
    \centering
    \includegraphics[width=\linewidth]{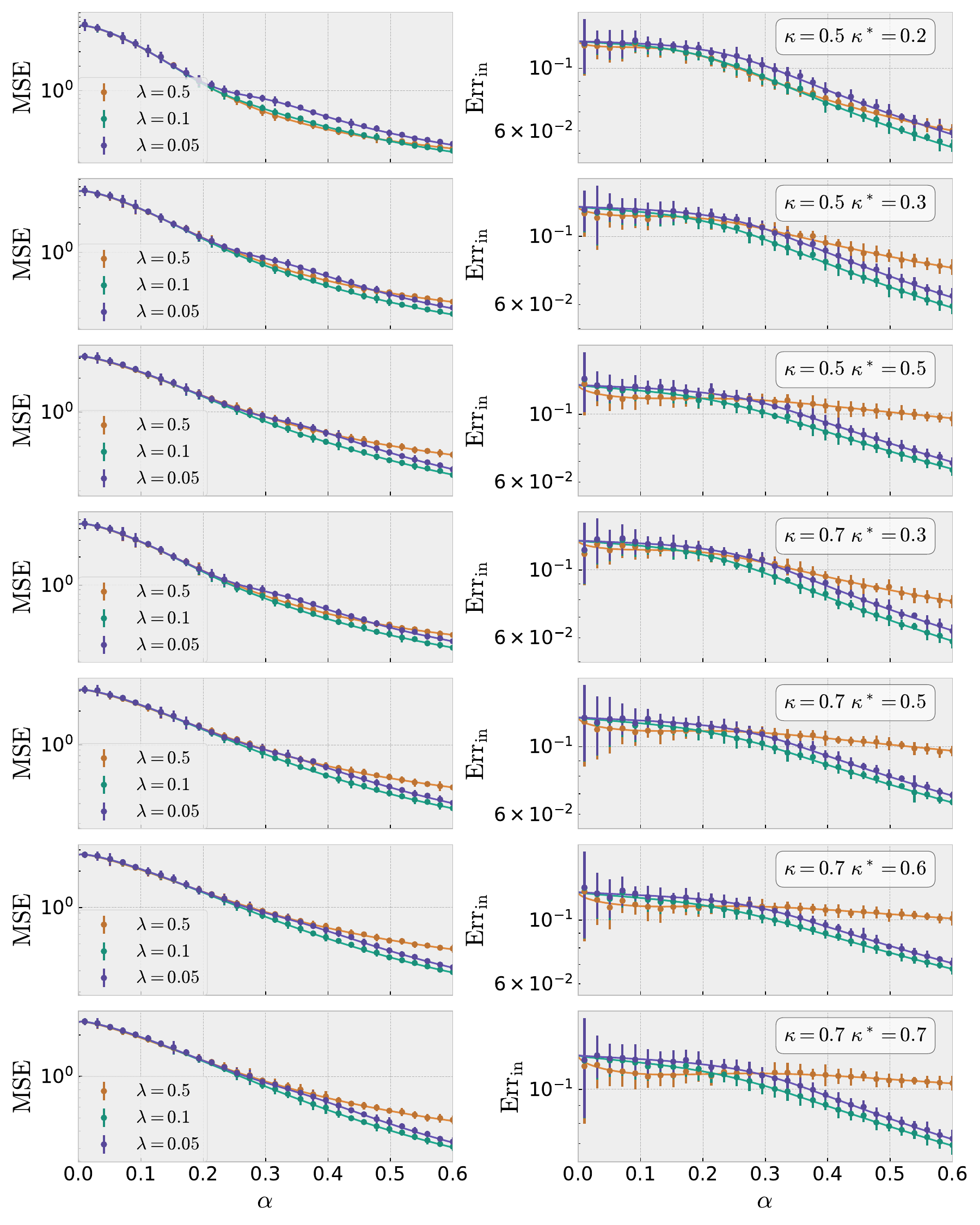}
    \vspace*{-0.3cm}
    \caption{Comparison between simulations of gradient descent, defined in \eqref{eq:GDdynamics}, and numerical integration of the system of equations~\eqref{eq:SystemResult2}, for $\Delta = 0.5$ and several values of the widths $\kappa, \kappa^*$ and regularization strength $\lambda$. MSE (left) and in-sample error (right), defined in equation \eqref{eq:insampleerror}, as a function of the sample complexity $\alpha$. Gradient descent simulations are averaged over 10 realizations of the initialization, teacher and data.}
    \label{fig:Curves3}
\end{figure}

\begin{figure}[ht]
    \centering
    \includegraphics[width=\linewidth]{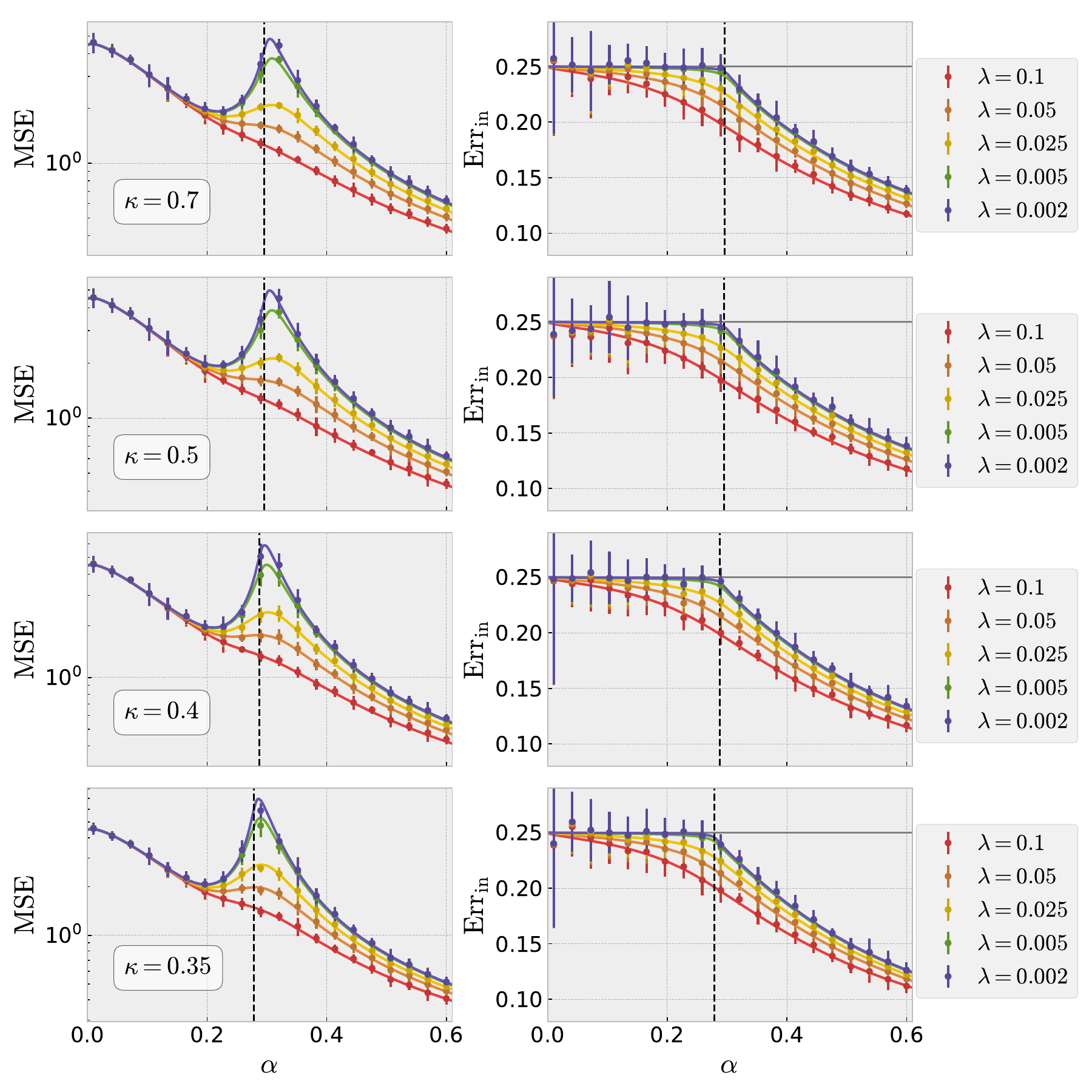}
    \vspace*{-0.3cm}
    \caption{Comparison between simulations of gradient descent, defined in \eqref{eq:GDdynamics}, and numerical integration of the system of equations~\eqref{eq:SystemResult2}, for $\kappa^* = 0.3, \Delta = 1.0$ and several values of the width $\kappa$ and regularization strength $\lambda$. MSE (left) and in-sample error (right), defined in equation \eqref{eq:insampleerror}, as a function of the sample complexity $\alpha$. Gradient descent simulations are averaged over 10 realizations of the initialization, teacher and data. The vertical dashed line indicates the interpolation threshold in the small regularization limit (see \cref{Prop:InterpolationThreshold}), and the horizontal gray line is the value $\Delta / 4$.}
    \label{fig:CurvesDD}
\end{figure}

Finally, \cref{fig:CurvesDD} shows the MSE and in-sample error (see equation \ref{eq:insampleerror}) in the double descent regime (introduced in \cref{Subsubsec:Overfitting}), corresponding to large values of $\Delta$ and small values of $\lambda$. In contrast with the previous cases, decreasing the regularization strength leads to larger values of both the MSE and the in-sample error. This indicates that under weak regularization, the gradient flow predictor fits the noise present in the training labels. Beyond the interpolation threshold, the in-sample error and MSE starts to decrease again, suggesting that the structure of the teacher is progressively learned. 

\FloatBarrier
\subsubsection{Perfect Recovery Threshold} \label{App:Subsubsec:SimusPRThreshold}

In this section, we give some numerical evidence regarding the perfect recovery threshold whose expression was conjectured in \cref{Conjecture:PRThreshold}. Because of the finite dimension and the finite time horizon at which our numerical simulations have been carried, the MSE always remains non-zero, but it is still possible to numerically investigate the perfect recovery threshold. 

\begin{figure}[H]
    \centering
    \includegraphics[width=\linewidth]{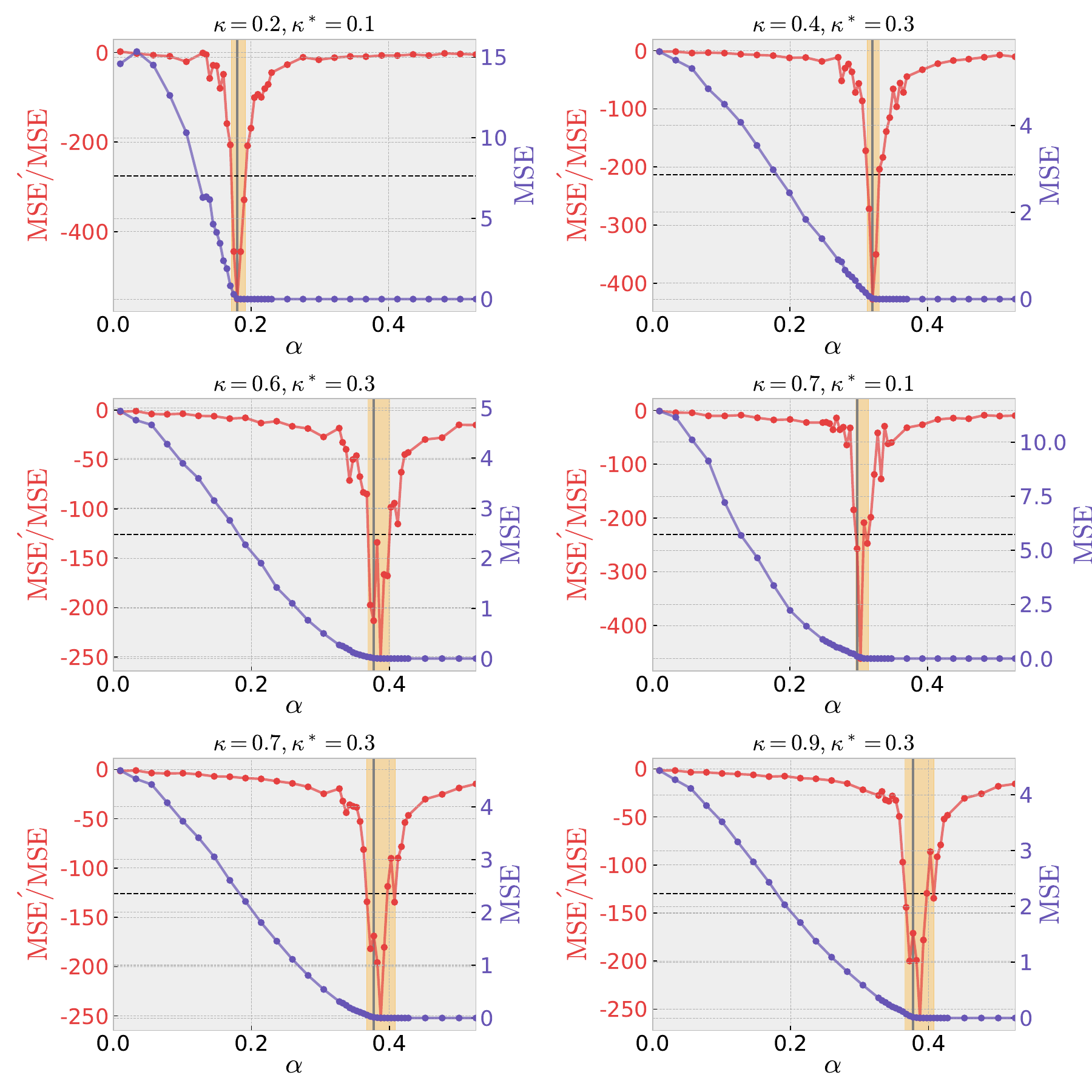}
    \vspace*{-0.5cm}
    \caption{Measure of the perfect recovery threshold from gradient descent simulations. Red: the logarithmic derivative $\mathrm{MSE}' / \mathrm{MSE}$ as a function of $\alpha$. This function exhibits a very sharp minimizer that coincides with the value of $\alpha$ where the MSE (purple curve) approaches zero. Horizontal dashed line: half depth of the minimizer, allowing to compute the uncertainty on the perfect recovery threshold (yellow region). Vertical gray line: conjectured value of the perfect recovery threshold in \cref{Conjecture:PRThreshold}.}
    \label{fig:PRMeasure}
\end{figure}

As $\alpha \to \alpha_\text{PR}$, the MSE goes continuously to zero, but non-smoothly. In the statistical physics vocabulary, this corresponds to a second-order phase transition. Such transitions are often associated with a scaling of the form:
\begin{equation} \label{eq:CriticalMSE}
    \mathrm{MSE}(\alpha) \underset{\alpha \to \alpha_\text{PR}}{\sim} \big( \alpha_\text{PR} - \alpha \big)^\theta, 
\end{equation}
for $\alpha < \alpha_\text{PR}$ and $\theta > 0$. Then:
\begin{equation*}
    \frac{\mathrm{MSE}'(\alpha)}{\mathrm{MSE}(\alpha)} \underset{\alpha \to \alpha_\text{PR}}{\sim} - \frac{\theta}{\alpha_\text{PR} - \alpha} \xrightarrow[\alpha \to \alpha_\text{PR}]{} - \infty. 
\end{equation*}
Therefore, the perfect recovery threshold can be located by identifying a point where the above quantity exhibits a sharp minimum. This is illustrated in \cref{fig:PRMeasure}. By computing the discrete derivative of $\log(\mathrm{MSE})$ with respect to $\alpha$, we identify this minimum and use it as a measure of the perfect recovery threshold. In addition, we quantify the uncertainty of this estimation by measuring the width of the minimum at half depth. 

This procedure, repeated over a large number of values for $(\kappa, \kappa^*)$, allows to compare the numerical value of the perfect recovery threshold with the one given in \cref{Conjecture:PRThreshold}. In \cref{fig:PRSummary}, we compare the predictions of \cref{Conjecture:PRThreshold} (continuous curves) with our numerical experiments (dots and error bars). This leads to a very convincing match and brings a numerical confirmation of our conjecture.

\begin{figure}[ht]
    \begin{minipage}{0.65\textwidth}
        \centering
        \includegraphics[width=\linewidth]{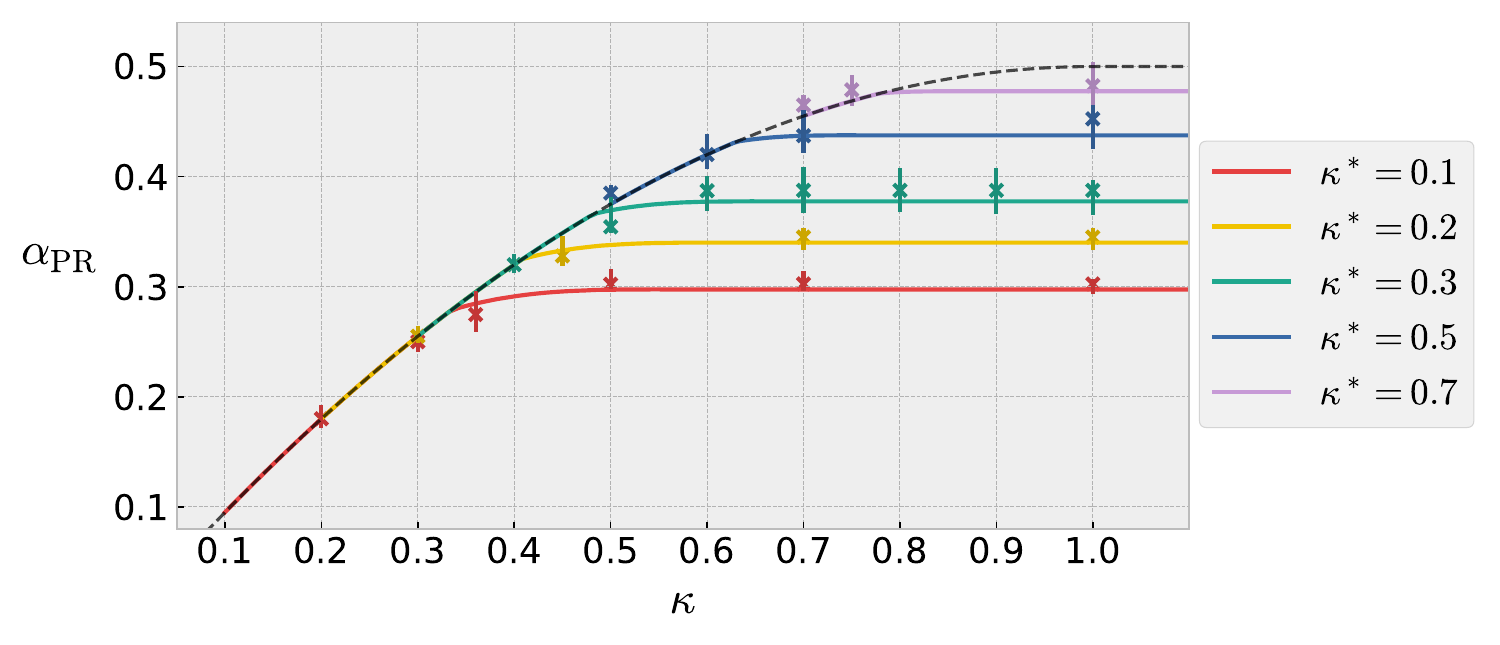}
    \end{minipage}
    \hfill
    \begin{minipage}{0.33\textwidth}
    \caption{Perfect recovery threshold $\alpha_\text{PR}$ as a function of $\kappa$ for different values of $\kappa^*$. Comparison between the prediction in equation \eqref{eq:PRThresholdUnreg} and the numerical estimation of $\alpha_\text{PR}$. The vertical bars indicate the uncertainty on the measure computed from \cref{fig:PRMeasure}.}
    \label{fig:PRSummary}
    \end{minipage}
\end{figure}

Additionally, the exponent $\theta$ in equation \eqref{eq:CriticalMSE} can be numerically estimated using standard linear regression techniques. Based on the data obtained from simulations, we observe that $\theta$ depends on the parameters $\kappa, \kappa^*$, but further work is required to formulate a precise conjecture and investigate this dependence. 

From a theoretical perspective, we expect that it is technically possible (although quite challenging) to derive an expression of this exponent. It would require working under the steady-state assumption, which under \cref{Conjecture:PRThreshold} should hold for large values of $\kappa$ and close enough to perfect recovery. In the end, such a calculation could be made possible by extending the results of \cref{App:Subsec:IntegralAsymptotics}.  

\end{document}